\documentclass[10pt]{amsart}
\usepackage{hyperref}
\usepackage[utf8]{inputenc}
\usepackage{amsmath}
\usepackage{amsthm}
\usepackage{amssymb}
\usepackage{tikz}
\usepackage{tikz-cd}
\usepackage{enumitem}
\usepackage{comment}
\usepackage{thmtools, thm-restate}
\usepackage{latexsym,amsmath,amsfonts,bbm,times,easyeqn, graphicx, amsthm, amssymb, physics, listings, framed, csvsimple, matlab-prettifier,placeins, multirow, booktabs, algorithm, algorithmicx, algpseudocode,lscape}
\usepackage[square,sort,comma,numbers]{natbib}
\usepackage{ulem}
\usepackage{chngcntr}

\DeclareMathOperator{\Spec}{Spec}

\DeclareMathOperator{\Hom}{Hom}
\DeclareMathOperator{\Gal}{Gal}

\DeclareMathOperator{\ch}{char}
\DeclareMathOperator{\id}{id}
\DeclareMathOperator{\GL}{GL}
\DeclareMathOperator{\GSp}{GSp}
\DeclareMathOperator{\ad}{ad}
\DeclareMathOperator{\Ad}{Ad}
\DeclareMathOperator{\Lie}{Lie}
\DeclareMathOperator{\Nil}{Nil}
\DeclareMathOperator{\res}{res}
\DeclareMathOperator{\Frob}{Frob}
\DeclareMathOperator{\Stab}{Stab}
\DeclareMathOperator{\St}{St}
\DeclareMathOperator{\Aut}{Aut}
\DeclareMathOperator{\Sp}{Sp}
\DeclareMathOperator{\End}{End}
\DeclareMathOperator{\diag}{diag}

\DeclareMathOperator{\SO}{SO}
\DeclareMathOperator{\SL}{SL}
\DeclareMathOperator{\gllie}{\mathfrak{gl}}
\DeclareMathOperator{\solie}{\mathfrak{so}}
\DeclareMathOperator{\RHom}{RHom}
\DeclareMathOperator{\LI}{\Lambda_I}
\DeclareMathOperator{\rec}{rec}
\DeclareMathOperator{\CNL}{CNL}
\DeclareMathOperator{\der}{der}
\DeclareMathOperator{\Iw}{Iw}
\DeclareMathOperator{\an}{an}
\DeclareMathOperator{\Kli}{Kli}
\DeclareMathOperator{\loc}{loc}

\DeclareMathOperator{\PGL}{PGL}
\DeclareMathOperator{\GT}{GT}
\DeclareMathOperator{\Ind}{Ind}

\DeclareMathOperator{\PGSp}{PGSp}
\DeclareMathOperator{\Spin}{Spin}
\newcommand\restr[2]{{
  \left.\kern-\nulldelimiterspace 
  #1 
  \vphantom{\big|} 
  \right|_{#2} 
  }}

\makeatletter
\def\@tocline#1#2#3#4#5#6#7{\relax
  \ifnum #1>\c@tocdepth 
  \else
    \par \addpenalty\@secpenalty\addvspace{#2}%
    \begingroup \hyphenpenalty\@M
    \@ifempty{#4}{%
      \@tempdima\csname r@tocindent\number#1\endcsname\relax
    }{%
      \@tempdima#4\relax
    }%
    \parindent\z@ \leftskip#3\relax \advance\leftskip\@tempdima\relax
    \rightskip\@pnumwidth plus4em \parfillskip-\@pnumwidth
    #5\leavevmode\hskip-\@tempdima
      \ifcase #1
       \or\or \hskip 1em \or \hskip 2em \else \hskip 3em \fi%
      #6\nobreak\relax
    \dotfill\hbox to\@pnumwidth{\@tocpagenum{#7}}\par
    \nobreak
    \endgroup
  \fi}
\makeatother
\usepackage{hyperref}

\newtheorem{lemma}{Lemma}[section]
\newtheorem{theorem}[lemma]{Theorem}
\newtheorem{definition}[lemma]{Definition}
\newtheorem{corollary}[lemma]{Corollary}
\newtheorem{proposition}[lemma]{Proposition}
\theoremstyle{remark}
\newtheorem{example}[lemma]{Example}
\newtheorem{remark}[lemma]{Remark}

\newtheorem{construction}[lemma]{Construction}

\newcommand{\gk}{\ensuremath{\hat{\mathfrak{g}}_k}}
\newcommand{\gA}{\ensuremath{\hat{\mathfrak{g}}_A}}
\newcommand{\Lzerobar}{\ensuremath{\overline{L_0}}}

\newcommand{\catname}[1]{\mathbf{#1}}

\newcommand{\gkv}{\ensuremath{\gk^{0,\vee} (1)}}
\newcommand{\GamKv}{\ensuremath{\Gamma_{K_v}}}

\newcommand{\Qlbar}{\ensuremath{\overline{\mathbb{Q}_l}}}
\newcommand{\Qpbar}{\ensuremath{\overline{\mathbb{Q}_p}}}

\newcommand{\Zpbar}{\ensuremath{\overline{\mathbb{Z}_p}}}
\newcommand{\OFv}{\ensuremath{\mathcal{O}_{F_v}}}
\newcommand{\piv}{\ensuremath{\varpi_v}}
\newcommand{\pone}{\ensuremath{{\mathfrak{p}_1}}}

\newcommand{\Fpbar}{\ensuremath{\overline{\mathbb{F}_p}}}
\newcommand{\splie}{\ensuremath{\mathfrak{sp}}}
\newcommand{\Gss}{\ensuremath{\Gamma^{\text{ss}}}}

\newcommand{\g}{\ensuremath{\hat{\mathfrak{g}}}}
\newcommand{\WD}{\ensuremath{\text{WD}}}
\newcommand{\Art}{\ensuremath{\text{Art}}}
\newcommand{\Fss}{\ensuremath{\textup{F-ss}}}
\newcommand{\CalO}{\ensuremath{\mathcal{O}}}
\newcommand{\CGhat}{\ensuremath{C_{\hat{G}}}}
\newcommand{\DvTW}{\ensuremath{\mathcal{D}_v^{\square,\textup{TW}}}}
\newcommand{\Sglobdp}{\ensuremath{\mathcal{S} = (\bar{\rho},S,\{\Lambda_v\}_{v \in S}, \psi, \{\mathcal{D}_v^\square\}_{v \in S})}}

\numberwithin{equation}{section}

\title{The Taylor--Wiles method for reductive groups}
\author{Dmitri Whitmore}

\begin{document}

\maketitle

\begin{abstract}
    We construct a local deformation problem for residual Galois representations $\bar{\rho}$ valued in an arbitrary reductive group $\hat{G}$ which we use to develop a variant of the Taylor--Wiles method.
    Our generalization allows Taylor--Wiles places for which the image of Frobenius is semisimple, a weakening of the regular semisimple constraint imposed previously in the literature.
    We introduce the notion of $\hat{G}$-adequate subgroup, our corresponding `big image' condition.
    When $\hat{G}$ is symplectic, special orthogonal, or simply connected and of exceptional type, and $\hat{G} \to \GL_n$ is a faithful irreducible representation of minimal dimension, we show that a subgroup is $\hat{G}$-adequate if it is $\GL_n$-irreducible and the residue characteristic is sufficiently large.
    
    We apply our ideas to the case $\hat{G} = \GSp_4$ and prove a modularity lifting theorem for abelian surfaces over a totally real field $F$ which holds under weaker hypotheses than in the work of Boxer--Calegari--Gee--Pilloni.
    We deduce some modularity results for elliptic curves over quadratic extensions of $F$.
\end{abstract}

\tableofcontents

\section{Introduction}

We introduce a new setup which provides a generalization of the Taylor--Wiles method for an arbitrary reductive group.
We apply our constructions to the context of proving modularity results for abelian surfaces over totally real fields, extending \cite{surfaces}.

\subsection{Main ideas and theorems}

Let $\hat{G}$ be a reductive group over $\mathbb{Z}$ and let $p$ a prime number not dividing the order of the Weyl group of $\hat{G}$.
Let $F$ be a number field with absolute Galois group $G_{F}$ and fix a continuous representation $\bar{\rho}: G_{F} \to \hat{G}(\mathbb{F}_p)$ which is absolutely $\hat{G}$-irreducible (see Definition \ref{Ghat irreducible}).
The Taylor--Wiles method is used to establish automorphy lifting theorems, providing instances of the global Langlands correspondence.
The goal of automorphy lifting is to take lifts $\rho_1,\rho_2: G_F \to~\hat{G}(\mathbb{Z}_p)$ of $\bar{\rho}$ for which $\rho_1$ is the Galois representation attached to an automorphic representation and show that the same is true of $\rho_2$.

In the Taylor--Wiles method one considers lifts of $\bar{\rho}$ for which ramification is allowed at certain auxiliary sets of places.
In this paper we consider finite places $v$ of $F$ for which
\begin{itemize}
    \item $\bar{\rho}$ is unramified at $v$
    \item the residue field $k(v)$ satisfies $\# k(v) \equiv 1 \mod p$
    \item the image of the Frobenius element $\bar{g} = \bar{\rho}(\Frob_v)$ is semisimple.
\end{itemize}
We call such a place $v$ a Taylor--Wiles place.
This is a generalization of previous notions for which $\bar{g}$ was always taken to be regular semisimple, such as in \cite{thorne2019}.

The main difficulty arising from our above notion is that the Taylor--Wiles method requires us to consider deformations of $\bar{\rho}$ for which the inertia subgroup $I_v$ has image in a torus.
While this is automatic in the regular semisimple case, we prove the following theorem to provide us with a `canonical' torus.
\begin{theorem} \label{Mg imprecise}
Let $A$ be a complete Noetherian local $\mathbb{Z}_p$-algebra of residue field $\mathbb{F}_p$ and let $g \in \hat{G}(A)$ be a lift of $\bar{g}$.
Then there is a canonical lift of the centralizer $M_{\bar{g}} = Z_{\hat{G}_{\mathbb{F}_p}}(\bar{g})$ to a closed reductive subgroup scheme $M_g$ of $\hat{G}_A$.
\end{theorem}
A more precise statement is given in Theorem \ref{Mg}.
Let $g$ be the image of a Frobenius lift at $v$ under a lift
$\rho: G_F \to \hat{G}(A)$
of $\bar{\rho}$.
We then define a local deformation problem by admitting all lifts $\rho$ for which $I_v$ is valued in the torus given by the connected center of $M_g$.
Since our deformation problem is in general not equal to the unrestricted deformation problem, some kind of local-global compatibility would be required to show that Galois representations (attached to those automorphic representations arising in the Taylor--Wiles method) satisfy our local deformation condition upon restriction to the decomposition group at a Taylor--Wiles prime.

In order to control the arising universal deformation rings, one requires the existence sets of Taylor--Wiles places for which a certain Galois cohomology group (the dual Selmer group) is forced to vanish.
Consequently, one typically imposes `big image' hypotheses on the representation $\bar{\rho}$ (e.g. the $\hat{G}$-abundant image condition of \cite{thorne2019} or vast representations of \cite{surfaces}) to guarantee the existence of such sets of Taylor--Wiles places.
In Proposition \ref{TW places existence} we show that the dual Selmer group can be killed off under much weaker assumptions on $\bar{\rho}$ if one allows our notion of Taylor--Wiles prime.
This motivates the following definition, which is our milder `big image' condition.
\begin{definition} \label{adequate imprecise}
We say $H \leq \hat{G}(\mathbb{F}_p)$ is $\hat{G}$-adequate if the following conditions hold:
\begin{enumerate}
    \item The following groups vanish
    \begin{enumerate}
        \item $H^0(H,\hat{\mathfrak{g}}^{0,\vee})$
        \item $H^1(H,\mathbb{F}_p)$
        \item $H^1(H,\hat{\mathfrak{g}}^{0,\vee})$
    \end{enumerate}
    \item For every non-zero simple $\mathbb{F}_p[H]$-submodule $W \leq \hat{\mathfrak{g}}^{0,\vee}$ there exists a semisimple element $h \in H$ such that for some $w \in W$ and $z \in \Lie Z(Z_{\hat{G}}(h)) \cap \g^0$ we have $w(z) \neq 0$. \label{submodule imprecise}
\end{enumerate}
\end{definition}
Here $\hat{\mathfrak{g}}^0$ denotes the Lie algebra of the derived group of $\hat{G}$ (over $\mathbb{F}_p$) equipped with the adjoint action of $\hat{G}(\mathbb{F}_p)$, $\hat{\mathfrak{g}}^{0,\vee}$ denotes the dual module and $\mathfrak{z}$ denotes the center of a Lie algebra.

Our notion of Taylor--Wiles place and $\hat{G}$-adequate subgroup are by analogy with \cite{thorne2012}.
There unramified places $v$ are considered for which the $\bar{\rho}(\Frob_v)$ acts as $\bar{\psi} \oplus \bar{s}$, where $\bar{\psi}$ is an eigenspace which is acted on semisimply.
A local deformation problem is then defined by taking lifts of the form $\psi \oplus s$, where the action on $\psi$ is through a scalar. 
This itself was a generalization of \cite{clozel2008automorphy}, where $\bar{\psi}$ was required to be $1$-dimensional.
Thus our generalization of allowing $\bar{\rho}(\Frob_v)$ to be any semisimple element, rather than regular semisimple, is analogous to allowing repeated eigenvalues in the case of $\GL_n$.

`Irreducibility implies adequacy' results have already been established in the case of $\hat{G} = \GL_n$ in \cite{thoappendix, guralnick2015adequate} outside of a set of primes depending on $n$.
We use this to deduce an analogous theorem for absolutely irreducible subgroups of some simple linear algebraic groups.
\begin{theorem} \label{intro gln irred implies adequate}
Let $\hat{G}$ be one of the following simple linear algebraic groups over $\Fpbar$:
\begin{itemize}
    \item the special orthogonal group $\SO_{n}$ for $n \geq 3$
    \item the symplectic group $\Sp_{n}$ for $n \geq 2$ even
    \item the simply connected exceptional groups $G_2,F_4,E_6,E_7,E_8$.
\end{itemize}
Let $\hat{G} \hookrightarrow \GL_n$ be a faithful irreducible representation of minimal dimension.
If $p \geq n+4$, $p \neq 2n \pm 1$ and $\Gamma \subset \hat{G}(\Fpbar)$ is a finite, $\GL_n$-irreducible subgroup, then $\Gamma$ is $\hat{G}$-adequate.

\end{theorem}

We show in Corollary \ref{classical gln irred equiv adequate} that the converse to Theorem \ref{intro gln irred implies adequate} holds when $\hat{G}$ is either $\Sp_{2n}$ or $\SO_{2n+1}$.
We also give examples of $\hat{G}$-adequate subgroups which are not $\GL_n$-irreducible when $\hat{G}$ is either $\SO_{n}$ for $n$ even or when $\hat{G}$ is a simply connected exceptional group.

Our methods are then applied with the goal of proving modularity results in the case of $\hat{G} = \GSp_4$ and $F$ totally real.
We are able to prove modularity theorems in which the assumption of \cite{surfaces} that $\bar{\rho}$ is `vast' is replaced by the assumption that $\bar{\rho}$ is $\GSp_4$-reasonable (see Definition \ref{reasonable}); this holds when $\bar{\rho}(G_{F(\zeta_p)})$ is $\GSp_4$-adequate.
This offers a significant upgrade, since condition \ref{submodule imprecise} in Definition \ref{adequate imprecise} always holds for irreducible subgroups of $\Sp_4(\mathbb{F}_p)$.
In contrast, the corresponding condition with `semisimple' replaced by `regular simple' fails for approximately half of the conjugacy classes of irreducible subgroups of $\Sp_4(\mathbb{F}_p)$ (for some small choices of $p$).
Many of our preparatory lemmas (such as the study of various Hecke algebras) are proven for an arbitrary reductive group and are therefore applicable in other settings.
Here are the modularity lifting theorems we prove.

\begin{theorem} \label{intro modular galois reps}
Let $F$ be a totally real field in which $p \geq 3$ splits completely.
Suppose that $\rho:G_F \to \GSp_4(\Qpbar)$ is a continuous representation satisfying:
\begin{enumerate}
    \item $\nu \circ \rho = \varepsilon^{-1}$
    \item $\bar{\rho}$ is $\hat{G}$-reasonable and tidy in the senses of Definition \ref{reasonable} and Definition \ref{defn: tidy}
    \item For every $v | p$, $\restr{\rho}{G_{F_v}}$ is $p$-distinguished weight $2$ ordinary in the sense of Definition \ref{defn: p-distinguished ordinary}
    \item There exists an automorphic representation $\pi$ of $\GSp_4(\mathbb{A}_F)$ of parallel weight 2 and central character $|\cdot|^2$, ordinary in the sense of \cite[Definition 2.4.25]{surfaces}, such that $\overline{\rho_{\pi,p}} \cong \bar{\rho}$, where $\rho_{\pi,p}$ is as in Corollary \ref{parallel weight 2 galois rep}.
    \item For every finite place $v$ of $F$, the representations $\restr{\rho}{G_{F_v}}$ and $\restr{\rho_{\pi,p}}{G_{F_v}}$ are pure.
\end{enumerate}
Then $\rho$ is modular and for every finite place $v$ of $F$ we have full local-global compatibility.
\end{theorem}

Here $\nu: \GSp_4 \to \mathbb{G}_m$ is the similitude character and $\varepsilon$ is the $p$-adic cyclotomic character.
A more precise statement of Theorem \ref{intro modular galois reps} is given in Theorem \ref{modular galois reps}

\begin{theorem} \label{intro modular abelian surfaces}
Let $F$ be a totally real field in which $p \geq 3$ splits completely.
Let $A/F$ be an abelian surface such that
\begin{enumerate}
    \item $A$ has good ordinary reduction at every $v | p$
    \item for each $v|p$, the unit root crystalline eigenvalues are distinct modulo $p$
    \item $A$ admits a polarization of degree prime to $p$.
\end{enumerate}
Suppose that the residual Galois representation $\overline{\rho_{A,p}}$ is such that
\begin{enumerate} [resume]
    \item $\overline{\rho_{A,p}}$ is $\hat{G}$-reasonable and tidy
    \item $\overline{\rho_{A,p}}$ is ordinarily modular: there exists $\pi$ an automorphic representation of $\GSp_4(\mathbb{A}_F)$ for which
    \begin{enumerate}
        \item $\pi$ has parallel weight $2$ and central character $|\cdot|^2$
        \item for every $v|p$, $\pi_v$ is unramified and ordinary
        \item $\overline{\rho_{A,p}} \cong \overline{\rho_{\pi,p}}$
        \item for every finite place $v$ of $F$, $\rho_{\pi,p}$ is pure.
    \end{enumerate}
\end{enumerate}
Then $A$ is modular.
\end{theorem}

Here the Galois representation $\rho_{A,p}$ is given by the $p$-adic cohomology group $H^1(A_{\bar{F}},\Qpbar)$.
Now let $K/F$ be an arbitrary quadratic extension.
We apply Theorem \ref{intro modular galois reps} to $p$-adic Galois representations of abelian surfaces arising as the restriction of scalars of an elliptic curve over $K$ and deduce a modularity for lifting theorem for elliptic curves $E/K$.

\begin{theorem} \label{intro modularity lifting elliptic curves}
Let $F$ be a totally real field in which $p \geq 3$ splits completely and let $K/F$ be a quadratic extension in which $p$ is unramified with Galois group $\{1,\sigma\}$. 
Let $E/K$ be an elliptic curve and let $\bar{\varrho} = \overline{\varrho_{E,p}}$ be the attached mod $p$ Galois representation of determinant $\varepsilon^{-1}$.
Suppose that the following conditions hold:
\begin{enumerate}
    \item $E$ has ordinary good reduction or multiplicative reduction at every place $w | p$
    \item $\bar{\varrho}(G_{K(\zeta_p)})$ is absolutely irreducible
    \item $\bar{\varrho}$ is ordinarily modular
    \item if $p = 3$ or $p=5$ then order of the subgroup of $\GL_2(\mathbb{F}_p) \times \GL_2(\mathbb{F}_p)$ generated by the image of $\bar{\varrho} \oplus \bar{\varrho}^\sigma$ together with the matrices $\diag(I,-I), \diag(-I,I)$ is not equal to $64$.
\end{enumerate}
Then $E$ is modular.
\end{theorem}

Theorem \ref{intro modularity lifting elliptic curves} often applies even when the image $\overline{\varrho_{E,p}}(G_{K(\zeta_p)})$ is small, such as when it is projectively dihedral.
In some such cases we are able to establish ordinary modularity of $\overline{\varrho_{E,p}}$ via automorphic induction of a character and hence deduce modularity of $E$ (see, for example, Corollary \ref{cm comp ind curves modular} when $p \geq 7$).
Here is a precise statement we prove when $p=5$ which provides infinitely many modular elliptic curves over $K$ without complex multiplication nor arising from base-change from a proper subfield.

\begin{corollary} \label{intro p small inf modular curves}
Let $F$ be a totally real field in which $5$ splits completely and let $K/F$ be a quadratic extension in which $5$ is unramified.
There exists representations $\bar{\varrho}:~G_K \to \GL_2(\mathbb{F}_5)$ which are CM-compositum induced from a character $\bar{\psi}: G_L \to \mathbb{F}_5^\times$ (see Definition \ref{def CM-compositum induced}), where $L/K$ is a quadratic extension, and $\bar{\psi}$ satisfies $\restr{(\bar{\psi}^2)^\tau}{G_{L(\zeta_p)}} \neq \restr{\bar{\psi}^2}{G_{L(\zeta_p)}}$ for $\tau \in G_F \setminus G_K$.
Fix such a representation $\bar{\varrho}$ and suppose that $E/K$ is an elliptic curve for which $\overline{\varrho_{E,5}} \cong \bar{\varrho}$ and that $E$ has ordinary good reduction or multiplicative reduction at every $w|5$.
Then $E$ is modular.
Moreover, the subset of $X_{\bar{\varrho}}(K)$ given by such elliptic curves which in addition do not have complex multiplication nor have $j$-invariant contained in a proper subfield of $K$ is of positive density (in the sense of Definition \ref{positive density}).
\end{corollary}

Moreover, by the second part of Lemma \ref{tidy CM comp ind}, the representations $\bar{\varrho}$ of Corollary \ref{intro p small inf modular curves} cannot be handled by the results of \cite{surfaces}, as the (twisted) inductions of $\bar{\varrho}$ to $G_F$ are never vast, illustrating the power of our methods.

\begin{remark}
As in \cite[1.4.1]{surfaces}, the main modularity results of Sections \ref{sec surfaces} and \ref{sec modular elliptic curves} (including Theorems \ref{intro modular galois reps}, \ref{intro modular abelian surfaces}, \ref{intro modularity lifting elliptic curves} and Corollary \ref{intro p small inf modular curves}) are conditional on the multiplicity formula for the discrete spectrum of $\GSp_4$, stated in \cite{arthur2004automorphic} (with the proofs yet to appear).
However, due to the work of many other authors, this result has been established subject only to proving the twisted weighted fundamental lemma (see \cite[1.6]{boxer2025modularity} for a more detailed summary).
\end{remark}

\subsection{Structure of the paper}

In Section \ref{def setup}, we prove in Theorem \ref{Mg} a precise version of Theorem \ref{Mg imprecise} and show that our `canonical' torus lift given by the connected center of $M_g$ satisfies many of the desired functoriality properties.
We also give an alternate construction in the case of lifting to the ring of integers of a finite extension of $\mathbb{Q}_p$, which is used in Section \ref{sect local global}.

Then in Section \ref{section deformation} we define our Taylor-Wiles local deformation problem by requiring that the inertia subgroup $I_v$ has image in the torus constructed in Section \ref{def setup}.
We then study the Galois cohomological interpretation of our local deformation problem and give a definition of $\hat{G}$-adequate subgroups, which is our weakened `big image' condition.

In Section \ref{sec adequate subgroup} we compare our notion of $\hat{G}$-adequate subgroups to absolute irreducibility.
From the corresponding result for $\GL_n$, we show an `irreducibility implies adequacy' type result for some simple groups.
We use computer algebra software \texttt{magma} to study the case of $\Sp_4$ in small characteristic for our applications in Sections \ref{sec surfaces} and \ref{sec modular elliptic curves}, and give a comparison between $\Sp_4$-adequacy and a `big image' condition seen previously in the literature using the tables given in Appendix \ref{sect appendix}.

Let $G$ be the reductive group dual to $\hat{G}$ and let $v$ be a Taylor--Wiles place.
In Section \ref{sec local computations}, we explore the representation theory of $G(F_v)$ in relation to our constructions of Section \ref{section deformation}.
We explain how our semisimple image of Frobenius gives rise to certain compact open groups $\mathfrak{p}_1$ and $\mathfrak{p}$ of $G(F_v)$, summarized by Construction \ref{parahoric construction}.
We prove Theorem \ref{satake}, a generalization of the Satake isomorphism, which we use to define certain abelian subalgebras of the Hecke algebras arising from our compact open subgroups.
These abelian subalgebras have maximal ideals corresponding to $\bar{g}$, and we study the localizations of modules of our Hecke algebras at these maximal ideals.
The results obtained are in full generality, but with the precise statements suited for our applications in the $\hat{G} = \GSp_4$ case.
To conclude the section, we restrict to the $\GSp_4$ case, where there is the local Langlands correspondence of \cite{gsp4llc} available.
We show in Proposition \ref{gsp4 p invariants} and Proposition \ref{gsp4 p1 invariants} that if our localized spaces of invariants of irreducible admissible representations of $\GSp_4(F_v)$ are non-zero, then their image under the local Langlands map is a semisimple Weil--Deligne representation.

The main goal of Section \ref{sect local global} is to prove Proposition \ref{local global deltav}.
In rough terms, this states that if we have a sufficiently strong local-global compatibility assumption between an automorphic representation $\pi$ and an associated Galois representation at the Taylor--Wiles places then the Galois representation will satisfy our deformation condition.
Moreover, we show the necessary compatibility between the action of inertia and the action of certain Hecke operators holds.
Our proof uses the alternate construction of $M_g$ of Section \ref{alternate Mg} to understand the passage to Frobenius-semisimplification of Weil--Deligne representations in the context of our deformation problem.

In Section \ref{sec surfaces}, we apply the theory developed in the earlier sections to the case of $\hat{G} = \GSp_4$ in order to extend the partial results on modularity of abelian surfaces over a totally real field $F$ given in \cite{surfaces}.
Using the methods of \cite{surfaces} as a framework, we explain a setup compatible with our generalized notion of Taylor--Wiles primes.
The majority of the work is done in Section \ref{sec tw prep}.
Here we prove analogues of many of the results of \cite[7]{surfaces} necessary for the Taylor--Wiles patching argument.
In particular, we explain the moduli interpretation of working with level structures given by the compact open subgroups of Section \ref{sec local computations} at Taylor--Wiles primes.
We use the main result of Section \ref{sect local global} to prove the existence of a Galois representation valued in a certain Hecke algebra satisfying our deformation condition.
The patching construction then goes through from our work in Sections \ref{section deformation} and \ref{sec local computations}.
In Section \ref{sect modularity results} we deduce our modularity lifting theorems which are analogous to those of \cite{surfaces} but hold under our weaker hypotheses on the images of the residual Galois representations.

Finally, in Section \ref{sec modular elliptic curves} we study representations of $G_F$ induced from $\GL_2$-valued representations of the absolute Galois group of quadratic extensions $K/F$.
From the main modularity lifting theorem of Section \ref{sect modularity results} we deduce modularity lifting theorems for elliptic curves over $K$.
We define the notion of CM-compositum induced representations (Definition \ref{def CM-compositum induced}), which are certain projectively dihedral residual Galois representations.
We give a construction of these representations for every $p \geq 5$ and prove modularity theorems for some elliptic curves with CM-compositum induced mod $p$ Galois representation.

\subsection{Acknowledgements}

The author is exceptionally grateful to Jack Thorne for inspiration of the topic of this paper, numerous insightful conversations and for helpful comments on earlier drafts.
The author thanks the anonymous referee for helpful feedback.

\section{Setup for a local deformation problem}\label{def setup}

Let $\hat{G}$ be a reductive group defined over $\mathbb{Z}$ (that is, a smooth affine group scheme $\hat{G} \to \Spec \mathbb{Z}$ such that every geometric fibre $\hat{G}_{\Bar{s}}$ is a connected reductive group).
Let $p$ be a prime number, which we will later assume is pretty good for $\hat{G}$ (see Definition \ref{pretty_good}).
Let $\mathcal{O}$ be the ring of integers of a finite extension of $E/\mathbb{Q}_p$ with uniformizer $\lambda$ and residue field $k = \mathcal{O}/\lambda \mathcal{O}$, a finite field of characteristic $p$. 
Denote by $\CNL_\CalO$ the category whose objects are complete Noetherian local $\mathcal{O}$-algebras together with a fixed isomorphism from the residue field to $k$, so that we may freely identify them without confusion. 

Motivated by wanting to impose constraints on representations valued in $\hat{G}(A)$ for $A \in \CNL_\CalO$, we will consider the following setup. Fix $A \in \CNL_\CalO$ with maximal ideal $m_A$ and let $\bar{g} \in \hat{G}(k)$ be a semisimple element with centralizer $M_{\bar{g}} = Z_{\hat{G}_k}(\bar{g})$, a possibly disconnected smooth closed $k$-subgroup with reductive identity component (\cite[Theorem 1.1.19]{conrad}).
Denote by $\hat{\mathfrak{g}}_k$ the Lie algebra of $\hat{G}_k$, and consider the adjoint action of $\Bar{g}$, a Lie algebra automorphism $\Ad(\Bar{g}): \hat{\mathfrak{g}}_k \to \hat{\mathfrak{g}}_k$. 
Let $\overline{L_0} = \{v \in \gk: \Ad(\Bar{g})(v) = v\}$ denote the eigenvalue-1 subspace. 
Let $\overline{L_1}$ denote the $\Ad(\Bar{g})$-invariant complement in $\gk$, which will be the direct sum of all other eigenspaces over an algebraic closure. 
Thus we have an $\Ad(\Bar{g})$-invariant decomposition $\gk = \overline{L_0} \oplus \overline{L_1}$.
We have the equality of Lie algebras $\Lie M_{\bar{g}} = \Lzerobar$ by \cite[9.1]{borelLAG}.

Now let $g \in \hat{G}(A)$ be a fixed choice of lift of $\bar{g}$ under the natural map $\hat{G}(A) \to \hat{G}(k)$.
Our main goal of this section is to show Theorem \ref{Mg}, which states that the Lie algebra $\gA$ admits a unique $\Ad(g)$-invariant lift $\gA = L_0 \oplus L_1$ of the decomposition $\gk = \overline{L_0} \oplus \overline{L_1}$ and that there exists a unique closed subgroup $M_g \subset \hat{G}_A$ defined over $A$ with reductive identity component and with the properties that the base change $(M_g)_k$ is naturally identified with $M_{\bar{g}}$, $\Lie M_g = L_0$ and that $g \in M_g(A)$.

Throughout this section we will allow ourselves to replace $k$ (and $\mathcal{O}$) by some finite extension, so long as the choice of extension depends only on $\bar{g}$.

\subsection{Lifting the Lie algebra}

\begin{lemma}\label{uniquedirectsumlift}
Let $f: A^n \to A^n$ be an $A$-linear map inducing $\bar{f}: k^n \to k^n$. Suppose we have a factorisation of the characteristic polynomial of $\bar{f}$ as $\det(X-\bar{f}) = \bar{p}(X)\bar{q}(X)$ with $\bar{p}(X)$ coprime to $\bar{q}(X)$. Then letting $\bar{N} = \ker(\bar{p}(\bar{f}))$, there exists a unique $f$-invariant direct summand lift $N$ of $\bar{N}$ to $A^n$. Moreover if $M$ is an $f$-invariant direct summand of $A^n$ with $\bar{M} \subset \bar{N}$ then $M \subset N$.
\end{lemma}
\begin{proof}
Since $\det(X-\bar{f})$ is monic and $A$ is a Henselian local ring (as $A$ is a complete Noetherian local ring) it follows by Hensel's lemma that there exists a factorisation 
$$\det(X-f) =p(X)q(X)$$ with $p(X) \mod m_A = \bar{p}(X)$ and $q(X) \mod m_A = \bar{q}(X)$.
Moreover, since $\bar{p}(X)$ and $\bar{q}(X)$ are coprime, we claim we can find power series $r(X),s(X) \in A[[X]]$ with $p(X)r(X) + q(X)s(X) = 1$. 
Indeed, by coprimality of $\bar{p}(X)$ and $\bar{q}(X)$ we can find $\bar{r}(X),\bar{s}(X) \in k[X]$ such that $\bar{p}(X) \bar{r}(X) + \bar{q}(X) \bar{s}(X) = 1$.
Then taking arbitrary lifts $\Tilde{r}(X),\Tilde{s}(X) \in A[X]$ of $\bar{r}(X),\bar{s}(X)$ respectively, we have that $p(X)\Tilde{r}(X) + q(X)\Tilde{s}(X) = 1 + h(X)$ for some polynomial $h(X)$ whose coefficients all lie in $m_A$. 
Multiplying both $\Tilde{r}(X)$ and $\Tilde{s}(X)$ by $\sum_{i=0}^\infty (-1)^i h(X)^i$ yields the desired $r(X)$ and $s(X)$.

Now we have $p(f) r(f) + q(f) s(f) = \id_{A^n}$ so letting $e_1 = p(f) r(f)$ and $e_2 = q(f)s(f)$ we have $e_1 e_2 = e_2 e_1 = r(f)s(f) p(f)q(f) = 0$ by Cayley-Hamilton.
Since $e_1 + e_2 = \id_{A^n}$ we have $e_1^2 = e_1^2 + e_1 e_2 = e_1$ and $e_2^2 = e_2$. 
Setting $N = e_2(A^n)$, we see that $N$ is an $f$-invariant direct summand, since $f$ commutes with $e_2$ and $A^n = N \oplus e_1(A^n)$. Note $\overline{e_2}(A^n) = \overline{e_2}(k^n) = \bar{N}$, so we have constructed a lift.

We next show that if $M$ is any $f$-invariant direct summand of $A^n$ with $\overline{M} \subset \overline{N}$ then $M \subset N$ with equality if and only if $\overline{M} = \overline{N}$, from which uniqueness of $N$ (and hence the lemma) follows.
We have that $M$ is a free submodule (since projective modules over a local ring are free).
Since $M$ is $f$-invariant, we can consider the decomposition $M = e_1(M) \oplus e_2(M)$.
Then $\overline{M} = \overline{e_1}(\overline{M}) \oplus \overline{e_2}(\overline{M})$. We supposed that $\overline{M} \subset \overline{N}$, and we know $\overline{e_1}(\overline{N}) = 0$ and $\overline{e_2}(\overline{N}) = \overline{N}$.
Thus $e_1(M)$ is a finite free $A$-module which is zero on reduction modulo $m_A$, so it is equal to $0$.
Hence $M = e_2(M) \subset e_2(A^n) = N$. 
Finally, if $\overline{M} = \overline{N}$ then we have that $M$ is a free $A$-submodule direct summand of $N$ of equal rank (as their reductions modulo $m_A$ have equal dimension), and we therefore have equality. 
\end{proof}

Applying Lemma \ref{uniquedirectsumlift} to the adjoint action of $g$ on $A^n$, we see that $\Lzerobar$ admits a unique $\Ad(g)$-invariant direct summand lift to $\gA$, and we write this decomposition as $\gA = L_0 \oplus L_1$. We can also use Lemma \ref{uniquedirectsumlift} to prove a generalisation which will be of use later.

\begin{lemma}\label{liftfamily}
Let $f_1,\ldots,f_m: A^n \to A^n$ be a commuting family of $A$-linear maps such that each $\bar{f_i}$ has all of its eigenvalues lying in $k$. For each vector $\lambda \in k^m$ let $$\overline{V_\lambda} = \{v \in k^n: (\bar{f_i}-\lambda_i)^n(v) = 0 \text{ for } i=1,\ldots,m\}.$$
Then reduction modulo $m_A$ puts the following sets in bijection:
\begin{enumerate}[label=(\roman*)]
    \item\label{type_a} The set of $k$-subspaces of $k^n$ of the form $\oplus_{\lambda \in S} \overline{V_\lambda}$ for some $S \subset k^m$
    \item\label{type_b} The set of $A$-submodule direct summands of $A^n$ invariant under each $f_i$ whose reduction modulo $m_A$ is of the type defined in \ref{type_a}.
\end{enumerate}
\end{lemma}

\begin{proof}
Let $e_{i,\lambda_i}$ denote the $f_i$-invariant projection from $A^n$ onto the unique $f_i$-invariant lift of the generalised $\lambda_i$-eigenspace for the action of $f_i$ on $k^n$, as in Lemma \ref{uniquedirectsumlift}. Then $e_{i,\lambda_i}$ is a polynomial in $f_i$, and hence commutes with any other such $e_{j,\lambda_j}$. We show firstly that $\overline{V_{\lambda}}$ admits a lift to a direct summand of $A^n$ which is invariant under every $f_i$, and that this lift is unique.

We have that $\overline{V_\lambda} = \cap_{i=1}^m \overline{e_{i,\lambda_i}}(k^n)$. We claim that $V_\lambda = \cap_{i=1}^m e_{i,\lambda_i}(A^n)$ is such a lift. To see this, let $e = \prod_{i=1}^m e_{i,\lambda_i}$ and observe that $V_\lambda' = e(A^n)$ is equal to $V_{\lambda}$. Clearly we have $V_\lambda' \subset V_\lambda$, as the projection operators all commute. 
Now let $v \in V_{\lambda}$. 
Note that since each $e_{i,\lambda_i}$ is a projection operator, we have $v = e_{1,\lambda_1}(v)$, $v = e_{2,\lambda_2}(v) = e_{2,\lambda_2}(e_{1,\lambda_1}(v))$, and so on. 
Thus we have $v = \prod_{i=1}^m e_{i,\lambda_i}(v) \in V_\lambda'$, establishing the claim.
We deduce that $e(A^n)$ is a lift of $\overline{V_\lambda} = \bar{e}(k^n)$.
Moreover, this lift is a direct summand of $A^n$.
Indeed, the operator $\Tilde{e} = \sum_{\mu \in k^n\setminus \{\lambda\}} \prod_{i=1}^m e_{i,\mu_i}$ satisfies $\Tilde{e} e = 0$, $\Tilde{e}^2 = \Tilde{e}$ and $\Tilde{e} + e = \id_{A^n}$.
Thus $\Tilde{e}(A^n)$ provides a direct sum complement to $e(A^n)$. 
This lift is invariant under each $f_i$, as each $f_i$ commutes with $e$ and so $f_i(e(A^n)) = e(f_i(A^n)) \subset e(A^n)$.

To see uniqueness, observe that if $N$ is a direct summand lift of $\overline{V_\lambda}$ that is invariant under each $f_i$ then $N \subset e_{i,\lambda_i}(A^n)$ for every $i$ by Lemma \ref{uniquedirectsumlift}.
It follows that $N \subset V_\lambda$ is a direct summand free $A$-submodule of the same rank (as their reductions modulo $m_A$ are equal), and hence are equal. 

By taking direct sums of such $V_\lambda$ we have established the existence of $f_i$-invariant direct summand lifts. 
To show uniqueness, suppose that $N \subset A^n$ is a direct summand $f_i$-invariant lift of $\oplus_{\lambda \in S} \overline{V_\lambda}$.
Then $N$ is free over $A$ and we can therefore write $N = \oplus_{\lambda \in S} V_\lambda'$, with each $V_\lambda'$ a direct summand $f_i$-invariant lift of $\overline{V_\lambda}$.
Since $N$ itself is a direct summand of $A^n$, it follows by the uniqueness of $V_\lambda$ shown already that $V_\lambda' = V_\lambda$ and we are done.
\end{proof}

The following lemma provides another perspective on the lifts of Lemma \ref{uniquedirectsumlift} in terms of topological nilpotents.

\begin{lemma}\label{topnilpotent}
Let $f: A^n \to A^n$ be an $A$-linear map inducing $\bar{f}: k^n \to k^n$. The $\bar{f}$-invariant decomposition $k^n = \overline{V_0} \oplus \overline{V_1}$, where $\overline{V_0} = \ker \bar{f}^n$ admits a unique $f$-invariant lift $A^n = V_0 \oplus V_1$. Then $$V_0 = \{v \in A^n| f^m(v) \to 0 \text{ as } m \to \infty\},$$ where convergence is with respect to the $m_A$-adic topology.
\end{lemma}
\begin{proof}
The existence and uniqueness of the decomposition follows from Lemma \ref{uniquedirectsumlift}. 
Let $$V_0' = \{v \in A^n| f^m(v) \to 0 \text{ as } m \to \infty\}.$$ 
If $v \in V_0$ then $\bar{f}^n(v \mod m_A) = 0$, so $f^n(v) \in m_A A^n \cap V_0 = m_A V_0$, since $V_0$ is a direct summand of $A^n$. Thus $f^n(V_0) \subset m_A V_0$ and so $f^{nu}(V_0) \subset (m_A)^u V_0$ for every $u \geq 1$.
It follows that $V_0 \subset V_0'$.

Now suppose $v \in V_0'$, and write $v = x+y$ with $x \in V_0$ and $y \in V_1$.
Then $y \in V_1 \cap V_0'$.
On the one hand, $f$ acts invertibly on $y$, since $\bar{f}$ has trivial kernel on $\overline{V_1}$.
On the other hand, $f^m(y) \to 0$ as $m \to \infty$, so $y \in (m_A)^m \cdot A^n$ for every $m$.
We conclude that $y = 0$ and $V_0 = V_0'$.
\end{proof}

\begin{lemma}
The $A$-submodule $L_0$ is a Lie subalgebra of $\gA$.
\end{lemma}
\begin{proof}
We need only check that $L_0$ is closed under the Lie bracket of $\gA$.
We will firstly consider the case that $A$ is Artinian, so $(m_A)^u = 0$ for some $u \geq 1$.
Let $f: \gA \to \gA$ denote the endomorphism $\Ad(g) - \id$ and define the increasing sequence of $A$-submodules $L_0^{(m)} = \{v \in L_0| f^m(v) = 0\}$. 
Since $L_0$ is the $f$-invariant lift of $\ker \bar{f}$, we have by Lemma \ref{topnilpotent} that $L_0 = \{v \in \gA| f^m(v) = 0 \text{ for some } m \geq 0\} = \bigcup_{m \geq 0} L_0^{(m)}$. 
Let $x \in L_0^{(n)}$ and $x' \in L_0^{(m)}$. 
We show that $[x,x'] \in L_0^{(n+m)}$ by induction on $n+m$. Notice that in the case $n=0$ we have $x = 0$ and so the result is clear. We need only show the inductive step then, where we may take both $n,m \neq 0$. Then write 
\begin{align*}
    \Ad(g)(x) &= x + y \\
    \Ad(g)(x') &= x'+ y'
\end{align*}
where $f^{n-1}(y) = f^{n-1}(f(x)) = 0$ and similarly $f^{m-1}(y') = 0$. Thus $y \in L_0^{(n-1)}$ and $y' \in L_0^{(m-1)}$. We show $f([x,x']) \in L_0^{(n+m-1)}$ to deduce then that $[x,x'] \in L_0^{(n+m)}$, completing the argument. We have
\begin{align*}
    f([x,x']) &= \Ad(g)([x,x']) - [x,x'] \\
            &= [\Ad(g)(x),\Ad(g)(x')] - [x,x'] \\
            &= [x+y,x'+y'] - [x,x'] \\
            &= [y,x'] + [x,y'] + [y,y']
\end{align*}
where we have used that $\Ad(g)$ is a Lie algebra homomorphism on $\gA$. By induction hypothesis, the three terms all lie in $L_0^{(n+m-1)}$ and we are done in the Artinian case.

Suppose now that $A$ is not necessarily Artinian.
Let $x,x' \in L_0$ and let $u \geq 1$.
We wish to show that $f^m([x,x']) \in (m_A)^u \gA$ for all $m$ sufficiently large.
We know that there exists $M \geq 1$ such that for all $m \geq M$ we have $f^m(x),f^m(x') \in (m_A)^u \gA$.
From the argument above applied to $A/(m_A)^u$, we see that $f^{2m}([x,x']) \in (m_A)^u \gA$.
So indeed we have for every $m \geq 2M$ that $f^m([x,x']) \in (m_A)^u \gA$ and we are done.
\end{proof}

\subsection{Lifting $M_{\bar{g}}$}

We will frequently wish to take the Lie algebra of a base-changed group scheme, and so we record the following lemma which will be of repeated use.

\begin{lemma}
 Let $G \to S$ be a smooth group scheme over a base scheme $S$ with identity section $e: S \to G$. Let $f: S' \to S$ be a morphism of schemes and let $G' = G \times_S S'$ be the base change.
 Let $e': S' \to G'$ be the base-change of $e$.
 Then we have an isomorphism
 \begin{align*}
     f^* e^* \Omega^1_{G/S} \cong (e')^* \Omega^1_{G'/S'}
 \end{align*}
 and, in particular, an isomorphism of Lie algebras
  \begin{align*}
     (\Lie G)_{S'} \cong \Lie G'.
 \end{align*}
\end{lemma}
\begin{proof}
Letting $g: G' \to G$ be the natural map, there is an isomorphism $g^* \Omega^1_{G/S} \cong \Omega^1_{G'/S'}$ (\cite[\href{https://stacks.math.columbia.edu/tag/01V0}{Tag 01V0}]{stacks-project}).
Pulling back by $e'$ we have
\begin{align*}
   (e')^* \Omega^1_{G'/S'}  \cong (e')^* g^* \Omega^1_{G/S} \cong f^* e^* \Omega^1_{G/S}
\end{align*}
by functoriality of pullback. 
Since pullback of locally free sheaves commutes with taking duals, we deduce the stated isomorphism of Lie algebras.
\end{proof}

We will require later that taking centers of certain reductive groups will commute with formation of the Lie algebras.
For such statements to hold we need to make assumptions on the characteristic of $k$. 
An example where such an equality does not hold is for $\SL_2$ in characteristic 2: the Lie algebra has 1-dimensional center spanned by the identity while $\SL_2$ has 0-dimensional center.
We recall the notion of pretty good characteristic, as in \cite[Definition 2.11]{smoothnesscentralizers}.

\begin{definition}\label{pretty_good}
Let $l$ be a prime number and $G$ a reductive group over a field $F$.
Denote the root datum of $G$ (with respect to some maximal torus over an algebraic closure of $F$) by $(X^*,\Phi,X_*,\Phi^\vee)$.
We say $l$ is a prime of pretty good characteristic for $G$ if for any subset $\Phi' \subset \Phi$ the abelian groups $X^*/\mathbb{Z}\Phi'$ and $X_*/\mathbb{Z}(\Phi')^\vee$ are $l$-torsion free.
\end{definition}

For a given group $G$, we are only considering the $l$-torsion of finitely many finitely generated abelian groups and hence only finitely many primes $l$ are not of pretty good characteristic for $G$.
In our setting we have that if $l$ is a prime of pretty good characteristic for $\hat{G}$ then it is also of pretty good characteristic for $M_{\bar{g}}^\circ$, since the roots of $M_{\bar{g}}^\circ$ are a subset of the roots of $\hat{G}$ (and the character and cocharacter groups are equal). 
From now on then we assume that $p = \ch k$ is a pretty good prime for $\hat{G}_k$.

One use of Definition \ref{pretty_good} to us comes from \cite[Theorem 3.3]{smoothnesscentralizers} and \cite[Lemma 3.1]{smoothnesscentralizers}.
Together these imply that if $\ch F = l$ is a prime of pretty good characteristic for $G$, then $Z_G(G)$ is smooth and in particular the Lie algebra $\mathfrak{g} = \Lie G$ is separable inside itself, in the sense that 
\begin{equation} \label{lie algebra commute center}
\dim Z_{G}(\mathfrak{g}) = \dim \mathfrak{z}(\mathfrak{g}).
\end{equation}
From this discussion we deduce the following lemma, which is use for computations in later sections.

\begin{lemma} \label{lie algebra mg commute center}
Let $\bar{H} = Z(M_{\bar{g}}^\circ)$.
Then there is an equality of Lie algebras $\Lie \bar{H} = \mathfrak{z}(\Lzerobar)$.
\end{lemma}
\begin{proof}
By \cite[Proposition 3.3.8]{conrad} we always have $\bar{H} = Z_{M_{\bar{g}}^\circ}(\Lzerobar)$, the kernel of the adjoint representation, and hence equality of their Lie algebras.
The containment $\Lie \bar{H} \subset \mathfrak{z}(\Lzerobar)$ is clear from differentiating the adjoint action of $\bar{H}$ on $\Lzerobar$.
Thus they are equal, as by equality \ref{lie algebra commute center} they are $k$-vector spaces of the same dimension. 
\end{proof}

\begin{remark}
In our applications in later sections, we will take $p$ to be a prime of order coprime to the Weyl group of $\hat{G}$.
This will, in particular, imply that $p$ is a prime of pretty good characteristic for $\hat{G}$ since $p$ will be a very good prime for $\hat{G}$ by \cite[Lemma 3.9]{thorne2019} and hence a pretty good prime by \cite[Lemma 2.12]{smoothnesscentralizers}.
\end{remark}

The next step in the construction of $M_g$ is to find an appropriate multiplicative type subgroup, which will be the center of $M_g$.
The key will be to find a suitable lift of a maximal torus of $\hat{G}_k$.
We recall some definitions and results from \cite{sga3} which will aid us in this goal.

Let $\mathfrak{g}$ be a Lie algebra of dimension $n$ over a field $F$. For any $F$-algebra $R$ and for any $a \in \mathfrak{g}_R = \mathfrak{g} \otimes_F R$ we can consider the endomorphism of $\mathfrak{g}_R$ given by $\ad(a)$. We may write its characteristic polynomial as $P_\mathfrak{g}(a,t) = t^n + c_{n-1}(a) t^{n-1} + \cdots + c_0(a)$. 
The nilpotent rank is the minimal $r \geq 0$ such that $c_r(a) \neq 0$ for some $a$. 
Let $r$ denote the nilpotent rank of $\mathfrak{g}$. 
\begin{definition} \label{ad-regular}
With notation as above, we say that $a \in \mathfrak{g}$ is ad-regular if $c_r(a) \neq 0$.
\end{definition}
\begin{remark}
The terminology `ad-regular' is not standard, but we use it to avoid conflict with the other notion of regular (which is that the centralizer has dimension as small as possible).
\end{remark}
If $a \in \mathfrak{g}$ is any element we let $\Nil(a,\mathfrak{g}) = \bigcup_{u \geq 0} \ker(\ad(a)^u: \mathfrak{g} \to \mathfrak{g})$, a Lie subalgebra of $\mathfrak{g}$. If $a \in \mathfrak{g}$ is ad-regular then $\rank_k(\Nil(a,\mathfrak{g})) = r$ and we call $\Nil(a,\mathfrak{g})$ a Cartan subalgebra of $\mathfrak{g}$.
More generally \cite[XIV Définition 2.4]{sga3}, if $S = \Spec R$ is any affine scheme and $\mathfrak{g}$ is a Lie algebra over $S$ we say a Lie subalgebra $\mathfrak{d} \subset \mathfrak{g}$ is a Cartan subalgebra if $\mathfrak{d}$ is locally a direct summand of $\mathfrak{g}$ and for every $s \in S$ we have that $\mathfrak{d}(s) \subset \mathfrak{g}(s)$ is a Cartan subalgebra in the sense described above. Here $\mathfrak{g}(s) = \mathfrak{g} \otimes_R k(s)$ is the fibre over $k(s)$.
\begin{proposition} \label{invariant lift cartan subalg}
Suppose that $a \in \gA$ is such that $\bar{a}$ is ad-regular in $\gk$.
Let $\mathfrak{d}$ denote the direct summand $\ad(a)$-invariant lift of $\Nil(\bar{a},\gk)$ to $\gA$ given by Lemma \ref{uniquedirectsumlift}.
Then $\mathfrak{d}$ is a Cartan subalgebra of $\gA$.
\end{proposition}
\begin{proof}
Let $S = \Spec A$.
Since $\hat{G}$ admits a maximal torus of rank $r \geq 0$, $\gA$ satisfies condition $(\text{C}_0)$ of \cite[XIV Proposition 2.9]{sga3} stating that the nilpotent rank of $\mathfrak{\hat{g}}(s)$ is a locally constant function of $s$ on $S$.
Then by \cite[Corollaire 2.10]{sga3}, the set $U$ of $s \in S$ for which $a(s)$ is ad-regular in $\mathfrak{\hat{g}}(s)$ is an open subset of $S$.
Thus since $U$ contains the closed point of $S$ corresponding to $m_A$, we have that $U = S$.

Let $f = \ad(a): \gA \to \gA$, and recall that the unique $f$-invariant lift of $\ker(\bar{f}^{\dim(\gk)}) = \Nil(\bar{a},\gk)$ is given by
\begin{equation}
    \mathfrak{d} = \{a \in \gA : \ad(a)^m (x) \to 0 \text{ as } m \to \infty\}
\end{equation}
by Lemma \ref{topnilpotent}.
We claim that $\mathfrak{d}$ is a Lie subalgebra. 
Note that $f$ is an $A$-linear derivation, in the sense that $f([y,z]) = [f(y),z] + [y,f(z)]$ for every $y,z \in \gA$.
So suppose $y,z \in \mathfrak{d}$, and let $u \geq 1$.
There exists $M \geq 1$ such that $f^m(y),f^m(z) \in (m_A)^u \gA$ for every $m \geq M$.
Now suppose $m \geq 2M$.
Then since $f$ is a derivation,
$$
f^m([y,z]) = \sum_{i=0}^m \binom{m}{i} [f^i(y), f^{m-i}(z)]
$$
and each term in the sum lies in $(m_A)^u \gA$.
Hence $[y,z] \in \mathfrak{d}$.

Observe that there are inclusions $\mathfrak{d}(s) \supset \Nil(a(s),\mathfrak{\hat{g}}(s))$ for every $s \in S$.
Indeed, let $\eta$ denote the $f$-invariant direct sum complement to $\mathfrak{d}$ in $\gA$.
Then $f$ acts invertibly on $\eta$.
Hence writing $\mathfrak{\hat{g}}(s) = \mathfrak{d}(s) \oplus \eta(s)$, we see that $\ad(a(s))$ acts invertibly on $\eta(s)$, which shows the containment.
This inclusion is an equality, since both are $k(s)$-vector spaces of dimension $r$, as $a(s)$ is ad-regular and $r$ is the nilpotent rank of $\mathfrak{\hat{g}}(s)$.
We have therefore shown that $\mathfrak{d}$ is a Cartan subalgebra of $\gA$.
\end{proof}

Let $S$ be a scheme and $G$ a group-scheme over $S$.
We say \cite[IX Définition 1.3]{sga3} $G$ is a torus if $G$ is locally isomorphic in the fpqc topology to a group of the form $\mathbb{G}_m^r$ for some $r \geq 0$.
We say \cite[XII Définition 1.3]{sga3} a subgroup-scheme $T$ of $G$ is a maximal torus if $T$ is a torus whose geometric fibres are maximal tori in the geometric fibres of $G$, in the classical sense.
By \cite[X Corollaire 3.3(ii)]{sga3}, reduction modulo $m_A$ induces an equivalence of categories between the categories of finite type and multiplicative type group schemes over $A$ and those over $k$. Thus if we have a torus over $A$ such that its reduction modulo $m_A$ is split, then our torus over $A$ must also be split.

Now suppose that $G$ is a smooth group scheme over $S$.
We say (\cite[4.1.1]{bouthier:hal-02390969}) that $G$ is root-smooth if for every geometric point $\bar{s}$ of $S$ and maximal torus $T$ of $G_{\bar{s}}$, every root $T \to \mathbb{G}_{m,\bar{s}}$ is a smooth morphism.
Equivalently, a root $\alpha$ is smooth if $X^*(T)/\mathbb{Z} \alpha$ is $\ch k(\bar{s})$-torsion free.
It therefore follows from Definition \ref{pretty_good} that $\hat{G}$ is root-smooth over $A$.
The following proposition is contained in \cite[Proposition 4.1.3]{bouthier:hal-02390969} and allows us to pass between Cartan subalgebras and maximal tori in our setting.

\begin{proposition}\label{cartan subalg max tori}
Let $G$ be a root-smooth reductive group over a scheme $S$.
The Cartan subalgebras of $\mathfrak{g}$ are precisely the Lie algebras of the maximal $S$-tori of $G$.
\end{proposition}

\begin{corollary}\label{x-centralizer-torus}
Suppose $x \in L_0$ is any choice of lift of an ad-regular element of $\Lzerobar$.
Then there exists a Lie subalgebra $\mathfrak{t}$ of $L_0$ containing $x$ and a maximal torus $T$ of $\hat{G}_A$ with Lie algebra $\mathfrak{t}$.
\end{corollary}
\begin{proof}
Let $\mathfrak{t}$ denote the unique $\ad(x)$-invariant lift of $\Nil(\bar{x},\gk)$ to $\gA$.
Then $\mathfrak{t}$ is a Cartan subalgebra of $\gA$ by Proposition \ref{invariant lift cartan subalg} and we can find a maximal torus $T$ with $\Lie T = \mathfrak{t}$ by Proposition \ref{cartan subalg max tori}.
Hence we just need to show that $\mathfrak{t} \subset L_0$.
Since $L_0$ is a Lie algebra, the adjoint action of $x$ on $\gA$ restricts to an $A$-linear endomorphism of $L_0$.
We can therefore lift $\mathfrak{\bar{t}}$ to an $\ad(x)$-invariant direct summand of $L_0$.
Since $L_0$ is itself a direct summand of $\gA$, by uniqueness in Lemma \ref{uniquedirectsumlift} this lift must equal $\mathfrak{t}$ and we are done.
\end{proof}

By extending $k$ if necessary, we will assume from now on that $M_{\bar{g}}$ contains a split maximal torus $\bar{T}$ whose Lie algebra $\mathfrak{\bar{t}}$ contains an ad-regular element $\bar{x}$ of $\gk$.
Let $x$ be an arbitrary lift of $\bar{x}$ to $L_0$ and let $T$ be the maximal torus arising from Corollary \ref{x-centralizer-torus}.
Necessarily $T$ is split, since we assumed $\bar{T}$ is split.
Now recall that $\bar{H} = Z(M_{\bar{g}}^\circ)$ is the scheme-theoretic centre of $M_{\bar{g}}^\circ$, a subgroup of $\bar{T}$ of multiplicative type.
We claim that there exists a unique subgroup $H \subset T$ lifting $\bar{H}$. 
Indeed, by \cite[VIII Corollaire 1.6]{sga3} we deduce an antiequivalence of categories between finite $\mathbb{Z}$-modules and split multiplicative groups of finite type on any connected scheme, and that this is compatible with base change to a connected scheme. 
So $\bar{H}$ corresponds to a quotient of the finite free module defining $\bar{T}$, and thus under the antiequivalence correspond to subgroup of $T$ lifting $H$.

We are now ready to define $M_g$ and show the desired properties.

\begin{theorem}\label{Mg}
Let $p$ be a pretty good prime for $G$ (in the sense of Definition \ref{pretty_good}).
Let $\bar{g} \in \hat{G}(k)$ be semisimple with centralizer $M_{\bar{g}}$, whose Lie algebra is denoted by $\Lzerobar$.
 Let $A \in \CNL_{\CalO}$ and let $g \in \hat{G}(A)$ be a lift of $\bar{g}$.
 \begin{enumerate}
     \item There exists a unique connected reductive closed subgroup-scheme $M_g^\circ$ of $\hat{G}_A$ with the following properties:

\begin{enumerate}[label=(\roman*)]
    \item\label{Mgbasechange} $(M_g^\circ)_k = M_{\bar{g}}^\circ$ as subgroups of $\hat{G}_k$
    \item\label{liealg} $\Lie M_g^\circ = L_0$
    \item\label{g_in_Mg} $g \in M_g^\circ(A)$.
\end{enumerate}
Here we recall that $\gA = L_0 \oplus L_1$ is the unique $\Ad(g)$-invariant lift of the $\Ad(\bar{g})$-invariant decomposition $\gk = \Lzerobar \oplus \overline{L_1}$.
\item There exists a unique multiplicative type subgroup scheme $S$ of $Z(M_g^\circ)$ such that $S_k = Z(M_{\bar{g}})$. 
Let $M_g = Z_{\hat{G}_A}(S)$.
Then $M_g$ has identity component $M_g^\circ$ and $(M_g)_k = M_{\bar{g}}$ as subgroups of $\hat{G}_k$.
 \end{enumerate}
\end{theorem}

\begin{proof}
Set $M_g^\circ = Z_{\hat{G}_A}(H)^\circ$ to be the identity component of the scheme-theoretic centralizer of $H$.
The group scheme $M_g^\circ$ is reductive and smooth over $A$ by the discussion preceding \cite[Remark 3.1.5]{conrad}.
We show \ref{Mgbasechange} firstly.
By definition of scheme-theoretic centralizer, we have that $Z_{\hat{G}_A}(H)_k = Z_{\hat{G}_k}(H_k)$.
We have an inclusion $M_{\bar{g}}^\circ \subset Z_{\hat{G}_k}(H_k)$ which is immediate from the definitions.
We also have an inclusion $Z_{\hat{G}_k}(H_k) \subset M_{\bar{g}}$ which follows because $\bar{g} \in H_k(k)$.
From the discussion in \cite[3.1]{sga3} we see that $(M_g^\circ)_k = (Z_{\hat{G}_A}(H)_k)^\circ$ and we deduce that \ref{Mgbasechange} holds from the above inclusions. 

Next we note that $\Lie M_g^\circ$ is a lift of $\Lzerobar$. Indeed, we have that taking Lie algebras is compatible with base-change, so $(\Lie M_g^\circ)_k = \Lie (M_g^\circ)_k = \Lie M_{\bar{g}}^\circ = \Lzerobar$ by \ref{Mgbasechange}. 
We have that $\Lie M_g^\circ = \Lie Z_{\hat{G}_A}(H) = (\gA)^H$ by \cite[Lemma 2.2.4]{conrad} and that $\Lie M_g$ is a direct summand of $\gA$. 
Indeed, the action of the split multiplicative group $H$ on the finite free $A$-module $\gA$ induces a grading indexed by the character group of $H$ (see \cite[A.8.8]
{psrgroups}). 
Then $(\gA)^H$ is the graded piece corresponding to the trivial character, hence is a direct summand.

We have direct summand submodules $L_0$ and $\Lie M_g$ of $\gA$, both containing $\mathfrak{t}$, with reductions modulo $m_A$ both equal to $\Lzerobar$. Since $\bar{T}$ is a split maximal torus in $\hat{G}_k$, $\gk$ admits a root space decomposition $\gk = \mathfrak{\bar{t}} \oplus \bigoplus_{\alpha} \hat{\mathfrak{g}}_{\alpha}$. 
These $\hat{\mathfrak{g}}_{\alpha}$ come in pairs, and are distinct subspaces for different pairs (even when $\ch k = 2$).
We also have that $\Lzerobar = \Lie M_{\bar{g}}$ admits such a decomposition, hence necessarily corresponding to a subset of pairs of roots. 
Applying Lemma \ref{liftfamily} to the adjoint action of a basis of $\mathfrak{t}$ on $\gA$, we see that $L_0$ must equal $\Lie M_g$.

We now show $\ref{g_in_Mg}$.
Since $Z(M_g^\circ) = \ker(\Ad: M_g^\circ \to \GL(L_0))$ by \cite[Proposition 3.3.8]{conrad} and $\Ad(g): L_0 \to L_0$ we see that $g$ normalizes $Z(M_g^\circ)$.
Note that $Z(M_g^\circ)$ lifts $Z(M_{\bar{g}}^\circ) = \bar{H}$ by definition of scheme theoretic centralizer.
Thus since $H \subset Z(M_g^\circ)$ are both multiplicative type subgroup schemes of a torus lifting $\bar{H}$ they must be equal.
Since $\bar{g} \in \bar{H}(k)$, the morphism of split multiplicative type groups $\Ad(\bar{g}): \bar{H} \to \bar{H}$ is trivial.
Thus $\restr{\Ad(g)}{H}$ is an automorphism which is trivial $\mod m_A$. 
By \cite[1.6, VIII]{sga3}, this implies that the conjugation action of $g$ on $H$ is trivial.
By definition, we have $Z_{\hat{G}_A}(H)(A) = \{z \in \hat{G}_A(A): \restr{\Ad(z)}{H} = \id_H\}$, so we have shown $g \in Z_{\hat{G}_A}(H)(A)$.
We deduce that $g \in M_g^\circ(A)$ since $\bar{g} \in M_g^\circ(k)$ and $\Spec A$ is connected.

Uniqueness of such a subgroup $M_g^\circ$ satisfying \ref{liealg} follows from \cite[XIV Proposition 3.12]{sga3} since $\mathfrak{t} \subset L_0$ is Cartan subalgebra of $\mathfrak{g}_A$.

Finally, we show the second part of the theorem.
We observe that $Z_{\hat{G}_k}(S_k) = M_{\bar{g}}$ since $\bar{g} \in S_k(k)$.
We therefore have $Z_{\hat{G}_k}(S_k)^\circ = M_{\bar{g}}^\circ$.
The properties of $M_g$ are now clear by definition of scheme theoretic centralizer and because taking identity components commutes with base change in this setting.
\end{proof}

\subsection{Functoriality properties of $M_{g}$}

Having constructed $M_g$, we now show some properties which will be useful in defining our deformation condition.

\begin{lemma} \label{center of Mg}
The center of $M_g$, $Z(M_g)$, lifts $Z(M_{\bar{g}})$.
\end{lemma}
\begin{proof}
This is immediate by functoriality of scheme-theoretic center.
\end{proof}

We now show that the construction of $M_g$ is compatible with changing $g$ by an element $h \in \ker(\hat{G}(A) \to ~\hat{G}(k))$, in the following senses.

\begin{lemma}\label{change_of_g}
Let $h \in \ker(\hat{G}(A) \to \hat{G}(k))$. Then 
\begin{enumerate} [label=(\roman*)]
    \item $M_{h g h^{-1}} = \Ad(h)(M_g)$; \label{conjugate g}
    \item if additionally $h \in M_g^\circ(A)$ then $M_{gh} = M_g$. \label{multiply g}
\end{enumerate}
\end{lemma}
\begin{proof}
We prove \ref{conjugate g} firstly.
We show that the Lie algebras of $M_{h g h^{-1}}$ and $\Ad(h)(M_g)$ coincide:
\begin{align*}
\Lie M_{h g h^{-1}} &= \{v \in \gA: (\Ad(hgh^{-1}) - 1)^u v  \to 0 \text{ as } m \to \infty\} \\
                    &= \{v \in \gA: (\Ad(h) \Ad(g) \Ad(h)^{-1} - 1)^u v  \to 0 \text{ as } m \to \infty\} \\
                    &= \{v \in \gA: \Ad(h)( \Ad(g)  - 1)^u \Ad(h)^{-1} v  \to 0 \text{ as } m \to \infty\} \\
                    &= \Ad(h) \{v \in \gA: ( \Ad(g)  - 1)^u v  \to 0 \text{ as } m \to \infty\} \\
                    &= \Ad(h) (\Lie M_g) \\
                    &= \Lie \Ad(h)(M_g).
\end{align*}
Equality of $M^\circ_{h g h^{-1}}$ and $\Ad(h)(M_g^\circ)$ then follows, just as in the proof of uniqueness in Theorem \ref{Mg}.
If $S$ is the unique subgroup scheme of $Z(M_g^\circ)$ lifting $Z(M_{\bar{g}})$ then $\Ad(h)(S)$ is the unique subgroup scheme of $Z(M_{h g h^{-1}}^\circ)$ lifting $Z(M_{\bar{g}})$ and the equality $M_{h g h^{-1}} = \Ad(h)(M_g)$ is then clear.

Now suppose in addition that $h \in M_g^\circ(A)$. 
We have a semisimple element $\bar{g} \in M_g^\circ(k)$ and a lift $gh \in M_g^\circ(A)$. 
Thus we can apply Theorem \ref{Mg} to the reductive group $M_g^\circ$ to define a connected, reductive, closed subgroup-scheme $N$ of $M_g^\circ$ with the properties
\begin{itemize} 
    \item $N$ lifts $Z_{M_{\bar{g}}^\circ}(\bar{g}) = M_{\bar{g}}^\circ$
    \item $\Lie N  = \{v \in \Lie M_g^\circ : (\Ad(gh)-1)^u v \to 0 \text{ as } u \to \infty\}$.
\end{itemize}
It follows that $N = M_g^\circ$, and therefore $\Lie M_g^\circ = \Lie N \subset \Lie M_{gh}^\circ$. 
By symmetry we have equality of Lie algebras and hence equality of $M_g^\circ$ and $M_{gh}^\circ$ just as above.
The constructions of $M_g$ and $M_{gh}$ are given by taking the centralizer of the same multiplicative type group scheme inside $\hat{G}_A$ and therefore also coincide.
\end{proof}

As an application, we obtain the following corollary which will be of use in identifying Selmer groups corresponding to our deformation condition in Section \ref{section deformation}.

\begin{corollary}\label{inertial_selmer}
Suppose $A = k[\epsilon]$ and identify $\gk = \ker(\hat{G}(k[\epsilon]) \to \hat{G}(k)) \subset \hat{G}(k[\epsilon])$. 
Then $\gk \cap M_{\bar{g}}(k[\epsilon]) = \gk \cap M_{g}(k[\epsilon])$ and $\gk \cap Z(M_{\bar{g}})(k[\epsilon]) = \gk \cap Z(M_{g})(k[\epsilon])$.
\end{corollary}
\begin{proof}
Observe firstly that for $h \in \gk$ we have $\gk \cap M_{g}(k[\epsilon]) = \gk \cap M_{h g h^{-1}}(k[\epsilon])$, since conjugation by $h$ acts trivially on $\gk$.
Since conjugation will also send centers to centers, we may freely replace $g$ by its conjugate under an element of $\gk$.
Now suppose we are in the situation that $g = \bar{g} \delta$ where $\delta \in M_{\bar{g}}(k[\epsilon])$. 
Then since $g$ is a lift of $\bar{g}$, we have $\delta \in \gk \cap M_{\bar{g}}(k[\epsilon]) = \Lie M_{\bar{g}}^\circ$ and so by Lemma \hyperref[multiply g]{\ref*{change_of_g} \ref*{multiply g}} we have $M_{\bar{g} \delta} = M_{\bar{g}}$.
Thus equality of their centers also holds.

Thus we just need to find some $h \in \gk$ such that $\bar{g}^{-1} h g h^{-1} \in M_{\bar{g}}(k[\epsilon])$. Writing $g = \bar{g} \gamma$ for $\gamma \in \gk$ we require, switching now to additive notation, that
$$\Ad({\bar{g}}^{-1})(h) + \gamma - h \in  \Lzerobar.$$ 
Letting $\overline{L_1}$ be the $\Ad(\bar{g})$-invariant complement to $\Lzerobar$, we know that $\Ad(\bar{g})-\id$ restricts to a vector space automorphism of this subspace.
Thus we may write $$\Ad(\bar{g})(\gamma) \equiv (\Ad(\bar{g}) - \id)(h) \mod \Lzerobar$$ for some $h \in \gk$. Taking $\Ad(\bar{g})^{-1}$ of this equality, we are done.
\end{proof}

We now show compatibility under changing the coefficient ring $A$.

\begin{lemma}\label{change_of_ring}
Let $\varphi: A \to B$ be a morphism in $\CNL_\mathcal{O}$.
Let $g' \in \hat{G}(B)$ denote the image of $g$ under the map $\hat{G}(\varphi): \hat{G}(A) \to \hat{G}(B)$.
Then under the natural identification of $(\hat{G}_A)_B$ with $\hat{G}_B$, we have $(M_g)_B = M_{g'}$ and $Z(M_g)^\circ_B = Z(M_{g'})^\circ$. In particular, we have the equality $Z(M_g)^\circ(B) = Z(M_{g'})^\circ(B)$.
\end{lemma}
    
\begin{proof}
We may identify $\Lie (M_g)_B$ with $\Lie M_g \otimes_A B$ by compatibility of Lie algebras under base-change. This isomorphism is also compatible with the action of $\Ad(g)$ in the sense that $\Ad(g')(v \otimes b) = \Ad(g)(v) \otimes b$ for any $v \in \gA$ and $b \in B$.
Hence we have a containment of Lie algebras $\Lie (M_g)_B \subset \Lie M_{g'}$.
Indeed, if $v \in \Lie M_g$ and $b \in B$ then $(\Ad(g')-1)^u(v \otimes b) = (\Ad(g)-1)^u(v) \otimes b \to 0$ as $u \to \infty$.
The Lie algebras are both free and direct summand $B$-submodules of equal ranks (as their reductions $\mod m_B$ coincide) and hence are equal.
We deduce that $(M_g^\circ)_B = M_{g'}^\circ$.
Letting $S$ be the unique subgroup of $Z(M_g^\circ)$ lifting $Z(M_{\bar{g}})$ we see that $S_B$ is a subgroup of $Z(M_{g'}^\circ)$ lifting $Z(M_{\bar{g}})$.
Thus $Z(M_g)_B = Z(M_{g'})$ and $(M_g)_B = M_{g'}$ from which the lemma follows.
\end{proof}

\subsection{A construction of $M_g^\circ$ over $\Qpbar$} \label{alternate Mg}

For the remainder of this section suppose that $A =~\mathcal{O}$.
We give a description of $(M_g^\circ)_{\Qpbar}$ and its center which depend only on the knowledge of the semisimple part of $g$ in its Jordan decomposition.

\begin{lemma}
Let $g = g_{ss} g_u$ be the multiplicative Jordan decomposition of $g$ in $\hat{G}(\Qpbar)$.
Let $T$ be any maximal torus of $Z_{\hat{G}_{\Qpbar}}(g_{ss})$ and let $$N_{g_{ss}} = \langle T, U_\alpha | \alpha \in \Phi(\hat{G},T), \alpha(g_{ss}) \equiv 1 \mod m_{\Zpbar} \rangle^\circ,$$
where $U_\alpha$ denotes the root group for $\alpha$.
Then $N_{g_{ss}} = (M_g^\circ)_{\Qpbar}$.
\end{lemma}

\begin{proof}
Let $L \leq \hat{\mathfrak{g}}_{\Qpbar}$ be given by the sum of generalized eigenspaces for the action of $\Ad(g_{ss})$ on $\hat{\mathfrak{g}}_{\Qpbar}$ with eigenvalues congruent to $1 \mod m_{\Zpbar}$.
We firstly show that $L$ coincides with $(\Lie M_g^\circ)_{\Qpbar}$.
The generalized eigenspaces for the actions of $\Ad(g)$ and $\Ad(g_{ss})$ on $\g_{\Qpbar}$ coincide.
It then follows because $\Lie M_g^\circ = L_0$ is given by the sum of topological nilpotents for the action of $\Ad(g) - 1$, while its $\Ad(g)$-invariant complement is given by the sum of topological nilpotents for the action of $\Ad(g) - \lambda$ for all eigenvalues $\lambda \in \Zpbar^\times$ with $\lambda \not\equiv 1 \mod m_{\Zpbar}$.

We next claim that 
$$\Lie N_{g_{ss}} = \langle \Lie T \oplus \bigoplus_{\alpha \in \Phi(\hat{G},T), \alpha(g_{ss}) \equiv 1 \mod m_{\Zpbar} } \hat{\mathfrak{g}}_\alpha \rangle$$
 coincides with $L$.
 Since the action of $\Ad(g_{ss})$ on $\Lie \hat{G}$ is the identity on $\Lie T$ and given by $\alpha(g_{ss})$ on a root space $\hat{\mathfrak{g}}_\alpha$, they must coincide as $L$ is a Lie algebra.
 Thus $N_g = (M_g^\circ)_{\Qpbar}$, as their Lie algebras coincide (and contain a Cartan subalgebra of $\mathfrak{g}_{\Qpbar}$).
\end{proof}

We deduce the following corollary, which we use in Section \ref{sect local global}.

\begin{corollary} \label{rational Mg}
If $g_1,g_2 \in \hat{G}(\CalO)$ are lifts of $\bar{g}$ and there exists $h \in \hat{G}(\Qpbar)$ for which $\Ad(h)(g_{1,ss}) = g_{2,ss}$ then $\Ad(h)((M_{g_1}^\circ)_{\Qpbar}) = (M_{g_2}^\circ)_{\Qpbar}$ and $\Ad(h)(Z(M_{g_1}^\circ)_{\Qpbar}) = Z(M_{g_2}^\circ)_{\Qpbar}$.
\end{corollary}

\section{Deformation theory}\label{section deformation}

Continue with notation as in the start of Section \ref{def setup}.
We make a further assumption on $p = \ch k$ by taking $p$ to be a very good prime for $\hat{G}$.
Let $F$ be a global field, so that $F$ is either a number field or the function field of a smooth, projective, geometrically connected curve $X$ over a finite field $\mathbb{F}_q$ (whose characteristic we suppose is prime to $p$).
Fix a separable closure $\overline{F}$ of $F$ and denote the absolute Galois group by $G_F$.
If $S$ is a finite set of finite places of $F$, let $F_S$ denote the maximal extension of $F$ unramified outside of $S$ with Galois group $G_{F,S} = \Gal(F_S/F)$.
We let $S_p$ denote the (possibly empty) set of places of $F$ dividing $p$.

Let $\bar{\rho}: G_{F} \to \hat{G}(k)$ be an absolutely $\hat{G}$-irreducible representation, continuous with respect to the usual profinite topology on $G_{F}$ and discrete topology on $\hat{G}(k)$, and unramified outside of fixed finite set $S$ of finite places of $F$ (so that we can, and will, view $\bar{\rho}$ as a representation of $G_{F,S}$).
Recall that absolutely $\hat{G}$-irreduciblity means that the image of $G_F$ is not contained inside (the $\Fpbar$-points of) any proper parabolic subgroup of $\hat{G}_{\Fpbar}$.
By \cite[Theorem 16.4]{Richardson1988ConjugacyCO}, this implies that $Z_{\hat{G}}(\bar{\rho}(G_F))$ is finite modulo $Z(\hat{G})$, and hence that the invariants of $\gk$ under the adjoint action of $G_F$ via $\bar{\rho}$ are exactly $\mathfrak{z}(\gk)$.
We let $\gk^0$ denote the Lie algebra of $\hat{G}_k^{\der}$, viewed as a $G_{F}$-module via the adjoint action under $\bar{\rho}$.
Since we are in very good characteristic for $\hat{G}$, we in fact have a decomposition $\gk = \gk^0 \oplus \mathfrak{z}(\gk)$ as $G_{F}$-modules.
By extending $k$ if necessary, we assume throughout this section that for every semisimple element $\bar{g}$ in the image of $\bar{\rho}$ there exists a split maximal torus $T$ defined over $k$ such that $\bar{g} \in T(k)$ and that $\Lie T$ contains an ad-regular element, in the sense of Definition \ref{ad-regular}.

If $v$ is a place of $F$, let $F_v$ denote the completion of $F$ with residue field $k(v)$ of finite cardinality $q_v$ and fix an embedding of $\overline{F}$ into a separable closure $\overline{F_v}$ of $F_v$, defining inertia and decomposition subgroups $I_v \subset G_{F_v} \subset G_F$.
We will let $\bar{\rho}_v = \restr{\bar{\rho}}{G_{F_v}}$.
If, for each $v \in S$, we are given $\Lambda_v \in \CNL_{\CalO}$ then we will let $\Lambda = \hat{\otimes}_{v \in S, \mathcal{O}} \Lambda_v$ be the completed tensor product and let $\CNL_{\Lambda}$ be the category of complete Noetherian local $\Lambda$-algebras with residue field $k$.
Let $\CGhat$ denote the cocenter of $\hat{G}$, the quotient of $\hat{G}$ by its derived group, and let $\nu: \hat{G} \to \CGhat$ be this quotient map.
Fix a continuous character $\psi: G_{F,S} \to \CGhat(\CalO)$ such that $\nu \circ \bar{\rho} = (\psi \mod \lambda)$, which may be thought of by analogy with fixing the determinant when $\hat{G} = \GL_n$.

\subsection{Deformation theory of representations}

Firstly, let us recall some of the theory of deformation problems, as in \cite{cssFLT}. 

\begin{definition}\label{deformation}
Let $\Gamma$ be a profinite group, $\bar{\rho}: \Gamma \to \hat{G}(k)$ be a continuous representation and $\Lambda \in \CNL_{\CalO}$.
\begin{itemize}
    \item Let $(\theta: \Lambda \to A)  \in \CNL_{\Lambda}$.
    A lift $\rho: \Gamma \to \hat{G}(A)$ is a continuous homomorphism satisfying $\rho \mod m_A = \bar{\rho}$.
    Such a lift $\rho$ is said to have similitude character $\psi$ if $\nu \circ \rho = \theta \circ \psi$.
    This defines the functor of lifts (also called framed deformations) $$D_{\bar{\rho}}^{\square}: \CNL_{\Lambda} \to \catname{Set}$$ sending $A \in \CNL_{\Lambda}$ to the set of lifts of similitude character $\psi$.
    \item If $A \in \CNL_{\Lambda}$, a deformation is a strict equivalence class of lifts of similitude character $\psi$, where two lifts in $D_{\bar{\rho}}^{\square}(A)$ are said to be strictly equivalent if they are conjugate by an element of $\ker(\hat{G}(A) \to \hat{G}(k))$.
    This defines a functor $$D_{\bar{\rho}}: \CNL_{\Lambda} \to \catname{Set}$$ assigning $A \in \CNL_{\Lambda}$ to the set of deformations valued in $A$.
    Note that if any representative of a strict equivalence class has similitude character $\psi$, then all of representatives do.
\end{itemize}
\end{definition}

We say a functor $D: \CNL_{\Lambda} \to \catname{Set}$ is representable if there exists $R \in \CNL_{\Lambda}$ and a natural equivalence between $D$ and the functor $h_R = \Hom_{\CNL_{\Lambda}}(R,-)$.
If $\mathcal{D},D: \CNL_{\Lambda} \to \catname{Set}$ are two functors with $\mathcal{D} \subset D$ a subfunctor, we say $\mathcal{D}$ is relatively representable to $D$ if the following condition holds: for every pair of morphisms $A \to C \xleftarrow{} B$ in $\CNL_{\Lambda}$ the diagram

\begin{center}

\begin{tikzcd}
\mathcal{D}(A \times_C B) \arrow[r] \arrow[d] & \mathcal{D}(A) \times_{\mathcal{D}(C)} \mathcal{D}(B) \arrow[d] \\
D(A \times_C B) \arrow[r]                     & D(A) \times_{D(C)} D(B)                                       
\end{tikzcd}
\end{center}
is cartesian. 
By Grothendieck's theorem on representability \cite[Proposition 3.1]{SB_1958-1960__5__369_0}, we see that if $D$ itself is representable, then $\mathcal{D}$ is representable if and only if $\mathcal{D}$ is relatively representable to $D$.
If $D$ is represented by $R$ and $\mathcal{D}$ is relatively representable then $\mathcal{D}$ will be represented by a quotient of $R$.

We will use notation similar to \cite[Section 7.1]{surfaces} in defining local and global deformation problems.
\begin{definition}[Local Deformation Problem] \label{local def problem}
Let $v \in S$.
Let $D_v^\square: \CNL_{\Lambda_v} \to \catname{Set}$ denote the functor of liftings for $\restr{\bar{\rho}}{G_{F_v}}: G_{F_v} \to \hat{G}(k)$.
A local deformation problem for $\restr{\bar{\rho}}{G_{F_v}}$ is a subfunctor $\mathcal{D}_v^\square$ of $D_v^\square$ which is relatively representable and is closed under strict equivalence.
\end{definition}

Since $G_{F_v}$ satisfies Mazur's $\Phi_p$-condition (see \cite[1.1]{mazur1989deformingGR}), the functor of lifts $D_v^\square$ is representable (see \cite[Theorem 1.2.2]{balaji2012g}) by some $R_v^{\square} \in \CNL_{\Lambda_v}$.
Hence any local deformation problem will be representable by a quotient of $R_v^{\square}$.
If $\mathcal{D}_v^\square$ is any local deformation problem, we will let $\mathcal{D}_v$ denote the functor of strict equivalence classes of $\mathcal{D}_v^\square$.
In particular, $D_v$ will denote the functor of all deformations of $\bar{\rho}_v$.

\begin{definition}[Global Deformation Problem] \label{global def problem}
A global deformation problem is a tuple
$$\Sglobdp$$
where for $v \in S$, $\mathcal{D}_v^\square$ is a local deformation problem.
For $A \in \CNL_{\Lambda}$, a lift $\rho: G_{F,S} \to \hat{G}(A)$ is of type $\mathcal{S}$ if $\nu \circ \rho$ equals the image of $\psi$ and $\restr{\rho}{G_{F_v}} \in \mathcal{D}_v^\square(A)$ (where we view $A \in \CNL_{\Lambda_v}$ via the natural map $\Lambda_v \to \Lambda$).
We let $\mathcal{D}_{\mathcal{S}}$ denote the functor sending $A \in \CNL_{\Lambda_v}$ to the set of strict equivalence classes of $A$-valued lifts of type $\mathcal{S}$.
\end{definition}

\begin{definition}[$T$-framed deformations] \label{T-framed}
Let $\mathcal{S} = (\bar{\rho},S,\{\Lambda_v\}_{v \in S}, \psi, \{\mathcal{D}_v\}_{v \in S})$ be a global deformation problem and let $T \subset S$.
If $T \subset S$ and $A \in \CNL_{\Lambda}$, a $T$-framed lift of type $\mathcal{S}$ is a tuple $(\rho,\{\gamma_v\}_{v \in T})$ where $\rho: G_{F,S} \to \hat{G}(A)$ is a lift of type $\mathcal{S}$ and $\gamma_v \in \ker(\hat{G}(A) \to \hat{G}(k))$ for every $v \in T$.
Two $T$-framed lifts $(\rho,\{\gamma_v\}_{v \in T})$ and $(\rho',\{\gamma'_v\}_{v \in T})$ are said to be strictly equivalent if there exists $\alpha \in \ker(\hat{G}(A) \to \hat{G}(k))$ such that $\rho' = \alpha \rho \alpha^{-1}$ and $\gamma'_v = \alpha\gamma_v$ for every $v \in T$.
A strict equivalence class of $T$-framed lifts is called a $T$-framed deformation and the functor sending $A$ to the set of $T$-framed deformations is denoted $\mathcal{D}_\mathcal{S}^T$.
\end{definition}

\begin{proposition} \label{representable}
Let $\Sglobdp$ be a global deformation problem.
The functor $\mathcal{D}_\mathcal{S}$ is representable by an object $R_{\mathcal{S}}$ of $\CNL_{\Lambda}$.
If $T \subset S$ then the functor $\mathcal{D}_\mathcal{S}^T$ is representable by an object $R^T_{\mathcal{S}}$ of $\CNL_{\Lambda}$.
\end{proposition}
\begin{proof}
This can be proved in the same way as the proof of \cite[Proposition 9.2]{patrikis2015}, since $\gk^{\bar{\rho}(G_{F,S})} = \mathfrak{z}(\gk)$.
\end{proof}

Now suppose $v \in S \setminus S_p$ is a place and $\mathcal{D}_v^\square: \CNL_{\Lambda_v} \to \catname{Set}$ is a local deformation problem.
As explained in \cite[7.2.6]{neukirch2013cohomology}, local Tate duality gives a perfect pairing
\begin{align*}
    H^1(G_{F_v},\gk^0) \times H^1(G_{F_v},\gkv) \to H^2(G_{F_v}, \mu),
\end{align*}
where $\mu = \bigcup_{n \in \mathbb{N}} \mu_n(\overline{F_v})$ and
\begin{itemize}
    \item if $F$ is a number field then $H^2(G_{F_v}, \mu) \cong \mathbb{Q}/\mathbb{Z}$,
    \item if $F$ has positive characteristic $l$, then $H^2(G_{F_v}, \mu) \cong \bigcup_{l \nmid n} \frac{1}{n}\mathbb{Z}/\mathbb{Z}$.
\end{itemize}
We will define the Selmer conditions for $\mathcal{D}_v$, which will be subspaces $\mathcal{L}_v \subset H^1(G_{F_v},\gk^0)$ and $\mathcal{L}_v^\perp \subset H^1(G_{F_v},\gkv)$ dual to each other under the local Tate pairing.

\begin{definition} \label{selmer condition}
Take $v$ and $\mathcal{D}_v^\square$ as above and let $\mathcal{D}_v$ be the arising functor of deformations.
Let $\mathcal{D}_v(k[\epsilon])$ be the tangent space of the functor $\mathcal{D}_v$.
The Selmer condition, $\mathcal{L}_v$, is the subspace given by the image of $\mathcal{D}_v(k[\epsilon])$ in $H^1(G_{F_v}, \gk^0)$ under the isomorphism $D_v(k[\epsilon]) \cong H^1(G_{F_v}, \gk^0)$.
The dual Selmer condition, $\mathcal{L}_v^\perp \subset H^1(G_{F_v}, \gkv)$, is the annihilator of $\mathcal{L}_v$ under the Tate pairing.
\end{definition}
Note that any local deformation problem defines a subspace of $H^1(G_{F_v}, \gk^0)$ rather than just $H^1(G_{F_v}, \gk)$, since we have fixed the similitude character $\psi$ (and hence the arising cocyles will have trivial image under projection to $\Lie Z(\hat{G})$).
We can now define the Selmer groups arising from a global deformation problem.
\begin{definition} \label{selmer group}
Let $\Sglobdp$ be a global deformation problem and let $T \subset S$ with $S_p \subset T$.
For $v \in S \setminus T$ let $\mathcal{L}_v$ and $\mathcal{L}_v^\perp$ be the subspaces arising from Definition \ref{selmer condition} applied to $\mathcal{D}_v^\square$.
The Selmer groups for $\mathcal{S}$ and $T$ are given by
\begin{align*}
H^1_{\mathcal{S},T}(\gk^0) &= \ker\left(H^1(G_{F,S},\gk^0) \to \prod_{v \in S \setminus T}\frac{H^1(G_{F_v},\gk^0)}{\mathcal{L}_v} \times \prod_{v \in T} H^1(G_{F_v},\gk^0)\right) \\
H^1_{\mathcal{S}^\perp,T}(\gkv) &= \ker\left(H^1(G_{F,S},\gkv) \to \prod_{v \in S \setminus T}\frac{H^1(G_{F_v},\gkv)}{\mathcal{L}_v^\perp}\right).
\end{align*}
\end{definition}

\subsection{Taylor--Wiles places}\label{section tw places}

We begin this section by giving a modified definition of what it means to be a Taylor--Wiles place.

\begin{definition}[Taylor--Wiles place] \label{TW place}
We say a finite place $v \not\in S$ of $F$ is a Taylor--Wiles place if $q_v \equiv 1 \mod p$ and $\bar{g}_v := \bar{\rho}(\phi_v) \in \hat{G}(k)$ is semisimple where $\phi_v \in G_{F_v}$ is any choice of lift of (geometric) Frobenius element.
Such a place is said to be of level $n \geq 1$ if $q_v \equiv 1 \mod p^n$.
\end{definition}

In \cite[Definition 4.1]{thorne2012}, there is no insistence of Taylor--Wiles places having image of Frobenius with a multiplicity one eigenvalue, as was the case for previous work in the $\hat{G} = \GL_n$ case.
As explained in the introduction, Definition \ref{TW place} is an attempt to make an analogous generalization for a general reductive group $\hat{G}$, by allowing $v$ for which $\bar{g}_v$ is only semisimple rather than regular-semisimple.
It is therefore necessary for us impose a deformation condition at such Taylor--Wiles places $v$ to ensure that the image of inertia subgroup at $v$ is still valued in a torus, as is automatic in the regular-semisimple case.

If $v$ is a Taylor--Wiles place, let $\phi_v \in G_{F_v}$ be a fixed choice of lift of the Frobenius element at $v$.
Let $\bar{g}_v$ be as in Definition \ref{TW place} and let $M_{\bar{g}_v} := Z_{\hat{G}_k}(\bar{g}_v)$ denote the centralizer of $\bar{g}_v$ in $\hat{G}_k$.
These are independent of the choice of $\phi_v$ since $\bar{\rho}$ is unramified at $v$.

\begin{definition}[Taylor--Wiles datum]\label{TW datum}
A Taylor--Wiles datum is a pair
$$
(Q,\{(\hat{T}_v,\hat{B}_v)\}_{v \in Q})
$$
where $Q$ is a set of Taylor--Wiles places and, for each $v \in Q$, $\hat{T}_v \subset \hat{G}_k$ is a split maximal torus defined over $k$ containing $Z(M_{\bar{g}_v})^\circ$ and $\hat{B}_v$ is a Borel subgroup of $\hat{G}_k$ containing $\hat{T}_v$.
\end{definition}

Fixing a maximal torus and Borel subgroup can be seen as analogous to fixing an ordering on the eigenvalues of $g_v$.
Now let $(Q,\{(\hat{T}_v,\hat{B}_v)\}_{v \in Q})$ be a Taylor--Wiles datum.
We will set $\Lambda_v = \CalO$ for $v \in Q$ and define a local deformation problem $\DvTW$ for $v \in Q$.

\begin{definition}\label{deformation_condition}
Let $v \in Q$ be a Taylor--Wiles place.
Define a subfunctor $\DvTW \subset D_v^\square$ whose image for $A \in \CNL_\CalO$ is given by those lifts $\rho_v: G_{F_v} \to \hat{G}(A)$ for which $\rho_v(I_v) \subset Z(M_{g_v})^\circ(A)$.
Here $g_v = \rho_v(\phi_v)$ and $M_{g_v}$ is the subgroup constructed in Theorem \ref{Mg}.
\end{definition}

\begin{remark}
We make the following two remarks on Definition \ref{deformation_condition}.
\begin{enumerate}
    \item Definition \ref{deformation_condition} is independent of the choice of lift of Frobenius, $\phi_v$, by Lemma \ref{change_of_g}.
    \item For $\rho_v \in \DvTW(A)$ the representation has image in $M_{g_v}(A)$ since $G_{F_v}$ is topologically generated by $\phi_v$ and inertia.
\end{enumerate}
\end{remark}

We need to show that $\DvTW$ is actually a functor and is relatively representable to the functor of all lifts.

\begin{lemma} \label{rel rep}
 Let $v$ be a Taylor--Wiles place and $\rho: G_{F_v} \to \hat{G}(A)$ be a lift of $\bar{\rho}_v$ in $\DvTW(A)$. Then
 \begin{enumerate}
     \item \label{i} if $\phi: A \to B$ is a morphism in $\CNL_{\CalO}$ and $\rho': G_{F_v} \to \hat{G}(B)$ denotes the composition of $\rho$ with $\hat{G}(\phi): \hat{G}(A) \to \hat{G}(B)$ then $\rho'$ lies in $\DvTW(B)$;
     \item \label{ii} if $A \to C$ and $B \to C$ are two morphisms in $\CNL_{\CalO}$ and $\rho_B \in \DvTW(B)$ with $\rho$ and $\rho_B$ inducing the same lift to $C$, then the arising representation to $\hat{G}(A \times_C B) \cong \hat{G}(A) \times_{\hat{G}(C)} \hat{G}(B)$ lies in $\DvTW(A \times_C B)$.  
\end{enumerate}
\end{lemma}

\begin{proof}
For \ref{i}, Lemma \ref{change_of_ring} immediately shows that $\rho'(I_v) \subset Z(M_{\rho'(\phi_v)})^\circ(B)$ and hence $\rho'$ lies in $\DvTW(B)$.
Let $D$ denote the fibre product $A \times_C B$.
For \ref{ii} it will suffice to show that under the isomorphism $\hat{G}(D) \cong \hat{G}(A) \times_{\hat{G}(C)} \hat{G}(B)$ we have $Z(M_{g_{D}})^\circ(D)$ is mapped to $Z(M_{g_A})^\circ(A) \times_{Z(M_{g_C})^\circ(C)} Z(M_{g_B})^\circ(B)$, where $g_R$ denotes the image of $\phi_v$ for the aforementioned representations with values in $\hat{G}(R)$.
By Lemma \ref{change_of_ring} again we have $Z(M_{g_{D}})^\circ(A) = Z(M_{g_A})^\circ(A)$, and similarly for $B,C$. So denoting $Y = Z(M_{g_{D}})^\circ$ we need $Y(D) = Y(A) \times_{Y(C)} Y(B)$, which indeed holds.
\end{proof}

From this we deduce $\DvTW$ is relatively representable to $D_v$ and hence the following corollary.

\begin{corollary}
The functor $\DvTW$ defines a local deformation problem, represented by an algebra $R_{v}^{\textup{TW}} \in \CNL_{\CalO}$.
\end{corollary}

Now let $\mathcal{S} = (\bar{\rho},S,\{\Lambda_v\}_{v \in S}, \psi, \{\mathcal{D}_v\}_{v \in S})$ be a global deformation problem.
We can form the augmented global deformation problem
$$\mathcal{S}_Q = (\bar{\rho},S \cup Q,\{\Lambda_v\}_{v \in S \cup Q}, \psi, \{\mathcal{D}_v\}_{v \in S \cup Q}),$$ where for every $v \in Q$ we take $\mathcal{D}_v = \DvTW$ (and $\Lambda_v = \CalO$, as before).
Hence for any $T \subset S$, the functors $\mathcal{D}_{\mathcal{S}_Q}$ and $\mathcal{D}_{\mathcal{S}_Q}^T$ are representable by respective objects $R_{\mathcal{S}_Q}, R_{\mathcal{S}_Q}^T \in \CNL_{\Lambda}$ by Proposition \ref{representable}.

Having now shown our deformation problem is representable, we can explore the Selmer conditions $\mathcal{L}_v$ and $\mathcal{L}_v^\perp$ arising from our local deformation problem $\DvTW$ at $v \in Q$.
From this we can find conditions for existence of sets of Taylor--Wiles places with generating sets for the universal deformation rings of our augmented deformation problems as small as possible, by controlling the arising dual Selmer groups.

For $v \in Q$, $\mathcal{L}_v$ is given by those classes in $H^1(G_{F_v},\gk^0)$ whose restriction to inertia lies in $\Lie Z(M_{\bar{g}_v})$. 
Indeed, we have by Corollary \ref{inertial_selmer} that the condition of a deformation to $k[\epsilon]$ defining a cocycle whose restriction to inertia lies in $\Lie Z(M_{\bar{g}_v})$ is equivalent to a representative lift $\rho$ having image of inertia in $Z(M_{\rho(\phi_v)})(k[\epsilon])$.

If $M$ is any $G_{F_v}$-module write $H^1_{\text{ur}}(G_{F_v},M) = \ker(\res: H^1(G_{F_v},M) \to H^1(I_v,M))$.
Let $N_v$ denote the $G_{F_v}$-submodule $\Lie Z(M_{\bar{g}_v}) \cap \gk^0$.
Then we may write
\begin{equation}
    \mathcal{L}_v = H^1_{\text{ur}}(G_{F_v},\gk^0) + H^1(G_{F_v},N_v).
\end{equation}
Indeed, if $\xi \in \mathcal{L}_v$ then we can choose $\gamma \in H^1(G_{F_v},N_v)$ with $\res(\gamma) = \res(\xi)$ (as $N_v$ is a trivial $G_{F_v}$-module and so any homomorphism $I_v \to N_v$ can be extended to $G_{F_v}$ by local class field theory). 
Then $\xi = (\xi - \gamma) +\gamma$ shows $\xi \in H^1_{\text{ur}}(G_{F_v},\gk^0) + H^1(G_{F_v},N_v)$.

We will further simplify this to a direct sum of terms whose dimensions can be computed.
Consider the following decomposition of $\gk$ into $G_{F_v}$-submodules:
\begin{equation} \label{gk decomposition}
        \gk^0 = N_v \oplus M_v \oplus L_{1,v}
\end{equation}
where $\gk = L_{0,v} \oplus L_{1,v}$ is the $\Ad(\bar{g}_v)$-invariant decomposition of $\gk^0$ into its eigenvalue $1$ subspace and complement, and $M_v$ is any choice of vector space complement to $N_v \subset L_{0,v}$ (necessarily a $G_{F_v}$-module, since $L_{0,v}$ is a trivial $G_{F_v}$-module).
Note then that since $Z(\hat{G}_k) \subset M_{\bar{g}_v}$ we have $\Lie M_{\bar{g}_v} / \Lie Z(\hat{G}_k) \cong L_{0,v}$.
We now claim that $H^1(G_{F_v},L_{1,v}) = 0$.
Indeed, by the Euler characteristic formula \cite[Lemma 2.9]{adt} we have $h^1(G_{F_v},L_{1,v}) = h^0(G_{F_v},L_{1,v}) + h^2(G_{F_v},L_{1,v})$. We have in turn by Tate duality that $h^2(G_{F_v},L_{1,v}) = h^0(G_{F_v},L_{1,v}^\vee(1))$.
Both $L_{1,v}$ and its Tate dual have trivial $\phi_v$-invariants (as the residual $p$-adic cyclotomic character is trivial on $G_{F_v}$, since $q_v \equiv 1 \mod p$) and so indeed $h^1(G_{F_v},L_{1,v}) = 0$.
So we can rewrite 
\begin{equation} \label{Lv decomposition}
    \mathcal{L}_v = H^1(G_{F_v},N_v) \oplus H^1_{\text{ur}}(G_{F_v},M_v).
\end{equation}

Now let $z = \dim Z(\hat{G}_k)$, $l_v = \dim_k(L_{0,v}) = \dim M_{\bar{g}_v} - z$ and $n_v = \dim_k(N_v) = \dim Z(M_{\bar{g}_v}) - z$.
We will compute $\dim_k(\mathcal{L}_v)$.
We will firstly show that $h^1(G_{F_v},k) = 2$ and $h^1_{\text{ur}}(G_{F_v},k) = 1$. The former follows by the local Euler characteristic formula and Tate duality:
\begin{align*}
h^1(G_{F_v},k) &= h^0(G_{F_v},k) +h^2(G_{F_v},k) \\
                &= 1 +h^0(G_{F_v},k^\vee(1)) = 2
\end{align*}
since the residual $p$-adic cyclotomic character is trivial.
The latter is exactly $\dim_k(\Hom(G_{F_v}/I_v,k))$, which equals 1 since any such homomorphism is determined by the image of Frobenius, of which all may occur.
Since $\mathcal{L}_v = H^1(G_{F_v},N_v) \oplus H^1_{\text{ur}}(G_{F_v},M_v)$ for trivial $G_{F_v}$-modules $N_v$ and $M_v$ of respective dimensions $n_v$ and $l_v-n_v$, it follows that 
\begin{equation} \label{dimension of Lv}
   \dim_k \mathcal{L}_v = l_v + n_v. 
\end{equation}

We will now consider the dual Selmer conditions $\mathcal{L}_v^\perp$, given by annihilators of $\mathcal{L}_v$ under the Tate pairing.
If $M$ is any $G_{F_v}$-module then the annihilator of $H^1_{\text{ur}}(G_{F_v},M)$ in $H^1(G_{F_v},M^\vee(1))$ under this pairing is given by the subspace $H^1_{\text{ur}}(G_{F_v},M^\vee(1))$.
We therefore see from equations \ref{gk decomposition} and \ref{Lv decomposition} that
$$\mathcal{L}_v^\perp = 
H^1\left(G_{F_v},M_v^\vee(1) \oplus L_{1,v}^\vee(1)\right) \cap \left(H^1_{\text{ur}}(G_{F_v}, M_v^\vee(1)) \oplus 
H^1(G_{F_v}, N_v^\vee(1) \oplus L_{1,v}^\vee(1)\right).$$
Since we already saw that $H^1(G_{F_v},L_{1,v}^\vee(1)) = 0$, we conclude that 
\begin{equation} \label{dual selmer condition}
\mathcal{L}_v^\perp = H^1_{\text{ur}}(G_{F_v},M_v^\vee(1)).
\end{equation}
Having explicitly computed the dual Selmer condition for a Taylor--Wiles deformation, we obtain a better understanding of the arising dual Selmer groups.

\begin{lemma} \label{dual selmer ur in Q}
Let $\mathcal{S} = (\bar{\rho},S,\{\Lambda_v\}_{v \in S}, \psi, \{\mathcal{D}_v\}_{v \in S})$ a global deformation problem and $T \subset S$ with $S_p \subset T$.
Let $Q$ be any set of Taylor--Wiles places and let $\mathcal{S}_Q$ denote the augmented global deformation problem.
Then
$$
H^1_{\mathcal{S}_Q^\perp,T}(\gkv) \cong \ker\left(H^1(G_{F,S},\gkv) \to \prod_{v \in Q} H^1_{\text{ur}}(G_{F_v},N_v^\vee(1)) \times \prod_{v \in S\setminus T} \frac{H^1(G_{F_v},\gkv)}{\mathcal{L}_v^\perp}\right).
$$
\end{lemma}
\begin{proof}
Recall by Definition \ref{selmer group} that
$$H^1_{\mathcal{S}_Q^\perp,T}(\gkv) = \ker\left(H^1(G_{F,S \cup Q}, \gkv) \to \prod_{v \in Q \cup S \setminus T} \frac{H^1(G_{F_v},\gkv)}{\mathcal{L}_v^\perp}\right),$$
where for $v \in Q$ we saw that $\mathcal{L}_v^\perp = H^1_{\text{ur}}(G_{F_v},M_v^\vee(1)) \leq H^1_{\text{ur}}(G_{F_v},\gkv)$.
We therefore see that any class $[\psi] \in H^1_{\mathcal{S}_Q^\perp,T}(\gkv)$ has $\res_{G_{F_v}} [\psi] \in H^1_{\text{ur}}(G_{F_v},\gkv)$.

We will firstly show that any class $[\psi] \in H^1_{\mathcal{S}_Q^\perp,T}(\gkv)$ lies in the subgroup $H^1(G_{F,S},\gkv) \leq H^1(G_{F,S \cup Q},\gkv)$. 
Let $Q = \{v_1,\ldots,v_r\}$ and for $0 \leq t \leq r$ let $Q_t = \{v_1,\ldots,v_t\}$
We show that if $[\psi] \in H^1(G_{F,S \cup Q_t},\gkv)$ for some $t \geq 1$ then in fact $[\psi] \in H^1(G_{F,S \cup Q_{t-1}},\gkv)$. To simplify notation, let $v$ denote $v_t$.
Inflation-restriction gives an exact sequence
\begin{align*}
    0 \to H^1(G_{F,S \cup Q_{t-1}},\gkv) \to H^1(G_{F,S \cup Q_t},\gkv) \to H^1(I_{v,t},\gkv),
\end{align*}
where $I_{v,t}$ is an inertia subgroup at a place of $F_{S \cup Q_{t-1}}$ above $v$ in $F_{S \cup Q_t}$.
Since $I_{v,t}$ may be identified as a quotient of $I_v$ (noting that, as $F_{S \cup Q_{t-1}}/F$ is unramified at $v$, we see that $I_{v,t}$ is equal to the inertia subgroup at $v$ in $F_{S \cup Q_t}$), we have an inclusion $H^1(I_{v,t},\gkv) \hookrightarrow H^1(I_v,\gkv)$.
Under this inclusion, $\res_{I_{v,t}} [\psi]$ has image $0$, since this coincides with $\res_{I_v} [\psi] = 0$ and we saw that $[\psi]$ defines an unramified class at $v$.
Thus $[\psi]$ naturally lies in $H^1(G_{F,S \cup Q_{t-1}},\gkv)$ completing the induction.

We next claim that if $v \in Q$, then $$\frac{H^1_{\text{ur}}(G_{F_v},\gkv)}{H^1_{\text{ur}}(G_{F_v},\gkv) \cap \mathcal{L}_v^\perp} \cong H^1_{\text{ur}}(G_{F_v},N_v^\vee(1)).$$
The result would then follow, since we already saw that any class in $H^1_{\mathcal{S}_Q^\perp,T}(\gkv)$ is unramified at $v$.
The claim follows immediately on writing $H^1(G_{F_v},\gkv) = H^1(G_{F_v},N_v^\vee(1)) \oplus H^1(G_{F_v},M_v^\vee(1))$, where we have recalled that $H^1(G_{F_v}, L_v^\vee(1)) = 0$.
\end{proof}

We now consider conditions under which we can show existence of sets of Taylor--Wiles places of level $N$ with vanishing dual Selmer group.
Our first notion is a modification of the definition of an abundant subgroup \cite[Definition 5.18]{thorne2019} to allow for our images of Frobenius elements to be only semisimple elements of $\hat{G}(k)$.

\begin{restatable}{definition}{adequate}
\label{adequate}
We say $H \leq \hat{G}(k)$ is $\hat{G}$-adequate if the following conditions hold:
\begin{enumerate}
    \item The following groups vanish
    \begin{enumerate}
        \item $H^0(H,\gk^{0,\vee})$ \label{H0 vanish}
        \item $H^1(H,k)$ \label{homs to k vanish}
        \item $H^1(H,\gk^{0,\vee})$
    \end{enumerate}
    \item For every non-zero simple $k[H]$-submodule $W \leq \gk^{0,\vee}$ there exists a semisimple element $h \in H$ such that for some $w \in W$ and $z \in \Lie Z(M_{h}) \cap \gk^0$ we have $w(z) \neq 0$. \label{submodule condition}
\end{enumerate}
\end{restatable}

When $M_{h}$ is connected we may write $\Lie Z(M_{h}) \cap \gk^0 = \mathfrak{z}((\gk^0)^h)$ by Lemma \ref{lie algebra mg commute center}.
If moreover $h \in H$ is regular semisimple we recover the contribution of $h$ in the original definition of $\hat{G}$-abundant (when $\hat{G}$ is semisimple and simply connected): 
the centralizer of $h$ is a maximal torus and $\mathfrak{z}((\gk^0)^h) = (\gk^0)^h$, so taking $z \in (\gk^0)^h$ and $w \in W$ with $w(z) \neq 0$ shows the restriction of $w$ to $W^h$ is non-zero.

The following lemma has an identical proof to that of \cite[Lemma 7.5.9]{surfaces}.
\begin{lemma} \label{lemma for tw primes}
Let $n \geq 1$.
Suppose that $H^0(\bar{\rho}(G_{F(\zeta_{p^n})}), \gk^{0,\vee}) = 0$.
Let $L$ be the fixed field of the module $\gk^{0,\vee}$.
Suppose that one of the following conditions holds
\begin{enumerate}
    \item $\zeta_p \not\in L$
    \item $H^1(\bar{\rho}(G_{F(\zeta_{p^n})}), \gk^{0,\vee}) = 0.$
\end{enumerate}
Then $H^1(\Gal(L(\zeta_{p^n})/F), \gk^{0,\vee}(1)) = 0.$
\end{lemma}

Let us now recall a form of the Chebotarev density theorem, which we will use to relate the existence of desired semisimple elements to a Frobenius element at some place of $F$.

\begin{theorem}[Chebotarev density theorem]
Let $T$ be a finite set of places of $F$. Then the set of Frobenius elements of the form $\{\Frob_w\}$ for $w$ a place of the maximal extension of $F$ unramified outside $T$ not dividing $T$ is a dense subset of $G_{F,T}$.
\end{theorem}

We are now ready to prove the standard result on the existence of sets of Taylor--Wiles places with vanishing dual Selmer group under suitable conditions on the image of $\bar{\rho}$.

\begin{proposition}\label{TW places existence}
Let $\mathcal{S} = (\bar{\rho},S,\{\Lambda_v\}_{v \in S}, \psi, \{\mathcal{D}_v\}_{v \in S})$ be a global deformation problem and let $T \subset S$ with $S_p \subset T$.
Let $n \geq 1$.
Suppose that $\bar{\rho}(G_{F(\zeta_{p^n})}) \subset \hat{G}(k)$ satisfies
condition \ref{submodule condition} of Definition \ref{adequate}, and if $\zeta_p$ is in the fixed field of $\gk^{0,\vee}$ then suppose further that $H^1(\bar{\rho}(G_{F(\zeta_{p^n})}), \gk^{0,\vee}) = 0$.
Let $q \geq h^1(G_{F,S},\gkv)$.
Then there exists a set $Q$ of Taylor--Wiles places of level $n$ such that 
\begin{itemize}
    \item $H^1_{\mathcal{S}_Q^\perp,T}(\gkv) = 0$
    \item $\#Q = q.$
\end{itemize}
\end{proposition}

\begin{proof}
By Lemma \ref{dual selmer ur in Q}, it will suffice to find for each non-zero class $[\psi] \in H^1(G_{F,S},\gkv)$ infinitely many Taylor--Wiles places $v$ of level $n$ with $[\psi]$ having non-zero image in $H^1_{\text{ur}}(G_{F_v},N_v^\vee(1))$.
Indeed, suppose that $t < q$ and that we have inductively found Taylor--Wiles places $Q_t = \{v_1,\ldots,v_t\}$ of level $n$ for which $h^1_{\mathcal{S}_{Q_t}^\perp,T}(\gkv) \leq q-t$.
If there still exists $0 \neq [\psi] \in H^1_{\mathcal{S}_{Q_t}^\perp,T}(\gkv)$ then we can find a Taylor--Wiles place $v_{t+1} \not\in Q_t$ of level $n$ for which $[\psi]$ has non-zero image in $H^1_{\text{ur}}(G_{F_v},N_{v_{t+1}}^\vee(1))$.
It then follows that $h^1_{\mathcal{S}_{Q_{t+1}}^\perp,T}(\gkv) \leq q-t-1$.
So inductively we can always find some $t \leq q$ and $Q_t$ for which $h^1_{\mathcal{S}_{Q_t}^\perp,T}(\gkv) = 0$, and then we are done on extending $Q_t$ by any other $q-t$ choices of Taylor--Wiles places of level $n$.

So let $[\psi] \in H^1(G_{F,S},\gkv)$ be a non-zero class and suppose that we can always find $\sigma \in G_{F(\zeta_{p^n})}$ such that $\bar{\rho}(\sigma)$ is semisimple and $\psi(\sigma)(N) \neq 0$, where $N = \Lie Z(Z_{\hat{G}_k}(\bar{\rho}(\sigma))) \cap \gk^0$.
We claim that by the Chebotarev density theorem we can find infinitely many places $v \not\in S$ with $\bar{\rho}(\sigma) = \bar{\rho}(\Frob_v)$, $\psi(\sigma) = \psi(\Frob_v)$ and $\Frob_v \in G_{F(\zeta_{p^n})}$. 
Such a $v$ would be a Taylor--Wiles place of level $n$ with $N_v = N$ as above and with $[\psi]$ having non-zero image in $H^1(G_{F_v},N_v^\vee(1))$.
Indeed, consider the non-empty open subset $U$ of $G_{F,S}$ given by the intersection of $\bar{\rho}^{-1} \{\bar{\rho}(\sigma) \}$, $\Gal(F_S/F(\zeta_{p^n}))$ and $\psi^{-1} \{\psi(\sigma)\}$. 
The Frobenius elements are dense in $G_{F,S}$ by Chebotarev, so there exists (infinitely many) Frobenius elements $\Frob_v$ satisfying $\bar{\rho}(\Frob_v) = \bar{\rho}(\sigma)$, $\psi(\Frob_v) = \psi(\sigma)$ and $\Frob_v \in G_{F(\zeta_{p^n})}$.

We now show the existence of such a $\sigma$.
Let $F_n = F(\zeta_{p^n})$ and let $H = \bar{\rho}(G_{F_n})$. 
Let $L$ be the fixed field of $\gk^{0,\vee}(1)$ and let $L_n = L(\zeta_{p^n})$.
We will show that the restriction of $\psi$ to $G_{L_n,S}$ is a non-zero homomorphism which we denote $f: G_{L_n,S} \to \gkv$. 
This would follow immediately from injectivity of the restriction map $$H^1(G_{F,S},\gkv) \to H^1(G_{L_n,S},\gkv)^{G_F}.$$
By inflation-restriction, the kernel of this map is given by
$$H^1(\Gal(L(\zeta_{p^n})/F), \gk^{0,\vee}(1))
$$
which therefore vanishes by Lemma \ref{lemma for tw primes}.

Thus $f(G_{L_n})$ is non-zero and moreover is a $G_{F_n,S}$-module, as for $\sigma \in G_{L_n}$ and $\tau \in G_{F_n,S}$ we have
\begin{align*}
    \psi(\tau \sigma \tau^{-1}) &= \psi(\tau) + \tau \psi(\sigma\tau^{-1}) \\
    &= \psi(\tau) + \tau(\psi(\sigma)+ \sigma \psi(\tau^{-1})) \\
    &= \psi(\tau) + \tau \psi(\sigma) + \tau \psi(\tau^{-1}) \\
    &= \psi(\id) + \tau \psi(\sigma) \\
    &= \tau \psi(\sigma).
\end{align*}
Hence we can take a non-zero simple $H$-submodule $W \leq f(G_{L_n,S})$. Now use that $H$ satisfies condition \ref{submodule condition} of Definition \ref{adequate} to find a $\sigma_0 \in H$ with $\bar{\rho}(\sigma_0)$ semisimple such that there exists a $w \in W$ and $z \in \Lie Z(Z_{\hat{G}_k}(\bar{\rho}(\sigma_0))) \cap \gk^0$ with $w(z) \neq 0$. 
We may write $w = \psi(\sigma)$ for some $\sigma \in G_{L_n,S}$. 
We claim that there exists $\tau \in G_{L_n}$ with $\psi(\tau \sigma_0)(\Lie Z(Z_{\hat{G}_k}(\bar{\rho}(\sigma_0))) \cap \gk^0) \neq 0$. If $\psi(\sigma_0)(z) \neq 0$ then we're done.
Otherwise $\psi(\tau \sigma_0) = \psi(\tau) + \psi(\sigma_0)$. 
So choose $\tau$ to be such that $\psi(\tau) = w$ and let $\sigma = \tau \sigma_0$.
Then $\psi(\sigma)(z) = \psi(\tau)(z) = w(z) \neq 0$, and $\Lie Z(Z_{\hat{G}_k}(\bar{\rho}(\sigma_0))) \cap \gk^0 = \Lie Z(Z_{\hat{G}_k}(\bar{\rho}(\sigma))) \cap \gk^0$, so we're done. 
\end{proof}

The hypotheses of Proposition \ref{TW places existence} motivates the following definition, which is a weakening of the notion of vast Galois representations of \cite[Definition 7.5.6]{surfaces}.

\begin{definition} \label{reasonable}
A Galois representation
$$
\bar{\rho}: G_F \to \hat{G}(k)
$$
is said to be $\hat{G}$-reasonable if the following conditions hold for all sufficiently large $n$
\begin{enumerate}
    \item $\bar{\rho}(G_{F(\zeta_{p^n})}) \subset \hat{G}(k)$ satisfies conditions \ref{H0 vanish} and \ref{submodule condition} of Definition \ref{adequate}
    \item if $\zeta_p$ is in the fixed field of $\gk^{0,\vee}$ then $H^1(\bar{\rho}(G_{F(\zeta_{p^n})}), \gk^{0,\vee}) = 0$.
   \end{enumerate}
\end{definition}

Note that while Definition \ref{adequate} is a condition on a subgroup, Definition \ref{reasonable} is a condition on a representation $\bar{\rho}$.
We can relate these notions as follows

\begin{lemma} \label{adequate implies reasonable}
Suppose that $H^1(\bar{\rho}(G_{F(\zeta_p)}), k) = 0$.
Then for every $n \geq 1$, $\bar{\rho}(G_{F(\zeta_p)}) = \bar{\rho}(G_{F(\zeta_{p^n})})$.
In particular, if $\bar{\rho}(G_{F(\zeta_p)})$ is $\hat{G}$-adequate then $\bar{\rho}$ is $\hat{G}$-reasonable.
\end{lemma}
\begin{proof}
We have that $\bar{\rho}(G_{F(\zeta_{p^n})})$ is a normal subgroup of $\bar{\rho}(G_{F(\zeta_p)})$ of $p$-power index.
If they weren't equal, there would therefore exist a non-trivial homomorphism $\bar{\rho}(G_{F(\zeta_p)}) \to \mathbb{Z}/p\mathbb{Z}$, which would contradict the hypothesis that $H^1(\bar{\rho}(G_{F(\zeta_p)}), k) = 0$.
\end{proof}

\subsection{Dual groups and the $\Delta_Q$-action}\label{dual group deltaq}

To begin this section, we will recall some properties of (pinned) root data over a general scheme from \cite{sga3} and the relationship to reductive group schemes.
Let $S$ be a scheme and $T$ a torus over $S$.
We recall the notion of a root datum in $T$ over $S$ from \cite[Exp XXII Définition 1.9]{sga3}.
This consists of
\begin{itemize}
    \item a finite subgroup scheme $\mathcal{R} \subset \Hom_{S-gp}(T,\mathbb{G}_{m,S})$
    \item a finite subgroup scheme $\mathcal{R}^\vee \subset \Hom_{S-gp}(\mathbb{G}_{m,S},T)$
    \item an isomorphism $\mathcal{R} \to \mathcal{R}^\vee$ denoted $\alpha \mapsto \alpha^\vee$
\end{itemize}
such that for every $S' \to S$
\begin{itemize}
    \item for every $\alpha \in \mathcal{R}(S')$ we have $\alpha \circ \alpha^\vee = 2$
    \item for every $\alpha,\beta \in \mathcal{R}(S')$ we have 
\begin{align*}
    \alpha - (\beta^\vee,\alpha) \beta \in \mathcal{R}(S') \quad \alpha^\vee - (\alpha^\vee,\beta) \beta^\vee \in \mathcal{R}^\vee(S').
\end{align*}
\end{itemize}
The root datum is called reduced if whenever $S' \to S$ is non-empty and $\alpha \in \mathcal{R}(S')$ we have $2 \alpha \not\in \mathcal{R}(S')$.
Suppose now that the $T$ torus is split, so that we can identify the characters and cocharacters of $T$ with finite free $\mathbb{Z}$-modules.
Then $\mathcal{R}$ (resp. $\mathcal{R}'$) can be identified with a finite subset $\Phi \subset X^*(T)$ (resp. $\Phi^\vee \subset X_*(T)$).
A set of simple roots $\Delta \subset \Phi$ is a subset such that each $\alpha \in \Phi$ can be written uniquely as a linear combination of elements of $\Delta$ with integral coefficients all of the same sign. 
A pinned root datum over $S$ is then the data of a root datum together with a system of simple roots (\cite[Exp XXIII 1.5]{sga3}).
We can denote such a pinned root datum by a tuple $$R = (X^*(T),\Phi,\Delta, X_*(T), \Phi^\vee, \Delta^\vee),$$ where $\Delta^\vee = \{\alpha^\vee : \alpha \in \Delta\}$.
We can then form the dual root datum
$$
R^\vee = ( X_*(T), \Phi^\vee, \Delta^\vee,X^*(T),\Phi,\Delta)
$$
given by interchanging the roles of $X^*(T)$ and $X_*(T)$.

Suppose now that $G$ is a reductive group scheme over $S$ together with a split maximal torus $T$.
A pinning of $G$ with respect to $T$, as in \cite[Exp XXIII Définition 1.1]{sga3} is the data of a root basis $\Delta$ of the arising set of roots together with sections $X_\alpha \in \Gamma(S,\mathfrak{g}_\alpha)^\times$ for every $\alpha \in \Delta$.
The functor assigning a pinned reductive group $G$ over $S$ to its pinned root datum $R(G)$ defines an equivalence of categories to the category of pinned reduced root data over $S$, as stated in \cite[Exp XXV Théorème 1.1]{sga3}.
The data of the sections $X_{\alpha}$ is to ensure faithfulness of this functor and does not change the arising pinned root datum.
If $B \subset G$ is a Borel subgroup containing $T$, then $B$ determines a unique choice of root basis $\Delta$ by insisting $\Delta \subset \Phi^+$, the positive roots with respect to $B$. Hence the choice of $B$ determines a pinned root datum.
We note that if we have an isomorphism of pinned root data then it is unique, by \cite[Exp XXIII Corollaire 5.5]{sga3}.
Given a reduced pinned root datum then, we can view this as a reduced pinned root datum over $\mathbb{Z}$ and hence this determines (up to isomorphism) a reductive group scheme over $\mathbb{Z}$ together with a split maximal torus and Borel subgroup containing it.

Now let $\hat{T} \subset \hat{B} \subset \hat{G}$ be a choice of split maximal torus and Borel subgroup of $\hat{G}$, determining a pinned root datum $R(\hat{G})$ and dual root datum $R(\hat{G})^\vee$.
Let $G$ be a reductive group scheme defined over either $\mathcal{O}_F$ in the case that $F$ is a number field, or $\mathbb{F}_q$ in the function field case, together with a split maximal torus $T$ and Borel subgroup $B$ with root datum $R(G)$ and suppose there exists an isomorphism $R(G) \cong R(\hat{G})^\vee$.
We see that for every such choice of split maximal torus $\hat{T}$ and Borel subgroup $\hat{B}$ containing $\hat{T}$ we get unique isomorphisms $X^*(\hat{T}) \cong X_*(T)$ and $X_*(\hat{T}) \cong X^*(T)$.
We will now suppose that we have fixed such a $\hat{T}$ and $\hat{B}$ and let $\iota$ denote this isomorphism of pinned root data.

Suppose that $\hat{P}$ is a standard parabolic subgroup of $\hat{G}$ (so $\hat{B} \subset \hat{P}$) and $\hat{L}$ is a standard Levi factor of $\hat{P}$ (so $\hat{T} \subset \hat{L}$).
Then the isomorphism of root data determines dual subgroups $P$ and $L$ of $G$, which are a standard parabolic and standard Levi subgroup respectively.
We will explain this now, following the explanation of \cite[2.3]{schmidtgsp4}.
Let $\hat{\Delta}$ be the set of simple roots of $\hat{G}$, and $\hat{\Delta}_{\hat{P}} \subset \hat{\Delta}$ be the set of simple roots corresponding to $\hat{P}$ under the bijection between subsets of simple roots and standard parabolic subgroups.
We can then form the parabolic subgroup of $P \subset G$ dual to $\hat{P}$ by defining $P$ to be the standard parabolic subgroup of $G$ corresponding to the subset of simple roots given by the image of $\hat{\Delta}^\vee_{\hat{P}}$ under $\iota$.
The dual Levi subgroup $L$ can then be taken to be the unique Levi factor containing the maximal torus $T$.
There will then be an isomorphism of pinned root data between that of $L$ and $\hat{L}$ induced by $\iota$, where the pinned root datum of $L$ (resp $\hat{L}$) is defined via the maximal torus $T$ (resp. $\hat{T}$) and Borel subgroup $B \cap L$ (resp. $\hat{B} \cap \hat{L}$).

There is also a duality induced between $Z(\hat{G})^\circ$ and $C_G$.
Recall here that the cocenter $C_G = G/G^{\der}$ is the quotient of $G$ by its derived group.
We may view $X_*(Z(\hat{G})^\circ)$ as a submodule of $X_*(\hat{T})$ and $X^*(C_G)$ as a submodule of $X^*(T)$; we claim that these submodules are identified via $\iota$.
This is immediate from the characterizations of these submodules as the annihilators of the roots of $\hat{G}$ (resp. coroots of $G$).
This description of the coroots of the connected center follows from viewing $Z(\hat{G}) = \cap_{\alpha \in \hat{\Delta}} \ker(\alpha)$, a consequence of $Z(\hat{G})$ being the kernel of the adjoint representation of $\hat{G}$.
This characterization of the cocenter, on the other hand, follows from the start of the proof of \cite[5.3.1]{conrad}.

Suppose now that $(Q,\{(\hat{T}_v,\hat{B}_v)\}_{v \in Q})$ is a Taylor--Wiles datum and let $v \in Q$.
The data of the split maximal torus $\hat{T}_v$ and Borel subgroup $\hat{B}_v$ thus determines unique isomorphisms $X^*(T) \cong X_*(\hat{T}_v)$ and $X_*(T) \cong X^*(\hat{T}_v)$ which we will denote by $\iota_v$ in both cases.
Let $\bar{g}_v = \bar{\rho}(\phi_v)$ and recall that the centralizer of $\bar{g}_v$ in $\hat{G}_k$ is a closed subgroup $M_{\bar{g}_v}$ of $\hat{G}_k$ with reductive identity component containing the torus $\hat{T}_v$.
It is not, in general, a Levi subgroup of $\hat{G}_k$ (consider, for example, a semisimple element of $\Sp_4(k)$ of order $2$ with $\ch k \neq 2$: its centralizer will be isomorphic to $\SL_2 \times \SL_2$, which is not the centralizer of any torus).
However, we can use the connected center of $M_{\bar{g}_v}$ to construct a Levi subgroup $\hat{L}_v$ of $\hat{G}_k$.

\begin{definition}[$\hat{L}_v$ and $L_v$] \label{Lv dual levi}
Let $\hat{L}_v = Z_{\hat{G}_k}(Z(M_{\bar{g}_v})^{\circ})$ denote the centralizer of the identity component of the center of $M_{\bar{g}_v}$.
This is a standard Levi subgroup of $\hat{G}_k$.
Let $L_v$ denote the dual Levi subgroup of $G$ arising from $\iota_v$.
\end{definition}

Note that $\hat{L}_v$ will always contain $M_{\bar{g}_v}$ and their connected centers will coincide.
As above, $\iota_v$ allows us to identify $X_*(Z(M_{\bar{g}_v})^{\circ}) \subset X_*(\hat{T}_v)$ with $X^*(C_{L_v}) \subset X^*(T)$.

Suppose now that $\rho_v \in \DvTW(A)$ for some $A \in \CNL_\CalO$ and let $g_v = \rho(\phi_v)$.
Then since $\rho_v$ lifts the unramified representation $\bar{\rho}_v$, we have by definition of $\DvTW$ that $\restr{\rho}{I_v}$ is valued in $\ker(Z(M_{g_v})(A) \to Z(M_{g_v})(k))$, where $M_{g_v}$ is as in Theorem \ref{Mg}.
This group is pro-$p$ and abelian, so $\restr{\rho}{I_v}$ factors through the maximal $p$-power quotient of $k(v)^\times$, viewing $k(v)^\times$ as a quotient of $I_v$.
We thus obtain a map in $\Hom(k(v)^\times, Z(M_{g_v})^{\circ}(A))$.
Given $\alpha \in X_*(C_{L_v})$ we have $\iota_v(\alpha) \in X^*(Z(M_{\bar{g}})^{\circ})$ which lifts to a unique homomorphism of $A$-tori, $Z(M_g)^{\circ} \to \mathbb{G}_{m,A}$.
We can thus define a bijection
\begin{align} \label{tensor hom adjunction}
    \Hom(k(v)^\times, Z(M_{g_v})^{\circ}(A)) &\to \Hom(X_*(C_{L_v}) \otimes k(v)^\times, A^\times) \\
    \varphi &\mapsto (\alpha \otimes x \mapsto \iota_v(\alpha)(\varphi(x))) \nonumber
\end{align}
which can be thought as an instance of tensor-hom adjunction.
Thus we obtain a map $C_{L_v}(k(v)) \to A^\times$, and it can be checked that this map factors through the maximal $p$-power quotient of $C_{L_v}(k(v))$.
Since $\restr{\psi}{I_v}$ is trivial, this map is also trivial on the image of $Z(G)(k(v))$ in $C_{L_v}(k(v))$, as $\iota_v$ induces an isomorphism $X_*(Z(G)^{\circ}) \cong X^*(C_{\hat{G}})$.
Note that since $p$ is a prime of very good characteristic, the images of $Z(G)^{\circ}(k(v))$ and $Z(G)(k(v))$ in the maximal $p$-power order quotient of $C_{L_v}(k(v))$ coincide.

We thus define $\Delta_v$ to be the maximal quotient of $C_{L_v}(k(v))$ which is of $p$-power order and for which the image of $Z(G)(k(v))$ is trivial.
From this discussion, we see that $\rho$ induces a group homomorphism $\Delta_v \to A^\times$, and it is clear that this construction is functorial in $A$.
The assignment of a representation to the $\Delta_v$-action is also constant on strict equivalence classes.
Indeed, if $\gamma \in \ker(\hat{G}(A) \to \hat{G}(k))$, set $g_v' = \Ad(\gamma)(g_v)$.
Then $\Ad(\gamma)(Z(M_g)) = Z(M_{g'})$ by Lemma \ref{change_of_g} and hence $\gamma$ will define an isomorphism $X^*(Z(M_{g_v})) \to X^*(Z(M_{g_v'}))$ which is the identity on reduction modulo $m_A$.

Associated to the Taylor--Wiles datum $(Q,\{(\hat{T}_v,\hat{B}_v)\}_{v \in Q})$, we can now set $\Delta_Q = \prod_{v \in Q} \Delta_v$, a product of cyclic $p$-groups.
The number of cyclic factors of $\Delta_Q$ is bounded by $\#Q \cdot (\rank(\hat{G}) - \rank(Z(\hat{G})^\circ))$.
Note that in previous implementations of the Taylor--Wiles method (in which $\overline{\rho}(\Frob_v)$ is required to be regular semisimple for a Taylor--Wiles prime $v$), this upper bound would further be an equality.
Suppose that $\Sglobdp$ is any global deformation problem, $T \subset S$ is any subset, and consider $\mathcal{S}_Q$, the augmented deformation problem.
If $\rho$ defines a class in $\mathcal{D}^T_{\mathcal{S}_Q}(A)$ then we obtain a map $\Delta_v \to A^\times$ for every $v \in Q$, and hence a map $\Delta_Q \to A^\times$.
In particular, we obtain a map $\Delta_Q \to (R_{\mathcal{S}_Q}^T)^{\times}$.
If we let $\mathfrak{a}_Q$ denote the ideal of $(R_{\mathcal{S}_Q}^T)^\times$ generated by the images of $\delta - 1$ for $\delta \in \Delta_Q$, we have an isomorphism $R_{\mathcal{S}_Q}^T/\mathfrak{a}_Q \to R^T_{\mathcal{S}}$.
This is clear from noting that, for $v \in Q$, a lift in $\DvTW(A)$ is unramified if and only it induces the trivial homomorphism $k(v)^\times \to Z(M_{g_v}))^\circ(A)$, which is equivalent under (\ref{tensor hom adjunction}) to the map $\Delta_v \to A^\times$ being trivial.

\subsection{Presentations of framed deformation rings in the number fields case} \label{number field def setup}

In this section let $F$ be a number field and suppose that $p > 2$.
Let $T \subset S$ with $S_p \subset T$.
Suppose that $\Sglobdp$ is a global deformation problem such that for every $v \in S \setminus T$ we have $\mathcal{D}_v^\square = D_v^\square$ and $\Lambda_v = \CalO$.
Then for $v \in S \setminus T$ we have $\mathcal{L}_v = H^1(G_{F_v},\gk^0)$.
Set $R^{T,\loc} = \hat{\otimes}_{v \in T} R_v^\square$.
This has the structure of a $\Lambda = \otimes_{v \in T} \Lambda_v$-algebra.
Now let $(Q,\{(\hat{T}_v,\hat{B}_v)\}_{v \in Q})$ be a Taylor--Wiles datum (with $Q$ possibly empty) and consider the augmented deformation problem $\mathcal{S}_Q$.
For $v \in T$ there is a natural transformation of functors $\mathcal{D}_{\mathcal{S}_Q}^T \to \mathcal{D}_v^\square$ on $\CNL_{\Lambda_v}$  given by $(\rho,\{\gamma_v\}_{v \in T}) \mapsto \gamma_v^{-1} \restr{\rho}{G_{F_v}} \gamma_v$.
Thus for each $v \in T$ we have a morphism $R_v^\square \to R_{\mathcal{S}_Q}^T$ in $\CNL_{\Lambda_v}$.
Tensoring, we obtain a map $R^{T,\loc} \to R_{\mathcal{S}_Q}^{T}$ in $\CNL_\Lambda$.
The Selmer groups in this case become
\begin{align*}
H^1_{\mathcal{S}_Q,T}(\gk^0) &= \ker\left(H^1(G_{F,S \cup Q},\gk^0) \to (\prod_{v \in Q}\frac{H^1(G_{F_v},\gk^0)}{\mathcal{L}_v} \times \prod_{v \in T} H^1(G_{F_v},\gk^0))\right) \\
H^1_{\mathcal{S}_Q^\perp,T}(\gkv) &= \ker\left(H^1(G_{F,S \cup Q},\gkv) \to \prod_{v \in S \setminus T} H^1(G_{F_v},\gkv) \times \prod_{v \in Q} \frac{H^1(G_{F_v},\gkv)}{\mathcal{L}^\perp_v}\right).
\end{align*}

\begin{proposition} \label{number field generators formula}
With the above setup, there exists a surjective morphism $R^{T,\loc}[[X_1,\ldots,X_g]] \to R_{\mathcal{S}_Q}^{T}$ in $\CNL_\Lambda$ with
\begin{align*}
    g = h^1_{\mathcal{S}_Q^\perp,T}(\gkv) - h^0(G_{F,S},\gkv) - \sum_{v | \infty} h^0(G_{F_v},\gk^0) \\
    + \sum_{v \in S \setminus T} h^0(G_{F_v},\gkv) + z(\#T - 1) + \sum_{v \in Q} n_v.
\end{align*}
Here $z = \dim Z(\hat{G}_k)$ and for $v \in Q$ we have $n_v = \dim Z(M_{\bar{\rho}(\phi_v)}) - z$.
\end{proposition}

\begin{proof}
Note firstly that if $A \to B$ is a morphism in $\CNL_{\Lambda}$ then taking $g = \dim_k m_B/(m_A,m_B^2)$ we can find a surjection $A[[X_1,\ldots,X_g]] \to B$ by sending the $X_i$ to lifts of a $k$-basis of $m_B/(m_A,m_B^2)$.
Viewed in terms of relative tangent spaces we have an exact sequence of $k$-vector spaces
$$
m_A/(m_\Lambda,m_A^2) \to m_B/(m_\Lambda,m_B^2) \to m_B/(m_A,m_B^2) \to 0
$$
so it is equivalent to compute the dimension of
$$
\ker\left[\left(m_B/(m_\Lambda,m_B^2)\right)^* \to \left(m_A/(m_\Lambda,m_A^2)\right)^*\right],
$$
the map on relative cotangent spaces.
For $R \in \CNL_{\Lambda}$ we can identify $(m_R/(m_\Lambda,m_R^2))^* \cong \Hom_{\CNL_{\Lambda}}(R,k[\epsilon])$.
Hence taking $A = R^{T,\loc}$ and $B = R_{\mathcal{S}_Q}^T$ now, we can understand these maps explicitly in terms of $k[\epsilon]$-points of the functors they represent.
By definition, the map will send a strict equivalence class to a tuple $(\rho,\{\gamma_v\}_{v \in T}) \mapsto (\gamma_v^{-1} \restr{\rho}{G_{F_v}} \gamma_v)_{v \in T}$.
Before continuing further, we introduce the following notation:
\begin{align*}
    Z^1_Q = \ker \left(Z^1(G_{F_{S \cup Q}},\gk^0) \to \prod_{v \in Q} \frac{H^1(G_{F_v},\gk^0)}{\mathcal{L}_v}\right) \\
    H^1_Q = \ker\left(H^1(G_{F_{S \cup Q}},\gk^0) \to \prod_{v \in Q} \frac{H^1(G_{F_v},\gk^0)}{\mathcal{L}_v}\right).
\end{align*}
Since we are working with $k[\epsilon]$-points, we can view the assignment above in terms of cocyles as a map
\begin{align*}
  \alpha_2: \left(Z^1_Q \times \prod_{v \in T} \gk \right)/ \sim &\to \prod_{v \in T} Z^1(G_{F_v},\gk^0)  \\
  (\psi,(\gamma_v)_{v \in T}) &\mapsto (\restr{\psi}{G_{F_v}} - \delta(\gamma_v))
\end{align*}
where $\delta$ is the usual coboundary map and the equivalence relation $\sim$ is given by $(\psi,(\gamma_v)_{v \in T}) \sim (\psi',(\gamma_v')_{v \in T})$ if there exists $\alpha \in \gk$ such that $\psi' = \psi + \delta(\alpha)$ and $\gamma_v' = \gamma_v + \alpha$ for every $v \in T$ (see Definition \ref{T-framed}).

Now consider the following commutative diagram, whose rows are short exact sequences.

\begin{tikzcd}
0 \arrow[r] & {\gk^0 \times \underset{v \in T} {\prod}\gk} \arrow[r, "\delta \times \id"] \arrow[d, "\alpha_1"] & Z^1_Q \times \underset{v \in T} {\prod} \gk  \arrow[r, "\alpha_3"] \arrow[d, "\alpha_2"] & H^1_Q \arrow[r] \arrow[d]                                 & 0 \\
0 \arrow[r] & {\underset{v \in T} {\prod} \frac{Z^0(G_{F_v},\gk)}{H^0(G_{F_v},\gk)}} \arrow[r, "\delta"]                                        & {\underset{v \in T} {\prod} Z^1(G_{F_v},\gk^0)} \arrow[r]                                                   & {\underset{v \in T} {\prod} H^1(G_{F_v},\gk^0)} \arrow[r] & 0
\end{tikzcd}
Here the map $\alpha_1$ is defined by $(\gamma,\{\gamma_v\}_{v \in T}) \mapsto (\gamma + \gamma_v)_{v \in T}$.
The map $\alpha_3$ sends $(\psi,\{\gamma_v\}_{v \in T}) \mapsto [\psi]$, forgetting the data of the $\{\gamma_v\}_{v \in T}$.
The final row is then just the product over $v \in T$ of the exact sequences $$0 \to \frac{Z^0(G_{F_v},\gk)}{H^0(G_{F_v},\gk)} \to Z^1(G_{F_v},\gk^0) \to H^1(G_{F_v},\gk^0) \to 0,$$ 
where we have used the isomorphism $\frac{Z^0(G_{F_v},\gk^0)}{H^0(G_{F_v},\gk^0)} \cong \frac{Z^0(G_{F_v},\gk)}{H^0(G_{F_v},\gk)}$.
Since $\alpha_1$ is surjective, it follows by the snake lemma that 
\begin{align*}
    g &= \sum_{v \in T} h^0(G_{F_v}, \gk) + \dim_k (\gk^0) - h^0(G_{F,S},\gk^0) + h^1_{\mathcal{S}_Q,T} - \dim \gk \\ &= \sum_{v \in T} h^0(G_{F_v}, \gk) - h^0(G_{F,S},\gk) + h^1_{\mathcal{S}_Q,T},
\end{align*}
where we have used the decomposition of $\gk$ as a direct sum of $\gk^0$ and a trivial submodule again and that the equivalence relation $\sim$ defines a quotient by a subspace of dimension $\dim_k \gk$ since $T \neq \emptyset$.

We have by \cite[Theorem 2.19]{DDT} that
\begin{align*}
    h^1_{\mathcal{S}_Q,T}(\gk^0) = h^1_{\mathcal{S}_Q^\perp,T}(\gkv) + h^0(G_{F,S},\gk^0) - h^0(G_{F,S}, \gkv) - \sum_{v| \infty} h^0(G_{F_v},\gk^0) \\ - \sum_{v \in T} h^0(G_{F_v},\gk^0) + \sum_{v \in S \setminus T} \left(h^1(G_{F_v},\gk^0 - h^0(F_v,\gk^0)\right) + \sum_{v \in Q} n_v,
\end{align*}
where we have used that $H^1(G_{F_v}, \gk) = 0$ if $v | \infty$ by the assumption that $\ch k > 2$.
The result then follows on applying the local Euler characteristic formula together with Tate duality at those places $v \in S \setminus T$.
\end{proof}

\begin{corollary}\label{TW places number field}
Continue with the same setup as above, but with $T = S$.
Let $n \geq 1$.
Suppose that $\bar{\rho}$ is $\hat{G}$-reasonable.
Let $q \geq h^1(G_{F,S},\gkv)$.
Then there exists a set $Q$ of Taylor--Wiles places of level $n$ such that 
\begin{itemize}
    \item $H^1_{\mathcal{S}_Q^\perp,S}(\gkv) = 0$
    \item $\#Q = q$
    \item There exists a surjection $R^{S,\loc}[[X_1,\ldots,X_g]] \to R_{S_Q,S}$ with $$g = \sum_{v \in Q} n_v - \sum_{v | \infty} h^0(G_{F_v},\gk^0) + z(\#S - 1).$$
\end{itemize}
Here $z = \dim  Z(\hat{G}_k)$ and $n_v = \dim Z(M_{\bar{\rho}(\phi_v)}) - z$, for $v \in Q$.
\end{corollary}
\begin{proof}
Taking $n$ sufficiently large as necessary, this is immediate from Proposition \ref{TW places existence} together with Proposition \ref{number field generators formula}.
\end{proof}

\section{$\hat{G}$-adequate subgroups} \label{sec adequate subgroup}

In this section let $p$ be a prime number and let $\hat{G}$ be a connected reductive group over $\Fpbar$.
We assume throughout this section that $p$ is a prime of very good characteristic for $\hat{G}$.
We fix a finite subgroup $\Gamma \leq \hat{G}(\Fpbar)$ (so $\Gamma$ is defined over a finite extension $k$ of $\mathbb{F}_p$).
We aim to relate the notions of $\hat{G}$-irreducibility and $\hat{G}$-adequacy of $\Gamma$.
We recall these firstly.
\begin{definition} \label{Ghat irreducible}
We say $\Gamma$ is $\hat{G}$-irreducible if, for every proper parabolic subgroup $P$ of $\hat{G}$, we have $\Gamma \not\subset P(\Fpbar)$.
\end{definition}
Note that this coincides with the usual definition of irreducibility in the case $\hat{G} = \GL_n$.
Now let $\g^0$ be the Lie algebra of the derived group $\hat{G}^{\der}$, a finite dimensional $\Fpbar$-vector space, and $\g^{0,\vee}$ its linear dual.
Both are $\Gamma$-modules, via the adjoint action of $\hat{G}$ on its Lie algebra.

\adequate*

Here we recall that $M_h$ is the centralizer of $h$ in $\hat{G}$.
It suffices to check these conditions when we replace each instance of $k$ with $\Fpbar$ in Definition \ref{adequate}, and we will do so throughout this section.
For computational purposes, we recall that if $h \in \hat{G}(\Fpbar)$ is semisimple with connected centralizer $M_h$ then
$\Lie Z(M_h) \cap \g^0 = \mathfrak{z}(\g^h) \cap \g^0 =  \mathfrak{z}((\g^0)^h)$.

Let $\Gss$ denote the set of semisimple elements of $\Gamma$; since we are working over a finite field these are exactly the elements of order prime to $p$.
An equivalent characterisation of condition \ref{submodule condition} in Definition \ref{adequate} is as follows:
\begin{align} \label{equiv characterisation}
    \sum_{\gamma \in \Gss} (\Lie Z(M_\gamma) \cap \g^0) = \g^0.
\end{align}
Indeed, letting $U \leq (\g^0)^\vee$ be the annihilator of $\sum_{\gamma \in \Gss} (\Lie Z(M_\gamma) \cap \g^0)$, $U$ is non-zero if and only if there exists a non-zero simple $\Fpbar[\Gamma]$-submodule $W \leq (\g^0)^\vee$ such that for every $\gamma \in \Gss$, every $w \in W$ and $z \in \Lie Z(M_\gamma) \cap \g^0$ we have $w(z) = 0$.

One way to compute $\Lie Z(M_\gamma)$ which we do not explicitly use but motivates the proof of some later results is as follows.
Suppose that $T \subset \hat{G}$ is a maximal torus with Weyl group $W = N_{\hat{G}}(T)/T$ and let $\gamma \in~\hat{T}(\Fpbar)$.
Then 
$$
\Lie Z(M_\gamma) = (\Lie T)^{\Stab_{W}(\gamma)}.
$$
Indeed we have
$Z(M_\gamma^\circ)^\circ = (T^{W_{M_{\gamma}^\circ}})^\circ$ by \cite{MOweylgroup}, where $W_{M_{\gamma}^\circ}$ is the Weyl group of the identity component of $M_\gamma$.
Since $M_\gamma$ is generated by $M_\gamma^\circ$ together with $\Stab_{W}(\gamma)$ we have
$Z(M_\gamma)^\circ = (Z(M_\gamma^\circ)^{\Stab_W(\gamma)})^\circ = (T^{\Stab_W(\gamma)})^\circ$ from which the claim follows on taking Lie algebras.

\subsection{$\GL_n$-irreducible implies $\hat{G}$-adequate} \label{subsect gln irred implies adequate}

We firstly try to answer the question of when does irreducibility imply adequacy.
In the case $\hat{G} = \GL_n$ these notions are already known to be equivalent away from finitely many primes.

\begin{theorem} \label{gln equiv}
Suppose that $p \geq n+4$ and that $p \neq 2n \pm 1$. 
Then $\Gamma \leq \GL_n(\Fpbar)$ is $\GL_n$-adequate if and only if $\Gamma$ is absolutely irreducible.
\end{theorem}
\begin{proof}
Note that each semisimple element $\gamma$ of $\GL_n(\Fpbar)$ has connected centralizer and that $\mathfrak{z}(\gllie_n^\gamma)$ is equal to both the span of the powers of $\gamma$ and to the span of the projections to the eigenspaces of $\gamma$ (viewing $\gllie_{n} = \End_{\Fpbar}(\Fpbar^n)$).
Hence if $\Gamma$ is absolutely irreducible then $\Gamma$ is $\GL_n$-adequate by \cite[Corollary 1.5]{guralnick2015adequate} and the converse is given by \cite[Lemma 1]{thoappendix}.
\end{proof}

In the rest of this section, let $\hat{G}$ be one of the following simple linear algebraic groups over $\Fpbar$:
\begin{itemize}
    \item the special orthogonal group $\SO_{n}$ for $n \geq 3$
    \item the symplectic group $\Sp_{2n}$ for $n \geq 1$
    \item the simply connected exceptional groups $G_2,F_4,E_6,E_7,E_8$.
\end{itemize}
We will use Theorem \ref{gln equiv} to prove an analogous statement for $\hat{G}$, showing that if $\Gamma$ is $\GL_n$-irreducible via a certain embedding $\hat{G} \hookrightarrow \GL_n$ then $\Gamma$ is $\hat{G}$-adequate.
We firstly prove some lemmas which will allow us to verify that a finite subgroup of $\hat{G}(\Fpbar)$ is $\hat{G}$-adequate by working in $\gllie_n$.
\begin{lemma} \label{module complement}
Let $\hat{G}$ be as above.
View $\hat{G} \subset \GL_n$ via an irreducible faithful representation of $\hat{G}$ of minimal dimension.
If $p$ is a prime of very good characteristic for $\hat{G}$ then the Lie algebra $\g \leq \gllie_n$ admits a $\hat{G}$-module complement.
\end{lemma}
\begin{proof}
In the proof of \cite[Theorem 3.9]{humphreys2011conjugacy}, it is shown that for some embedding $\hat{G} \to \GL_n$ there is a $\hat{G}$-module complement to $\g$ inside $\gllie_n$.
For the cases of the classical simple groups, these embeddings are of minimal dimension.
However, the cases of the exceptional groups 
are instead handled by considering the inclusion $\g \to \gllie(\g)$ given by the adjoint representation.
There it is shown that $\g$ admits a $\hat{G}$-module complement by showing that the trace form $(X,Y) \mapsto \tr(XY)$ on $\gllie(\g)$ is invertible on restriction to $\g$ in primes of very good characteristic by computing the determinant of the restriction.
We perform a similar computation with respect to the faithful irreducible representation of minimal dimension, computing the restriction of the trace form $B$ on $\gllie_n$ to $\g$.

Note that $\hat{G}$ is defined over $\mathbb{Z}$ by means of a Chevalley basis for the Lie algebra.
Let $V$ be a highest weight module giving rise to such a faithful irreducible representation.
Then $V$ is defined over $\mathbb{Q}$ and there exists a finite set $S$ of primes such that this representation is defined over the localization $\mathbb{Z}_S$.
The set $S$ can be taken to consist of those primes appearing in the denominators of the images of the Chevalley basis in $\gllie(V)$.
This minimal set $S$ together with the determinant of the trace form is computed in $\texttt{magma}$ and displayed below in Table \ref{det trace form}.
In the case $\hat{G} = E_8$, the smallest dimension faithful irreducible representation is given by the $248$-dimensional adjoint representation and is computed in the proof of \cite[Theorem 3.9]{humphreys2011conjugacy}.
We are done on observing that the prime numbers not appearing in $S \cup \{p : v_p( \det(\restr{B}{\g})) > 0\}$ are exactly the primes of very good characteristic for $\hat{G}$.

\FloatBarrier

\begin{table}[h] \label{det trace form}
\caption{Determinant of the trace form on restriction to $\g$}
\begin{tabular}{@{}cccc@{}}
\toprule
$\hat{G}$ & $n$   & $\det(\restr{B}{\g})$ & $S$        \\ \midrule
$G_2$     & 7   & $2^{14} \cdot 3^7$           & $\emptyset$ \\
$F_4$     & 26  & $2^{78} \cdot 3^{52}$                         & $\emptyset$          \\
$E_6$     & 27  & $2^{78} \cdot 3^{77}$                         & $\{3\}$          \\
$E_7$     & 56  & $-2^{265} \cdot 3^{133}$                         & $\{2\}$           \\
$E_8$     & 248 & $2^{496} \cdot 3^{248} \cdot 5^{248}$                         & $\emptyset$          \\ \bottomrule
\end{tabular}
\end{table}

\FloatBarrier

\end{proof}

\begin{lemma} \label{compute gln adequacy}
Let $\hat{G}$ be an arbitrary closed reductive subgroup of $\GL_n$ over $\Fpbar$ and suppose that there exists a decomposition $\gllie_n = \g \oplus \mathfrak{m}$ as $\hat{G}$-modules.
Let $\pi: \g \oplus \mathfrak{m} \to \g$ be the projection onto the first factor.
Then for any set $S$ of semisimple elements of $\hat{G}(\Fpbar)$ we have
$$ \pi\big(\sum_{\gamma \in S} \mathfrak{z}(\gllie_{n}^\gamma)\big) \subset \sum_{\gamma \in S} \mathfrak{z}(\g^\gamma).$$
\end{lemma}
\begin{proof}

Observe that it suffices to show the inclusion 
\begin{equation} \label{projection equality}
    \pi(\mathfrak{z}(\gllie_{n}^\gamma)) \subset \mathfrak{z}(\g^\gamma)
\end{equation}
for any $\gamma \in \hat{G}(\Fpbar)$ semisimple.
Now suppose that $x+y \in \mathfrak{z}(\gllie_{n}^\gamma)$ with $x \in \g$ and $y \in \mathfrak{m}$.
We just need to show that $\pi(x+y) = x \in \mathfrak{z}(\g^\gamma)$. If $w \in \g^\gamma$, then $w \in \gllie_n^\gamma$ and hence $$0 = [x+y,w] = [x,w] + [y,w].$$
Differentiating the action of $\hat{G}$ on $\gllie_n$, we see that $[\g,\g] \subset \g$ and $[\g,\mathfrak{m}] = [\mathfrak{m},\g] \subset \mathfrak{m}$.
Therefore $[x,w] \in \g$ and $[y,w] \in \mathfrak{m}$, so since these subspaces are disjoint both terms in the above sum must equal zero.
We deduce that $x \in \mathfrak{z}(\g^\gamma)$, as desired.
\end{proof}

We temporarily focus on the case where $\hat{G}$ is special orthogonal or symplectic.
Let $\hat{H}$ be either $\mathrm{O}_n$ or $\Sp_{n}$, defined as a subgroup of $\GL_{n}$ by those matrices preserving a non-degenerate symmetric or alternating bilinear form $J$ on a vector space $V = \Fpbar^{n}$ (so $n$ is necessarily even in the case of $\Sp_n$).
We then let $\hat{G} = \SL_{n} \cap \hat{H}$ be the subgroup of determinant one matrices.
We have an irreducible faithful representation $\hat{G} \to \GL_n = \GL(V)$.
We assume that $n \geq 3$ if $\hat{G} = \SO_n$.

An explicit realisation of $\g$ as a subalgebra of $\gllie_n$ is given by
$$
\g = \{A \in \gllie_n : A^T J + J A = 0\}.
$$
It follows that a $\hat{G}$-invariant complement to $\g$ inside $\gllie_n$ is given by
$$
\mathfrak{m} = \{A \in \gllie_n : A^T J - J A = 0\},
$$
since we assumed $p>2$.
We let $\pi: \gllie = \g \oplus \mathfrak{m} \to \g$ denote the arising $\hat{G}$-equivariant projection.

The following lemma shows that the inclusion \ref{projection equality} is an equality in most cases and thereby allows us to compute centers of centralizers by working inside $\gllie_n$.

\begin{lemma} \label{spanning equality}
Suppose that $\hat{G}$ is special orthogonal or symplectic and that $\gamma \in \hat{G}(\Fpbar)$ is semisimple.
Then
\begin{enumerate}
    \item the inclusion
    \begin{equation}
        \pi(\mathfrak{z}(\gllie_{n}^\gamma)) \subset \mathfrak{z}(\g^\gamma)
\end{equation}
fails to be an equality if and only if both $\hat{G} = \SO_n$ and $2 \in \{m_1, m_{-1}\}$;
    \item the inclusion
\begin{equation} \label{classical inclusion}
        \pi(\mathfrak{z}(\gllie_{n}^\gamma)) \subset \Lie Z(Z_{\hat{G}}(\gamma))
\end{equation}
always holds and this inclusion fails to be an equality if and only if both $\hat{G} = \SO_n$ and $ \{0,2\} = \{m_1, m_{-1}\}$.
\end{enumerate}
Here for $\alpha \in \Fpbar^\times$, $m_\alpha$ is the dimension of the $\alpha$-eigenspace for the action of $\gamma$ on $V$.
\end{lemma}
\begin{proof}
The bilinear form $J$ is non-degenerate on restriction to the sum of eigenspaces $X_\alpha = V_\alpha + V_{\alpha^{-1}}$.
Since the centralizer of $\gamma$ inside $\GL_n$ is equal to $\prod_{\alpha} \GL(V_\alpha)$, 
we have $Z_{\hat{H}}(\gamma) = \prod_{\{\alpha, \alpha^{-1}\}} Z_{\hat{H}_\alpha}(\restr{\gamma}{X_\alpha})$.
Here $\hat{H}_\alpha = \hat{H} \cap \GL(X_\alpha)$ is the subgroup of $\GL(X_\alpha)$ which preserves $\restr{J}{X_\alpha}$.
We see that $$Z_{\hat{H}_\alpha}(\restr{\gamma}{X_\alpha}) \cong \begin{cases}
\GL_{m_\alpha} &\text{ if } \alpha \neq \alpha^{-1} \\
\hat{H}_\alpha &\text{ if } \alpha = \alpha^{-1},
\end{cases}$$
where the copy of $\GL_{m_\alpha}$ is embedded into $\GL(V_\alpha) \times \GL(V_{\alpha^{-1}})$ with determinant one image.
The centralizer $Z_{\hat{G}}(\gamma)$ is therefore isomorphic to $$(\SL(X_1 \oplus X_{-1}) \cap (\hat{H}_1 \times \hat{H}_{-1})) \times \prod_{\{\alpha,\alpha^{-1}\} \not\in \{\{1\},\{-1\}\}} \GL_{m_\alpha}.$$
We claim that 
$$\pi(\mathfrak{z}(\gllie_{n}^\gamma)) = \bigoplus_{\{\alpha,\alpha^{-1}\} \not\in \{\{1\},\{-1\}\}} \mathfrak{z}(\gllie_{m_\alpha})$$
where, writing $\gllie(V) = \oplus_{\beta,\gamma} \Hom(V_\beta,V_\gamma)$, we identify $\mathfrak{z}(\gllie_{m_\alpha})$ with the subspace spanned by the map $V \to V$ which restricts to $\pm \id$ on $V_{ \alpha^{\pm 1}}$ and restricts to $0$ on every $V_\beta$ for $\beta \neq \alpha^{\pm 1}$. 
This is because, for both choices of $\pm$, $\mathfrak{z}(\gllie(X_{\pm 1})) \cap \Lie \hat{G} = 0$, as there are no non-zero scalar matrices in the Lie algebra of an orthogonal or symplectic group (recalling that $p>2$).
For $\alpha \neq \pm 1$ an eigenvalue of $\gamma$, we have $\Lie Z(\GL_{m_\alpha}) = \mathfrak{z}(\gllie_{m_\alpha})$.
We deduce that the inclusion $\pi(\mathfrak{z}(\gllie_{n}^\gamma)) \subset \Lie Z(Z_{\hat{G}}(\gamma))$ always holds.

It remains to check in what cases we have either
\begin{equation} \label{eqn: lie algebra center nonzero}
\mathfrak{z} (\Lie (\SL(X_1 \oplus X_{-1}) \cap (\hat{H}_1 \times \hat{H}_{-1}))) \neq 0
\end{equation}
or
\begin{equation} \label{eqn: center of lie algebra nonzero}
\Lie Z(\SL(X_1 \oplus X_{-1}) \cap (\hat{H}_1 \times \hat{H}_{-1})) \neq 0. 
\end{equation}
For the former, we are computing the center of the direct sum of Lie algebras of symplectic or special orthgonal groups, which is non-zero if and only if at least one of the Lie algebras is $\mathfrak{so}_2$. 
Hence \ref{eqn: lie algebra center nonzero} holds if and only if $\hat{G}$ is special orthogonal and $2 \in \{m_1,m_{-1}\}$.
For the latter, we similarly see that $J$ must be symmetric and that $2 \in \{m_1,m_{-1}\}$ if \ref{eqn: center of lie algebra nonzero} is to hold; without loss of generality suppose that $m_1 = 2$.
If $m_{-1} > 0$ then $\hat{H}_{-1}$ is isomorphic to a disconnected group, $\mathrm{O}_{m_{-1}}$.
Thus for any pair of matrices $(A,B)$ defining a point of $Z(\SL(X_1 \oplus X_{-1}) \cap (\hat{H}_1 \times \hat{H}_{-1}))$ it must be the case that $A$ commutes with every matrix in $\hat{H}_1 \cong \mathrm{O}_2$ and that $Z(\SL(X_1 \oplus X_{-1}) \cap (\hat{H}_1 \times \hat{H}_{-1}))$ must therefore be a finite group scheme with zero Lie algebra.
Otherwise, if $m_{-1} = 0$ then $\Lie Z(\SL(X_1 \oplus X_{-1}) \cap (\hat{H}_1 \times \hat{H}_{-1})) \cong \Lie \SO_2 \neq 0$.
\end{proof}

We now come to the main theorem of this section which gives convenient criteria for $\hat{G}$-adequacy in a number of cases.

\begin{theorem} \label{gln irred implies adequate}
Let $\hat{G}$ be one of the following simple linear algebraic groups over $\Fpbar$:
\begin{itemize}
    \item the special orthogonal group $\SO_{n}$ for $n \geq 3$
    \item the symplectic group $\Sp_{n}$ for $n \geq 2$ even
    \item the simply connected exceptional groups $G_2,F_4,E_6,E_7,E_8$.
\end{itemize}
Let $\hat{G} \hookrightarrow \GL_n$ be a faithful irreducible representation of minimal dimension.
If $p \geq n+4$, $p \neq 2n \pm 1$ and $\Gamma \subset \hat{G}(\Fpbar)$ is a finite, $\GL_n$-irreducible subgroup, then $\Gamma$ is $\hat{G}$-adequate.
\end{theorem}

\begin{proof}

By Theorem \ref{gln equiv} we know that $\Gamma$ is $\GL_n$-adequate, hence
$$\sum_{\gamma \in \Gss} \mathfrak{z}(\gllie_{n}^\gamma) = \gllie_{n}.$$
By Lemma \ref{module complement} we may write $\gllie_{n} = \g \oplus \mathfrak{m}$ as $\Gamma$-modules and we let $\pi$ denote the $\hat{G}$-equivariant projection $\gllie_n \to \g$.
Then by Lemma \ref{compute gln adequacy} for the cases where $\hat{G}$ is simply connected, and Lemma \ref{spanning equality} for the case $\hat{G} = \SO_n$, it follows that
$$\g = \pi\big(\sum_{\gamma \in \Gss} \mathfrak{z}(\gllie_{n}^\gamma)\big) \subset \sum_{\gamma \in \Gss} \Lie Z(M_\gamma)$$
and the spanning condition therefore holds.

We now check the cohomological conditions of Definition \ref{adequate} using the $\hat{G}$-module decomposition of $\gllie_n$ once again.
For $0 \leq i \leq 1$ we have $$0 = H^i(\Gamma, \gllie_{n}) = H^i(\Gamma, \gllie_{n}^\vee) = H^i(\Gamma, \g^\vee) \oplus H^i(\Gamma, \mathfrak{m}^\vee)$$ and $H^1(\Gamma,\Fpbar) = 0$ by $\GL_n$-adequacy of $\Gamma$.
\end{proof}

To conclude this section, we give some lemmas which in some cases make the computation of projections of centers of centralizers and their spans more straightforward.

\begin{lemma} \label{lemma: classical projection equal intersection}
Suppose $\hat{G}$ is either symplectic or special orthogonal.
Then for every semisimple $\gamma \in \hat{G}(\Fpbar)$ we have
$\pi(\mathfrak{z}(\gllie_n^\gamma)) = \mathfrak{z}(\gllie_n^\gamma) \cap \g.$
\end{lemma}
\begin{proof}
 We only show that $\pi(\mathfrak{z}(\gllie_n^\gamma)) \subset \mathfrak{z}(\gllie_n^\gamma) \cap \g$ since the reverse containment is clear.
Suppose that $x+y \in \mathfrak{z}(\gllie_n^\gamma)$ with $x \in \g$ and $y \in \mathfrak{m}$.
We just need to show that whenever $z \in \gllie_n^\gamma = \g^\gamma \oplus \mathfrak{m}^\gamma$ we have $[x,z] = 0$.
Suppose that $z \in \g^\gamma$ firstly.
Then $$0 = [x+y,z] = [x,z] + [y,z]$$ with $[x,z] \in \g$ and $[y,z] \in \mathfrak{m}$.
It follows that $[x,z] = 0$.
From the explicit descriptions of $\mathfrak{g}$ and $\mathfrak{m}$, it is easy to see that $[\mathfrak{m},\mathfrak{m}] \subset \g$.
Thus if $z \in \mathfrak{m}^\gamma$ then $[x,z] \in \mathfrak{m}$ and $[y,z] \in \g$.
We similarly see that $[x,z] = 0$.
\end{proof}

\begin{lemma} \label{lemma: compute with intersections}
Let $\hat{H} \subset \hat{G}$ be reductive groups defined over $\Fpbar$.
Suppose that we can write $\g = \mathfrak{\hat{h}} \oplus \mathfrak{m}$ as $\hat{H}$-modules.
Let $S \subset \hat{H}(\Fpbar)$ be a finite set of semisimple elements such that for every $s \in S$ we have
$
\mathfrak{z}(\mathfrak{\hat{h}}^\gamma) = \mathfrak{z}(\g^\gamma) \cap \mathfrak{\hat{h}}.
$
Then
$$
\sum_{\gamma \in S} \mathfrak{z}(\mathfrak{\hat{h}}^\gamma) = \left(\sum_{\gamma \in S} \mathfrak{z}(\g^\gamma) \right)\cap \mathfrak{\hat{h}}.
$$
\end{lemma}
\begin{proof}
 Let $\pi: \g \to \mathfrak{\hat{h}}$ denote the $\hat{H}$-equivariant projection arising from the decomposition $\g = \mathfrak{\hat{h}} \oplus \mathfrak{m}$.
We have inclusions
$$
\left(\sum_{\gamma \in S} \mathfrak{z}(\g^\gamma) \right)\cap \mathfrak{\hat{h}} \subset \pi(\sum_{\gamma \in S} \mathfrak{z}(\g^\gamma) ) \subset \sum_{\gamma \in S} \mathfrak{z}(\mathfrak{\hat{h}}^\gamma)
$$
where the second inclusion follows by the same argument as the proof of Lemma \ref{compute gln adequacy}.
The reverse containment 
$$
\sum_{\gamma \in S} \mathfrak{z}(\mathfrak{\hat{h}}^\gamma) \subset \left(\sum_{\gamma \in S} \mathfrak{z}(\g^\gamma) \right)\cap \mathfrak{\hat{h}}
$$
is clear from the hypothesis of the lemma and this completes the proof.
\end{proof}

\subsection{When does $\hat{G}$-adequate imply $\GL_n$-irreducible}

In some cases, a converse to Theorem \ref{gln irred implies adequate} holds.
In all cases where a converse holds, the proof only uses the spanning condition of $\hat{G}$-adequacy.
We firstly show that the spanning condition of $\hat{G}$-adequacy implies $\hat{G}$-irreducibility in general and that if $p$ is sufficiently large then the cohomological conditions are implied by $\hat{G}$-irreducibility.

\begin{proposition} \label{adequate implies irred}
Let $\hat{G}$ be an arbitrary reductive group over $\Fpbar$ and let $\Gamma \subset \hat{G}(\Fpbar)$ be a finite subgroup.
\begin{enumerate}
\item If $H \leq \hat{G}$ is a closed connected smooth subgroup of rank equal to that of $\hat{G}$ such that $\Gamma \subset H(\Fpbar)$, then
$$
\sum_{\gamma \in \Gss} (\Lie Z(M_\gamma) \cap \g^0) \subset \Lie H.$$
\item
In particular, if
$$
\sum_{\gamma \in \Gss} (\Lie Z(M_\gamma) \cap \g^0)  = \g^0
$$
  then $\Gamma$ is $\hat{G}$-irreducible.
\item If $p \geq 3 + 2\dim \g^0$ and $\Gamma$ is $\hat{G}$-irreducible, then $H^1(\Gamma, \g^0)$ and $H^1(\Gamma, \Fpbar)$ both vanish.
\end{enumerate}\end{proposition}
\begin{proof}
Let $\gamma \in \Gss$. 
We claim that $\Lie Z(M_\gamma) \cap \g^0 \leq \Lie H$.
Since $\gamma$ is semisimple we can find a maximal torus $T$ of $H$ for which $\gamma \in T(\Fpbar)$, as is true in any connected linear algebraic group.
Since $T$ is also a maximal torus of $\hat{G}$ and contained in $M_\gamma$, it follows that $(\Lie Z(M_\gamma) \cap \g^0) \subset \Lie T$ which shows the claim.
The second point immediately follows from this, as any parabolic subgroup is smooth and connected of maximal rank.
 
For the final point, it suffices to replace $\Gamma$ by its image in the adjoint quotient of $\hat{G}$.
Since $\Gamma$ is $\hat{G}$-irreducible, $\Gamma$ cannot normalize a non-trivial $p$-subgroup by \cite[Proposition 3.1]{borel1971elements}.
Applying \cite[Theorem A]{guralnick1999small}, we see that $\g^0$ is a completely reducible representation of $\Gamma$ and that $H^1(\Gamma, \g^0) = 0$.
The remainder of the proposition now follows from Lemma \ref{lemma: completely reducible implies h1(fp) vanishes} below.
\end{proof}

\begin{lemma} \label{lemma: completely reducible implies h1(fp) vanishes}
Let $\Gamma \subset \GL_n(\Fpbar)$ be a finite subgroup acting completely reducibly on $\Fpbar^n$.
If $p > 2n+1$ then $H^1(\Gamma,\Fpbar) = 0$.
\end{lemma}
\begin{proof}
In the case that $\Gamma$ acts irreducibly, the result follows from Theorem \ref{gln equiv}.
Now proceed by induction on $n$ and let $\Fpbar^n = V_1 \oplus V_2$ be a non-trivial decomposition of $\Gamma$-submodules.
Let $\Gamma_i$ denote the image of $\Gamma$ inside $\GL(V_i)$ and let $K = \ker(\Gamma \to \Gamma_2)$.
Then $\Gamma_2$ acts completely reducibly on $V_2$ and $K$ acts completely reducibly on $V_1$ by Clifford theory, since (the image of) $K$ is normal inside $\Gamma_1$ and $\Gamma_1$ acts completely reducibly on $V_1$.
The result now follows from inflation-restriction and the induction hypothesis.
\end{proof}

\begin{remark}
It is also straightforward to see that if $\Gamma \subset \hat{G}(\Fpbar)$ is a $\hat{G}$-adequate subgroup then $Z_{\hat{G}}(\Gamma) = Z(\hat{G})$.
Indeed, if $g \in Z_{\hat{G}}(\Gamma)(\Fpbar)$ then $\Ad(g)$ acts trivially on $\sum_{\gamma \in \Gss} \Lie Z(Z_{\hat{G}}(\gamma))$, which would imply that $g$ lies in the kernel of the adjoint representation if $\Gamma$ is $\hat{G}$-adequate.

\end{remark}

\subsubsection{The case of classical simple groups}

We return to the case where $\hat{G}$ is either special orthogonal or symplectic in this subsection.
We begin with the converse to Theorem \ref{gln irred implies adequate} when $\hat{G}$ is of type B or C.

\begin{corollary} \label{classical gln irred equiv adequate}
Let $p \geq 3$, let $\hat{G}$ be either $\SO_n$ with $n \geq 3$ odd or $\Sp_{n}$ with $n \geq 2$ even. 
Let $\Gamma \subset \hat{G}(\Fpbar)$ be a finite subgroup.
If $\sum_{\gamma \in \Gss} \Lie Z(M_\gamma) = \g$ then $\Gamma$ is $\GL_{n}$-irreducible.
In particular, if $\Gamma$ is $\hat{G}$-adequate then $\Gamma$ is $\GL_n$-irreducible, and the converse holds if both $p \geq n+4$ and $p \neq 2n \pm 1$.
\end{corollary}

\begin{proof}
By Theorem \ref{gln irred implies adequate}, we just need to show that if $\Gamma$ is $\GL_{n}$-reducible then
$$
\sum_{\gamma \in \Gss} \Lie Z(M_\gamma) \neq \g.
$$
Let $U \leq V \cong \Fpbar^n$ be a proper $\Gamma$-invariant subspace and let $P$ be the proper subgroup of $\GL_n$ defined by the stabilizer of $U$.
Since the inclusion $\hat{G} \hookrightarrow \GL_{n}$ is an irreducible representation, $\hat{G}$ is not a subgroup of $P$, and hence $\g$ is not a subalgebra of $\Lie P = \mathfrak{p}$.
On the other hand, we claim that $$\sum_{\gamma \in \Gss} \Lie Z(M_\gamma) \leq \mathfrak{p},$$
in which case $\Gamma$ cannot be $\hat{G}$-adequate.
Let $\gamma \in \Gss$.
By Lemma \ref{spanning equality} we have
$$\Lie Z(M_\gamma) = \mathfrak{z}(\g^\gamma) = \pi(\mathfrak{z}(\gllie_n^\gamma)),$$
since when $\hat{G} = \SO_n$ we have assumed that $n$ is odd, so that the sum of the dimensions of $\pm 1$-eigenspaces for $\gamma$ on $V$ is also odd.
The result then follows from Lemma \ref{lemma: classical projection equal intersection} and the fact that
$\mathfrak{z}(\gllie_n^\gamma) \leq \mathfrak{p}$ (just as in Proposition \ref{adequate implies irred} for example).

\end{proof}

We now relate $\GL_n$-irreducibility to properties more intrinsic to $\hat{G}$.
We say a subspace $U \leq V$ is isotropic if $\restr{J}{U \times U} = 0$ where $J$ is the bilinear form on $V$ defining $\hat{G}$.
Recall then that the parabolic subgroups of $\hat{G}$ are exactly the stabilizers of flags of isotropic subspaces of $V$.
Thus a finite subgroup $\Gamma \leq \hat{G}(\Fpbar)$ is $\hat{G}$-irreducible if and only if there does not exist a non-zero isotropic $\Fpbar[\Gamma]$-submodule $U$ of $V$.

\begin{lemma}\label{semisimple isotropic}
Let $\Gamma \subset \hat{G}(\Fpbar)$ be a finite group.
Then either $V$ is a direct sum of simple $\Fpbar[\Gamma]$-submodules $V = \bigoplus_i U_i$, with $\restr{J}{U_i \times U_i}$ non-degenerate for every $i$, or $\Gamma$ is not $\hat{G}$-irreducible.
\end{lemma} 

\begin{proof}
Let $U \leq V$ be a simple submodule.
If $\restr{J}{U \times U}$ is non-degenerate then the annihilator of $U$ with respect to $J$ provides a $\Gamma$-invariant complement and we are done by induction on $\dim V$.
Suppose instead that $\restr{J}{U \times U}$ is degenerate.
Since $U$ is irreducible, Schur's lemma implies that $J$ induces the zero map $U \to U^\vee$.
Thus $U$ is a non-zero isotropic $\Fpbar[\Gamma]$-submodule and $\Gamma$ is not $\hat{G}$-irreducible.
\end{proof}

\begin{proposition} \label{symp irred}
Let $\hat{G}$ be either $\SO_n$ or $\Sp_{n}$, let $\Gamma \subset \hat{G}(\Fpbar)$ be a finite subgroup and let $p \geq 3$.
If $\hat{G} = \SO_n$ let $\hat{H} = \mathrm{O}_n$ and suppose that $n \geq 3$; if $\hat{G} = \Sp_n$ let $\hat{H} = \hat{G}$.
Then the following are equivalent:
\begin{enumerate}
    \item $\Gamma$ is $\hat{G}$-irreducible and $Z_{\hat{H}}(\Gamma) = Z(\hat{H}) \cong \mu_2$ \label{spn irred centralizer}
    \item $\Gamma$ is an irreducible subgroup of $\GL_{n}(\Fpbar)$.
\end{enumerate}
\end{proposition}

\begin{proof}
Suppose that \ref{spn irred centralizer} does not hold.
If $\Gamma$ is not $\hat{G}$-irreducible then it must be contained inside the $\Fpbar$-points of a proper parabolic subgroup $P$, which stabilizes some non-zero isotropic subspace of $V$.
Thus $V$ is not irreducible as a $\Gamma$-module.
Otherwise, if $Z_{\hat{H}}(\Gamma) \neq Z(\hat{H})$ then $Z_{\GL_{n}}(\Gamma) \neq Z(\GL_n)$ since $Z(\GL_n) \cap \hat{H} = Z(\hat{H})$.
By Schur's lemma, $\Gamma$ is therefore not $\GL_n$-irreducible.

Now suppose that $\Gamma \subset \GL_{n}(\Fpbar)$ is reducible.
We may suppose that $\Gamma$ is $\hat{G}$-irreducible; we then have to show that $Z_{\hat{H}}(\Gamma) \neq \mu_2$.
Since $\Gamma$ is ($\GL_n$-)reducible, by Lemma \ref{semisimple isotropic} we may therefore write $V = X \oplus Y$ where $X,Y$ are non-zero $\Fpbar[\Pi]$-modules such that the restrictions of $J$ to both $X$ and $Y$ are non-degenerate.
 Consider the automorphisms of $V$ given by all choices of $\pm \id$ on each summand.
Such maps will commute with the action of $\Gamma$ and preserve $J$ and thus define elements of $\hat{H}(\Fpbar)$.
We have shown $Z_{\hat{H}}(\Gamma)(\Fpbar)$ has order at least $4$ and hence $Z_{\hat{H}}(\Gamma)$ cannot equal $\mu_2$.
\end{proof}

\subsubsection{Counter-examples to converse implications}

Let $\hat{G}$ be either a classical simple group or simply connected exceptional simple group viewed as a subgroup of $\GL_n$ via a faithful irreducible representation of minimal dimension.
Let $\Gamma \subset \hat{G}(\Fpbar)$ be a finite subgroup.
We have seen already that if $p$ is sufficiently large and $\Gamma$ is $\GL_n$-irreducible then $\Gamma$ is $\hat{G}$-adequate.
We saw that the converse also held when $\hat{G}$ is either $\SO_{2n+1}$ or $\Sp_{2n}$.
In this subsection we give counter-examples to show that the converse statement does not hold for $\SO_{2n}$ and the exceptional groups.

We have also seen that $\hat{G}$-adequacy of $\Gamma$ implies $\hat{G}$-irreducibility of $\Gamma$ and that $Z_{\hat{G}}(\Gamma) = Z(\hat{G})$.
In the following example we see that the converse implication to this statement also fails (even when $p$ is large and $\Gamma$ arises from a connected algebraic subgroup) in the case that $\hat{G} = \Spin_{16}$, the simply connected cover of $\SO_{16}$.

\begin{example}
When $\hat{G} = \SO_{2n}$, the existence of semisimple elements for which $1$ is an eigenvalue of multiplicity $2$ and $-1$ is not an eigenvalue can lead to subgroups $\Gamma$ for which
$$
\sum_{\gamma \in \Gss} \Lie Z(Z_{\SO_{2n}}(\Gamma)) = \g
$$
without $\Gamma$ being $\GL_{2n}$-irreducible.
An example is given by a copy of $\SO_{2n-1}(\mathbb{F}_p) \subset \SO_{2n}(\Fpbar)$ for all $p$ large enough (depending on $n$).
However, it isn't the case in general that every $\GL_{2n-1}$-irreducible subgroup $\Gamma$ of $\SO_{2n-1}(\Fpbar)$ is $\SO_{2n}$-adequate (even when $p$ is sufficiently large).
For example, let $k$ be a finite subfield of $\Fpbar$ and let
$$
\Gamma := \mathrm{im}(\PGL_4(\mathbb{F}_q) \xrightarrow{\Ad} \SO_{15}(\Fpbar) \to \SO_{16}(\Fpbar)).
$$
Then each semisimple element $\gamma$ has $1$ as an eigenvalue with multiplicity at least $4$, hence by Lemma \ref{spanning equality} we have
$$
\sum_{\gamma \in \Gss} \Lie Z(Z_{\SO_{16}}(\gamma)) = \sum_{\gamma \in \Gss} \pi(\mathfrak{z}(\gllie_{16}^\gamma)).
$$
Since $\Gamma$ acts reducibly on the $16$-dimensional representation, for each $\gamma \in \Gss$ we have $\mathfrak{z}(\gllie_{16}^\gamma) \subset \gllie_{15} \oplus \gllie_{1}$ and since $\solie_{16} \not\subset \gllie_{15} \oplus \gllie_{1}$ we see from the above equality that $\Gamma$ is not $\SO_{16}$-adequate.

In fact, taking $\Tilde{\Gamma}$ to be the preimage of $\Gamma$ under the surjection $\Spin_{16}(\Fpbar) \xrightarrow{f} \SO_{16}(\Fpbar)$, we claim that $\Tilde{\Gamma}$ is not $\Spin_{16}$-adequate.
We claim that this provides an example of an algebraic $\hat{G}$-irreducible subgroup $H$ of a simply connected reductive group $\hat{G}$ with trivial centralizer modulo center but whose $k$-points are never $\hat{G}$-adequate whenever $k$ is a finite field of very good characteristic.
Firstly, because $\Spin_{16}$ is simply connected, we have for $\Tilde{\gamma} \in \Tilde{\Gamma}^{\mathrm{ss}}$ that $\Lie Z(Z_{\Spin_{16}}(\Tilde{\gamma})) = \mathfrak{z}(\solie_{16}^\gamma)$, where $\gamma = f(\Tilde{\gamma})$.
It follows from the proof of Lemma \ref{spanning equality} that 
$$\mathfrak{z}(\solie_{16}^\gamma) \subset \mathfrak{z}(\gllie_{16}^\gamma) + \End(V_{-1}),$$
where $V_{-1}$ denotes the $-1$-eigenspace for the action of $\gamma$ on the standard $16$-dimensional representation $V$.
Since $V_{-1}$ is a subspace of the $15$-dimensional $\Tilde{\Gamma}$-invariant subspace of $V$ we see that $$\mathfrak{z}(\solie_{16}^\gamma) \subset \solie_{15} \subsetneq \solie_{16}.$$
Thus $\Lie Z(Z_{\Spin_{16}}(\Tilde{\gamma}))$ lies inside the same proper subspace of $\solie_{16}$ for every $\Tilde{\gamma} \in \Tilde{\Gamma}^{\mathrm{ss}}$ and so $\Tilde{\Gamma}$ is not $\Spin_{16}$-adequate.
It is also clear that $\Tilde{\Gamma}$ is $\Spin_{16}$-irreducible, since $\Gamma$ is $\SO_{16}$-irreducible, and that $Z_{\Spin_{16}}(\Tilde{\Gamma}) \subset f^{-1}(Z_{\SO_{16}}(\Tilde{\Gamma})) = Z(\Spin_{16})$.
\end{example}

We now focus on the case where $\hat{G}$ is a simply connected simple group of exceptional type.
We show that the source of subgroups given by the image of (the $k$-points of) a principal $\SL_2$-homomorphism to $\hat{G}$ are, in some cases, $\hat{G}$-adequate yet act reducibly on the minimal faithful irreducible representation of $\hat{G}$.

\begin{example} [Principal $\SL_2$-homomorphisms]
A regular nilpotent element of $\g$ defines, by the Jacobson--Morozov theorem, such a homomorphism $\rho: \SL_2 \to \hat{G}$.
Let $k$ be a finite subfield of $\Fpbar$ over which this map is defined, let $\Gamma = \rho(\SL_2(k))$ and suppose that $p$ is sufficiently large.
We show that $\Gamma$ is $\hat{G}$-adequate.
Moreover, if $\hat{G}$ is one of $F_4$, $E_6$, $E_7$ or $E_8$ then $\Gamma$ acts reducibly on a faithful irreducible representation of $\hat{G}$ of minimal dimension. 

We firstly decompose both the standard representation $V$ of $\hat{G}$ and the adjoint representation with respect to $\rho(\SL_2) \subset \hat{G}$.
The regular nilpotent element $x \in \g$ defining $\rho$ can be written as $x = \sum_{\alpha \in \Delta} x_\alpha$ where $\Delta$ is a choice of positive simple roots with respect to a maximal torus $T$ of $\hat{G}$ and each $x_\alpha \in \g_{\alpha}$ is a non-zero element of the root space.
Let $\lambda: \mathbb{G}_m \to T$ be the cocharacter given by the restriction of $\rho$ to the diagonal maximal torus of $\SL_2$.
Then $\lambda = \sum_{\beta^\vee > 0} \beta^\vee$ is the sum of positive coroots of $T$.
This is because $\lambda$ is the unique cocharacter for which $\langle \lambda, \alpha \rangle = 2$ for every $\alpha \in \Delta$ which must hold in order for the adjoint action of the image of $\lambda$ on $x$ to arise from a homomorphism from $\SL_2$.
It is straightforward to compute the image of $\lambda$ inside $\GL(V)$ and $\GL(\g)$ and this determines $\rho$.
The results are displayed in Table \ref{table: principal sl2 decomposition}, with $V_m$ denoting the irreducible representation of $\SL_2$ of dimension $m$.

\begin{table}[h]
\caption{Decompositions of $V$ and $\g$ with respect to a principal $\SL_2$}
 \label{table: principal sl2 decomposition}
\begin{tabular}{|l|ll|}
\hline
Group & Standard Representation $V$       & Adjoint Representation $\g$                                                             \\ \hline
$G_2$ & $V_7$                          & $V_{11} \oplus V_3$                                                                   \\ \hline
$F_4$ & $V_{17} \oplus V_9$              & $V_{23} \oplus V_{15} \oplus V_{11} \oplus V_3$                                           \\ \hline
$E_6$ & $V_{17} \oplus V_9 \oplus V_1$   & $V_{23} \oplus V_{17} \oplus V_{15} \oplus V_{11} \oplus V_9 \oplus V_3$                    \\ \hline
$E_7$ & $V_{28} \oplus V_{18} \oplus V_{10}$ & $V_{35} \oplus V_{27} \oplus V_{23} \oplus V_{19} \oplus V_{15} \oplus V_{11} \oplus V_3$ \\ \hline
$E_8$ & \multicolumn{2}{l|}{$V_{59} \oplus V_{47} \oplus V_{39} \oplus V_{35} \oplus V_{27} \oplus V_{23} \oplus V_{15} \oplus V_{3}$}        \\   \hline                                                                                     
\end{tabular}
\end{table}

We see in all cases that $\g$ decomposes as a direct sum of $r$ non-trivial representations of $\SL_2$, where $r$ is the rank of $\hat{G}$.
Moreover, since $\lambda$ has non-trivial action on each root space we see that the trivial eigenspace for the action of $\mathbb{G}_m$ on $\g$ arising from $\lambda$ is given by $\Lie T$.
Since each irreducible representation of $\SL_2$ occurring in $\g$ has a one-dimensional subspace in which $\mathbb{G}_m$ acts trivially via $\lambda$, we see that $\Lie T$ intersects each such irreducible representation non-trivially.
Note that, taking $\# k$ large enough, we can always find $\gamma \in \lambda(k^\times)$ such that $\gamma$ is regular semisimple.
This is because the Lie algebra of the centralizer of $\gamma$ is given by the sum of $\Lie T$ together with those root spaces on which $\gamma$ acts trivially.
Since $\alpha \circ \lambda$ is non-trivial for each root $\alpha$, we see that the image of every element of $\Fpbar^\times$ of sufficiently large order (not depending on $p$) under $\lambda$ must be regular semisimple.
Now since $\hat{G}$ is simply connected, the centralizer of $\gamma$ inside $\hat{G}$ is connected and equal to $T$.
Each irreducible summand of $\g$ with respect to the action of $\SL_2$ remains irreducible on restriction to $\Gamma$ and has non-trivial intersection with $\Lie T$.
Since
$$
\sum_{\gamma \in \Gss} \Lie Z(Z_{\hat{G}}(\gamma))
$$
is itself a $\Gamma$-submodule of $\g$, we conclude that this submodule must equal all of $\g$.

Taking $p$ sufficiently large, it follows from Proposition \ref{adequate implies irred} that the cohomological conditions of $\hat{G}$-adequacy are also satisfied by $\Gamma$.   
Finally, we see from Table \ref{table: principal sl2 decomposition} that the action of $\Gamma$ on $V$ is reducible provided that $\hat{G}$ is not of type $G_2$.
\end{example}

From the computation in the above example, it follows that such a $\Gamma$ is not only $\hat{G}$-adequate but also $\hat{G}$-abundant, in the sense of \cite{thorne2019}.
Examples of $G_2$-adequate (or even $\hat{G}$-abundant) but $\GL_7$-reducible subgroups have also be found, and we deduce that the analogue of Corollary \ref{classical gln irred equiv adequate} fails for all simply connected exceptional groups.

\begin{corollary}
Let $\hat{G}$ be a simply connected simple group of exceptional type.
Then for every $p$ sufficiently large, there exists a finite subgroup $\Gamma \subset \hat{G}(\Fpbar)$ which is $\hat{G}$-adequate and which acts reducibly on a faithful representation of $\hat{G}$ of minimal dimension. 
\end{corollary}
\begin{proof}
By the above example, we need only treat the case of $\hat{G} = G_2$.
Let $\phi: \Gamma \to G_2(k)$ be an abstract Coxeter homomorphism in the sense of \cite[Definition 10.3]{thorne2019}, where $k$ is some finite subfield of $\Fpbar$ which is sufficiently large.
Such a $\phi$ exists and $\phi(\Gamma)$ is $\hat{G}$-abundant (hence $\hat{G}$-adequate) by \cite[Corollary 10.8]{thorne2019}.
Moreover, $\phi(\Gamma)$ must act reducibly on the standard $7$-dimensional representation $V$ of $G_2$, since $\phi(\Gamma)$ is contained inside the normalizer of a maximal torus which itself admits a $1$-dimensional subrepresentation of $V$. 
\end{proof}

\subsection{$\Sp_4$-adequate subgroups of $\Sp_4(\mathbb{F}_p$)} \label{sect subgroups}

For our applications in Section \ref{sec surfaces}, we now focus on the case $\hat{G} = \Sp_4$ and study the extent to which Theorem \ref{gln irred implies adequate} fails when $p$ is small.
The following lemma provides an answer.
It is the result of computations carried out via the computer algebra software \texttt{magma} when $p \in \{3,5,7\}$ and otherwise follows from Theorem \ref{gln irred implies adequate}.
We use Lemmas \ref{spanning equality}, \ref{lemma: classical projection equal intersection} and \ref{lemma: compute with intersections} to verify the spanning condition by working inside $\gllie_4$, where $\mathfrak{z}(\gllie_4^\gamma)$ is given by the span of powers of $\gamma$.

\begin{lemma} \label{Sp4(Fp)}
Let $p>2$ be prime.
Suppose $\Gamma \leq \Sp_4(\mathbb{F}_p)$.
If $\Gamma$ is $\GL_4$-absolutely irreducible then $\Gamma$ is $\hat{G}$-adequate, unless $\Gamma$ is one of the finitely many groups listed in Table \ref{subgroups of sp4(f3)} in the appendix.
Here $.$ refers to a non-split extension, $\rm wr$ the wreath product and $:$ a semidirect product.
In all cases, the spanning condition of $\hat{G}$-adequacy holds and $H^0(\Gamma,\splie_4)$ vanishes.

\end{lemma}

We now give an example in the setting of subgroups of $\GSp_4(\mathbb{F}_3)$ of how our spanning condition compares to the similar but more restrictive conditions seen previously in the literature.
Let $\Gamma' \leq \GSp_4(\mathbb{F}_3)$ be such that $\Gamma = \Gamma' \cap \Sp_4(\mathbb{F}_3)$ is absolutely $\GL_4$-irreducible.
Suppose that the similitude character $\nu$ defines a surjection $\Gamma' \to \mathbb{F}_3^\times$, so that $\Gamma$ is an index $2$ subgroup of $\Gamma'$.
In Table \ref{subgroups of gsp4(f3)} of the Appendix, we record some properties of the conjugacy classes of such $\Gamma'$.
The column headings (A) and (B) refer to whether the following conditions on $H$ hold:
\begin{enumerate}[label=(\Alph*)]
    \item For every simple $\Fpbar[\Gamma]$-submodule $W \leq \splie_4^\vee$ there exists a semisimple element $\gamma \in \Gamma$ such that for some $w \in W$ and $z \in \mathfrak{z}(\splie_4^\gamma)$ we have $w(z) \neq 0$. \label{submodule condition again}   
    \item For every simple $\Fpbar[\Gamma]$-submodule $W \leq \splie_4^\vee$ there exists a regular semisimple element $\gamma \in \Gamma$ such that $W^\gamma \neq 0$. \label{abundant condition}
\end{enumerate}
Condition \ref{abundant condition} in some form is seen in both the definition of enormous subgroup of \cite[Definition 7.5.3 (E3)]{surfaces} and $\hat{G}$-abundant subgroup of \cite[Definition 5.18(ii)]{thorne2019}.
Condition \ref{submodule condition again}, equivalent to our spanning condition, appears in Definition \ref{adequate} and is implied by condition \ref{abundant condition} (since $\splie_4 \cong \splie_4^\vee$ as $\Gamma$-modules).
The column `Adequate' refers to whether $H$ is $\Sp_4$-adequate, in the sense of Definition \ref{adequate}.

The remaining columns are included for applications to modularity of abelian surfaces.
The column heading `Tidy' indicates whether $G$ is a tidy subgroup in the sense of Definition \ref{defn: tidy}.
The column `Induced' refer to whether the underlying $4$-dimensional module for $G$ is no longer absolutely irreducible on restriction to some index $2$ subgroup of $G$, and `Split-induced' refers to whether this reducibility may take place over $\mathbb{F}_p$ (so it is split induced).

As already noted in Lemma \ref{Sp4(Fp)}, condition \ref{submodule condition again} holds in every case, while condition \ref{abundant condition} holds for only 12 of the possible 25 choices of $\Gamma'$.
A similar computation shows that when $p=5$ (resp. $p=7$) condition \ref{submodule condition again} always holds while condition \ref{abundant condition} holds for $36/69$ (resp. $41/86$) possible images.
The importance of these conditions on simple submodules is that to prove an analogue of Proposition \ref{TW places existence} in which the Frobenius image at a Taylor--Wiles place is required to be regular semisimple, it would be very difficult to weaken condition \ref{abundant condition} on some image of the representation.
Moreover, the failure of the cohomological vanishing conditions in small characteristic do not automatically eliminate representations with image $\Gamma'$, as we saw with the definition of $\hat{G}$-reasonable representations (Definition \ref{reasonable}).

\section{Local computations} \label{sec local computations}

In this section, we will carry out computations with Hecke algebras of reductive groups over a non-archimedean local field with respect to various compact open subgroups.
We will firstly prove a version of the Satake isomorphism at a slightly deeper level in Theorem \ref{satake} and compose it with the maps of \cite{bk} to define certain abelian subalgebras of Hecke algebras.
We will then use the Bernstein presentation of the Iwahori-Hecke algebra to prove a variation of a standard result about modules over this algebra in Section \ref{parahoric invariants subsection}.
We will also define various data attached to a semisimple element of the dual group in Construction \ref{parahoric construction}.
Finally, we will use Jacquet modules to deduce properties of parabolically induced representations and relate these to the local Langlands correspondence for $\GSp_4$ in Propositions \ref{gsp4 p invariants} and \ref{gsp4 p1 invariants}.

\subsection{A Satake isomorphism} \label{sect satake}

Throughout this section, we will let $F_v$ be a non-archimedean local field with ring of integers $\OFv$, uniformizer $\varpi_v$ and residue field $k(v)$ of order $q_v$.
Let $\mathcal{O}$ be as in Section \ref{def setup}, the ring of integers of a finite extension of $\mathbb{Q}_p$ with residue field $k$ of characteristic $p$.
We will assume that $q_v$ is coprime to $p$ and suppose that $\CalO$ contains a fixed squareroot $q_v^{\frac{1}{2}}$ of $q_v$.
Let $| \cdot |_v$ denote the absolute value on $F_v$; its squareroot has image in $q_v^{\frac{1}{2} \mathbb{Z}}$, which we will view as a subgroup of either $\mathbb{R}^\times_{\geq 0}$ or $\mathcal{O}^\times$.

For now, let $G$ be a reductive group defined over $\OFv$ with a split maximal torus $T$ and Borel subgroup $B$ containing $T$ with unipotent radical $N$.
We will often abuse notation and write $G$ for $G(F_v)$ (and similarly for other groups) when it is clear from context.
We will consider various Hecke algebras $\mathcal{H}(H,U)$ for closed algebraic subgroups $H \leq G$ and compact open subgroups $U \leq H$; these are the algebras of compactly supported, locally constant, $U$-biinvariant functions $H \to \mathcal{O}$ which form an $\mathcal{O}$-algebra under convolution.

Let $G^{\der}$ be the derived subgroup of $G$, defined over $\OFv$ and compatible with base change by \cite[Theorem 5.3.1]{conrad}. 
Let $C_G = G/G^{\der}$ be the cocenter of $G$, a torus which we will assume is split.
Fix a quotient $\Delta$ of the finite abelian group $C_G(k(v))$ of order coprime to $|W_G|$.
We may then consider the composition $\theta$ of the natural maps
\begin{align*}
    G(\OFv) \to G(k(v)) \to C_G(k(v)) \to \Delta
\end{align*}
and define $K_1 = \ker \theta$, a finite index subgroup of $K = G(\OFv)$.
We may realise $C_G$ as the quotient of the maximal torus $T$ by a split subtorus $T'$, as described in Section \ref{dual group deltaq}.
It follows that $T(k(v))$ surjects onto $C_G(k(v))$, and thus we may take coset representatives for $K_1$ which lie inside $T(\OFv)$.

The Satake isomorphism  \cite[Theorem 4.1]{cartier} is an isomorphism
\begin{align*}
    \mathcal{S}: \mathcal{H}(G(F_v),K) &\to \mathcal{H}(T(F_v),T(F_v) \cap K)^{W_G} \\
    f &\mapsto (t \mapsto \delta_B(t)^{\frac{1}{2}} \int_N f(tn) dn).
\end{align*}
We will show that the same formula gives an isomorphism
\begin{align*}
    \mathcal{S}_1: \mathcal{H}(G(F_v),K_1) \to \mathcal{H}(T(F_v),T(F_v) \cap K_1)^{W_G}.
\end{align*}
Here $\delta_B: T(F_v) \to \mathbb{R}^\times_{\geq 0}$ is the homomorphism given by $\delta_B(t) = |\det(\restr{\Ad(t)}{\Lie N})|_v$.

Let $\Phi^+$ denote the positive roots of $G$ with respect to the choice of Borel subgroup $B$ and $\Lambda^+ = \{\lambda \in X_*(T): \langle \lambda, \alpha \rangle \geq 0 \text{ for all } \alpha \in \Phi^+\}$.
Let $\rho \in X_*(T) \otimes \mathbb{Q}$ be the half-sum of all positive coroots of $G$.
Then for $\mu \in X_*(T)$ we have $\delta_B(\mu(\varpi_v))^{\frac{1}{2}} = q_v^{-\langle \mu, \rho \rangle}$, so that $(\delta_B)^{\frac{1}{2}}$ is valued in $q_v^{\frac{1}{2} \mathbb{Z}}$.
Fix for each $\delta \in \Delta$ a preimage $\Dot{\delta} \in T(\OFv)$. 
The well-known variation of the Cartan decomposition $G(F_v) = \bigsqcup_{\lambda \in \Lambda^+} K \lambda(\varpi_v) K$ yields a further decomposition 
\begin{align*}
    G(F_v) = \bigsqcup_{\lambda \in \Lambda^+, \delta \in \Delta} K_1 \lambda(\varpi_v) \Dot{\delta} K_1.
\end{align*}
Indeed, $K_1 \triangleleft K$ is a normal subgroup with coset representatives given by $\{\Dot{\delta}\}_{\delta \in \Delta}$, and these coset representatives all commute with $\lambda(\varpi_v) \in T(\OFv)$.
Thus we have a basis of $\mathcal{H}(G(F_v),K_1)$ given by the characteristic functions $\phi_{(\lambda,\delta)} = [K_1 \lambda(\varpi_v) \Dot{\delta} K_1]$ for $\lambda \in \Lambda^+$ and $\delta \in \Delta$.

\begin{theorem}
The Satake map 
\begin{align*}
    \mathcal{S}_1: \mathcal{H}(G(F_v),K_1) &\to \mathcal{H}(T(F_v),T(F_v) \cap K_1)^{W_G}\\
    f &\mapsto (t \mapsto \delta_B(t)^{\frac{1}{2}} \int_N f(tn) dn)
\end{align*} is an isomorphism of $\mathcal{O}$-algebras.
\end{theorem}\label{satake}
\begin{proof}
We will follow the approach of the usual Satake isomorphism presented in \cite[Theorem 4.1]{cartier}. 
We will normalise all Haar measures on subgroups $\Gamma \leq G(F_v)$ to give $\Gamma \cap K_1$ measure 1 and write $T_1 = T(F_v) \cap K_1$.
The groups $G,N,T,K$ and $K_1$ are all unimodular.
For the group $B$ we may define a left Haar measure by the formula
\begin{align} \label{intB intM intN}
    \int_B f(b) d_l b = \int_T \int_N f(tn) dt dn
\end{align}
with the required normalisation.
The map $\mathcal{S}_1$ has image in $\mathcal{H}(T,T \cap K_1)$, as $f$ is biinvariant under $K_1$ and due to the fact that $\mathcal{S}_1 f(t) = \delta_B(t)^{-\frac{1}{2}} \int_N f(nt) dn$ by \cite[IV (19)]{cartier}.

We firstly show the map defines a homomorphism. 
To do this, we may write $\mathcal{S}_1$ as the composition of restricting functions to $\mathcal{H}(B,B \cap K_1)$ and then applying the same integral formula.
The proof that the second map is a homomorphism is exactly as in the usual case, since it is largely independent of choice of compact open subgroup.
So we show restriction $\mathcal{H}(G,K_1) \to \mathcal{H}(B, B \cap K_1)$ is a homomorphism.
If $f \in \mathcal{H}(G,K_1)$ we have by \cite[IV (4)]{cartier} that
\begin{align*}
    \int_G f(u) du = \frac{1}{|\Delta|}\int_K \int_B f(bk) dk d_l b,
\end{align*}
recalling we have normalised the Haar measures with respect to intersections with $K_1$.
Decomposing $K$ into $K_1$-cosets then we have 
\begin{align*}
    \int_G f(u) du = \sum_{\delta \in \Delta} \frac{1}{|\Delta|}\int_{K_1} \int_B f(b\Dot{\delta} k ) dk d_l b.
\end{align*}
The modular function for $B$ is given by $\Delta_B(tn) = \delta_B(t)^{-1}$ for $t \in T$ and $n \in N$ by \cite[Equation (11)]{cartier}.
So right translation by $\Dot{\delta} \in T(\OFv)$ has no effect since $\delta_B(\Dot{\delta}) = 1$.
It follows that
\begin{align} \label{intG intK1B}
    \int_G f(u) du = \int_{K_1} \int_B f(b k ) dk d_l b.
\end{align}
If $g \in \mathcal{H}(G,K_1)$ and $x \in B$ then
\begin{align*}
    (f \star g)(x) &= \int_G f(u) g(u^{-1} x) du \\
    &= \int_{K_1} \int_B f(b k) g((b k)^{-1} x) dk d_l b \\
    &= \int_B f(b ) g(b^{-1} x) d_l b \text{ by } K_1 \text{ biinvariance},
\end{align*}
which is exactly the convolution of the restrictions of $f$ and $g$ to $\mathcal{H}(B,B \cap K_1)$ evaluated on $x$.

Next we show that the image of $\mathcal{S}_1$ is contained inside $\mathcal{H}(T(F_v),T(F_v) \cap K_1)^{W_G}$.
Since $G$ is defined over $\OFv$, the Weyl group is also defined over $\OFv$.
Since our coset representatives for $\Delta$ lie in $T(\OFv)$ we can take representatives of the Weyl group to lie in $N_G(T)(F_v) \cap K_1$ by taking any representatives in $G(\OFv)$ and multiplying through by suitable $\dot{\delta}$.
It then would suffice to show that $\mathcal{S}_1f(x t x^{-1}) = \mathcal{S}_1 f(t)$ for every $x \in N_G(T)(F_v) \cap K_1$.
This can then be proved in exactly the same way as \cite{cartier} by firstly showing that this holds for $t \in T(\OFv)$ regular.

Finally, we show that we do indeed have a bijection.
We recall that the functions $\phi_{(\lambda,\delta)} = [K_1 \lambda(\varpi_v) \Dot{\delta} K_1]$ for $\lambda \in \Lambda^+$ and $\delta \in \Delta$ form a basis for $\mathcal{H}(G,K_1)$.
Letting $\Phi_{(\lambda,\delta)} = [T_1 \lambda(\varpi_v) \Dot{\delta} T_1]$, we have a basis of $\mathcal{H}(T,T_1)$ given by $\{\Phi_{(\lambda,\delta)}\}_{(\lambda,\delta) \in X_*(T) \times \Delta}$.
Summing up $\Phi_{(\lambda,\delta)}$ over $\lambda$ in a Weyl orbit gives a basis $\Phi'_{(\lambda,\delta)}$ for $\mathcal{H}(T,T_1)^{W_G}$, which may be paramaterized by $(\lambda,\delta) \in \Lambda^+ \times \Delta$.
Recall there is a partial ordering on $\Lambda^+$ given by $\lambda \geq \mu$ if and only if $\lambda - \mu$ is a sum of positive coroots.

We may write $\mathcal{S}_1 \phi_{(\lambda,\delta)} = \sum_{(\mu,\epsilon) \in \Lambda^+ \times \Delta} c(\lambda,\mu,\delta,\epsilon) \Phi'_{(\mu,\epsilon)}$.
We will show that the coefficients $c(\lambda,\mu,\delta,\epsilon)$ satisfy the following:
\begin{enumerate}
    \item $c(\lambda,\lambda,\delta,\epsilon) = 0$ for $\epsilon \neq \delta$ \label{diffdelta}    \item $c(\lambda,\lambda,\delta,\delta) \in \mathcal{O}^\times$ \label{same}
    \item If $c(\lambda,\mu,\delta,\epsilon) \neq 0$ then $\lambda \geq \mu$. \label{diffboth} 
\end{enumerate}
Together these give imply the map is an isomorphism. 
Indeed, if $0 \neq \sum \alpha_{(\lambda,\delta)} \phi_{(\lambda,\delta)} \in \mathcal{H}(G,K_1)$ then let $\lambda \in \Lambda^+$ and $\delta \in \Delta$ be such that $ \alpha_{(\lambda,\delta)} \neq 0$ and for every $\mu > \lambda$  and $\epsilon \in \Delta$ we have $\alpha_{(\mu,\epsilon)} = 0$.
Then the coefficient of $\Phi'_{(\lambda,\delta)}$ in the image under $\mathcal{S}_1$ is non-zero and this shows injectivity.
For surjectivity, induction on $\lambda$ shows that $\Phi'_{(\lambda,\delta)}$ is in the image of $\mathcal{S}_1$.

We have 
\begin{align} 
c(\lambda,\mu,\delta,\epsilon) &= \mathcal{S}_1 \phi_{(\lambda,\delta)}(\mu(\varpi_v) \Dot{\epsilon}) \nonumber \\ 
&= \delta_B(\mu(\varpi_v))^{-\frac{1}{2}} \int_N \phi_{(\lambda,\delta)}(n \mu(\varpi_v) \Dot{\epsilon}) dn \label{satake coefficients integral}\\
&= q_v^{\langle \mu, \rho \rangle} \mu_N(N \mu(\varpi_v) \Dot{\epsilon} \cap K_1 \lambda(\varpi_v) \Dot{\delta} K_1) \label{satake coefficients Nmeasure}\\
&= q_v^{\langle \mu, \rho \rangle} \mu_G(N \mu(\varpi_v) \Dot{\epsilon} K_1 \cap K_1 \lambda(\varpi_v) \Dot{\delta} K_1) \nonumber
\end{align}
where $\mu_N$ and $\mu_G$ denote the respective Haar measures on $N$ and $G$, and with the final equality following from the bijection
\begin{align*}
    (N s \cap K_1 t K_1)/(N \cap K_1) \leftrightarrow (N s K_1 \cap K_1 t K_1)/K_1
\end{align*}
for any $s,t \in T$.
For \ref{diffdelta} we need to show that for $\epsilon \neq \delta$ that
$$N \lambda(\varpi_v) \Dot{\epsilon} K_1 \cap K_1 \lambda(\varpi_v) \Dot{\delta} K_1 = \emptyset.$$
Suppose then that there exists $x,y \in K_1$ and $n \in N$ with $x \lambda(\varpi_v) \Dot{\epsilon} y = n \lambda(\varpi_v) \Dot{\delta}$; we will show $\delta = \epsilon$.
Then since $N$ is a normal subgroup of $B$ we may write $$\Dot{\epsilon} y \Dot{\delta}^{-1} = \lambda(\varpi_v)^{-1} x^{-1} \lambda(\varpi_v) n'$$
for some $n' \in N$.
Applying the quotient map $\pi: G \to C_G$ to the above equality, we obtain
$$
\pi(\Dot{\epsilon}) \pi(y) \pi(\Dot{\delta}^{-1}) = \pi(x)^{-1},
$$
where we have used that $\pi(N(F_v)) = \{1\}$.
Both sides have image contained in $C_G(\OFv)$, and passing to the images under the map $C_G(\OFv) \to \Delta$, we are left with $\epsilon \delta^{-1} = 1$.

Now we show \ref{same}.
Since  $\Dot{\delta} K_1 = K_1 \Dot{\delta}$ we have that
$$c(\lambda,\lambda,\delta,\delta) = q_v^{\langle \lambda, \rho \rangle} \mu_G(N \lambda(\varpi_v) K_1 \cap K_1 \lambda(\varpi_v) K_1).$$
Certainly the inclusion $N \lambda(\varpi_v) K_1 \cap K_1 \lambda(\varpi_v) K_1) \supset \lambda(\varpi_v) K_1$ holds, so $$\mu_G(N \lambda(\varpi_v) K_1 \cap K_1 \lambda(\varpi_v) K_1) \geq 1.$$
We have by \cite[4.4.4(ii)]{bruhattits} that $N \lambda(\varpi_v) K \cap K \lambda(\varpi_v) K = \lambda(\varpi_v) K$, so
$$N \lambda(\varpi_v) K_1 \cap K_1 \lambda(\varpi_v) K_1)\subset \bigsqcup_{\epsilon \in \Delta} \lambda(\varpi_v) \Dot{\epsilon} K_1 = \lambda(\varpi_v) K.$$
If $\lambda(\varpi_v)^{-1} x \lambda(\varpi_v) y \in \Dot{\epsilon} K_1$ for $x,y \in K_1$ and $\epsilon \in \Delta$, then applying $\pi$ and taking images in $\Delta$ we see that $\epsilon = 1$.
Thus $c(\lambda,\lambda,\delta,\delta) = q_v^{\langle \lambda, \rho \rangle} \in \mathcal{O}^\times$.

Finally, to show $\ref{diffboth}$ we can reduce to the case of $K_1 = K$ since if 
\begin{equation*} \label{intersection nonempty}
       N \mu(\varpi_v) \Dot{\epsilon} K_1 \cap K_1 \lambda(\varpi_v) \Dot{\delta} K_1 \neq \emptyset 
\end{equation*}
then $N \mu(\varpi_v) K \cap K \lambda(\varpi_v) K \neq \emptyset$.
It is shown in \cite[(4.4.4)(i)]{bruhattits} that if this latter intersection is non-empty then $\lambda \geq \mu$.
\end{proof}

From the proof of the isomorphism, we see that the Satake map is compatible with passage to a quotient $\Tilde{\Delta}$ of $\Delta$ (inducing a compact open subgroup $\Tilde{K_1}$), in the following sense.
\begin{lemma}\label{satake group change} 
The following diagram of $\mathcal{O}$-algebras commutes: \\
\begin{tikzcd}
\phi_{(\lambda,\delta)} \arrow[d, maps to] & \mathcal{H}(G,K_1) \arrow[r, "\mathcal{S}_1"] \arrow[d] & \mathcal{H}(T,T \cap K_1)^{W_G} \arrow[d] & \Phi'_{(\lambda,\delta)} \arrow[d, maps to] \\
\phi_{(\lambda,\Tilde{\delta})}                    & \mathcal{H}(G,\Tilde{K_1}) \arrow[r, "\Tilde{\mathcal{S}_1}"]           & \mathcal{H}(T,T \cap \Tilde{K_1})^{W_G}           & \Phi'_{(\lambda,\Tilde{\delta})}                   
\end{tikzcd} \\
where $\lambda \in \Lambda^+$ and $\Tilde{\delta}$ denotes the image in $\Tilde{\Delta}$ of $\delta \in \Delta$.
\end{lemma}
\begin{proof}
We may write
\begin{align*}
    \mathcal{S}_1(\phi_{(\lambda,\delta)}) &= \sum_{(\mu,\epsilon) \in \Lambda^+ \times \Delta} c_{(\lambda,\mu,\delta,\epsilon)} \Phi'_{(\mu,\epsilon)} \\
    \Tilde{\mathcal{S}_1}(\phi_{(\lambda,\Tilde{\delta})}) &= \sum_{(\mu,\Tilde{\epsilon}) \in \Lambda^+ \times \Tilde{\Delta}} \Tilde{c}_{(\lambda,\mu,\Tilde{\delta},\Tilde{\epsilon})} \Phi'_{(\mu,\Tilde{\epsilon})},
\end{align*}
as in the proof of Theorem \ref{satake}.
Letting $\theta: \Delta \to \Tilde{\Delta}$ be the quotient map, we must show that the coefficients satisfy the following relation
\begin{align*}
    \sum_{(\delta,\epsilon) \in \theta^{-1}(\Tilde{\delta}) \times \theta^{-1}(\Tilde{\epsilon})} c_{(\lambda,\mu,\delta,\epsilon)} = \Tilde{c}_{(\lambda,\mu,\Tilde{\delta},\Tilde{\epsilon})}.
\end{align*}
We saw in equality \ref{satake coefficients Nmeasure} that these coefficients are given by
\begin{align*}
c_{(\lambda,\mu,\delta,\epsilon)} &= q_v^{\langle \mu, \rho \rangle} \mu_N(N \mu(\varpi_v) \Dot{\epsilon} \cap K_1 \lambda(\varpi_v) \Dot{\delta} K_1) \\
&= q_v^{\langle \mu, \rho \rangle} \mu_N(N \mu(\varpi_v) \cap K_1 \lambda(\varpi_v) \Dot{\delta} \Dot{\epsilon}^{-1} K_1),  
\end{align*}
where $\mu_N(K_1 \cap N) = 1$.
We can write 
$$
\Tilde{K_1} \lambda(\varpi_v) \Dot{\Tilde{\delta}} \Dot{\Tilde{\epsilon}}^{-1} \Tilde{K_1} = \bigsqcup_{\eta \in \ker(\theta)} K_1 \lambda(\varpi_v) \Dot{\delta} \Dot{\epsilon}^{-1} \Dot{\eta} K_1.
$$
It follows that
\begin{align*}
    \sum_{(\delta,\epsilon) \in \theta^{-1}(\Tilde{\delta}) \times \theta^{-1}(\Tilde{\epsilon})} c_{(\lambda,\mu,\delta,\epsilon)} = \frac{\#\Delta}{\#\Tilde{\Delta}} q_v^{\langle \mu, \rho \rangle} \mu_N(\Tilde{K_1} \lambda(\varpi_v) \Dot{\Tilde{\delta}} \Dot{\Tilde{\epsilon}}^{-1} \Tilde{K_1}).
\end{align*}
This in turn equals $\Tilde{c}_{(\lambda,\mu,\Tilde{\delta},\Tilde{\epsilon})}$, as this factor of $\#\Delta/\#\Tilde{\Delta}$ is exactly the ratio between the respective Haar measures on $N$ for which either $K_1 \cap N$ or $\Tilde{K_1} \cap N$ have measure 1.

Finally, all the maps in the diagram are maps of commutative $\mathcal{O}$-algebras, since the horizontal maps are algebra isomorphisms and the map $\mathcal{H}(T,T \cap K_1)^{W_G} \to \mathcal{H}(T,T \cap \Tilde{K_1})^{W_G}$ is clearly an algebra homomorphism.
\end{proof}

We will now offer a different perspective on the map $\mathcal{S}_1$, following arguments similar to those of \cite{HKP10}.
Our goal will be to show that $\mathcal{S}_1$ will respect formation of $K_1$-invariants of parabolic inductions of characters of $T$.
For the rest of this section only, we will take all Hecke algebras to have $\Qpbar$-coefficients.

Firstly we will introduce some more notation.
Let $\mathfrak{t}_1 = K_1 \cap T(F_v) = \ker(T(\OFv) \to \Delta)$.
Let $\mathcal{H}_{K_1} = \mathcal{H}(G,K_1)$ and $\mathcal{A}_1 = \Qpbar[T/\mathfrak{t}_1]$.
If $\chi: T(F_v) \to \Qpbar^\times$ is a smooth characterwe will let $i_B^G(\chi)$ denote the normalized parabolic induction, given by those smooth functions $\phi: G(F_v) \to \Qpbar$ satisfying $\phi(bg) = \delta_B(b)^{\frac{1}{2}} \chi(b) \phi(g)$ for every $b \in B$ and $g \in G$.

Let $M_{1} = C_c(\mathfrak{t}_1 N \backslash G / K_1)$ denote the space of compactly supported $\Qpbar$-valued functions on $G(F_v)$, left-invariant under $\mathfrak{t}_1 N(F_v)$ and right-invariant under $K_1$.
Then $M_1$ naturally has the structure of a right $\mathcal{H}_{K_1}$-module via the convolution giving $K_1$ measure 1.
The Iwasawa decomposition $G = BK$ implies that 
\begin{equation} \label{gen iwasawa decomp}
    G = B K_1 = T N K_1
\end{equation}
and the elements $m^{(\lambda,\delta)} := 1_{\mathfrak{t}_1 N \lambda(\varpi_v) \Dot{\delta} K_1}$ for $\lambda \in X_*(T)$ and $\delta \in \Delta$ form a $\Qpbar$-vector space basis for $M_1$.
There is also a left $\mathcal{A}_1$-module structure on $M_1$, defined by 
$$\lambda(\varpi_v) \dot\delta \cdot m^{(\mu,\epsilon)} = \delta_B(\lambda(\varpi_v))^{\frac{1}{2}} m^{(\mu+\lambda,\delta \epsilon)}.$$
These actions are compatible, and from the above basis we see that $M_1$ is a free module over $\mathcal{A}_1$ of rank $1$.
This induces a homomorphism $\mathcal{H}_{K_1} \to \mathcal{A}_1$ sending $h \mapsto h^\vee$, where
\begin{equation} \label{satake property}
    m \star h = h^\vee \cdot m.
\end{equation}

\begin{lemma}
The map $h \mapsto h^\vee$ is given by $\mathcal{S}_1$.
\end{lemma}
\begin{proof}
Let $\lambda \in \Lambda^+$ and $\delta \in \Delta$.
Let $h = \phi_{(\lambda,\delta)}$ and write $$h^\vee = \sum_{(\mu,\epsilon) \in X_*(T) \times \Delta} d(\lambda,\mu,\delta,\epsilon) \Phi_{(\mu,\epsilon)}.$$
We will show that the coefficients $d(\lambda,\mu,\delta,\epsilon)$ are exactly given as in the proof of Theorem \ref{satake}.
Consider applying the equality \ref{satake property} to $m = m^{(1,1)}$ and evaluating these functions at $\mu(\varpi_v) \Dot{\epsilon}$.
We obtain
\begin{align} \label{computing hvee}
   &\int_G 1_{\mathfrak{t}_1 N K_1}(g) \phi_{(\lambda,\delta)}(g^{-1} \mu(\varpi_v) \Dot{\epsilon}) dg \notag \\ 
   = &\sum_{(\nu,\gamma) \in X_*(T) \times \Delta} d(\lambda,\nu,\delta,\gamma) \delta_B(\nu(\varpi_v))^{\frac{1}{2}} m^{(\nu,\gamma)} (\mu(\varpi_v) \Dot{\epsilon}). 
\end{align}
By \ref{intG intK1B} and \ref{intB intM intN}, the left hand side of \ref{computing hvee} becomes
\begin{align*}
  \int_{K_1} \int_T \int_N 1_{\mathfrak{t}_1 N K_1}(tnk) \phi_{(\lambda,\delta)}((tnk)^{-1} \mu(\varpi_v) \Dot{\epsilon}) dk dt dn \\
  = \int_{N} \phi_{(\lambda,\delta)}(n^{-1} \mu(\varpi_v) \Dot{\epsilon}) dn  = \int_{N} \phi_{(\lambda,\delta)}(n \mu(\varpi_v) \Dot{\epsilon}) dn,
\end{align*}
since $\phi_{(\lambda,\delta)}$ is bi-invariant under $K_1 \supset \mathfrak{t}_1$ and $N$ is unimodular.
The right hand side of \ref{computing hvee} equals
$d(\lambda,\mu,\delta,\epsilon) \delta_B(\mu(\varpi_v))^{\frac{1}{2}}$.
We therefore see that
$$
d(\lambda,\mu,\delta,\epsilon) = \delta_B(\mu(\varpi_v))^{-\frac{1}{2}} \int_{N} \phi_{(\lambda,\delta)}(n \mu(\varpi_v) \Dot{\epsilon}) dn,
$$
which is exactly formula \ref{satake coefficients integral} (which is still valid for computing the coefficient of $\Phi_{(\lambda,\delta)}$, even when $\mu \not\in \Lambda^+$).
\end{proof}

We will now describe involutions of the algebras $\mathcal{H}_{K_1}$ and $\mathcal{A}_1$, and show that these are compatible with $\mathcal{S}_1$.
This will allow us to convert between left and right $\mathcal{H}_{K_1}$-modules.
Let $\iota_{K_1}$ be the involution on $\mathcal{H}_{K_1}$ defined by $\iota_{K_1}(f)(x) = f(x^{-1})$ for every $x \in G(F_v)$ and $f \in \mathcal{H}_{K_1}$.
A special case of this will be the involution $\iota_{\mathcal{A}_1}$ of $\mathcal{A}_1$, given by $\iota_{\mathcal{A}_1}(\Phi_{(\lambda,\delta})) = \Phi_{(-\lambda,\delta^{-1})}$.

\begin{lemma} \label{satake involution}
The involutions $\iota_{K_1}$ and $\iota_{\mathcal{A}_1}$ are compatible with the Satake map $\mathcal{S}_1$, in the sense that
$$
\mathcal{S}_1 \circ \iota_{K_1} = \iota_{\mathcal{A}_1} \circ \mathcal{S}_1.
$$
\end{lemma}
\begin{proof}
We will prove this directly from the coefficient formulas $\ref{satake coefficients integral}$ and $\ref{satake coefficients Nmeasure}$.
We just need to show that
$$
c(\lambda,\mu,\delta,\epsilon) = c(-\lambda,-\mu,\delta^{-1},\epsilon^{-1})
$$
for every $\lambda,\mu \in \Lambda^+$ and $\delta,\epsilon \in \Delta$.
By \cite[IV (19)]{cartier} and our coefficient formulas, we can write
\begin{align*} 
c(\lambda,\mu,\delta,\epsilon) &=
  \delta_B(\mu(\varpi_v))^{-\frac{1}{2}} \int_N \phi_{(\lambda,\delta)}(n \mu(\varpi_v) \Dot{\epsilon}) dn \\ 
  &= \delta_B(\mu(\varpi_v))^{\frac{1}{2}} \int_N\phi_{(\lambda,\delta)}(\mu(\varpi_v) \Dot{\epsilon} n) dn \\
  &= \delta_B(\mu(\varpi_v))^{\frac{1}{2}} \mu_N(\mu(\varpi_v) \Dot{\epsilon}N  \cap K_1 \lambda(\varpi_v) \Dot{\delta} K_1) \\
  &= \delta_B(\mu(\varpi_v)^{-1})^{-\frac{1}{2}} \mu_N(N \mu(\varpi_v)^{-1} \Dot{\epsilon}^{-1} \cap K_1 \lambda(\varpi_v)^{-1} \Dot{\delta}^{-1} K_1),
\end{align*}
where we have used unimodularity of $N$ in the final line.
Since these formulas do not depend on the choice of representatives of $\Delta$ in $T(\OFv)$, the final line is exactly $c(-\lambda,-\mu,\delta^{-1},\epsilon^{-1})$ and we are done.
\end{proof}
\begin{remark}
A proof of Lemma \ref{satake involution} similar to that of \cite[4.4]{HKP10} should also be possible by creating a $\mathcal{A}_1$-valued, $G$-invariant perfect $\mathcal{A}_1$-sesquilinear pairing on a module $M~ \supset~M_{1}$ and applying the relation \ref{satake property}.
\end{remark}

Now let $\chi: T(F_v) \to \Qpbar^\times$ be a smooth character.
There is an isomorphism of right $\mathcal{H}_{K_1}$-modules 
\begin{equation}\label{K1 invariants iso}
    \Qpbar \otimes_{\chi, \mathcal{A}_1} M_{1} \xrightarrow{\sim} i_B^G(\chi^{-1})^{K_1},
\end{equation}
where the action of $\mathcal{H}_{K_1}$ on the $K_1$-invariants of the parabolically induced representation is by first applying $\iota_{K_1}$ followed by the usual left-action.
We can now prove the desired compatibility between our Satake isomorphism and such parabolic inductions.

\begin{proposition} \label{satake parabolic induction}
Let $\chi: T(F_v) \to \Qpbar^\times$ be a smooth character and $\pi = i_B^G \chi$.
The space of $K_1$-invariants, $\pi^{K_1}$, is at most one-dimensional, and if $\pi^{K_1} \neq 0$ then $\restr{\chi}{\mathfrak{t}_1}$ is trivial and for every
$f \in \mathcal{H}_{K_1}$ and $v \in \pi^{K_1}$ we have $$f v = \chi(\mathcal{S}_1 f) v.$$
\end{proposition}
\begin{proof}
Suppose that $0 \neq \varphi \in \pi^{K_1}$.
Then writing $G = T N K_1$ as in \ref{gen iwasawa decomp}, we see that $\varphi(t n k) = \delta_B(t)^{\frac{1}{2}} \chi(t) \varphi(1)$.
Since $\varphi$ is determined by $\varphi(1)$, we see that $\dim \pi^{K_1} = 1$ and $\varphi(1) = \varphi(t) = \chi(t) \varphi(1)$ if $t \in T(F_v) \cap K_1 = \mathfrak{t}_1$.

Now let $f \in \mathcal{H}_{K_1}$ and let $m \in \Qpbar \otimes_{\chi^{-1}, \mathcal{A}_1} M_{1}$.
If we can show that $m \star \iota_{K_1}(f) = \chi(\mathcal{S}_1(f)) m$ then we would be done using the isomorphism \ref{K1 invariants iso}.
By equation \ref{satake property}, we have
\begin{align*}
    m \star \iota_{K_1}(f) &= \mathcal{S}_1(\iota_{K_1}(f)) m \\
                        &= \iota_{\mathcal{A}_1}(\mathcal{S}_1(f)) m \qquad \quad \text{ by Lemma } \ref{satake involution} \\
                        &= \chi^{-1}(\iota_{\mathcal{A}_1}(\mathcal{S}_1(f))) m \\
                        &= \chi(\mathcal{S}_1(f)) m,
\end{align*}
where we have used that $\chi^{-1} \circ \iota_{\mathcal{A}_1} = \chi$ as maps $\mathcal{A}_1 \to \Qpbar$, which can be checked on basis elements $\Phi_{(\lambda,\delta)}$ for $\mathcal{A}_1$. 
\end{proof}

\subsection{An abelian subalgebra of the $\mathfrak{p}_1$-Hecke algebra} \label{p_1}

Now consider the following setup.
Let $G$ be a reductive group over $\OFv$, with $T \subset B \subset G$ a split maximal torus contained in a Borel subgroup of $G$ with Levi decomposition $B = TU$. 
We continue with the notation of writing $H$ for $H(F_v)$ when $H \leq G$ is a closed algebraic subgroup.
Let $P = LN$ be a standard parabolic subgroup of $G$, with unipotent radical $N$, and $L$ a Levi factor containing $T$.
Let $W_G$ (resp. $W_L$) denote the Weyl group of $G$ (resp. $L$).
Let $\delta_P: P(F_v) \to \CalO^\times$ be the modulus character for $P$, defined by $\delta_P(p) = |\det(\restr{\Ad(p)}{\Lie N})|_v$.
Let $(X^*(T),\Phi,X_*(T),\Phi^\vee)$ denote the root datum of $G$ with respect to this choice of maximal torus and let $\Phi^+$ denote the positive roots with respect to the Borel subgroup $B$.
We assume that $p$ is coprime to the order of $W_G$ (and hence of $W_L$).

Let $\mathfrak{p}$ be the parahoric subgroup given by the preimage of $P(k(v))$ in $G(\OFv)$.
We will apply the Satake map of Section \ref{sect satake} to the reductive group $L$, defined over $\OFv$.
Let $C_L = L/L^{\der}$ be the cocenter of $L$ and let $\Delta$ be a quotient of $C_L(k(v))$ of order coprime to the order of $W_G$.
Define the further compact open subgroup $\mathfrak{p}_1 \subset \mathfrak{p}$ given by the kernel of the composition
\begin{align*}
    \mathfrak{p} \to P(k(v)) \to L(k(v)) \to \Delta.
\end{align*}
Thus we have an isomorphism $$\mathcal{S}_1^{-1}: \mathcal{H}(T,T \cap \mathfrak{p}_1)^{W_L} \to \mathcal{H}(L,L \cap \mathfrak{p}_1)$$ by Theorem \ref{satake}.
We will show the existence of an injective algebra homomorphism $$\mathcal{T}_1: \mathcal{H}(L,L \cap \mathfrak{p}_1) \to \mathcal{H}(G,\mathfrak{p}_1),$$ so that the composition will allow us to view $\mathcal{H}(T,T \cap \mathfrak{p}_1)^{W_L}$ as an abelian subalgebra of $\mathcal{H}(G,\mathfrak{p}_1)$.

To do this, we will use the notion of strongly $(P,\mathfrak{p}_1)$-positive elements as in \cite[6.16]{bk}.
We say $z \in L$ is positive if $z (N \cap \mathfrak{p}_1) z^{-1} \subset N \cap \mathfrak{p}_1$ and $z^{-1} (\bar{N} \cap \mathfrak{p}_1) z \subset \bar{N} \cap \mathfrak{p}_1$, where $\bar{N}$ denotes the unipotent radical of the opposite parabolic subgroup to $P$.
We say further that $\zeta \in Z(L)$ is strongly positive if $\zeta$ is positive and for all pairs of compact open subgroups $H_1,H_2 \leq N$ and $K_1,K_2 \leq \bar{N}$ there exists $n,m \geq 0$ such that 
\begin{align*}
    \zeta^n H_1 \zeta^{-n} \leq H_2 \\
    \zeta^{-m} K_1 \zeta^m \leq K_2.
\end{align*}
Write $L^+$ for the set of positive elements of $L$.
Our homomorphism will come from \cite[Theorem 7.2]{bk} provided we can find a strongly $(P,\mathfrak{p}_1)$-positive element $\zeta \in Z(L)$ whose characteristic function $[\mathfrak{p}_1 \zeta \mathfrak{p}_1]$ is invertible in $\mathcal{H}(G,\mathfrak{p}_1)$.

\begin{proposition} \label{T1 subalgebra}
The map 
\begin{align*}
\mathcal{T}_1: \mathcal{H}(L^+, \mathfrak{p}_1 \cap L^+) &\to \mathcal{H}(G, \mathfrak{p}_1) \\
[(\mathfrak{p}_1 \cap L^+) z (\mathfrak{p}_1 \cap L^+)] &\mapsto \delta_P(z)^{\frac{1}{2}} [\mathfrak{p}_1 z \mathfrak{p}_1]
\end{align*}
is an injective algebra homomorphism.
Suppose that $\lambda \in X_*(T)$ take values in $Z(L)$ and satisfies $\langle \lambda, \alpha \rangle > 0$ for every $\alpha \in \Phi^+ \setminus \Phi_L$.
Then $\lambda(\varpi_v)$ is a strongly $(P,\mathfrak{p}_1)$-positive element and $[\mathfrak{p}_1 \lambda(\varpi_v) \mathfrak{p}_1]$ is invertible in $\mathcal{H}(G, \mathfrak{p}_1)$.
The map $\mathcal{T}_1$ extends to an injective algebra homomorphism
\begin{align*}
\mathcal{T}_1: \mathcal{H}(L, \mathfrak{p}_1 \cap L) &\to \mathcal{H}(G, \mathfrak{p}_1).
\end{align*}
\end{proposition}
\begin{proof}
The map on the Hecke algebra with positive support is a twist of the injective algebra homomorphism stated in \cite[6.12]{bk}.
The proof of \cite[Lemma 6.14]{bk} shows that $\lambda(\varpi_v)$ is strongly $(P,\mathfrak{p}_1)$-positive for $\lambda$ as in the statement of the proposition.
By \cite[2.2.9, 2.18]{psrgroups}, we can find such a $\lambda \in X_*(T)$.

Next we show that $[\mathfrak{p}_1 \lambda(\varpi_v) \mathfrak{p}_1]$ is invertible.
Letting $\Iw^{(1)}$ be the pro-p-Iwahori subgroup, given by the preimage of $U(k(v))$ in $G(\OFv)$, and $\mathcal{H}^{(1)} = \mathcal{H}(G,\Iw^{(1)})$, we have by \cite[Corollary 1]{vig} that $[\Iw^{(1)} \lambda(\varpi_v) \Iw^{(1)}]$ is invertible in $\mathcal{H}^{(1)}$.
Let $W^{(1)} = W_G \rtimes (T(F_v)/T(\OFv) \times T(k(v))$, which we will view as a quotient of $N_G(T)(F_v)$.
We will extend the usual length function on $(W_G \rtimes T(F_v)/T(\OFv))$ to $l: W^{(1)} \to \mathbb{Z}_{\geq 0}$ satisfying $l(\alpha w \beta) = l(w)$ for any $w \in W^{(1)}$ and $\alpha,\beta \in T(k(v))$, as in \cite[Proposition 1]{vig}.
If $w \in W^{(1)}$, choose a representative $n \in N_G(T)(F_v)$ and let $T_w = [\Iw^{(1)} n \Iw^{(1)}]$.
The Iwahori-Matsumoto presentation of $\mathcal{H}^{(1)}$ \cite[Theorem 1]{vig} states that $\mathcal{H}^1$ has $\{T_w\}_{w \in W^{(1)}}$ as a basis with multiplication defined by the braid relations
\begin{center}
 $T_w T_{w'} = T_{w w'}$ whenever $w,w' \in W^{(1)}$ satisfy $l(w w') = l(w) + l(w')$   
\end{center}
and certain quadratic relations.
By definition of $l$ and the same proof of \cite[Lemma 2.2]{lusztig89}, it can be seen that for every $w \in W_L$ and $\alpha \in T(k(v))$ the basis elements $T_{w \alpha}$ and $T_{\lambda(\varpi_v)}$ commute.

Let $\Gamma = \ker(\pi: L(k(v)) \to \Delta)$.
If we choose representatives $\Dot{w}$ lying in $L^{\der}(k(v)) \subset \Gamma$ for $w \in W_L$, we claim there is a decomposition of the group
$$\Gamma \cdot N(k(v)) = \bigsqcup_{w \in W_L, \alpha \in T(k(v)) \cap \Gamma} U(k(v)) \Dot{w} \alpha U(k(v)).$$
Firstly, observe that the right hand side is actually a disjoint union.
Indeed, suppose that $u \Dot{w} \alpha v = \Dot{x} \beta$ with $u,v \in U(k(v))$, $\alpha,\beta \in T(k(v))$ and $w,x \in W_L$.
Then by the Bruhat decomposition for $G$ we have that $w = x$.
Writing the equality as
$ (\Dot{w} \alpha) (\Dot{w} \alpha)^{-1} u (\Dot{w} \alpha) v = \Dot{w} \beta$,
we deduce that $y^{-1} u y v = \alpha^{-1} \beta$ for some $y$.
It follows that $y^{-1} u y \in B(k(v))$, so must lie in $U(k(v))$, as it is unipotent.
Therefore $\alpha^{-1} \beta \in T(k(v)) \cap U(k(v))$, so $\alpha = \beta$ as required.

Suppose that $m  \in \Gamma$ and $n \in N(k(v))$.
By the Bruhat decomposition for $L(k(v))$, we may write $m = x \Dot{w} y$ where $x,y \in (B \cap L)(k(v))$ and $w \in W_L$. 
Then since $B = T U$, we may write $x = s x'$ and $y = t y'$ with $x',y' \in L(k(v))$ unipotent and $s,t \in T(k(v))$.
It follows that we can write $m = z \Dot{w} \alpha y'$ with $z \in L(k(v))$ unipotent and $\alpha \in T(k(v))$.
Applying $\pi$ to both sides, we see that $1 = \pi(m) = \pi(z \Dot{w} \alpha y') = \pi(\alpha)$.
We deduce that $\alpha \in \Gamma \cap T(k(v))$ and since $N \subset U$ we have $mn \in U(k(v)) \Dot{w} \alpha U(k(v))$.

To show the other containment, we just need to show $U(k(v)) \subset \Gamma N(k(v))$, so let $u \in U(k(v))$.
Since $u \in P(k(v)) = L(k(v)) N(k(v))$, we can write $u = l n$ with $l \in L(k(v))$ and note that $l$ is necessarily unipotent.
Therefore $\pi(l) = 1$ and $l \in \Gamma$, as required.

We see that $$\mathfrak{p}_1 = \bigsqcup_{w \in W_L, \alpha \in \Gamma \cap T(k(v))} \Iw^{(1)} \Dot{w} \Dot{\alpha} \Iw^{(1)},$$ where for each $\alpha \in \Gamma \cap T(k(v))$, $\Dot{\alpha}$ is a choice of lift of $\alpha$ to $T(\OFv)$.
Therefore the indicator function $[\mathfrak{p}_1]$ in $\mathcal{H}^{(1)}$ commutes with $[\Iw^{(1)} \lambda(\varpi_v) \Iw^{(1)}]$, and they have product $[\mathfrak{p}_1 \lambda(\varpi_v) \mathfrak{p}_1]$.
Taking the product of the inverse in $[\Iw^{(1)} \lambda(\varpi_v) \Iw^{(1)}]$ in $\mathcal{H}^{(1)}$ with $[\mathfrak{p}_1]$ then yields the desired inverse in $\mathcal{H}(G,\mathfrak{p}_1)$.

To conclude the proof then, we apply \cite[Theorem 7.2]{bk} with our invertible element to see that $\mathcal{T}_1$ does indeed extend to an algebra homomorphism on $\mathcal{H}(L,\mathfrak{p}_1 \cap L)$.
\end{proof}

We have constructed our desired abelian subalgebra.

\begin{corollary}
The composition $$\mathcal{T}_1 \circ \mathcal{S}_1^{-1}: \mathcal{H}(T,T \cap \mathfrak{p}_1)^{W_L} \to\mathcal{H}(G,\mathfrak{p}_1)$$
is an injective $\CalO$-algebra homomorphism.
\end{corollary}

Now suppose that $\Delta$ has $p$-power order (note that $p$ is coprime to $\#W_G$ by assumption).
Consider the Hecke algebras of the torus $T$ at levels $T \cap \mathfrak{p}$ and $T \cap \mathfrak{p}_1$: these are just group algebras of suitable quotients of $T$, and we have a natural map of $\mathcal{O}$-algebras
\begin{align*}
    \mathcal{O}[T/(T \cap \mathfrak{p}_1)] \to \mathcal{O}[T/(T \cap \mathfrak{p})]
\end{align*}
which also restricts to a map of $W_L$-invariants.
The maps induce bijections on maximal ideals.
Indeed, we can find a section of the above map which gives a presentation $\mathcal{O}[T/(T \cap \mathfrak{p}_1)] \cong \mathcal{O}[T/(T \cap \mathfrak{p})][\Delta]$ using the Teichmuller map.
It suffices to show there is a bijection modulo $\lambda$; but then the presentation reduces to one of the form $k[T/(T~\cap~\mathfrak{p})][y_1,\ldots,y_r]/(y_i^{p^{a_i}})$, where $y_i+1$ are generators of $\Delta$ and $a_i \geq 1$, since $k$ is characteristic $p$.
The bijection follows, as the residual map is given by quotienting of nilpotents.

Now suppose $N$ is any $\mathcal{H}(G,\mathfrak{p}_1)$-module which is finite free over $\mathcal{O}$.
Then $N$ receives the structure of a module over $\mathcal{O}[T/(T \cap \mathfrak{p}_1)]^{W_L}$ via the composition $\mathcal{T}_1 \circ \mathcal{S}_1^{-1}: \mathcal{O}[T/(T \cap \mathfrak{p}_1)]^{W_L} \hookrightarrow \mathcal{H}(G,\mathfrak{p}_1)$.
Similarly, the invariants $N^\mathfrak{p} = N^\Delta$ (viewing $\Delta = \mathfrak{p}/\mathfrak{p}_1$) receive the structure of a module over $\mathcal{O}[T/(T \cap \mathfrak{p})]^{W_L}$.
We can relate the module to the algebra at level $\mathfrak{p}$ in the following way.

\begin{proposition}
Let $N$ be as above.
Let $\mathfrak{n}_1$ be any maximal ideal of $\mathcal{O}[T/(T \cap \mathfrak{p}_1)]^{W_L}$ and $\mathfrak{n}_0$ the corresponding maximal ideal of $\mathcal{O}[T/(T \cap \mathfrak{p})]^{W_L}$.
The localisations $N_{\mathfrak{n}_1}$ and $(N^{\mathfrak{p}})_{\mathfrak{n}_0}$ are submodules of $N$ and $N^{\mathfrak{p}}$ respectively, and under this identification we have the equality $(N_{\mathfrak{n}_1})^{\Delta} = (N^{\mathfrak{p}})_{\mathfrak{n}_0}$.
\end{proposition}
\begin{proof}
The proposition is clear provided we know that the two actions of $\mathcal{O}[T/(T \cap \mathfrak{p}_1)]^{W_L}$ (viewing either as a subalgebra of $\mathcal{H}(G,\mathfrak{p}_1)$ or passing to $\mathcal{O}[T/(T \cap \mathfrak{p})]^{W_L}$ and viewing as a subalgebra of $\mathcal{H}(G,\mathfrak{p})$) on $N^\mathfrak{p}$ coincide. 
By Lemma \ref{satake group change} it suffices to show the actions of the $\phi_{(\lambda,\delta)} \in \mathcal{H}(L,L \cap \mathfrak{p}_1)$ and $\phi_{\lambda} \in \mathcal{H}(L,L \cap \mathfrak{p}_1)$ agree.
We know the maps $\mathcal{T}$ and $\mathcal{T}_1$ on positive elements, and the respective maps extend uniquely, so it suffices to show that for $z = \lambda(\varpi_v) \dot{\delta} \in L^+$ we have a bijection $\mathfrak{p}_1 z \mathfrak{p}_1/\mathfrak{p}_1 \to \mathfrak{p} z \mathfrak{p}/\mathfrak{p}$.
Decomposing $\mathfrak{p}_1 z \mathfrak{p}_1~=~\bigsqcup_i z_i \mathfrak{p}_1$, it is easy to see that $\mathfrak{p} z \mathfrak{p} = \bigsqcup_i z_i \mathfrak{p}$, as $z$ commutes with the representatives $\dot{\delta}$ and $\mathfrak{p}_1 \triangleleft \mathfrak{p}$.
\end{proof}

\subsection{Parahoric invariants} \label{parahoric invariants subsection}

Continue with notation as in the previous subsection.
Assume for the rest of Section \ref{sec local computations} that $q_v \equiv 1 \mod p$.
Let $\mathfrak{b} = \{g \in G(\OFv): g \mod (\varpi_v) \in B(k(v))\}$ denote the Iwahori-subgroup of $G(F_v)$ and let $\mathcal{H} = \mathcal{H}(G,\mathfrak{b})$ denote the Iwahori-Hecke algebra of compactly supported, $\mathcal{O}$-valued functions on the double coset space $\mathfrak{b}\backslash G(F_v)/\mathfrak{b}$.
We will make use of the Bernstein presentation of $\mathcal{H}$.
This presents $\mathcal{H}$ as a twisted tensor product of $\mathcal{A} := \mathcal{O}[X_*(T)]$ and $\CalO[\mathfrak{b}\backslash \mathfrak{g} / \mathfrak{b}]$.
We refer the reader to \cite{HKP10} for more details, but note that the inclusion $\mathcal{A} \to \mathcal{H}$ arises from sending $\lambda \mapsto q_v^{- \langle\rho, \lambda \rangle} [\mathfrak{b} \lambda(\varpi_v) \mathfrak{b}]$, where $\lambda$ is a dominant cocharacter and $\rho$ is the usual half sum of positive roots.
Since $q_v \equiv 1 \mod p$, the Bernstein presentation of $\mathcal{H}$ allows us to identify the reduction of $\mathcal{H}$ modulo $(\lambda)$ as $\overline{\mathcal{H}} = k[X_*(T) \rtimes W_G]$ (see for example \cite[Lemma 7.5]{thorne2019}).

Continue to let $P \supset B$ be a standard parabolic subgroup of $G$ with standard Levi factor $L \supset T$.
Let $\mathfrak{p} = \{g \in G(\OFv): g \mod \lambda \in P(k(v))\}$ be the parahoric subgroup associated to the parabolic subgroup $P$ (and similarly $\mathfrak{g}$ with respect to $G$).
For each $w \in W_G$ we will fix a coset representative $\Dot{w} \in G(F_v)$. For $U \subset G(F_v)$ a compact open subset biinvariant under $\mathfrak{b}$, let $[U] \in \mathcal{H}$ denote the indicator function on $U$.  
The Bruhat decompositions for $L(k(v))$ and $G(k(v))$ imply the following equalities in $\mathcal{H}$:
\begin{align*}
    [\mathfrak{p}] &= \sum_{w \in W_L} [\mathfrak{b} \Dot{w} \mathfrak{b}]\\
     [\mathfrak{g}] &= \sum_{w \in W_G} [\mathfrak{b} \Dot{w} \mathfrak{b}].
\end{align*}

The choice of parabolic subgroup $P$ determines an injective algebra homomorphism $\mathcal{T}_P: \mathcal{H}(L, \mathfrak{p} \cap L) \to \mathcal{H}(G,\mathfrak{p})$, as in Proposition \ref{T1 subalgebra} applied to the case where $\Delta$ is trivial (so $\pone = \mathfrak{p}$).
Thus $\mathcal{T}_B$ allows us to view $\mathcal{A}$ as a subalgebra of $\mathcal{H}$, and this map is easily seen to agree with the inclusion of $\mathcal{A}$ via the Bernstein presentation.
Consider the further subalgebra $\mathcal{B} = \mathcal{O}[X_*(T)]^{W_L} \leq \mathcal{A}$.
If we let $\mathcal{S}_P$ denote the isomorphism in Proposition \ref{satake} applied to the case where the reductive group is the Levi subgroup $M$ and the group $\Delta$ is trivial, the composition $\mathcal{T}_P \circ \mathcal{S}_P^{-1}: \mathcal{B} \to \mathcal{H}(G,\mathfrak{p})$ is an injective algebra homomorphism.
If $N$ is any left $\mathcal{H}$-module $N$ then $[\mathfrak{p}]N$ is a left $\mathcal{H}(G,\mathfrak{p})$-module.
The map $\mathcal{T}_P \circ \mathcal{S}_P^{-1}$ allows us to further view $[\mathfrak{p}]N$ as a $\mathcal{B}$-module.
This action coincides with the restriction of the above action of $\mathcal{A}$ to $\mathcal{B}$ by \cite[Proposition 3.4]{adequacy}.

Suppose $\mathfrak{n}$ is a maximal ideal of $\mathcal{A}$ with residue field $k$. 
Intersection of $\mathfrak{n}$ with $\mathcal{B}$ defines a maximal ideal of $\mathcal{B}$, with maximal ideals $\mathfrak{n}$ and $\mathfrak{n}'$ having the same image if and only if they lie in the same orbit under the action of $W_L$.
This map is also surjective, and hence maximal ideals of $\mathcal{B}$ are in bijection with $W_L$-orbits of maximal ideals of $\mathcal{A}$. 
This bijection continues to hold on taking both of these rings modulo $(\lambda)$.

\begin{proposition} \label{parahoric invariants iso}
Let $N$ be a left $\mathcal{H}$-module, which is either a finite free $\CalO$-module or a finite dimensional $k$-vector space.
Let $\mathfrak{n}_0$ be a maximal ideal of $\mathcal{B}$ and suppose that the centralizer of some (equivalently, any) maximal ideal of $\mathcal{A}$ above $\mathfrak{n}_0$ in $W_G$ is contained in $W_L$.
Let $\mathfrak{q}$ denote the unique maximal ideal of $\mathcal{O}[X_*(T)]^{W_G}$ lying below $\mathfrak{n}_0$.
 Then multiplication by $[\mathfrak{g}]$ defines an isomorphism $[\mathfrak{g}]: ([\mathfrak{p}]N)_{\mathfrak{n}_0} \to ([\mathfrak{g}] N)_{\mathfrak{q}}$.
\end{proposition}

\begin{proof}
In the case where $N$ is a finite free $\mathcal{O}$-module, the claims will hold true if and only if they hold true for $N/\lambda N$.
Hence we will take $N$ to be a finite dimensional $k$-vector space from now on.

As observed, $\bar{\mathcal{H}} = k[X_*(T) \rtimes W_G]$ is just the group algebra of a semi-direct product. 
We may decompose $N = \oplus_{\mathfrak{m}} N_{\mathfrak{m}}$ as a module over $\overline{\mathcal{A}} = k[X_*(T)]$, since the localisations are given by the generalized eigenspaces for the action of $k[X_*(T)]$ on $N$.
If $\mathfrak{n}$ is a maximal ideal of $\overline{\mathcal{B}} = k[X_*(T)]^{W_L}$, let $\mathcal{O}_{\mathfrak{n}}$ denote the set of maximal ideals of $\overline{\mathcal{A}}$ above $\mathfrak{n}$.
We claim firstly that $([\mathfrak{p}] N)_{\mathfrak{n}} = [\mathfrak{p}] ( \oplus_{\mathfrak{m} \in \mathcal{O}_{\mathfrak{n}}} N_{\mathfrak{m}})$.
Indeed, if $n \in N$ is in the generalized eigenspace for the character of $\overline{\mathcal{A}}$ defining $\mathfrak{m}$, we have that $(\sum_{w \in W_L} w) n$ is in the generalized eigenspace for the restriction of this character to $\overline{\mathcal{B}}$; the other containment is similar.

Now let $M = \oplus_{\mathfrak{m} \in \mathcal{O}_{\mathfrak{n}_0}} N_{\mathfrak{m}}$ and $\pi: N = \oplus_{\mathfrak{m}} N_{\mathfrak{m}} \to M$ denote the natural projection.
We next claim that the composition $\pi \circ [\mathfrak{g}]:  ([\mathfrak{p}]N)_{\mathfrak{n}_0} \to M$ has image in $([\mathfrak{p}]N)_{\mathfrak{n}_0}$ and is given by multiplication by $\#W_L \in k^\times$.
To see this, let $n \in ([\mathfrak{p}]N)_{\mathfrak{n}_0}$ and suppose firstly that $w \in W_G\setminus W_L$.
The only maximal ideals $\mathfrak{m}$ of $\overline{\mathcal{A}}$ for which the image of $w n$ in the localisation $N_\mathfrak{m}$ could be non-zero are such that $w^{-1} \mathfrak{m} \in \mathcal{O}_{\mathfrak{n}_0}$.
We cannot have both $\mathfrak{m}$ and $w \mathfrak{m}$ in $\mathcal{O}_{\mathfrak{n}_0}$, since this would contradict that the centralizer of $\mathfrak{m}$ is contained in $W_L$ and therefore the composition $\pi \circ w$ is zero.
If $w \in W_L$, then $w n = n$ for any $n \in ([\mathfrak{p}]N)_{\mathfrak{n}_0}$, and the claim then follows.

Lastly, we show that $\restr{\pi}{([\mathfrak{g}]N)_{\mathfrak{q}}}$ is an injection with image contained in $([\mathfrak{p}]N)_{\mathfrak{n}_0}$.
We will then be done, since $\frac{1}{\#W_L} \pi$ will provide the inverse to $[\mathfrak{g}]$ on $([\mathfrak{p}]N)_{\mathfrak{n}_0}$ (and the image of $([\mathfrak{p}]N)_{\mathfrak{n}_0}$ under $[\mathfrak{g}]$ is supported only at maximal ideals in $\mathcal{O}_{\mathfrak{n}_0}$ as a module over $\overline{\mathcal{A}}$, and hence only at the maximal ideal $\mathfrak{q}$ as an $\mathcal{O}[X_*(T)]^{W_G}$-module.
Let us suppose that $0 \neq n \in ([\mathfrak{g}]N)_{\mathfrak{q}}$.
Then since $n = \frac{1}{\# W_L} [\mathfrak{p}] n \in [\mathfrak{p}] N$ there must exist a maximal ideal $\mathfrak{n}$ of $\overline{\mathcal{B}} = k[X_*(T)]^{W_L}$ for which the image of $n$ in $([\mathfrak{p}] N)_{\mathfrak{n}}$ is non-zero.
The maximal ideals $\mathfrak{n}$ and $\mathfrak{n}_0$ must be conjugate, so there exists a $w \in W_G$ such that $wn$ has non-zero image in $([\mathfrak{p}] N)_{\mathfrak{n}_0} \subset M$.
Since $wn = n$, we have shown that $\restr{\pi}{[\mathfrak{g}]N}$ is indeed injective.
Multiplication by $w \in W_L$ commutes with $\pi$, so $\pi(([\mathfrak{g}]N)_{\mathfrak{q}}) \subset [\mathfrak{p}] N \cap M = ([\mathfrak{p}]N)_{\mathfrak{n}_0}$ and we are done.
\end{proof}

\begin{remark}
While the statement of Proposition \ref{parahoric invariants iso} is really about a module over the parahoric Hecke algebra, the proof used the Bernstein presentation of the Iwahori-Hecke algebra.
It would be of interest to know whether similar presentations exist of the parahoric Hecke algebras, at least in the case when $q_v=1$.
\end{remark}

Now let $\hat{G}$ be a reductive group over $k$ and suppose that $\hat{T}$ is a split maximal torus contained in a Borel subgroup $\hat{B}$ of $\hat{G}$.
Suppose that the arising pinned root datum is dual to that of $G$, and let $\iota$ denote the induced isomorphisms $X_*(T) \cong X^*(\hat{T})$ and $X^*(T) \cong X_*(\hat{T})$ (just as in Section \ref{dual group deltaq}).
If we take some $\bar{g} \in \hat{T}(k)$, the isomorphism $\iota$ will allow us to define various objects, firstly via the same construction of a Levi subgroup of $G$ in Definition \ref{Lv dual levi}.

\begin{construction} \label{parahoric construction}
Let $\bar{g} \in \hat{T}(k)$. Then we can form the following:
\begin{enumerate}
    \item $M_{\bar{g}} = Z_{\hat{G}_k}(\bar{g})$, the scheme-theoretic centralizer of $\bar{g}$.
    \item A subtorus $\hat{S}$ of $\hat{T}$. This will be given by the connected center of $M_{\bar{g}}$, $\hat{S} =~Z(M_{\bar{g}})^\circ$.
    \item A standard Levi subgroup $\hat{L}$ of $\hat{G}$ with connected center $\hat{S}$. 
    The scheme-theoretic centralizer of $\hat{S}$, $\hat{L} = Z_{\hat{G}_k}(\hat{S})$, will be a Levi subgroup of $\hat{G}$ containing $M_{\bar{g}}$ and let $\hat{P}$ denote the standard parabolic subgroup of $\hat{G}$ admitting $\hat{L}$ as a Levi factor.
    \item A standard parabolic subgroup $P$ of $G$ with unipotent radical $N$ and standard Levi subgroup $L$.
    These will be the dual standard parabolic subgroup and Levi subgroup to $\hat{P}$ and $\hat{L}$ via $\iota$, as in Section \ref{dual group deltaq}.
    \item A parahoric subgroup $\mathfrak{p} \leq G(\OFv)$.
    This will be given by those $g \in G(\OFv)$ for which $g \mod \piv \in P(k(v))$.
    \item A finite $p$-group $\Delta$.
    This will denote the maximal $p$-power quotient of $C_L(k(v))$ such that the image of $Z(G)(k(v))$ is trivial. 
    Here $C_L = L/L^{\der}$ is the cocenter of $L$.
    \item A compact open subgroup $\pone \leq \mathfrak{p}$, normal in $\mathfrak{p}$, with quotient $\Delta$.
  This will be given by the kernel of the composition $$\mathfrak{p} \to P(k(v)) \to L(k(v)) \to C_L(k(v)) \to \Delta.$$
\end{enumerate}
\end{construction}

Associated to the element $\bar{g} \in \hat{T}(k)$, we obtain a character
\begin{align*}
    \bar{\chi}: X_*(T) &\to k^\times \\
                \alpha &\mapsto \iota(\alpha)(\bar{g})
\end{align*}
and the same formula defines a bijection $\hat{T}(k) \cong \Hom(X_*(T),k^\times)$.
The next lemma will place restrictions on the centralizer of $\bar{\chi}$ under action of the Weyl group $W_G$, which will help us verify the hypotheses of Proposition \ref{parahoric invariants iso} in cases of interest.

\begin{lemma} \label{weyl stabilizer}
The following inclusions hold:
\begin{enumerate}
    \item $\Stab_{W_{\hat{G}}}(\bar{g}) \leq W_{\hat{L}}$ \label{gbar stab}
    \item $\Stab_{W_G} \bar{\chi} \leq W_{L}$ \label{char stab}
\end{enumerate}
with equalities if $M_{\bar{g}}$ is a Levi subgroup.
\end{lemma}
\begin{proof}
By \cite[Theorem 3.5.3]{carter1985finite}, the identity component of the centralizer of $\bar{g}$, $M_{\bar{g}}^{\circ}$, is generated by the maximal torus $\hat{T}$ together with the root groups $U_{\alpha}$ corresponding to roots $\alpha$ for which $\alpha(\bar{g}) = 1$, while the centralizer $M_{\bar{g}}$ is generated by $M_{\bar{g}}^\circ$ together with those Weyl elements of $W_{\hat{G}}$ which centralize $\bar{g}$.
Since $\hat{L} = Z_{\hat{G}_k}(Z(M_{\bar{g}})^\circ)$ always contains $M_{\bar{g}}$, \ref{gbar stab} therefore holds.
If $M_{\bar{g}} = \hat{L}$ is a Levi subgroup, then $M_{\bar{g}}^{\circ} = M_{\bar{g}}$ and those elements of $W_{\hat{G}}$ centralizing $\bar{g}$ lie in $M_{\bar{g}}(\bar{k})$, showing the reverse containment.

By \cite[14.8]{borelLAG} we may view $W_G$ and $W_{\hat{G}}$ as subgroups of $\Aut(X^*(T)) \cong \Aut(X_*(\hat{T}))$ generated by reflections corresponding to simple roots of $G$ (or simple coroots of $\hat{G}$).
Therefore, $\iota$ induces an isomorphism between $W_G$ and $W_{\hat{G}}$.
This isomorphism restricts to an isomorphism between $W_{L}$ and $W_{\hat{L}}$, and also between $\Stab_{W_G}(\bar{\chi})$ and $\Stab_{W_{\hat{G}}}(\bar{g})$.
Thus \ref{char stab} follows immediately from \ref{gbar stab}.
\end{proof}

We can extend the map $\bar{\chi}$ to a map $\mathcal{A} \to k$, defining a maximal ideal $\mathfrak{m}$ of $\mathcal{A}$.
From Lemma \ref{weyl stabilizer} and Proposition \ref{parahoric invariants iso} we deduce the following corollary.

\begin{corollary} \label{gbar parahoric invariants iso}
Let $N$ be a left $\mathcal{H}$-module, which is either a finite free $\CalO$-module or a finite dimensional $k$-vector space.
Let $\bar{g} \in \hat{T}(k)$ and form $\mathfrak{p}$ and $L$ as in Construction \ref{parahoric construction}.
 Let $\mathfrak{n}_0$ be the maximal ideal of $\mathcal{B} = \CalO[X_*(T)]^{W_L}$ corresponding to $\bar{g}$ and let $\mathfrak{q}$ denote the unique maximal ideal of $\mathcal{O}[X_*(T)]^{W_G}$ lying below $\mathfrak{n}_0$.
 Then multiplication by $[\mathfrak{g}]$ defines an isomorphism $[\mathfrak{g}]: ([\mathfrak{p}]N)_{\mathfrak{n}_0} \to ([\mathfrak{g}] N)_{\mathfrak{q}}$.
\end{corollary}

\subsection{Jacquet module computations} \label{sec jacquet computations}

Continue with the setup of Section \ref{parahoric invariants subsection}.
This required fixing some $\bar{g} \in \hat{T}(k)$, and it allowed us to define the data of Construction \ref{parahoric construction}.
We saw already that the maximal ideals of $\mathcal{O}[T/(T \cap \pone)]^{W_L}$ and $\mathcal{O}[T/(T \cap \mathfrak{p})]^{W_L}$ are naturally in bijection.
We will let $\mathfrak{n}_1$ (respectively $\mathfrak{n}_0$) denote the maximal ideal of $\mathcal{O}[T/(T \cap \pone)]^{W_L}$ (respectively $\mathcal{O}[T/T(\OFv)]^{W_L}$) corresponding to $\bar{\chi}$.

If $\sigma$ is an admissible $\Qpbar[L(F_v)]$-module, we will let $i_P^G \sigma$ denote the usual normalised induction of $\sigma$, so that left translation by $p \in P(F_v)$ acts via $\delta_P(p)^{\frac{1}{2}} \sigma(p)$.
If $\rho$ is a $\Qpbar[G(F_v)]$-module then let $r_N(\rho)$ denote the normalised Jacquet module of $\rho$, viewed as a representation of $L(F_v)$.
It will be of use to compute Jacquet modules of parabolically induced representations.
A formula for this is given by \cite[2.12]{BZ77} and we will combine this with \cite[2.11(a)]{BZ77} to understand the indexing of the stated filtration.
The following lemma will help us understand invariants of smooth representations of $G(F_v)$ in terms of the invariants of the Jacquet module.

\begin{lemma} \label{jacquet module invariants}
Let $\pi$ be a smooth $\Qpbar[G(F_v)]$-module.
The quotient map gives isomorphisms
\begin{align*}
   \pi^\mathfrak{p} &\xrightarrow{\sim} r_N(\pi)^{(\mathfrak{p} \cap L)} \\
    \pi^{\mathfrak{p}_1} &\xrightarrow{\sim} r_N(\pi)^{({\mathfrak{p}_1} \cap L)} .
\end{align*}
\end{lemma}
\begin{proof}
We saw already in the proof of Proposition \ref{T1 subalgebra} that there exists a strongly $(P,\mathfrak{p}_1)$-positive element of $L$.
The proof then goes through exactly as in the proof of \cite[Proposition 3.2 (4)]{adequacy}.
\end{proof}

\begin{lemma} \label{principal series loc inv}
Let $\chi: T(F_v) \to \Qpbar^\times$ be a smooth character of $T$ and $\pi = i_B^G \chi$.
The space of localized invariants $(\pi^\pone)_{\mathfrak{n}_1}$ is at most one-dimensional.
  If $(\pi^\pone)_{\mathfrak{n}_1} \neq 0$ then
  $\chi$ is valued in $\Zpbar^\times$, there exists $w \in W_G$ such that the action of $\mathcal{O}[T/(T \cap \pone)]^{W_L}$ on $(\pi^\pone)_{\mathfrak{n}_1}$ is through $w \chi$, and this $w \chi$ lifts $\bar{\chi}$.
\end{lemma}
\begin{proof}
By \cite[2.12 and 2.11(a)]{BZ77}, we can write the semisimplification of the Jacquet module $r_N(\pi)$ as
$$\bigoplus_{w \in W_L \backslash W_G} i_{B \cap L}^L w \chi.$$
Since taking invariants under a compact open subgroup is an exact functor, it suffices to compute the dimension of these localised invariants in the semisimplification of $r_N(\pi)$.
By Lemma \ref{jacquet module invariants} then we just need to compute
$$(i_{B \cap L}^L w \chi)^{(\pone \cap L)}$$
for $w \in W_G$.
Viewed as a module over $\mathcal{O}[T/(T \cap \pone)]^{W_L}$, this is exactly $(w \chi)^{(T \cap \pone)}$ by Proposition \ref{satake parabolic induction}.
Now suppose that $(\pi^\pone)_{\mathfrak{n}_1} \neq 0$. 
We claim there must be some $w \chi$ valued in $\Zpbar^\times$ and lifting $\bar{\chi}$. 

If $((w \chi)^{(T \cap \mathfrak{p}_1)})_{\mathfrak{n}_1} \neq 0$ but $w \chi$ is not integral then there exists $t \in T(F_v)$ with $\alpha = w \chi(t) \equiv 0 \mod m_{\Zpbar}$.
The element $\alpha$ is integral over $\mathcal{O}$ and hence satisfies $\alpha^n + c_{n-1} \alpha^{n-1} + \cdots + c_0 = 0$ for some $c_i \in \mathcal{O}$.
We can take these $c_i \in \lambda \mathcal{O}$, as the coefficients in the minimal polynomial are sums and products of the Galois conjugates of $\alpha$, which will all have positive $p$-adic valuation.
Then as elements of $\mathcal{O}[T/(T \cap \pone)]^{W_L}$, $t^n + c_{n-1} t^{n-1} + \cdots + c_0 \equiv t^n \not\equiv 0 \mod \mathfrak{n}_1$, so the localisation $((w \chi)^{(T \cap \mathfrak{p}_1)})_{\mathfrak{n}_1} = 0$.

We must then have that $w \chi$ lifts $\bar{\chi}$.
Indeed, let $t \in T(F_v)$ with $\alpha = w \chi(t) \in \Zpbar^\times$ and minimal polynomial over $\mathcal{O}$ given by $f(X) \in \mathcal{O}[X]$.
We must have $f(t) - f(\alpha) \in \mathfrak{n}_1$, in order for the localisation to be non-zero.
Thus $f(t) \equiv 0 \mod \mathfrak{n}_1$.
Let $\bar{\beta} \in k^\times$ denote the image of $t$ under $\bar{\chi}$ and let $\bar{f}(X)$ be the reduction of $f(X)$ modulo $(\lambda)$.
We must have $\bar{f}(\bar{\beta}) = 0$.
Writing $\bar{f}(X) = \bar{g}(X) \bar{h}(X)$, where $\bar{g}(X) = (X- \bar{\beta})^s$ and $\bar{g}(X),\bar{h}(X) \in k[X]$ are coprime, if $\bar{h}(X)$ was not equal to $1$ then we could lift this factorization to a non-trivial factorisation of $f(X)$ in $\mathcal{O}[X]$ by Hensel's lemma. So we must have that $\bar{f}(X) = (X-\bar{\beta})^n$ and $\alpha$ lifts $\bar{\beta}$ proving the claim.

Note that there is at most one such coset $W_L w$ across $W_L \backslash W_G$ by Lemma \ref{weyl stabilizer}.
Moreover, we can find a choice of $w$ for which $w \chi$ actually lifts $\bar{\chi}$, viewing both as characters of $\CalO[T/(T \cap \pone)]$, by going up.
\end{proof}
In the same way, we can prove the following lemma.

\begin{lemma} 
Let $\chi: T(F_v) \to \Qpbar^\times$ be a smooth character of $T$ and $\pi = i_B^G \chi$.
The space of localized invariants  $(\pi^\mathfrak{p})_{\mathfrak{n}_0}$ is at most one-dimensional.
  If $(\pi^\mathfrak{p})_{\mathfrak{n}_0} \neq 0$ then
  $\chi$ is valued in $\Zpbar^\times$, there exists $w \in W_G$ such that
  the action of $\mathcal{O}[T/(T \cap \mathfrak{p})]^{W_L}$ on $(\pi^\mathfrak{p})_{\mathfrak{n}_0}$ is through $w \chi$, and this $w \chi$ lifts $\bar{\chi}$.
\end{lemma}

\subsection{Computations in the case $G = \GSp_4$} \label{gsp4 local}

We will continue with notation as in the previous section, temporarily suspending our assumption that $q_v \equiv 1 \mod p$, and now restricting to the case where $G = \GSp_4$ and $F_v$ is characteristic zero.
There is an isomorphism $\hat{G} \cong \GSp_4$; in everything that follows we will identify $\hat{G}$ with $\GSp_4$ via the isomorphism of root data given by \cite[(2.13)]{roberts2007local}.
A local Langlands correspondence for $\GSp_4$ has been established in \cite[Main Theorem]{gsp4llc}, and we will work with $\Qpbar$ coefficients by fixing an isomorphism $\iota: \mathbb{C} \xrightarrow{\sim} \Qpbar$ compatible with our choice of $q_v^{\frac{1}{2}}$.
As in \cite[2.3]{surfaces}, we will let $\rec_{\GT,p}$ denote the local Langlands map sending the equivalence class of a smooth irreducible $\Qpbar$-valued representation of $\GSp_4(F_v)$ to a Weil--Deligne representation of $W_{F_v}$ valued in $\GSp_4(\Qpbar$).

Suppose once again that $q_v \equiv 1 \mod p$ now.
The goal of this section is to show that if $\pi$ is an irreducible admissible representation of $\GSp_4(F_v)$ for which a certain space of localized invariants are non-zero, then $\rec_{\GT,p}(\pi) = (V_\pi,N_\pi)$ is unramified.
We will make this explicit via the description given in \cite[Table A.7]{roberts2007local}, where we will need to check that $N_\pi=0$.
This description matches up, as stated directly after \cite[Proposition 13.1]{thetacorr}.
The following propositions are similar in statement and allow us to conclude $N_\pi = 0$ in cases where the localised $\mathfrak{p}$ or $\pone$ invariants are non-zero.

\begin{proposition} \label{gsp4 p invariants}
Let $\pi$ be an admissible irreducible $\Qpbar[G(F_v)]$-module.
Suppose that $(\pi^{\mathfrak{p}})_{\mathfrak{n}_0} \neq 0$.
Then $\pi$ is a subquotient of a parabolically induced representation $i_B^G \chi$ for some unramified smooth character $\chi: T(F_v) \to \Zpbar^\times$.
The characters through which $\mathcal{O}[T/T(\OFv)]^{W_L}$ act on $\pi^\mathfrak{p}$ are $W_G$-conjugates of $\chi$ and there exists $w \in W_G$ such that $w \chi$ lifts $\bar{\chi}$.
The localized invariants $(\pi^{\mathfrak{p}})_{\mathfrak{n}_0}$ are 1-dimensional and the action of $\mathcal{O}[T/T(\OFv)]^{W_L}$ is through $w \chi$.
Finally, if $\rec_{\GT,p}(\pi) = (V_\pi,N_\pi)$ is the Weil--Deligne representation associated to $\pi$ under the local Langlands correspondence for $\GSp_4$ then $N_\pi = 0$ and there is an isomorphism of $\mathcal{O}[T/T(\OFv)]^{W_G}$-modules $(\pi^\mathfrak{p})_{\mathfrak{n}_0} \to \pi^\mathfrak{g}$.
\end{proposition}

\begin{proposition} \label{gsp4 p1 invariants}
Let $\pi$ be an admissible irreducible $\Qpbar[G(F_v)]$-module.
Suppose that $(\pi^{\mathfrak{p}_1})_{\mathfrak{n}_1} \neq 0$. Then $\pi$ is a subquotient of a parabolically induced representation $i_B^G \chi$ for some tamely ramified smooth character $\chi: T(F_v) \to \Zpbar^\times$.
The characters through which $\mathcal{O}[T/(T \cap \pone)]^{W_L}$ act on $\pi^{\mathfrak{p}_1}$ are $W_G$-conjugates of $\chi$ and there exists $w \in W_G$ such that $w \chi$ lifts $\bar{\chi}$.
The localized invariants $(\pi^{\mathfrak{p}_1})_{\mathfrak{n}_1}$ are 1-dimensional and the action of $\mathcal{O}[T/(T \cap \pone)]^{W_L}$ is through $w \chi$.
Finally, if $\rec_{\GT,p}(\pi) = (V_\pi,N_\pi)$ is the Weil--Deligne representation associated to $\pi$ under the local Langlands correspondence for $\GSp_4$ then $N_\pi = 0$.
\end{proposition}

We will prove Proposition \ref{gsp4 p1 invariants} firstly, since the proofs are mostly very similar and the nature of the computations are slightly more intricate in this case.

\begin{proof}[Proof of Proposition \ref{gsp4 p1 invariants}]
The first three points are a standard consequence of Frobenius reciprocity together with Lemma \ref{principal series loc inv}; it follows from \cite[Lemma 7.4(i)]{thorne2019}, for example.
To show the final point, it is enough by the previous discussion about the map $L$ to work case by case using the descriptions given in \cite[A.5 Table A.7]{roberts2007local}.
The cases will be determined by $M_{\bar{g}}$ up to conjugacy.
In all but one case, $M_{\bar{g}}$ will be a Levi subgroup and $\bar{\chi}$ will have centralizer in $W_G$ equal to $W_L$ by Lemma \ref{weyl stabilizer}.
In such cases we will look for characters $\chi: T(F_v) \to \Zpbar^\times$ for which the centralizer of $\chi \mod m_{\Zpbar}$ in $W_G$ is conjugate to $W_L$, since some Weyl conjugate of $\chi$ will lift $\bar{\chi}$.
To conclude, we will show that those irreducible subquotients $\pi$ of $i_B^G \chi$ for which $N_\pi \neq 0$ satisfy $(\pi^{\mathfrak{p}_1})_{\mathfrak{n}_1} = 0$.
\begin{enumerate}
    \item Suppose firstly that $\bar{g}$ is regular semisimple; this is the only case for which $L = T$ is a maximal torus and $W_L$ is trivial.
    The stabilizer in $W_G$ of the reduction modulo $(\lambda)$ of some $w \chi$ must be trivial. 
    Hence by the discussion of \cite[1.1]{schmidtgsp4} the only possibility is that $\pi$ is an irreducible principal series representation (occurring in group I in \cite[A.5 Table A.7]{roberts2007local}).
    So $N_\pi=0$ immediately follows.
    \item Next suppose that the stabilizer $M_{\bar{g}}$ of $\bar{g}$ in $\hat{G}$ is conjugate to a Levi subgroup of the Klingen parabolic subgroup.
    Then $L \cong \GL_2 \times \GL_1$ is in fact conjugate to a Levi subgroup of the Siegel parabolic subgroup of $G$ (\cite[Proof of Lemma 2.3.1]{roberts2007local}).
    We know that $W_L$ is the stabilizer of $\bar{\chi}$ and so the only possibilities are groups I or II of \cite[A.5 Table A.7]{roberts2007local}).
    Thus we just need to eliminate the possibility that $\pi$ is conjugate to a representation of the form $\varphi \St_{\GL_2} \rtimes \sigma$ for some smooth characters $\varphi,\sigma$.
    
    Without loss of generality, we will take $\pi = \varphi \St_{\GL_2} \rtimes \sigma$ to be an induced representation from the Siegel parabolic, where $\St_{\GL_2} \rtimes \sigma$ is defined on the Levi factor $M$, a $W_G$-conjugate of $L$.
    We must then take $\bar{\chi}$ to be a reduction of the character of $T$ given by $w \chi$, where $\chi := \nu^{\frac{1}{2}} \varphi \otimes \nu^{-\frac{1}{2}} \varphi \otimes \sigma$, $\nu$ is the absolute value on $F_v$ (notation as in \cite{schmidtgsp4}) and $w \in W_G$.
    We will show that $(\pi^\pone)_{\mathfrak{n}_1} = 0$ to obtain a contradiction.
    We will compute Jacquet modules and take $(L \cap \pone)$-invariants again.
    By \cite[2.12 and 2.11(a)]{BZ77} and Lemma \ref{jacquet module invariants}, we can write $(\pi^\pone)$, up to semisimplification as
    $$\bigoplus_{x \in W_L \backslash W_G / W_M, x(M) \neq L} (i_{B \cap L}^L x s \chi)^{(L \cap \pone)}.$$
    Here $s \in W_M$ is the non-identity element, occurring since the Jacquet module of the Steinberg representation is given by the $s$ twist of the original character for which the Steinberg representation is a submodule of the parabolic induction.
    Moreover, those $x$ for which $x(M) = L$ yield the $(L \cap \pone)$-invariants of a Steinberg representation, which vanish and therefore do not contribute to the sum.
    
    On restricting to $T$, those characters possibly occurring in $(i_{B \cap L}^L x s \chi)^{(L \cap \pone)}$ are of the form $y x s \chi$ for $y \in W_L$.
    To show $(\pi^\pone)_{\mathfrak{n}_1} = 0$, it would suffice to show  that if $w \chi = y x s \chi$ then $x(M) = L = w(M)$.
    We may suppose then that $w^{-1} y x s = 1$ since $\chi$ has trivial stabilizer.
    We have $w^{-1} W_L w = W_M$, so since $w^{-1} y w \in W_M$ and $s \in W_M$, it follows that $w^{-1} x = (w^{-1} y w)^{-1} s^{-1} \in W_M$ and hence $w(M) = x(M)$.
    \item Suppose that $M_{\bar{g}}$ is conjugate to a Levi subgroup of the Siegel parabolic subgroup.
    Then $L$ is conjugate to a Levi subgroup of the Klingen parabolic, isomorphic to $\GL_1 \times \GSp_2 \cong \GL_1 \times \GL_2$.
    The stabilizer of $\bar{\chi}$ is $W_L$ and the only possibility we need to consider is that $\pi$ is conjugate to a representation $\varphi \rtimes \sigma \St_{\GSp_2}$.
    $\bar{\chi}$ must be a lift of a $W_G$-conjugate of $\chi = \varphi \otimes \nu \otimes \nu^{-\frac{1}{2}} \sigma$.
    The remaining analysis to show $(\pi^\pone)_{\mathfrak{n}_1} = 0$ is then similar to the previous case.
    
    \item The remaining possibilities are those for which $L = G$ (though we do not necessarily have $M_{\bar{g}} = \hat{G}$).
    In this case $\Delta$ is trivial and $\pone = \mathfrak{g}$.
    It follows that $N_{\pi} = 0$ from inspection of \cite[A.10 Table A.15 and A.5 Table A.7]{roberts2007local}.
\end{enumerate}
\end{proof}

\begin{proof}[Proof of Proposition \ref{gsp4 p invariants}]
The majority of the proof is very similar to the proof of Proposition \ref{gsp4 p1 invariants}.
We will thus only show the existence of an isomorphism $(\pi^\mathfrak{p})_{\mathfrak{n}_0} \to \pi^\mathfrak{g}$ using the explicit possible descriptions of $\pi$, in contrast to the proof of Proposition \ref{parahoric invariants iso}.
As expected, we may observe from \cite[A.10 Table A.15]{roberts2007local}) that those $\pi$ for which $N_\pi = 0$ are exactly those for which $\pi^\mathfrak{g} \neq 0$
and that in this case $\pi^\mathfrak{g}$ is 1-dimensional.
Since $(\pi^\mathfrak{p})_{\mathfrak{n}_0}$ is also 1-dimensional, with the actions of $\mathcal{O}[T/T(\OFv)]^{W_G}$ given on both by the $W_G$-orbit of $\chi$, we see they are isomorphic as $\mathcal{O}[T/T(\OFv)]^{W_G}$-modules.
\end{proof}

\section{Local-global compatibility} \label{sect local global}

We begin this section by proving some lemmas about when torus-valued representations are conjugate in the Weyl group of an ambient reductive group.
We use these in proving Proposition \ref{local global deltav}, which implies that if we have a sufficiently strong local-global compatibility result then we can deduce a statement necessary for the Taylor--Wiles method applied to our generalized notion of Taylor--Wiles places (as in Section \ref{section tw places}).

\begin{lemma} \label{ccl equiv wccl}
Let $H$ be a (possibly disconnected) reductive group over an algebraically closed field $K$ and let $T$ be a maximal torus in $H^\circ$.
Then any tuples $(s_1,\ldots,s_n)$ and $(t_1,\ldots,t_n)$ in $T^n(K)$ are conjugate via the diagonal action of $H(K)$ if and only if they are conjugate in $N_H(T)(K)$.
\end{lemma}

\begin{proof}
We proceed by induction on $n$, so suppose that $n>0$ and that we can find $w \in N_H(T)$ with $\Ad(w) (s_1,\ldots,s_{n-1}) = (t_1,\ldots,t_{n-1})$.
Thus by firstly conjugating by $w$, we may assume that $s_i = t_i$ for every $i < n$.
Then we know that $s = s_n$ and $t = t_n$ are conjugate by an element of $H(K)$ which centralizes $s_i$ for every $i < n$.
Let $Z = Z_{H}(s_1,\ldots,s_{n-1})$ be the centralizer of $s_1,\ldots,s_{n-1}$.
Then $Z$ is a (possibly disconnected) reductive group over $K$ containing $T$ as a maximal torus.
Thus by replacing $H$ by $Z$, we just need to prove the lemma in the case where $n=1$.

We may write $H(K) = H^\circ(K) \cdot N_H(T)(K)$, since $g \in H(K)$ will conjugate $T$ into some other maximal torus in $H^\circ$, and these two maximal tori will be conjugate by some element of $H^\circ(K)$. 
By writing an element of $H(K)$ which conjugates $s$ into $t$ as a product of an element of $H^\circ(K)$ and of $N_H(T)(K)$, we may without loss of generality take $s$ and $t$ to be conjugate by an element of $H^\circ(K)$.
It is then well-known (see e.g. \cite[Proposition 3.7.1]{carter1985finite}) that $s$ and $t$ must then be conjugate by an element of $N_{H^\circ}(T)$.
\end{proof}

The next lemma uses V. Lafforgue's notion of pseduocharacters.
A description of them and some properties can be found in \cite[Section 4]{thorne2019}; in particular we will use the construction of a representation from a pseudocharacter.

\begin{lemma} \label{equal pschar wccl}
Let $\Gamma$ be a group.
Let $H$ be a reductive group over an algebraically closed field $K$ and let $\iota: T \to H$ be the inclusion of a maximal torus of $H$, defining a Weyl group $W = W(H,T)$.
For $i \in \{1,2\}$, let $\chi_i: \Gamma \to T(K)$ be characters and $\rho_i = \iota \circ \chi_i: \Gamma \to H(K)$.
Then the following are equivalent:
\begin{enumerate}
    \item The characters $\chi_i$ are conjugate by an element of $W$ \label{char conjugate}
    \item The representations $\rho_i$ are conjugate via an element of $H(K)$ \label{rep conjugate}
    \item The $H$-pseudocharacters $\tr \rho_i$ are equal. \label{pschar equal}
\end{enumerate}
\end{lemma}
\begin{proof}
That \ref{char conjugate} implies \ref{rep conjugate} is clear, and $\ref{rep conjugate}$ implies \ref{pschar equal} is also clear from the definition of $\tr \rho_i$.
To prove that \ref{pschar equal} implies \ref{char conjugate}, we shall use the construction of representations from pseudocharacters given in \cite[Theorem 4.5]{thorne2019}.
As in the proof, we can choose $n \geq 1$ and $\delta = (\delta_1,\ldots,\delta_n) \in \Gamma^n$ such that the centralizer in $H$ of the Zariski closure of the subgroup of $H(K)$ generated by $\rho_1(\delta_1),\ldots,\rho_1(\delta_n)$ has minimal dimension, and among those tuples for which this dimension is minimised we choose one for which the component group of this centralizer has minimal order.

Now let $\gamma \in \Gamma$, and consider $(\delta_1,\ldots,\delta_n,\gamma) \in \Gamma^{n+1}$ defining a point of $(H^{n+1}//H)(K)$ via $\tr \rho_i$.
The construction then shows that $\rho_i(\gamma)$ is given by the final entry of the unique tuple in $H^{n+1}(K)$ whose $H$-orbit is closed, whose image in $(H^{n+1}//H)(K)$ is the point defined above and whose first $n$ entries are given by $(\rho_i(\delta_1),\ldots,\rho_i(\delta_n))$ (note condition (i) in \cite[Theorem 4.5]{thorne2019} is automatically satisfied for any choice of tuple of elements of $\Gamma$).
If we can show that the tuples $(\rho_i(\delta_1),\ldots,\rho_i(\delta_n))$ for $i \in \{1,2\}$ are conjugate by some element of $W$ then by uniqueness of the above tuples, $\rho_i(\gamma)$ will be conjugate by the same element of $W$ for every $\gamma \in \Gamma$.
Since we are assuming $\tr \rho_1 = \tr \rho_2$, the tuples $(\rho_i(\delta_1),\ldots,\rho_i(\delta_n))$ are conjugate by some element of $H(K)$.
By Lemma \ref{ccl equiv wccl} they are therefore conjugate by some element of $W$ and we are done.
\end{proof}

Now let $F$ be a global field of characteristic different from $p$ and fix $\bar{\rho}: G_F \to \hat{G}(k)$ as in Section \ref{section deformation}.
We will also fix a lift $\rho: G_F \to \hat{G}(\mathcal{O})$ (which we will assume is unramified outside a finite set of places of $F$) and let the composition $\hat{G}(\CalO) \to C_{\hat{G}}(\CalO)$ determine the choice of similitude character.
Let $v$ be a Taylor--Wiles place, as in Definition \ref{TW place}.
Let $\phi_v \in G_{F_v}$ be a choice of Frobenius element and let $\varpi_v = \Art_{F_v}(\phi_v) \in F_v^\times$ be the corresponding uniformizer.
Then set $g = \rho(\phi_v)$ and $\bar{g} = \bar{\rho}(\phi_v)$.
In Section \ref{def setup}, we saw the construction of a subgroup $M_{\bar{g}}$ of $\hat{G}_k$ defined over $k$, together with a lift $M_g$ defined over $\CalO$.
Fix a split maximal torus $\hat{T}$ of $\hat{G}_k$ containing $Z(M_{\bar{g}})^\circ$ and a Borel subgroup $\hat{B}$ containing $\hat{T}$.
We also fix a lift of $\hat{T}$ defined over $\CalO$ and containing $Z(M_{g})^\circ$; we denote this lift by $\hat{T}$ as well.

Let $G$ be the split reductive group dual to $\hat{G}$, defined over $\mathcal{O}_F$, as described in Section \ref{dual group deltaq}, and let $T \subset B$ be a split maximal torus and Borel subgroup of $G$.
By Construction \ref{parahoric construction}, the data of $\bar{g}$ and our maximal torus and Borel subgroup of $\hat{G}_k$ determine unique isomorphisms $X_*(\hat{T}) \cong X^*(T)$ and $X^*(\hat{T}) \cong X_*(T)$ both denoted by $\iota$, dual Levi subgroups $L \subset G$ and $\hat{L} \subset \hat{G}_k$, and compact open subgroups $\mathfrak{p}_1 \subset \mathfrak{p} \subset G(F_v)$.
We also have a $p$-group $\Delta_v = \mathfrak{p}/\pone$, the maximal $p$-power quotient of $C_L(k(v))$ for which the image of $Z(G)(k(v))$ is trivial.
We saw already in Section \ref{dual group deltaq} that the subtorus $Z(M_{\bar{g}})^\circ \subset \hat{T}$ is dual to the quotient torus of $T$ given by the cocenter of the Levi subgroup, $C_L$, in the sense that $X_*(Z(M_{\bar{g}})^\circ) \subset X_*(\hat{T})$ is identified with $X^*(C_L) \subset X^*(T)$ via $\iota$.

Let $\bar{\chi}: T(F_v) \to k^\times$ be the unramified character whose dual $\bar{\chi}^\vee: F_v^\times \to \hat{T}(k)$ satisfies $\bar{\chi}^\vee \circ \Art_{F_v}^{-1} = \restr{\bar{\rho}}{W_{F_v}}$ on including $\hat{T}$ into $\hat{G}_k$.
The character $\bar{\chi}$ determines a maximal ideal $\mathfrak{n}_1$ of the subalgebra $\CalO[T(F_v)/(T(F_v) \cap \pone)]^{W_L}$ of the Hecke algebra $\mathcal{H}(G(F_v),\pone)$.

Now suppose that $\pi_v$ is a smooth admissible irreducible representation of $G(F_v)$ which satisfies $((\pi_v)^\pone)_{\mathfrak{n}_1} \neq 0$.
By Lemma \ref{principal series loc inv}, we know that $\pi_v$ is a subquotient of a parabolic induction of some character $\chi: T(F_v) \to \overline{\mathbb{Q}_p}^\times$, which we may take to be a lift of $\bar{\chi}$ and that the action of $\CalO[T(F_v)/(T(F_v) \cap \pone)]^{W_L}$ on $((\pi_v)^\pone)_{\mathfrak{n}_1}$ is through $\chi$.
Via $\iota$ we obtain a dual character $\chi^\vee: F_v^\times \to\hat{T}(\overline{\mathbb{Q}_p})$ which in fact has image in $\hat{T}(\overline{\mathbb{Z}_p})$ and lifts $\restr{\bar{\rho}}{W_{F_v}} \circ \Art_{F_v}$.

We note here that $\chi^\vee(\CalO_{F_v}^\times) \subset Z(M_g)^\circ(\Qpbar)$.
Indeed, it suffices to show that if $\alpha \in X^*(\hat{T})$ and $x \in \CalO_{F_v}^\times$ with $\restr{\alpha}{Z(M_g)^\circ}$ trivial then $\alpha(\chi^\vee(x)) = 1$.
Equivalently, we need to show that $\chi(\iota(\alpha)(x)) = 1$.
We must have $\iota(\alpha)(x) \in \pone$, since $\iota(\alpha)(x)$ will have trivial image in $C_L(\mathcal{O}_{F_v})$, and hence also trivial image in $C_L(k(v))$.
To conclude, recall that $\chi$ is trivial on restriction to $\pone \cap T(F_v)$.

Recall the local deformation problem $\DvTW$ of Definition \ref{deformation_condition}.
The following proposition will allow us to deduce that $\restr{\rho}{G_{F_v}}$ defines a lift in $\DvTW(\CalO)$ and an explicit description of the arising map $\Delta_v \to \mathcal{O}^\times$ in terms of the action of the Hecke algebra on $\pi_v$ whenever we have sufficiently strong local-global compatibility result.

\begin{proposition} \label{local global deltav}
Suppose that $\WD(\restr{\rho}{G_{F_v}})^{\Fss} \cong ((\chi \otimes \varphi)^\vee \circ \Art_{F_v}^{-1},0)$ as Weil--Deligne representations over $\hat{G}(\overline{\mathbb{Q}_p})$, where $\varphi: G(F_v) \to \Qpbar^\times$ is some character which factors through $C_G(F_v)$ and is trivial on restriction to $G(\mathcal{O}_{F_v})$.
Then $\restr{\rho}{G_{F_v}}$ defines a lift in $\DvTW(\CalO)$ and the induced map $\Delta_v \to \mathcal{O}^\times$ coincides with the scalar action of $\Delta_v$ on $((\pi_v)^\pone)_{\mathfrak{n}_1}$ arising from viewing $\Delta_v = T(\mathcal{O}_{F_v})/(T(\mathcal{O}_{F_v}) \cap \pone)$ as Hecke operators.
\end{proposition}

\begin{proof}
Note firstly that since $\varphi$ factors through $C_G(F_v)/C_G(\mathcal{O}_{F_v})$, the only effect it has on the Weil--Deligne representations is to multiply through by an unramified representation valued in $Z(\hat{G})(\Qpbar)$.
Therefore it has no effect on whether $\rho \in \DvTW(\CalO)$ (aside from changing the similitude character potentially), nor on any arising $\Delta_v$-actions.
Hence the claims of the proposition hold true if and only if they hold true when $\varphi = 1$, so we will assume from now on that $\varphi = 1$.

Since $\WD(\restr{\rho}{G_{F_v}})^{\Fss}$ is semisimple, it follows that $\WD(\restr{\rho}{G_{F_v}})^{\Fss} = (r,0)$ with $r(\phi_v) = g_{ss}$, where $g = g_{ss} g_u$ is the Jordan decomposition of $g$ into its semisimple and unipotent parts, but $\restr{\rho}{I_{v}} = \restr{r}{I_{v}}$.
We claim that $r(I_{v}) \subset Z(M_g^\circ)(\Qpbar)$.
Since $\chi^\vee(\varpi_v) \in \hat{T}(\Zpbar)$ lifts $\bar{g}$, we can form $(M_{\chi^\vee(\varpi_v)})_{\Zpbar}$ as in Section \ref{def setup}.
Then the equality $$Z(M_{g}^\circ)_{\Zpbar} = Z(M_{\chi^\vee(\varpi_v)}^\circ)_{\Zpbar}$$ follows from Lemma \ref{center of Mg} and rigidity of tori, since both are closed subgroup schemes of $\hat{T}_{\Zpbar}$ lifting $Z(M_{\bar{g}}^\circ)_{\Fpbar}$.
Since 
$$\chi^\vee(\mathcal{O}_{F_v}^\times) \subset Z(M_{g}^\circ)(\Qpbar) = Z(M_{\chi^\vee(\varpi_v)}^\circ)(\Qpbar)$$
and $r \cong \chi^\vee \circ \Art_{F_v}^{-1}$, the claim follows by Corollary \ref{rational Mg}.

Since $g \in M_g^\circ(\Qpbar)$ (by Theorem \ref{Mg}), functoriality of the Jordan decomposition implies that $g_{ss} \in M_g^\circ(\Qpbar)$.
Since $g_{ss}$ lies inside (the $\Qpbar$-points of) a maximal torus of $(M_g^\circ)_{\Qpbar}$ and any maximal torus of $(M_g^\circ)_{\Qpbar}$ contains $(Z(M_g^\circ))_{\Qpbar}$, we see that $r$ is contained inside (the $\Qpbar$-points of) a maximal torus of $M_g^\circ$ by the above claim.
As all maximal tori in a connected reductive group over an algebraically closed field are conjugate, we can choose $h \in M_g^\circ(\Qpbar)$ such that $\Ad(h)(r)$ has image in $\hat{T}(\Qpbar)$.
Now, we know that $\Ad(h)(r)$ and $\chi^\vee \circ \Art_{F_v}^{-1}$ are $\hat{G}(\Qpbar)$-conjugate and both are valued in $\hat{T}(\Qpbar)$.
Therefore Lemma \ref{equal pschar wccl} implies that they are in fact conjugate by some element $w \in W(\hat{G},\hat{T})$, whose class may be represented by an element $n \in N_{\hat{G}}(\hat{T})(\Zpbar)$.
Our goal will be to show that $n \in M_g(\Qpbar)$.

Observe that the $(M_g^\circ)_{\Qpbar}$-pseudocharacters of $r$ and $\restr{\rho}{W_{F_v}}$ coincide.
This is a consequence of \cite[Theorem 5.2]{Richardson1988ConjugacyCO}, as given $\boldsymbol{x} = (g^{a_i} z_i)_i$ with $a_i \in \mathbb{Z}$ and $z_i \in Z(M_g^\circ)(\Qpbar)$ for $1 \leq i \leq n$, we may take $\boldsymbol{y} = ((g_{ss})^{a_i} z_i)_i$ to be the `Levi' part of the Levi decomposition of $\boldsymbol{x}$.
These both have the same image in $((M_g^\circ)^n//M_g^\circ)(\Qpbar)$ since the closed orbit defined by conjugation on $\boldsymbol{y}$ is the unique closed orbit in the closure of the orbit of $\boldsymbol{x}$ (see \cite[Proposition 3.7(ii)]{thorne2019}).

Let $\Theta_1$ denote the $(M_g^\circ)_{\Qpbar}$-pseudocharacter attached to $r$ and let $\Theta_2$ be the $(M_g^\circ)_{\Qpbar}$-pseudocharacter given by $\tr (\chi^\vee \circ \Art_{F_v}^{-1})$.
Both $\Theta_1$ and $\Theta_2$ are defined over $\Zpbar$ (since both are given by traces of representations valued in $M_g^\circ(\Zpbar)$) and are lifts of $\overline{\Theta} = \tr(\bar{\rho})$, an $M_g^\circ$-pseudocharacter over $\Fpbar$.
The representations $\bar{\rho}$ and the mod $p$ reduction of $\Ad(h)(r)$ have image contained in $\hat{T}(\Fpbar)$ and by Lemma \ref{equal pschar wccl} they are conjugate by some element $\bar{x} \in N_{(M_{\bar{g}}^\circ)_{\Fpbar}}(\hat{T}_{\Fpbar})$.
We also know that these representations are conjugate by the image $\bar{n}$ of $n$ in $N_{\hat{G}}(\hat{T})(\Fpbar)$, since $\bar{\rho}$ is the mod $p$ reduction of $\chi^\vee \circ \Art_{F_v}^{-1}$.
The elements $\bar{x}$ and $\bar{n}$ of $W(\hat{G}_{\Fpbar},\hat{T}_{\Fpbar})$ can differ only up to an element centralizing the image of $\bar{\rho}$.
Thus their difference must lie in $M_{\bar{g}}(\Fpbar)$, as it centralizes $\bar{g}$.
We see that $\bar{n} \in N_{M_{\bar{g}}}(\hat{T}_k)(\Fpbar)$, and hence also $n \in N_{M_{g}}(\hat{T})(\Qpbar)$, as desired.

Thus $r$ and $\chi^\vee \circ \Art_{F_v}^{-1}$ are conjugate in $M_g(\Qpbar)$, with $\chi^\vee \circ \Art_{F_v}^{-1}(I_v) \subset Z(M_{g})^\circ(\Qpbar)$ implying that $\restr{\chi^\vee \circ \Art_{F_v}^{-1}}{I_v}$ is unchanged under this conjugation.
It follows that that $\rho$ and $\chi^\vee \circ \Art_{F_v}^{-1}$ have equal restrictions to $I_{v}$, since we saw that $\restr{\rho}{I_v} = \restr{r}{I_v}$.
We are now done, since the action of $T(\mathcal{O}_{F_v})/(T(\mathcal{O}_{F_v}) \cap \pone)$ on $((\pi_v)^\pone)_{\mathfrak{n}_1}$ is via $\chi$ and this will exactly coincide with the $\Delta_v$-action arising from $\chi^\vee \circ \Art_{F_v}^{-1}$ constructed in Section \ref{dual group deltaq}.
\end{proof}

\section{Modularity of abelian surfaces} \label{sec surfaces}

In this section we will apply the constructions of the earlier sections to showing the modularity of some abelian surfaces over totally real fields, extending the results of \cite{surfaces}; for this we will need analogues of many of the results of \cite[Section 7]{surfaces}.
Since many of the arguments will work in exactly the same way, we will focus only on explaining the modifications arising from working with our more general notion of Taylor--Wiles place.

\subsection{Setup and notation} \label{setup for modularity}

In this section we continue with the setup of Section \ref{section deformation} but let $F$ be a totally real number field in which $p \geq 3$ splits completely and take $G = \GSp_4$, defined over $\mathcal{O}_F$ as those invertible matrices which scale a fixed symplectic form $J$ on a finite free $\mathcal{O}_F$-module $V$ of rank $4$.
Explicitly, we will fix an isomorphism $V \cong (\mathcal{O}_F)^4$ and take
\[
J = \begin{pmatrix}
0 & 0 & 0 & 1 \\
0 & 0 & 1 & 0 \\
0 & -1 & 0 & 0 \\
-1 & 0 & 0 & 0
\end{pmatrix}
\]
in this basis.
We will fix an isomorphism $\iota: \mathbb{C} \to \Qpbar$, through which we will determine a squareroot in $\Qpbar$ of every prime number and assume that these are all contained in our coefficient ring $\CalO$ (which we may since we need only adjoin finitely many elements).
Let $T \leq \GSp_4$ be usual diagonal split maximal torus and $B \leq \GSp_4$ the Borel subgroup of upper triangular matrices.
We can identify $\hat{G}$ with $\GSp_4$ just as in Section \ref{gsp4 local}, and take $\hat{G}$ to be defined over $\mathbb{Z}$.
We let $p \geq 3$ be the residue characteristic of $\CalO$ and assume that $p$ that splits completely in $F$.
Note that $p$ is coprime to the order of $W_G$ and hence of very good characteristic for $\hat{G}$.
We let $S_p$ denote the set of places of $F$ which divide $p$.

Let $\bar{\rho}: G_F \to \GSp_4(k)$ be a continuous representation.
We recall the following definitions from \cite{surfaces} which will be of repeated use.
\begin{definition} \label{defn: tidy}
A representation $\bar{\rho}: G_F \to \GSp_4(k)$ is said to be tidy if there exists $g \in \bar{\rho}(G_F)$ such that $\nu(g) \neq 1$ and for every pair of eigenvalues $\bar{\alpha} \neq \bar{\beta}$ of $g$ we have $\bar{\alpha}/\bar{\beta} \neq \nu(g)$.
\end{definition}
\begin{definition} \label{defn: p-distinguished ordinary}
Let $A$ be either $\CalO$ or $k$, and let $v \in S_p$.
A representation $\rho_v: G_{F_v} \to \GSp_4(A)$ is said to be $p$-distinguished weight 2 ordinary if for each place $v \in S_p$ we have an isomorphism
$$
\rho_v \cong
\begin{pmatrix}
&\chi_{\alpha_v} &0 &* &* \\
& &\chi_{\beta_v}  &* &*  \\
& & &\chi_{\beta_v}^{-1} \varepsilon^{-1}  &0  \\
& & & &\chi_{\alpha_v}^{-1} \varepsilon^{-1}  \\
\end{pmatrix}
$$
with $\alpha_v \not\equiv \beta_v \mod m_A$.
Here for $\lambda \in A^\times$, $\chi_{\lambda}: G_{F_v} \to A^\times$ is the unramified character satisfying $\chi_{\lambda}(\phi_v) = \lambda$ for any choice of Frobenius lift $\phi_v$.

\end{definition}
We will suppose that $\bar{\rho}$ satisfies all of the hypotheses of \cite[7.8.1]{surfaces} except replacing vastness by the assumption that $\bar{\rho}$ is $\GSp_4$-reasonable.
We record explicitly what this means.

The representation $\bar{\rho}$ is $\GSp_4$-reasonable, in the sense of Definition \ref{reasonable}, and tidy, in the sense of Definition \ref{defn: tidy}.
For each $v \in S_p$, $\restr{\bar{\rho}}{G_{F_v}}$ is $p$-distinguished weight $2$ ordinary, in the sense of Definition \ref{defn: p-distinguished ordinary} (and fix an ordering $(\bar{\alpha}_v,\bar{\beta}_v)$ on the eigenvalues of $\bar{\rho}(\Frob_v))$.
We suppose that there exists a finite set of finite places $R$ of $F$, disjoint from $S_p$, such that for each $v \in R$, $q_v \equiv 1 \mod p$ and, if $p=3$, then $q_v \equiv 1 \mod 9$.
We suppose that $\bar{\rho}$ is unramified outside of $S_p \cup R$.
For $v$ a finite place of $F$ let $\Iw_v$ be the Iwahori subgroup of $G(F_v) = \GSp_4(F_v)$ defined with respect to the Borel subgroup $B$.
Consider $\pi$ an ordinary cuspidal automorphic representation of $\GSp_4(\mathbb{A}_F)$ of parallel weight $2$ and central character $|\cdot|^2$ such that for $v \in R \cup S_p$ that $\pi_v^{\Iw_v} \neq 0$ and for $v \not\in R \cup S_p$ that $\pi_v^{\GSp_4(\CalO_{F_v})} \neq 0$.
There exists an associated continuous semi-simple Galois representation $\rho_{\pi,p}: G_F \to \GL_4(\Qpbar)$ (see \cite[Theorem 4.6]{mok_2014}), and we assume that there exists such a $\pi$ as above for which $\overline{\rho_{\pi,p}} \cong \bar{\rho}$.
Extending $k$ if necessary, we also assume throughout that $k$ is large enough for the results of Section \ref{section deformation} to apply to $\bar{\rho}$.

We will consider the global deformation problem described in \cite[7.7]{surfaces}.
For this we fix, for each $v \in R$, a pair of characters $\chi_v = (\chi_{v,1},\chi_{v,2})$ where $\chi_{v,i}: \CalO_{F_v}^\times \to \CalO^\times$ are trivial modulo $\lambda$.
Let $\chi = (\chi_v)_{v \in R}$, and let $\chi$ also denote the character $\prod_{v \in R} \Iw_v \to~\CalO^\times$ arising from viewing each (necessarily $p$-power order) $\chi_v$ as a character on $T(k_v)$ (via the first two diagonal entries) and inflating.
We take the similitude character $\psi$ of our lifts of $\bar{\rho}$ as in Section \ref{section deformation} to be given by $\varepsilon^{-1}$, the inverse of the cyclotomic character (and assume that $\bar{\rho}$ has its similitude character given by the reduction of $\varepsilon^{-1}$).
We can find, by tidiness of $\bar{\rho}(G_F)$ and the Chebotarev density theorem, an unramified place $v_0$ such that 
$q_{v_0} \not\equiv 1 \mod p$, $\ch k(v_0) > 5$ and such that no pair of eigenvalues of $\bar{\rho}(\Frob_{v_0})$ are in the ratio $q_{v_0}$.
Set $S = S_p \cup R \cup \{v_0\}$.

We let $I \subset S_p$ and write $I^c = S_p \setminus I$.
For $v \in I$, we let $\Lambda_v = \CalO[[\OFv^\times(p)]]$, where $\OFv^\times(p) = 1+p \OFv$.
For $v \in I^c$, set $\Lambda_v = \CalO[[(\OFv^\times(p))^2]]$.
For $v \in S_p$ fix a choice $\varsigma_v$ of either $\bar{\alpha}_v$ or $\bar{\beta}_v$, and let $\bar{\varsigma} = (\bar{\varsigma_v})_{v \in S_p}$.
We have $p$-adic local deformation problems $\mathcal{D}_v^{P}$ for $v \in I$ and $\mathcal{D}_v^{B,\bar{\varsigma}_v}$ for $v \in I^c$ described in \cite[7.3]{surfaces}.

For $v \in R$, we let $\Lambda_v = \CalO$ and consider the Ihara avoidance local deformation problem $\mathcal{D}_v^{\chi_v}$ of \cite[7.4.5]{surfaces}, given roughly by lifts for which the eigenvalues under the image of inertia are given according to $\chi_v$.
At the place $v_0$ we will let $\Lambda_v = \mathcal{O}$ and consider the unrestricted local deformation problem $D_{v_0}^\square$.
The global deformation problem $\mathcal{S}_{\chi}^{I,\varsigma}$ is then defined by these local deformation problems at the places of $S$.

Now fix a Taylor--Wiles datum $(Q,\{(\hat{T}_v,\hat{B}_v)\}_{v \in Q})$ as in Definition \ref{TW datum}.
We will consider the augmented global deformation problem $\mathcal{S}_{\chi,Q}^{I,\varsigma}$ arising from our Taylor--Wiles datum, as in Section \ref{section tw places}.
Denote the arising deformation functor by $\mathcal{D}_{\chi,Q}^{I,\varsigma}$, represented by $R_{\chi,Q}^{I,\varsigma} \in \CNL_\Lambda$.

If $v \in Q$, then $\bar{g}_v := \bar{\rho}(\Frob_v)$ is a semisimple element of $\hat{G}(k) = \GSp_4(k)$ and defines a connected reductive subgroup $M_{\bar{g}_v} = Z_{\hat{G}_k}(\bar{g}_v)$ of $(\GSp_4)_k$. 
The data of the maximal torus $\hat{T}_v$ of $M_{\bar{g}_v}$ and Borel subgroup $\hat{B}_v$ of $(\GSp_4)_k$ containing $\hat{T}_v$ yields dual Levi subgroups $L_v$ of $G = \GSp_4$ and compact open subgroups $\mathfrak{p}_{1,v} \leq \mathfrak{p}_v$ of $\GSp_4(F_v)$ as in Construction \ref{parahoric construction}.
We also obtain maximal ideals $\tilde{m}_v$ of $\CalO[T(F_v)/(T(F_v) \cap \mathfrak{p}_{1,v})]^{W_{L_v}}$, which are the ideals denoted by $\mathfrak{n}_1$ in Section \ref{jacquet module invariants}.
For $v \in R \cup \{v_0\}$, let $\Iw_{1,v}$ denote maximal the pro-$v$ subgroup of $\Iw_v$, which will be given by the preimage of $U(k(v))$ in $\GSp_4(\OFv)$, where $U$ is the unipotent radical of $B$.
For $v \in R$, we will let $L_v = T$.

\begin{remark}\label{rmk notation diff}
The notation $\Iw_{1,v}$ is inconsistent with the notation $\mathfrak{p}_{1,v}$ for places $v \in Q$ (but the same as in \cite{surfaces}).
The quotient $\Iw_v/\Iw_{1,v}$ can be identified with $T(k(v))$.
On the other hand, if $P_v=B$ for $v \in Q$ then the quotient $\mathfrak{p}_v/\mathfrak{p}_{1,v}$ would be the maximal $p$-power quotient of $T(k(v))/Z(G)(k(v))$.
\end{remark}

Let $\mathbb{A}_F^{\infty,p}$ denote the ring of finite adeles of $F$ away from places dividing $p$.
We will define compact open subgroups of $\GSp_4(\mathbb{A}_F^{\infty,p})$ in a similar manner to \cite[Definition 7.8.2]{surfaces}.
We warn the reader that our definitions will not coincide with these exactly, even in the case where $Q$ consists of places with $\bar{\rho}(\phi_v)$ regular semisimple for the reasons of Remark \ref{rmk notation diff}.
The definition of $K^p = \prod_v K_v$ is the same; we let:
\begin{itemize}
    \item $K_v = \GSp_4(\OFv)$, if $v \not\in S$
    \item $K_v = \Iw_{1,v}$, if $v \in R \cup \{v_0\}$.
\end{itemize}
Define compact open subgroups $K^p_{\Iw}(Q),K^p_0(Q),K_1^p(Q) \leq K^p$ and $K_0^p(Q,R)$ as follows:
\begin{itemize}
    \item $K^p_{\Iw}(Q)_v = K^p_0(Q)_v = K^p_1(Q)_v = K_0^p(Q,R) = K^p_v$, if $v \not\in Q \cup R$
    \item $K^p_{\Iw}(Q)_v = \Iw_v$, $K^p_0(Q)_v = K^p_0(Q,R)_v = \mathfrak{p}_v$ and $K_1^p(Q) = \mathfrak{p}_{1,v}$, if $v \in Q$
    \item $K^p_{\Iw}(Q)_v = K^p_0(Q)_v = K^p_1(Q)_v = \Iw_{1,v}$ and $K^p_0(Q,R)_v = \Iw_v$, if $v \in R$.
\end{itemize}
Lastly let $K_1^p(Q,R) = K_1^p(Q)$.
The extra compact open subgroup $K^p_{\Iw}(Q)$ will be helpful for carrying out computations using the well-understood structure of the Iwahori-Hecke algebra at places $v \in Q$.
It follows from the existence of $v_0$ and \cite[Lemma 7.8.3]{surfaces} that these compact open subgroups are neat.

For $I \subset S_p$ we will let $K^p(I) = \prod_{v \in I} \Kli_v \prod_{v \in I^c} \Iw_v$, where $\Kli_v$ is the parahoric subgroup of $G(F_v)$ associated to the Klingen parabolic subgroup of $G$, the latter given by block upper triangular matrices corresponding to the partition $4 = 1+2+1$.
These will determine our level structure at $v \in S_p$.

We will now define analogous objects to those constructed in the discussion following on from \cite[Lemma 7.8.3]{surfaces}.
Firstly, let
\begin{align*}
    \Tilde{\mathbb{T}} = &\bigotimes_{v \not\in S} \CalO[\GSp_4(F_v)//\GSp_4(\OFv)] \\
    \Tilde{\mathbb{T}}^Q = &\bigotimes_{v \not\in S \cup Q} \CalO[\GSp_4(F_v)//\GSp_4(\OFv)]    
\end{align*}
be the natural rings of Hecke operators for these levels at finite places away from $S$ (respectively $S \cup Q$).
For $v \not\in S$, define the following elements of $\mathbb{T}$
\begin{itemize}
    \item $T_{v,0} = [\GSp_4(\OFv) \diag(\varpi_v,\varpi_v,\varpi_v,\varpi_v)\GSp_4(\OFv)]$
    \item $T_{v,1} = [\GSp_4(\OFv) \diag(\varpi_v,\varpi_v,1,1)\GSp_4(\OFv)]$
    \item $T_{v,2} = [\GSp_4(\OFv) \diag(\varpi_v^2,\varpi_v,\varpi_v,1)\GSp_4(\OFv)]$,
\end{itemize}
where we are viewing $\mathcal{H}(\GSp_4(F_v),\GSp_4(\OFv))$ as a subalgebra of $\mathbb{T}$.
Then let $$Q_v(X) = X^4 - T_{v,1} X^3 + (q_v T_{v,2} + (q_v^3 + q_v)T_{v,0})X^2 - q_v^3 T_{v,0} T_{v,1}X + q_v^6 T_{v,0}^2 \in \Tilde{\mathbb{T}}[X],$$
where $v \not\in S$ has residue field of order $q_v$.
We let $\Tilde{m}^{an} \subset \Tilde{\mathbb{T}}$ denote the maximal ideal corresponding to $\bar{\rho}$ in the same way as in this discussion, and similarly for $\Tilde{m}^{an,Q} \subset \Tilde{\mathbb{T}}^Q$.
We will also let $\Tilde{\mathbb{T}}^I$ and $\Tilde{\mathbb{T}}^{I,Q}$ be defined in the same way by adjoining Hecke operators dependent on $I$ at places $v \in S_p$, and form the maximal ideals $\Tilde{m}^{I,\varsigma}$ and $\Tilde{m}^{I,\varsigma,Q}$ of these respective rings in the same way, corresponding to eigenvalues determined by $\bar{\rho}$ and the choice of $\varsigma$.

We define a $\Lambda_I[\Delta_Q]$-module
$$M^{\chi,I,\varsigma,Q} = \RHom^0_{\Lambda_I}(M^{\bullet,I}_{K_1^p(Q)}, \LI)_{\Tilde{m}^{I,\varsigma,Q},\Tilde{m}_Q,\chi,|.|^2}$$
and a $\Lambda_I$-module
$$M^{\chi,I,\varsigma} = \RHom^0_{\Lambda_I}(M^{\bullet,I}_{K^p}, \LI)_{\Tilde{m}^{I,\varsigma},\chi,|.|^2}.$$
Here 
\begin{itemize}
\item The complexes $M^{\bullet,I}_{K_1^p(Q)}$ and $M^{\bullet,I}_{K^p}$ are as defined in \cite[Theorem 4.6.1]{surfaces} at the prime-to-$p$ levels $K_1^p(Q)$ and $K^p$ respectively
\item The localizations at $\Tilde{m}^{I,\varsigma,Q}$ and $\Tilde{m}^{I,\varsigma}$ are as above
    \item $\Tilde{m}_Q$ is the maximal ideal of the Hecke algebra at places $v \in Q$
$$\bigotimes_{\CalO, v \in Q} \CalO[T(F_v)/(T(F_v) \cap \mathfrak{p}_{1,v})]^{W_{L_v}}$$
defined by the maximal ideals $\tilde{m}_v$ of the tensorands
\item $\chi$ denotes taking $\chi$-coinvariants for the action of $\prod_{v \in R} T(k(v))$
\item $|\cdot|^2$ denotes fixing the central character, by taking coinvariants under the Hecke operators $T_{v,0} - q_v^{-2}$ for $v \not\in S$.
\end{itemize}

Finally in this section, we will let $\mathbb{T}^{\chi,I,\varsigma,Q}$ denote the $\LI$-subalgebra of $\End_{\LI}(M^{\chi,I,\varsigma,Q})$ generated by the image of  
$\Tilde{\mathbb{T}}^{I,Q}$.
Similarly, let $\mathbb{T}^{\chi,I,\varsigma}$ denote the $\LI$-subalgebra of $\End_{\LI}(M^{\chi,I,\varsigma})$ generated by the image of  
$\Tilde{\mathbb{T}}^I$.

\subsection{Preparation for the Taylor--Wiles method} \label{sec tw prep}

In this section fix $I \subset S_p$ and let $K_p = K_p(I)$ be the level structure at $v \in S_p$.
Let $X_{K_p K^p_0(Q,R),\Sigma}$, $X_{K_p K^p_1(Q,R),\Sigma}$ denote the compactified Shimura varieties at the respective levels $K_p K^p_0(Q,R)$ and $K_p K^p_1(Q,R)$ with respect to a compatible choice $\Sigma$ of polyhedral cone decompositions (as in \cite[Theorem 3.5.1]{surfaces}).
For $v \in Q \cup R$, we have a parabolic subgroup $P_v$ defining a parahoric subgroup $\mathfrak{p}_v \leq \GSp_4(F_v)$ (these are given by the standard Borel subgroup $B$ and Iwahori subgroup respectively in the case $v \in R$).
This parabolic subgroup $P_v$ may be viewed as the subgroup of $\GSp_4$ which preserves a filtration $0 = V_{v,0} \leq \ldots \leq V_{v,n_v}$ of isotropic subspaces of $V$ for some $0 \leq n_v \leq 2$.
Thus from the definition of the moduli problem described in \cite[3.3]{surfaces},over the interior $Y_{K_p K^p_0(Q,R)}$  we obtain a filtration of $A[v]$ by $k(v)$-group schemes, which we will denote $H_{v,0} \leq \ldots \leq H_{v,n_v}$.
To show the natural map between the compactified Shimura varieties at the respective levels is finite étale we shall be interested in extending the top exterior powers of the graded pieces of these filtrations.

\begin{proposition} \label{deltaq torsor}
For $v \in Q \cup R$, the group schemes $J_{v,m+1} := \det(H_{v,m+1}/H_{v,m})$ defined over $Y_{K_p K^p_0(Q,R)}$ for $0 \leq m < n_v$ may be extended to rank 1 finite étale $k(v)$-group schemes over $X_{K_p K^p_0(Q,R),\Sigma}$.
Let $J_{v,0}$ denote the constant $k(v)$-group scheme of rank 1 defined over $X_{K_p K^p_0(Q,R),\Sigma}$.
Let $\Xi_Q = \ker(\prod_{v \in Q} C_{L_v}(k(v)) \to \Delta_Q)$.
The map $X_{K_p K^p_1(Q,R),\Sigma} \to X_{K_p K^p_0(Q,R),\Sigma}$ is then finite étale with group $\Delta_Q \times \prod_{v \in R} T(k(v))$ and we can identify $X_{K_p K^p_1(Q,R),\Sigma}$ as the torsor of $\Xi_Q$-orbits of trivialisations of the extended group schemes $J_{v,m}$ for each $v \in Q \cup R$ and $0 \leq m \leq n_v$.
\end{proposition}

\begin{proof}
We shall follow an argument along the lines of the proof of \cite[Proposition 7.8.10]{surfaces}.
We firstly explain how the second part of the Proposition follows from the first.
We claim that the torsor of $\Xi_Q$-orbits of trivialisations of the extended group schemes $J_{v,m}$ over $Y_{K_p K^p_0(Q,R)}$ is isomorphic to $Y_{K_p K^p_1(Q,R)}$.
If $\{V_{v,i}\}_{0 \leq i \leq n_v}$ is the filtration of $V$ corresponding to the parabolic subgroup $P_v$ then the Levi subgroup $L_v$ is isomorphic to an extension of $\mathbb{G}_m$ (corresponding to the similitude factor) by $\GL_{r_1} \times \ldots \times \GL_{r_{n_v}} \times \Sp_{2 r_0}$  where $r_i = \dim V_{v,i} - \dim V_{v,{i-1}}$ for $1 \leq i \leq n_v$ and $r_0 = 2 - \sum_{i=1}^{n_v} r_i$.
It then follows that the map $P_v \to C_{L_v} \cong \mathbb{G}_m^{\{0,1,\ldots,n_v\}}$ is given by the similitude character onto the $0$-index copy of $\mathbb{G}_m$, and taking determinants of the various $\GL_{r_i}$ factors onto the copy of $\mathbb{G}_m$ indexed by $i$.
Then giving a $\ker(P_v(k(v)) \to C_{L_v}(k(v)))$-orbit of isomorphisms $V \otimes_{\mathcal{O}_F} k(v) \to A[v]$ of $\Pi_1(Y_{K_p K^p_0(Q,R)},\bar{s})$-modules is equivalent to giving a $P_v(k(v))$-orbit, together with isomorphism between $J_{v,m}$ and the trivial $\Pi_1(Y_{K_p K^p_0(Q,R)},\bar{s})$-module $k(v)$ for every $0 \leq m \leq n_v$.
Similarly, giving a $\ker(P_v(k(v)) \to \Delta_Q)$-orbit of such isomorphisms is equivalent to giving a $P_v(k(v))$ orbit, together with a $\Xi_Q$-orbit of isomorphisms $J_{v,m} \to k(v)$, where $\Xi_Q$ acts on the isomorphism indexed by $m$ via the map $C_L(k(v))$ to $k(v)^\times$ onto the factor indexed by $m$.
As in the proof of \cite[Proposition 7.8.10]{surfaces} the torsor of $\Xi_Q$-orbits of trivializations and $X_{K_p K^p_1(Q,R),\Sigma}$ are both normal, finite flat over $X_{K_p K^p_0(Q,R),\Sigma}$ and generically equal, so they are isomorphic.

We thus have to show we can extend each $J_{v,m}$.
Just as in the proof of \cite[Proposition 7.8.10]{surfaces} we may reduce to the following situation: we have an open immersion of schemes $S \subset T$ together with an abelian scheme $A$ defined over $S$, and a semi-abelian scheme $B$ of constant toric rank defined over $T$ such that $A$ is the quotient of $B$ by a finite free $\mathcal{O}_F$-module on restricting to $S$, as in Mumford's construction.
In that proof we saw that $B[v]$ is already defined as a finite étale group scheme over the whole of $T$, and that $A[v]/B[v]$ is a constant group scheme on $S$.
For $m>0$ write $J_{v,m} = \det(D/C)$, where $C \leq D \leq A[v]$ are étale subgroup schemes defined on $S$.
Then we have the short exact sequence
$$0 \to \frac{D \cap B[v]}{C \cap B[v]} \to \frac{D}{C} \to \frac{D + B[v]}{C + B[v]} \to 0,$$
from which it follows on taking determinants that
$$\det \frac{D}{C} \cong \det (\frac{D \cap B[v]}{C \cap B[v]}) \otimes \det(\frac{D + B[v]}{C + B[v]}).$$
The first tensorand is an exterior power of a subquotient of $B[v]$ defined over $T$, and therefore extends to $T$.
The second tensorand is an exterior power of a subquotient of $A[v]/B[v]$, which is a constant group scheme on $S$ and thus extends to $T$, so we are done.
\end{proof}

The following is the direct analogue of \cite[Proposition 7.8.11]{surfaces}.

\begin{proposition}\label{finite free balanced}
If $I = \emptyset$ then $M^{\chi,I,\varsigma,Q}$ is a finite free $\Lambda_I[\Delta_Q]$-module.
If $\#I = 1$ then $M^{\chi,I,\varsigma,Q}$ is a balanced $\Lambda_I[\Delta_Q]$-module, in the sense of \cite[Definition 2.10.1]{surfaces}.
\end{proposition}

\begin{proof}
Having shown Proposition \ref{deltaq torsor}, the proof is now identical to the proof of \cite[Proposition 7.8.11]{surfaces}.
\end{proof}

The final result we need about the module $M^{\chi,I,\varsigma,Q}$ for the Taylor--Wiles method to go through is the analogue of \cite[Proposition 7.9.8]{surfaces}.

\begin{proposition} \label{deltaq coinvariants}
The map $M^{\chi,I,\varsigma,Q} \to M^{\chi,I,\varsigma}$ arising from the inclusion $K_1^p(Q) \subset K^p$ induces an isomorphism $(M^{\chi,I,\varsigma,Q})_{\Delta_Q} \to M^{\chi,I,\varsigma}$.
\end{proposition}
\begin{proof}
We again follow the proof of \cite[Proposition 7.9.8]{surfaces}.
Recall that the natural map 
$$\bigotimes_{\CalO, v \in Q} \CalO[T(F_v)/(T(F_v) \cap \mathfrak{p}_{1,v})]^{W_{L_v}} \to \bigotimes_{\CalO, v \in Q} \CalO[T(F_v)/(T(F_v) \cap \mathfrak{p}_{v})]^{W_{L_v}}$$
induces a bijection on maximal ideals, since $\Delta_Q$ is of $p$-power order.
Let $\Tilde{m}_{0,Q}$ be the maximal ideal corresponding to $\Tilde{m}_{Q}$ under this bijection.
We can then form $$M^{\chi,I,\varsigma,Q}_{K_0^p(Q)} = \RHom^0_{\Lambda_I}(M^{\bullet,I}_{K_0^p(Q)}, \LI)_{\Tilde{m}^{I,\varsigma,Q},\Tilde{m}_{0,Q},\chi,|.|^2}.$$
From Proposition \ref{deltaq torsor}, we deduce in the same way as the proof of \cite[Proposition 7.9.8]{surfaces} that we have an isomorphism
$$(M^{\chi,I,\varsigma,Q})_{\Delta_Q} \cong M^{\chi,I,\varsigma,Q}_{K_0^p(Q)}.$$
Then to show that there is an isomorphism
$$M^{\chi,I,\varsigma,Q}_{K_0^p(Q)} \cong M^{\chi,I,\varsigma}$$
it suffices to prove by \cite[Lemma 7.9.7]{surfaces} that the composition
$$
H^i(M^{\bullet,I}_{K_0^p(Q)} \otimes k)_{\Tilde{m}^{\an,Q},\Tilde{m}_{0,Q}} \to 
H^i(M^{\bullet,I}_{K_0^p(Q)} \otimes k)_{\Tilde{m}^{\an,Q}} \to 
H^i(M^{\bullet,I}_{K^p} \otimes k)_{\Tilde{m}^{\an,Q}}
\to 
H^i(M^{\bullet,I}_{K^p} \otimes k)_{\Tilde{m}^{\an}}
$$
is an isomorphism for every $i$.
The action of the Hecke algebra on the various complexes is described in \cite[3.9]{surfaces}.
If we have compact open subgroups $K' \leq K$ which differ only at places $v \in Q$, then we have the equalities
\begin{align*}
    1_{K 1 K'} 1_{K' 1 K} &= [K:K'] \\
    1_{K' 1 K} 1_{K 1 K'} &= [K],
\end{align*}
where $[K]$ is the Hecke operator arising from the indicator function on $K$, a $K'$-biinvariant function.
Let $N_i = H^i(M^{\bullet,I}_{K_{\Iw}^p(Q)} \otimes k)_{\Tilde{m}^{\an,Q}}$, viewed as a module over the product of Iwahori--Hecke algebras at places $v \in Q$.
The indices \[[K^p:K_{\Iw}^p(Q)] = \prod_{v \in Q} \frac{q_v^{4} (q_v^2-1)(q_v^4-1)}{q_v^4 (q_v-1)^2 } = \prod_{v \in Q} (q_v+1)^2(q_v^2+1)\] (see \cite[3.1.2]{o1978symplectic}) and $[K_0^p(Q):K_{\Iw}^p(Q)]$ are units in $k$, as follows from the assumption that $p>2$ and Definition \ref{TW place}. 
Thus via the above equalities of Hecke operators we can identify $H^i(M^{\bullet,I}_{K^p} \otimes k)_{\Tilde{m}^{\an,Q}}$ (resp. $H^i(M^{\bullet,I}_{K_0^p(Q)} \otimes k)_{\Tilde{m}^{\an,Q}}$) with the image of $N_i$ under the Hecke operator $[K^p]$ (resp. $[K_0^p(Q)]$).
This is identification is also compatible with the localizations at the various maximal ideals of the abelian subalgebras of the parahoric and unramified Hecke algebras at places $v \in Q$.
We are done on applying Corollary \ref{gbar parahoric invariants iso} since the above composition becomes identified with the isomorphism
$$
([K_0^p(Q)]N_i)_{\Tilde{m}^{\an,Q},\Tilde{m}_{0,Q}} \xrightarrow{\sim} ([K^p(Q)]N_i)_{\tilde{m}^{\an}}
$$
given by the corollary, noting that $\tilde{m}^{\an}$ is defined at places $v \in Q$ by the maximal ideal of the unramified Hecke algebra $\mathcal{O}[X_*(T)]^{W_G}$ corresponding to the $W_G$-orbit of $\bar{g}_v$.
\end{proof}

We now state a theorem on the existence of Galois representations associated to a classical weight cuspidal automorphic respresentation $\pi$.
Local-global compatibility will then imply these Galois representations define a class in the image of our deformation functor and that the $\Delta_Q$-action is the expected one arising from the local factors of $\pi$.

\begin{theorem} \label{gsp4 galois rep local glob}
Suppose $\pi$ is a cuspidal automorphic representation of $\GSp_4(\mathbb{A}_F)$ of classical weight $\kappa = (k_v,l_v)$ with $k_v \geq l_v > 2$ and $k_v \equiv l_v \mod 2$ for all $v | \infty$ with central character $|.|^2$.
Suppose moreover that $(\pi_f^{K^p_1(Q) K_p(\emptyset)} \otimes \bar{E})^{\chi,|.|^2}_{\Tilde{m}^{I,\varsigma,Q}, \Tilde{m}_Q} \neq 0$.
Then there exists an associated continuous irreducible Galois representation $\rho_{\pi,p}: G_F \to \GSp_4(\bar{E})$ which satisfies, for each finite place $v \nmid p$,
$$\WD(\restr{\rho_{\pi,p}}{G_{F_v}})^{\Fss} \cong \rec_{\GT,p}(\pi_v \otimes |\nu|^{-3/2}).$$
We may conjugate $\rho_{\pi,p}$ to a representation $\rho: G_F \to \GSp_4(\mathcal{O}_{E_\pi})$ where $E_\pi/E$ is some finite extension with ring of integers $\mathcal{O}_{E_\pi}$ and residue field $k_\pi$ such that $\rho$ lifts $\bar{\rho} \otimes_k k_\pi$.

Then for each $v \in Q$ and any such conjugate $\rho$, we have $\rho \in \mathcal{D}_{v,\CalO_{E_\pi}}^{\text{TW}}(\mathcal{O}_{E_{\pi}})$.
The composition $\Delta_v \to (R_{v,\CalO_{E_\pi}}^{\text{TW}})^\times \to \mathcal{O}_{E_\pi}^\times$ coincides with the scalar action of $\Delta_v$ on the space $(\pi_f^{K^p_1(Q) K_p(\emptyset)} \otimes \bar{E})^{\chi,|.|^2}_{\Tilde{m}^{I,\varsigma,Q}, \Tilde{m}_Q}$.
\end{theorem}
Recall here that $\rec_{\GT,p}$ is the local Langlands correspondence described in Section \ref{gsp4 local}.

\begin{proof}[Proof of Theorem \ref{gsp4 galois rep local glob}]
The existence of $\rho_{\pi,p}$ is stated in \cite[Theorem 2.7.2]{surfaces} and to show the desired properties from this we will firstly explain why $\rho_{\pi,p}$ is conjugate to a lift of $\bar{\rho}$.
Continuity of $\rho_{\pi,p}$ implies that it is conjugate to a representation valued in the ring of integers of some such extension of $\mathcal{O}$.
As in the proof of \cite[Theorem 7.9.4]{surfaces}, it follows that the reduction of $\rho$ on a dense subset of $G_F$ (given by Frobenius elements at places outside $S_p \cup Q \cup R \cup \{v_0\}$) has characteristic polynomial coinciding with that of $\bar{\rho}$, and so the two residual representations are conjugate in $\GSp_4(k_\pi)$.
Thus we may take such a conjugate $\rho$, lifting $\bar{\rho} \otimes_k k_\pi$.
By absolute irreducibility of $\bar{\rho}$, $\rho$ itself must also be absolutely irreducible and hence the stated local-global compatibility holds by \cite[Theorem 2.7.2]{surfaces} again.

Now let $v \in Q$ be a Taylor--Wiles place.
We know that $\pi_v$ satisfies $((\pi_v)^{\mathfrak{p}_{1,v}})_{\tilde{m}_v} \neq 0$.
By Proposition \ref{gsp4 p1 invariants}, we know that $\rec_{\GT,p}(\pi_v)$ is a semisimple Weil--Deligne representation with $\rec_{\GT,p}(\pi_v) = (\chi^\vee \circ \Art_{F_v}^{-1},0)$ for some character $\chi: G(F_v) \to \Qpbar^\times$, with $\pi_v$ an irreducible subquotient of the parabolic induction $i_B^G \chi$.
Note then that $\rec_{\GT,p}(\pi_v \otimes |\nu|^{-3/2})$ equals $((\chi \otimes |\nu|^{-3/2})^\vee \circ \Art_{F_v}^{-1},0)$ by \cite[Proposition 2.4.6]{surfaces}.
Hence we are done on applying Proposition \ref{local global deltav}.
\end{proof}

\begin{theorem} \label{hecke galois rep}
There exists a continuous Galois representation $\rho^{\chi,I,\varsigma,Q}: G_F \to \GSp_4(\mathbb{T}^{\chi,I,\varsigma,Q})$ defining a class in $\mathcal{D}_{\chi,Q}^{I,\varsigma}(\mathbb{T}^{\chi,I,\varsigma,Q})$ and such that the composition $\Delta_Q \to (R_{\chi,Q}^{I,\varsigma})^\times \to (\mathbb{T}^{\chi,I,\varsigma,Q})^\times$ defines the same action of $\Delta_Q$ on $M^{\chi,I,\varsigma,Q}$ as the action arising from viewing $\Delta_Q$ as a set of Hecke operators at places in $Q$.
Moreover, we have
\begin{enumerate}
    \item If $v \not\in S_p \cup R \cup \{v_0\} \cup Q$, then $\det(X- \rho^{\chi,I,\varsigma,Q}(\Frob_v)) = Q_v(X)$.
    \item If $v \in I$, then
    $$
\restr{\rho^{\chi,I,\varsigma,Q}}{G_{F_v}}    \cong
\begin{pmatrix}
 \lambda_{\Tilde{\alpha}_v} \theta_v & 0 & * & * \\
 0 & \lambda_{\Tilde{\beta}_v} \theta_v & * & * \\
 0 & 0 & \lambda_{\Tilde{\beta}_v}^{-1} \theta_v^{-1} \varepsilon^{-1} & 0 \\
0 & 0 & 0 & \lambda_{\Tilde{\alpha}_v}^{-1} \theta_v^{-1} \varepsilon^{-1}   
\end{pmatrix}    
    $$
    \item If $v \in I^c$, then
    $$
\restr{\rho^{\chi,I,\varsigma,Q}}{G_{F_v}}    \cong
\begin{pmatrix}
 \lambda_{U_{v,1}} \theta_{v,1} & 0 & * & * \\
 0 & \lambda_{U_{v,2}/U_{v,1}} \theta_{v,2} & * & * \\
 0 & 0 & \lambda_{U_{v,2}/U_{v,1}}^{-1} \theta_{v,2}^{-1} \varepsilon^{-1} & 0 \\
0 & 0 & 0 & \lambda_{U_{v,1}}^{-1} \theta_{v,1}^{-1} \varepsilon^{-1}   
\end{pmatrix}    .
    $$    
\end{enumerate}

\end{theorem}

\begin{proof}

We will follow the proof of \cite[Theorem 7.9.4]{surfaces} and suppose that $I= \emptyset$ firstly.
Consider the set $W$ of cuspidal automorphic representations $\pi$ of classical weight $\kappa = (k_v,l_v)$ with $k_v \geq l_v \geq 4$ and $k_v \equiv l_v \equiv 2 \text{ or } p+1 \mod 2(p-1)$ such that $e(\emptyset) (\pi_f^{K^p_1(Q) K_p(\emptyset)} \otimes \bar{E})^{\chi,|.|^2}_{\Tilde{m}^{\emptyset,\varsigma,Q}, \Tilde{m}_Q} \neq 0$.
Here $\pi_f$ is the finite part of $\pi$ and $e(\emptyset)$ is the ordinary projector of \cite[2.4]{surfaces}.
For each such $\pi \in W$, we have an associated Galois representation $\rho_{\pi,p}: G_F \to \GSp_4(\CalO_{E_\pi})$, with $\rho_{\pi,p}$ and $\CalO_{E_\pi}$ as in Theorem \ref{gsp4 galois rep local glob}.
We have an inclusion
\begin{equation} \label{inclusion of module}
 M^{\chi,\emptyset,\varsigma,Q} \subset \prod_{\kappa} \bigoplus_{\pi \in W, \text{wt}(\pi) = \kappa} (e(\emptyset) (\pi_f^{K^p_1(Q) K_p(\emptyset)} \otimes \bar{E})^{\chi,|.|^2}_{\Tilde{m}^{I,\varsigma,Q}, \Tilde{m}_Q})^*
\end{equation}
respecting the actions of the Hecke operators $\Tilde{\mathbb{T}}^{\chi,I,\varsigma,Q}$ at places away from $Q$ and 
$$\bigotimes_{\CalO, v \in Q} \CalO[T(F_v)/(T(F_v) \cap \mathfrak{p}_{1,v})]^{W_{L_v}}.$$
The above inclusion allows us to view $\mathbb{T}^{\chi,I,\varsigma,Q}$ as a subalgebra of $$A = \{(a,(a_\pi)) \in k \times \prod_{\pi \in W} \CalO_{E_\pi} : a_\pi \mod m_{E_\pi} = a\}$$ given by the scalar action of a Hecke operator on the space of localised invariants of $\pi$.

The same construction of the proof of \cite[Theorem 7.9.4]{surfaces} yields a lift $\rho^{\chi,I,\varsigma,Q}: G_F \to \GSp_4(\mathbb{T}^{\chi,I,\varsigma,Q})$ when $I = \emptyset$, and the desired properties away from places $v \in Q$ can be proved in the same way.
We can assume from the construction of $\rho^{\chi,I,\varsigma,Q}$ that the composition of $\rho^{\chi,I,\varsigma,Q}$ with the maps $\mathbb{T}^{\chi,I,\varsigma,Q} \to A \to \CalO_{E_\pi}$ will yield a representation conjugate in $\ker(\GSp_4(\CalO_{E_\pi}) \to \GSp_4(k_\pi))$ to $\rho_{\pi,p}$.
This composition will be valued in the ring of integers of a totally ramified extension of $\CalO$ and will satisfy the consequences of Theorem \ref{gsp4 galois rep local glob}.
We may write $\mathbb{T}^{\chi,I,\varsigma,Q}$ as an inverse limit of Artinian quotients $\mathbb{T}_n = \mathbb{T}^{\chi,I,\varsigma,Q}/I_n$ where each $\mathbb{T}_n$ can be written as a finite fibre product over $k$ of quotients of some such rings $\CalO_{E_\pi}$.
Then the image of $\rho^{\chi,I,\varsigma,Q}$ in $\GSp_4(\mathbb{T}_n)$ will lie in $\DvTW(\mathbb{T}_n)$ by Lemma \ref{rel rep}, and hence $\rho^{\chi,I,\varsigma,Q}$ itself will lie in $\DvTW(\mathbb{T}^{\chi,I,\varsigma,Q})$.

For each $\pi \in W$, the composition $\Delta_v \to (\mathbb{T}^{\chi,I,\varsigma,Q})^\times \to \CalO_{E_\pi}^\times$ will also coincide with the action of $\Delta_v$ on $e(\emptyset) (\pi_f^{K^p_1(Q) K_p(\emptyset)} \otimes \bar{E})^{\chi,|.|^2}_{\Tilde{m}^{\emptyset,\varsigma,Q}, \Tilde{m}_Q}$ by Theorem \ref{gsp4 galois rep local glob}.
The inclusion \ref{inclusion of module} implies that the map $\Delta_v \to (\mathbb{T}^{\chi,I,\varsigma,Q})^\times$ defines an action on $M^{\chi,\emptyset,\varsigma,Q}$ coinciding with the one arising from viewing $\Delta_v \subset \CalO[T(\CalO_{F_v})/(T(\CalO_{F_v}) \cap \mathfrak{p}_{1,v})]^{W_{L_v}}$.

We are now done in the case where $I = \emptyset$.
The case $I \neq \emptyset$ can now be proved by induction on $\#I$ via the same argument of \cite[Theorem 7.9.4]{surfaces}.
\end{proof}

To finish this section, we will state the analogue of \cite[Corollary 7.9.6]{surfaces}, which we will use in proving our modularity results.

\begin{corollary} \label{parallel weight 2 galois rep}
Suppose that $\pi$ is cuspidal automorphic representation of $\GSp_4(\mathbb{A}_F)$ of parallel weight 2, and that $\pi$ is ordinary in the sense of \cite[Definition 2.4.25]{surfaces}.
There exists a continuous semisimple Galois representation $\rho_{\pi,p}: G_F \to \GL_4(\Qpbar)$ satisfying
\begin{enumerate}
    \item If $v \nmid p$ is such that $\pi_v$ is unramified, then $\restr{\rho_{\pi,p}}{G_{F_v}}$ is unramified and
    $$
    \det(X - \rho_{\pi,p}(\Frob_v)) = Q_v(X).
    $$
\end{enumerate}
Suppose moreover that $\overline{\rho_{\pi,p}}$ is $\hat{G}$-reasonable and that for each $v|p$, the ordinary Hecke parameters $\alpha_v$ and $\beta_v$ satisfy $\overline{\alpha_v} \neq \overline{\beta_v}$.
Then $\rho_{\pi,p}$ can be conjugated into $\GSp_4(\Qpbar)$ and satisfies
\begin{enumerate}[resume]
    \item $\nu \circ \rho_{\pi,p} = \varepsilon^{-1}$
    \item For each $v \nmid p$, we have
    $$
    \WD(\restr{\rho_{\pi,p}}{G_{F_v}})^{\text{ss}} \cong \rec_{\GT,p}(\pi_v \otimes | \nu|^{-3/2})^{\text{ss}}
    $$
    \item For each $v | p$, we have
    $$
    \restr{\rho_{\pi,p}}{G_{F_v}} \cong 
    \begin{pmatrix}
 \lambda_{\alpha_v} \theta_v & 0 & * & * \\
 0 & \lambda_{\beta_v} \theta_v & * & * \\
 0 & 0 & \lambda_{\beta_v}^{-1}  \varepsilon^{-1} & 0 \\
0 & 0 & 0 & \lambda_{\alpha_v}^{-1}  \varepsilon^{-1}   
\end{pmatrix}    
    .$$
\end{enumerate}
\end{corollary}
\begin{proof}
This follows from Theorem \ref{hecke galois rep} and \cite{mok_2014} in the same way as in the proof of \cite[Corollary 7.9.6]{surfaces}.
\end{proof}

\subsection{Patching}

We are now ready to carry out a patching argument analogous to \cite[7.11]{surfaces}.
Most of the notation and constructions will be similar, so we will firstly highlight the main difference arising from our more general notion of Taylor--Wiles places of Section \ref{section tw places} in the following lemma.
Let $q = h^1(G_{F,S},\splie_4^\vee(1))$.

\begin{lemma} \label{compatible tw data}
There exists $0 \leq t \leq 2q$ such that for every $N \geq 1$ there exists a Taylor--Wiles datum
$$
(Q_N,\{(\hat{T}_v,\hat{B}_v)\}_{v \in Q_N})
$$
of level $N$ with $Q_N$ as in Corollary \ref{TW places number field}, $\#Q_N = q$ and such that $\Delta_{Q_N}$ has $t$ cyclic factors as an abelian group.
\end{lemma}
\begin{proof}
We can, for each $N \geq 1$, fix a Taylor--Wiles datum
$
(Q_N,\{(\hat{T}_v,\hat{B}_v)\}_{v \in Q_N})
$
of level $N$ with $Q_N$ as in Corollary \ref{TW places number field} and $\#Q_N = q$.
These $(Q_N)_{N \geq 1}$ may induce $p$-groups $\Delta_{Q_N}$ for which the number of cyclic factors may differ with $N$.
The number of such cyclic factors is bounded by $2q$ (since each $\Delta_v$ is the maximal $p$-power quotient of $S(k(v))$, where $S$ is a torus of rank at most $2$), so by the pigeonhole principle we can find an increasing sequence $(N_i)_{i \geq 1}$ for which our $\Delta_{Q_{N_i}}$ all have an equal number $0 \leq t \leq 2q$ of cyclic factors.
Since a Taylor--Wiles datum of level $N$ is also a Taylor--Wiles datum of level $N'$ for $N' \leq N$, we can use our subsequence to obtain Taylor--Wiles data as in the statement of the lemma.
\end{proof}

Fix for $N \geq 1$ Taylor--Wiles data $$
(Q_N,\{(\hat{T}_v,\hat{B}_v)\}_{v \in Q_N})
$$ as in Lemma \ref{compatible tw data}, and when $N = 0$ we will take $Q_0 = \emptyset$.
Let $g = t - 4 [F:Q] + \#S - 1$, with $t$ as in Lemma \ref{compatible tw data} for our choice of Taylor--Wiles data.
Note that this coincides with the $g$ of Corollary \ref{TW places number field}.
Indeed we have that $z = \dim Z(\GSp_4) = 1$, $n_v$ is the number of cyclic factors in $\Delta_v$ and for each $v | \infty$ we have $h^0(G_{F_v},\splie_4) = 4$, which follows from a straightforward computation using that $\bar{\rho}$ is odd in the sense of \cite[Definition 7.6.2]{surfaces}, as the similitude character is the inverse cyclotomic character.
Set $\Delta_{\infty} = \mathbb{Z}_p^t$ and, for each $N \geq 1$, let $\Delta_N = \Delta_{Q_N}$ and fix a surjection $\Delta_{\infty} \twoheadrightarrow \Delta_N$.
Set $\Delta_0$ to be the trivial group equipped with the unique map from $\Delta_\infty$.

With this notation in place, the rest of the setup will be almost identical to that of \cite[7.11]{surfaces}, but we shall include it for clarity.
We fix $I \subset S_p$ with $\#I \leq 1$ and let $\Lambda = \Lambda_I$.
By enlarging $E$ if required, we suppose that $E$ contains a primitive $p$th root of unity, and that if $p=3$ then $E$ contains a primitive $9$th root of unity.
Then let $\mathcal{T} = \Lambda[(\prod_{v \in S} \widehat{\GSp_4})/ \widehat{\mathbb{G}_m}]$, a power series ring in $11 \#S - 1$ variables over $\Lambda$.
Set $S_{\infty} = \mathcal{T}[[\Delta_{\infty}]]$ and let $\mathbf{a} = \ker(S_{\infty} \to \Lambda)$ be the augmentation ideal.
Fix, for each $v \in R$, a pair of non-trivial characters $\chi_v = (\chi_{v,1},\chi_{v,2})$ as in Section \ref{setup for modularity} with $\chi_{v,1} \neq \chi_{v,2}^{\pm 1}$.
In what follows, write $\chi$ for these choice of characters and $1$ for the trivial characters.

Let $N \geq 0$.
We will write $S_N = \mathcal{T}[\Delta_N]$.
Recall the augmented global deformation problems of Section \ref{setup for modularity}.
Let $R_N^{1,I,\varsigma} = R_{1,Q_N}^{I,\varsigma}$ and $R_N^{\chi,I,\varsigma} = R_{\chi,Q_N}^{I,\varsigma}$.
Let $R^{1,I,\varsigma,\loc} = R^{S,\loc}_{\mathcal{S}_{1}^{I,\varsigma}}$ and $R^{\chi,I,\varsigma,\loc} = R^{S,\loc}_{\mathcal{S}_{\chi}^{I,\varsigma}}$ be the completed tensor product of deformation rings described in Section \ref{number field def setup}.

There are still canonical isomorphisms
\begin{equation} \label{loc rings iso mod p}
R^{1,I,\varsigma,\loc}/(\lambda) \cong R^{\chi,I,\varsigma,\loc}/(\lambda)    
\end{equation}
and, for every $N \geq 0$,
\begin{equation} \label{def rings iso mod p}
    R_N^{1,I,\varsigma}/(\lambda) \cong R_N^{\chi,I,\varsigma}/(\lambda).
\end{equation}
For $N \geq 1$, we saw in Section \ref{dual group deltaq} that our Taylor--Wiles datum give both $R_N^{1,I,\varsigma}$ and $R_N^{\chi,I,\varsigma}$ the structure of a $\Lambda[\Delta_N]$-algebra, and it is easy to see that there are isomorphisms
\begin{align}
    R_N^{1,I,\varsigma} \otimes_{\Lambda[\Delta_N]} \Lambda \cong R_0^{1,I,\varsigma} \label{1 ring iso mod delta} \\
    R_N^{\chi,I,\varsigma} \otimes_{\Lambda[\Delta_N]} \Lambda \cong R_0^{\chi,I,\varsigma} \label{chi ring iso mod delta}  
\end{align}
compatible with the isomorphisms \ref{def rings iso mod p}.

For each $N \geq 0$, fix representatives $\rho_N^{1,I,\varsigma}$ of $\mathcal{D}_{\mathcal{S}_{1,Q_N}^{I,\varsigma}}$ and $\rho_N^{\chi,I,\varsigma}$ of $\mathcal{D}_{\mathcal{S}_{\chi,Q_N}^{I,\varsigma}}$ which are compatible with \ref{def rings iso mod p} and also compatible with the isomorphisms \ref{1 ring iso mod delta} and \ref{chi ring iso mod delta}.
As explained in \cite[Lemma 7.1.6]{surfaces}, the natural map of functors $\mathcal{D}^S_{\mathcal{S}_{1,Q_N}^{I,\varsigma}} \to \mathcal{D}_{\mathcal{S}_{1,Q_N}^{I,\varsigma}}$ has the structure of a torsor with group $\mathcal{T}$, and our choice of representatives determine an isomorphism $R^{S,\loc}_{\mathcal{S}_{1,Q_N}^{I,\varsigma}} \cong R_N^{1,I,\varsigma} \hat{\otimes}_\Lambda \mathcal{T}$ (and similarly, an isomorphism 
$R^{S,\loc}_{\mathcal{S}_{\chi,Q_N}^{I,\varsigma}} \cong R_N^{\chi,I,\varsigma} \hat{\otimes}_\Lambda \mathcal{T}$).
Therefore, we obtain an $R^{1,I,\varsigma,\loc}$-algebra structure on $R_N^{1,I,\varsigma} \hat{\otimes}_\Lambda \mathcal{T}$ and a $R^{\chi,I,\varsigma,\loc}$-algebra structure on $R_N^{\chi,I,\varsigma} \hat{\otimes}_\Lambda \mathcal{T}$, and these are again compatible with \ref{loc rings iso mod p} and \ref{def rings iso mod p}.
Now let $R_\infty^{1,I,\varsigma}$ denote the ring of formal power series in $g$ variables over $R^{1,I,\varsigma,\loc}$ and similarly for $R_\infty^{\chi,I,\varsigma}$ over $R^{\chi,I,\varsigma,\loc}$; we then have
\begin{equation} \label{infty rings iso mod p}
    R_\infty^{1,I,\varsigma}/(\lambda) \cong R_\infty^{\chi,I,\varsigma}/(\lambda)
\end{equation}
via \ref{loc rings iso mod p}.
As in Corollary \ref{TW places number field}, we can find local $\Lambda$-algebra surjections $R_\infty^{1,I,\varsigma} \twoheadrightarrow R_N^{1,I,\varsigma} \hat{\otimes}_\Lambda \mathcal{T}$ and $R_\infty^{\chi,I,\varsigma} \twoheadrightarrow R_N^{\chi,I,\varsigma} \hat{\otimes}_\Lambda \mathcal{T}$, which we may again assume are compatible modulo $(\lambda)$ in the sense of \ref{infty rings iso mod p} and \ref{def rings iso mod p}, and with the isomorphisms \ref{1 ring iso mod delta} and \ref{chi ring iso mod delta}.

We will apply the patching criterion of \cite[Proposition 7.10.1]{surfaces} together with Theorem \ref{hecke galois rep}, Proposition \ref{finite free balanced} and Proposition \ref{deltaq coinvariants}.
Before we summarize the main consequences of this in a proposition, we will introduce the following notation.
If $R \in \CNL_{\CalO}$ we will let $R'$ denote the maximal $\CalO$-flat quotient of $R$.

\begin{proposition} \label{patched modules}
There exists the following data:
\begin{itemize}
    \item Maps $S_\infty \to R_\infty^{1,I,\varsigma}$ and $S_\infty \to R_\infty^{\chi,I,\varsigma}$ of $\LI$-algebras
    \item Finite $S_\infty$-modules $M_\infty^{1,I,\varsigma}$ and $M_\infty^{\chi,I,\varsigma}$
    \item Isomorphisms
\begin{align}
    M_\infty^{1,I,\varsigma}/\mathbf{a} \cong M^{1,I,\varsigma} \label{1 module mod a}   \\
    M_\infty^{\chi,I,\varsigma}/\mathbf{a} \cong M^{\chi,I,\varsigma} \label{chi module mod a}
\end{align}
    \item Maps $ R_\infty^{1,I,\varsigma} \to \End_{S_\infty}(M^{1,I,\varsigma})$ and $ R_\infty^{\chi,I,\varsigma} \to \End_{S_\infty}(M^{\chi,I,\varsigma})$ of $S_\infty$-algebras which respect the respective actions of $R_0^{1,I,\varsigma}$ and $R_0^{\chi,I,\varsigma}$ on $M^{1,I,\varsigma}$ and $M^{\chi,I,\varsigma}$
    \item Isomorphisms $M_\infty^{1,I,\varsigma}/\lambda M_\infty^{1,I,\varsigma} \cong M_\infty^{\chi,I,\varsigma}/\lambda M_\infty^{\chi,I,\varsigma}$ respecting \ref{infty rings iso mod p} and $M^{1,I,\varsigma}/\lambda M^{1,I,\varsigma} \cong M^{\chi,I,\varsigma}/\lambda M^{\chi,I,\varsigma}$
    respecting \ref{def rings iso mod p}, which are also compatible with \ref{1 module mod a} and \ref{chi module mod a}.
\end{itemize}
The modules $M_\infty^{1,I,\varsigma}$ and $M_\infty^{\chi,I,\varsigma}$ satisfy the following properties:
\begin{enumerate}
    \item If $\#I = 0$ then $M_\infty^{1,I,\varsigma}$ and $M_\infty^{\chi,I,\varsigma}$ are free $S_\infty$-modules \label{infty module free}
    \item If $\#I = 1$ then $M_\infty^{1,I,\varsigma}$ and $M_\infty^{\chi,I,\varsigma}$ are balanced $S_\infty$-modules, in the sense of \cite[Definition 2.10.1]{surfaces} \label{infty module balanced}
    \item $M_\infty^{1,I,\varsigma}$ (resp. $M_\infty^{\chi,I,\varsigma}$) is a maximal Cohen--Macaulay module over $R_\infty^{1,I,\varsigma}$ (resp. $R_\infty^{\chi,I,\varsigma}$). \label{infty module cm}
\end{enumerate}
\end{proposition}
\begin{proof}
The existence of the stated data satisfying properties \ref{infty module free} and \ref{infty module balanced} follows in the same way as in \cite[7.11]{surfaces} with our above setup in place.
In order to satisfy the patching criteria of \cite[7.10.1]{surfaces}, we are appealing to Theorem \ref{hecke galois rep}, Proposition \ref{finite free balanced} and Proposition \ref{deltaq coinvariants}, which are the direct analogues of the corresponding statements of \cite{surfaces}.

It remains to show \ref{infty module cm}, which is the statement corresponding to \cite[Proposition 7.11.2]{surfaces}.
The same proof would go through, provided we have the analogue of \cite[Equation 7.11.1]{surfaces}:
\begin{equation} \label{numerical coincidence}
    \dim (R_\infty^{1,I,\varsigma})' = \dim (R_\infty^{\chi,I,\varsigma})' = \dim S_\infty - \# I.
\end{equation}
The same argument still shows that $S_\infty$ is formally smooth over $\LI$ of relative dimension $t + 11\#S - 1$, while $(R_\infty^{1,I,\varsigma})'$ and $(R_\infty^{\chi,I,\varsigma})'$ are both equidimensional of relative dimension $g + 10\#S + 4[F : Q] - \#I$ over $\LI$.
Since $g = t - 4 [F:Q] + \#S - 1$, the equalities in \ref{numerical coincidence} hold and we are done.
\end{proof}

\subsection{Modularity results for abelian surfaces} \label{sect modularity results}

We can now use the patched modules of Proposition \ref{patched modules} to prove modularity lifting theorems, by exactly the same methods of \cite{surfaces}.
From the work of the previous sections, this amounts to replacing the vastness assumptions on $\bar{\rho}$ by our assumption that $\bar{\rho}$ is $\GSp_4$-reasonable.
The following is the analogue of \cite[Theorem 8.4.1]{surfaces}, and is a more precise version of Theorem \ref{intro modular galois reps} stated in the introduction.

\begin{theorem} \label{modular galois reps}
Let $F$ be a totally real field in which $p \geq 3$ splits completely.
Suppose that $\rho:G_F \to \GSp_4(\Qpbar)$ is a continuous Galois representation satisfying:
\begin{enumerate}
    \item $\nu \circ \rho = \varepsilon^{-1}$
    \item $\bar{\rho}$ is $\GSp_4$-reasonable and tidy in the senses of Definition \ref{reasonable} and Definition \ref{defn: tidy}
    \item For every $v | p$, $\restr{\rho}{G_{F_v}}$ is $p$-distinguished weight $2$ ordinary in the sense of Definition \ref{defn: p-distinguished ordinary}
    \item There exists an automorphic representation $\pi$ of $\GSp_4(\mathbb{A}_F)$ of parallel weight 2 and central character $|\cdot|^2$, ordinary in the sense of \cite[Definition 2.4.25]{surfaces}, such that $\overline{\rho_{\pi,p}} \cong \bar{\rho}$, where $\rho_{\pi,p}$ is as in Corollary \ref{parallel weight 2 galois rep}.
    \item For every finite place $v$ of $F$, the representations $\restr{\rho}{G_{F_v}}$ and $\restr{\rho_{\pi,p}}{G_{F_v}}$ are pure.
\end{enumerate}
Then $\rho$ is modular, in the sense that there exists an ordinary automorphic representation $\pi'$ of $\GSp_4(\mathbb{A}_F)$ of parallel weight 2 and central character $|\cdot|^2$ satisfying $\rho_{\pi',p} \cong \rho$.
For every finite place $v$ of $F$ we have full local-global compatibility, in the sense that
$$
\WD(\restr{\rho}{G_{F_v}})^{\Fss} \cong \rec_{\GT,p}(\pi'_v \otimes |\nu|^{-3/2}).
$$
\end{theorem}

\begin{proof}
The theorem now follows from Proposition \ref{patched modules} and Corollary \ref{parallel weight 2 galois rep} by the same arguments used to prove \cite[Theorem 8.4.1]{surfaces}.
More precisely, one could use the two results mentioned to firstly prove analogues of \cite[Theorem 7.13.6, Theorem 8.2.1, Corollary 8.2.2, Lemma 8.3.2]{surfaces} in which vastness of $\bar{\rho}$ is replaced by the assumption that $\bar{\rho}$ is $\GSp_4$-reasonable, from which the theorem would follow by the same proof word for word.
\end{proof}

We can now apply Theorem \ref{modular galois reps} to deduce a modularity lifting theorem for abelian surfaces.
We recall as in \cite[Definition 9.1.8]{surfaces} that a Galois representation $\rho: G_F \to \GSp_4(\Qpbar)$ is said to be modular if there exists a cuspidal automorphic representation $\pi$ of $\GSp_4(\mathbb{A}_F)$ of parallel weight 2 and central character $|\cdot|^2$ such that $\overline{\rho_{\pi,p}} \cong \bar{\rho}$.
We then say an abelian surface $A/F$ is modular if the associated Galois representation $\rho_{A,l}$ (given by the $l$-adic cohomology group $H^1(A_{\bar{F}},\Qlbar)$) is modular for some prime number $l$.
The following modularity lifting theorem is the analogue of \cite[Proposition 10.1.1]{surfaces} with the identical proof via Theorem \ref{modular galois reps}.

\begin{theorem} \label{modular abelian surfaces}
Let $F$ be a totally real field in which $p \geq 3$ splits completely.
Let $A/F$ be an abelian surface such that
\begin{enumerate}
    \item $A$ has good ordinary reduction at every $v | p$
    \item for each $v|p$, the unit root crystalline eigenvalues $\alpha_v,\beta_v$ are distinct modulo $p$
    \item $A$ admits a polarization of degree prime to $p$.
\end{enumerate}
Suppose that the residual Galois representation $\overline{\rho_{A,p}}$ is such that
\begin{enumerate} [resume]
    \item $\overline{\rho_{A,p}}$ is $\GSp_4$-reasonable and tidy
    \item $\overline{\rho_{A,p}}$ is ordinarily modular: there exists $\pi$ an automorphic representation of $\GSp_4(\mathbb{A}_F)$ for which
    \begin{enumerate}
        \item $\pi$ has parallel weight $2$ and central character $|\cdot|^2$
        \item for every $v|p$, $\pi_v$ is ordinary
        \item $\overline{\rho_{A,p}} \cong \overline{\rho_{\pi,p}}$
        \item for every finite place $v$ of $F$, $\rho_{\pi,p}$ is pure.
    \end{enumerate}
\end{enumerate}
Then $A$ is modular.
\end{theorem}

The assumption that $\bar{\rho}$ is tidy is still enforced in the above modularity lifting theorems.
However, due to the other assumptions on $\bar{\rho}$, only a small number of possible images $\Gamma' = \bar{\rho}(G_F)$ are ruled out in practice.
The assumptions that $p$ is unramified in $F$ and that the similitude character is given by $\varepsilon^{-1}$ imply that the image $\Gamma'$ must surject onto $\mathbb{F}_p^\times$ via the similitude character.
The intersection $\Gamma = \bar{\rho}(G_{F(\zeta_p)})$ of $\Gamma'$ with $\Sp_{4}(k)$ must also be absolutely $\GL_4$-irreducible for $\bar{\rho}$ to be $\GSp_4$-reasonable.
When $k = \mathbb{F}_3$, 22 of the 25 such possible images $\Gamma'$ are tidy, as can be seen from Table \ref{subgroups of gsp4(f3)}.
A similar story holds for $k = \mathbb{F}_5$ and $k = \mathbb{F}_7$ with 65/69 and 86/86 respective possible images being tidy.

Now suppose that $K/F$ is a quadratic extension by a totally real field.
One source of modular mod $p$ Galois representations comes from inductions of representations $\bar{\varrho}:G_K \to \GL_2(\mathbb{F}_q)$ for some small $q$.
While we study these inductions more closely in Section \ref{sec modular elliptic curves}, we have included the following variation of \cite[Proposition 10.1.3]{surfaces} in this section for convenience.
In combination with Theorem \ref{modular abelian surfaces} it yields more explicit criteria on the mod $p$ Galois representation of an abelian surface $A/F$ under which we can establish modularity of $A$.

\begin{proposition} \label{modular induced rep}
Let $q \in \{5,7,9\}$.
Suppose that $K/F$ is a quadratic extension by a totally real field in which $p$ is unramified and $\bar{\varrho}: G_K \to \GL_2(\mathbb{F}_q)$ is a continuous representation of determinant $\varepsilon^{-1}$ which is semistable ordinary of weight $0$, in the sense of Definition \ref{defn semistable ordinary}.
Let $\bar{\rho} = \Ind_{G_K}^{G_F} \bar{\varrho}: G_F \to \GSp_4(\mathbb{F}_q)$ be the induced representation with similitude character $\varepsilon^{-1}$ as in the start of Section \ref{subsec modularity lifting elliptic}.
Then
\begin{enumerate}
    \item $\bar{\rho}$ is ordinarily modular, in the sense that there exists an automorphic representation $\pi$ of $\GSp_4(\mathbb{A}_F)$ such that
    \begin{enumerate} [label=(\alph*)]
        \item $\pi$ has parallel weight $2$ and central character $|\cdot|^2$
        \item for every $v|p$, $\pi_v$ is ordinary
        \item $\bar{\rho} \cong \overline{\rho_{\pi,p}}$
        \item for every finite place $v$ of $F$, $\rho_{\pi,p}$ is pure.
    \end{enumerate}
    \item If for every place $v \in S_p$ which splits into places $w,w'$ of $K$ we have 
\begin{equation*}
 \restr{\bar{\varrho}}{G_{K_w}} \cong 
\begin{pmatrix}
\chi_w & * \\
 0 & \chi_w^{-1} \varepsilon^{-1}
\end{pmatrix}   
\end{equation*}
with $\chi_w$ unramified and not equal to its $\Gal(K/F)$-conjugate then $\bar{\rho}$ is $p$-distinguished weight 2 ordinary.
\item Suppose that $\bar{\varrho}(G_{K(\zeta_p)})$ and  $\bar{\rho}(G_{F(\zeta_p)})$ are both absolutely $\GL_n$-irreducible.
If $p=3$ suppose that $\bar{\varrho}$ is valued in $\GL_2(\mathbb{F}_p)$.
If $p=3$ or $p=5$ suppose that $\bar{\rho}(G_F)$ is not isomorphic to the group $C_2.C_2\wr C_4$ of order 128 (group ID 128,77).
Then $\bar{\rho}$ is tidy and $\GSp_4$-reasonable.
\end{enumerate}
\end{proposition}

\begin{proof}
As in the proof of \cite[Proposition 10.1.3]{surfaces}, it follows from \cite[Theorem A]{blgg} that modularity of $\bar{\varrho}$ implies that $\bar{\varrho}$ is the mod $p$ Galois representation attached to an ordinary Hilbert modular form, and that $\bar{\rho}$ is therefore ordinarily modular by automorphic induction.
Modularity of $\bar{\varrho}$ then follows by the known cases of Serre's conjecture stated in \cite[Theorem 1.1]{Allene2108064118}.
The second part can be proved in the same way as Proposition \ref{induction ordinary}.
For the third part, tidiness follows from Lemma \ref{tidy inductions} and $\bar{\rho}$ is $\GSp_4$-reasonable by the same computation as in the proof of Proposition \ref{inductions adequate}.
\end{proof}

\section{Inductions from $\GL_2$ and modularity of elliptic curves} \label{sec modular elliptic curves}

\subsection{Modularity lifting for elliptic curves} \label{subsec modularity lifting elliptic}

Suppose for now that $K/F$ is an arbitrary quadratic extension of number fields and that $\bar{\varrho}: G_K \to \GL_2(\Fpbar)$ is a continuous representation with $\det \bar{\varrho} = \varepsilon^{-1}$.
The induced representation $$\bar{\rho} = \Ind_{G_K}^{G_F} \bar{\varrho}: G_F \to \GL_4(\Fpbar)$$ may be viewed as taking values in $\GSp_4(\Fpbar)$ with similitude character $\varepsilon^{-1}$.
We recall how this is done in \cite[2.2]{surfaces}.
Let $V$ denote the underlying representation of $\bar{\varrho}$ and let $\sigma \in G_F \setminus G_K$.
Writing $\bar{\rho} = V \oplus \sigma V$ and choosing an arbitrary symplectic form $\wedge$ on $V$, we can extend this form to $V \oplus \sigma V$ by taking $V$ and $\sigma V$ to be orthogonal and by setting $\sigma v_1 \wedge \sigma v_2 = \varepsilon^{-1}(\sigma) v_1 \wedge v_2$.
As noted in \cite[2.2]{surfaces}, this is compatible with formation of restriction of scalars of an elliptic curve $E/K$ in the sense that if $A = \Res_{K/F} E$ has associated mod $p$ Galois representation $\overline{\rho_{A,p}}$ then this is exactly the induction of the mod $p$ Galois representation $\overline{\varrho_{E,p}}$ of determinant $\varepsilon^{-1}$ given by the dual of $E[p]$.

We make some observations about $G$, the image of the induced representation $\bar{\rho}$.
Certainly $G \subset \Delta \rtimes \mathbb{Z}/2\mathbb{Z}$, where $$\Delta = \{(A,B) \in \GL_2(\mathbb{F}_p) \times  \GL_2(\mathbb{F}_p) | \det A = \det B\} \subset \GSp_4(\mathbb{F}_p)$$
and $\mathbb{Z}/2\mathbb{Z}$ permutes the two factors of $\GL_2(\mathbb{F}_p)$.
The similitude character of $(A,B) \in \Delta$ is easily seen to be equal to $\det A = \det B$. 
We then have the following lemma about tidiness of induced representations.

\begin{lemma} \phantomsection \label{tidy inductions}
Let $K/F$ be a quadratic extension of number fields with $p \geq 3$ unramified in $K$.
Suppose that $\bar{\varrho}: G_K \to \GL_2(\mathbb{F}_p)$ is a continuous representation with $\det \bar{\varrho} = \varepsilon^{-1}$ such that $\bar{\varrho}(G_{K(\zeta_p)})$ is absolutely irreducible and that the induced representation $\bar{\rho} = \Ind_{G_K}^{G_F} \bar{\varrho}$ is absolutely irreducible.
If $p=3$ or $p=5$ suppose that $\bar{\rho}(G_F)$ is not isomorphic to the group $C_2.C_2\wr C_4$ of order 128 (group ID 128,77).
Then $\bar{\rho}$ is tidy.
\end{lemma}
\begin{proof}
Let $G = \bar{\rho}(G_F)$ and $\Gamma = \bar{\varrho}(G_K)$.
By Definition \ref{defn: tidy}, $G$ is tidy if there exists $g \in G$ such that $\nu(g) \neq 1$ and for every pair of distinct eigenvalues $\alpha \neq \beta$ of $g$ we have $\frac{\alpha} { \beta} \neq \nu(g)$.
A computation in \texttt{magma} gives the result when $p \leq 7$, so suppose throughout the rest of this proof that $p>7$.
Consider $G \cap \Delta = \bar{\rho}(G_K)$, isomorphic to a subgroup of $\Gamma \times \Gamma$ (since $\bar{\varrho}^\sigma(G_K) \cong \Gamma$, where $\Gal(K/F) = \{1,\sigma\}$).
Goursat's lemma states that subgroups of $\Gamma \times \Gamma$ with surjective projections onto the two factors of $\Gamma$ are parameterised by triples
$$
(H_1 \triangleleft \Gamma,H_2 \triangleleft \Gamma, \varphi: \Gamma/H_1 \xrightarrow{\sim} \Gamma/H_2)
$$
with the corresponding subgroup given by
$$
\{(\gamma_1,\gamma_2) \in \Gamma \times \Gamma : \varphi(\gamma_1 H_1) = \gamma_2 H_2\}. 
$$
We show tidiness of $G$ by giving a lower bound on the number of elements of $G$ of the form $(\gamma_1,\gamma_2)$, where each $\gamma_i$ is a scalar matrix.

We firstly bound the order of the center of $\Gamma$ by an argument similar to \cite[Lemma 3]{dt94}.
Since $\det \bar{\varrho} = \varepsilon^{-1}$, we have a surjection
$$
\det: (\Gamma/Z(\Gamma))^{\mathrm{ab}} \to \mathbb{F}_p^\times/\det(Z(\Gamma)).
$$
By the classification of subgroups of $\PGL_2(\mathbb{F}_p)$, the abelianization of $\Gamma/Z(\Gamma)$ must isomorphic to a subgroup of $(\mathbb{Z}/2\mathbb{Z})^2$ since $\Gamma$ acts absolutely irreducibly.
Thus $\det Z(\Gamma)$ must be index $2$ inside $\mathbb{F}_p^\times$.
Therefore, if $p \equiv 1 \mod 4$ we must have $Z(\Gamma) \cong \mathbb{F}_p^\times$ and otherwise we only have that $\# Z(\Gamma) \geq \frac{p-1}{2}$.

Now suppose that $H \triangleleft \Gamma$ is a normal subgroup.
We consider the possibilities for the index of the image of $Z(\Gamma)$ inside of $Z(\Gamma/H)$.
The short exact sequence of groups
\begin{equation} \label{SES of groups}
    1 \to Z(\Gamma) H/H \to \Gamma/H \to \Gamma/(Z(\Gamma) H) \to 1
\end{equation}
implies that $[Z(\Gamma/H): Z(\Gamma) H/H] \mid \# Z(\Gamma/(Z(\Gamma) H))$.
Thus we need to compute the possibilities for the center of a quotient of an irreducible subgroup of $\PGL_2(\mathbb{F}_p)$.
If $\Gamma$ is projectively dihedral then any quotient will either be $D_{2n}$ (where $n \geq 1$) or trivial.
Hence in that case $Z(\Gamma/(Z(\Gamma) H))$ is isomorphic to a subgroup of $(\mathbb{Z}/2\mathbb{Z})^2$.
Otherwise, $Z(\Gamma/(Z(\Gamma) H))$ must be trivial unless $\Gamma/Z(\Gamma)$ is isomorphic to either $A_4$ or $S_4$.
In the case of $A_4$, $Z(\Gamma/(Z(\Gamma) H))$ must have order 1 or 3.
In the case of $S_4$, $Z(\Gamma/(Z(\Gamma) H))$ must have order 1 or 2.

Now suppose that $G \cap \Delta$ is defined (on fixing an isomorphism $\Gamma \cong \bar{\varrho}^\sigma(G_K)$) by subgroups $H_1,H_2 \triangleleft \Gamma$ and an isomorphism $\varphi: \Gamma/H_1 \xrightarrow{\sim} \Gamma/H_2$.
Then $\varphi$ restricts to an isomorphism $Z(\Gamma/H_1) \xrightarrow{\sim} Z(\Gamma/H_2)$.
Consider the following diagram 
$$
\begin{tikzcd}
                                             &                                               & Z(\Gamma/H_1)                                                                                        &                                                            \\
Z(\Gamma) \arrow[r, two heads]               & Q_1 \arrow[ru, " \leq 4"', hook] &                                                                                                    & \varphi^{-1}(Q_2) \arrow[lu, " \leq 4", hook] \\
Z \arrow[rr, two heads] \arrow[u, "m", hook] &                                               & Q_1 \cap \varphi^{-1}(Q_2) \arrow[ru, hook] \arrow[lu, "m \leq 4", hook] &                                                           
\end{tikzcd}
$$
where
\begin{itemize}
\item $Q_i = Z(\Gamma)H_i/H_i$ for $i=1,2$,
    \item $Z \leq Z(\Gamma)$ is the preimage of $Q_1 \cap \varphi^{-1}(Q_2)$
under the quotient map $Z(\Gamma) \twoheadrightarrow Q_1$,
\end{itemize}
and where the labels on the arrows describe the index of an inclusion of subgroups.
It follows from the above paragraph that $$m = [Z(\Gamma): Z] = [Q_1: Q_1 \cap \varphi^{-1}(Q_2)] \leq 4.$$
In the case where $Z(\Gamma/(Z(\Gamma) H_1)) \cong (\mathbb{Z}/2\mathbb{Z})^2$ we claim that $m \leq 2$ (possibly on interchanging $H_1$ and $H_2$).
Indeed, suppose not and note firstly that $\Gamma/(Z(\Gamma) H_1)$ must be isomorphic to $(\mathbb{Z}/2\mathbb{Z})^2$ in order for its center to be order $4$.
The short exact sequence \ref{SES of groups} then gives that $\Gamma/H_1$ is given by an extension of $(\mathbb{Z}/2\mathbb{Z})^2$ by a central subgroup isomorphic to $\mathbb{Z}/k\mathbb{Z}$ for some $k \geq 1$.
We deduce that $Z(\Gamma/H_1)$ is isomorphic to one of
$$
\mathbb{Z}/k \mathbb{Z}, \quad \mathbb{Z}/k \mathbb{Z} \oplus \mathbb{Z}/2\mathbb{Z}, \quad \mathbb{Z}/2k \mathbb{Z} \oplus \mathbb{Z}/2\mathbb{Z} \quad \text{or} \quad \mathbb{Z}/k \mathbb{Z} \oplus (\mathbb{Z}/2\mathbb{Z})^2.
$$
Since $m=4$, the cyclic subgroups $Q_1$ and $\varphi^{-1}(Q_2)$ must both be of index $4$ with intersection of index $16$ (else we could interchange $H_1$ with $H_2$ and get $m \leq 2$). 
These subgroups therefore generate $Z(\Gamma/H_1)$.
Since $Z(\Gamma/H_1)$ cannot be generated by two cyclic subgroups of index $4$ when $\# Z(\Gamma/H_1)$ is divisible by $16$ by consideration of the cases above, we deduce that we always have $m \leq 3$.

From the definition of $Z$, for each $A \in Z$ there exists $B \in Z(\Gamma)$ with $g = \diag(A,B) \in G$.
Since $A$ and $B$ are both scalar matrices, we must have $B = \pm A$ in order for the equality $\det A = \det B$ to hold.
Hence if $\alpha \neq \beta$ are distinct eigenvalues of $g$ then their ratio must be $-1$.
Thus if $Z$ is not a $2$-group then $G$ is tidy.
If $Z$ is a $2$-group of order at least $8$ then either $Z$ contains an element with a primitive $8$th root of unity as an eigenvalue or $G$ contains an element of the form $\zeta_4 I$ where $\zeta_4$ is a primitive $4$th root of unity. 
We deduce that if the order of $Z$ does not divide 4 then $G$ is tidy.

We saw above that $Z$ is a subgroup of $Z(\Gamma)$ of index at most $3$.
If $p \equiv 3 \mod 4$ then we need $6$ to not be divisible by $p-1$ to guarantee that $\#Z$ does not divide $4$; this holds since $p > 7$.
If $p \equiv 1 \mod 4$ then $Z(\Gamma) \cong \mathbb{F}_p^\times$ and $Z$ is a subgroup of $\mathbb{F}_p^\times$ of index at most $3$.
We therefore need to eliminate the cases that $(p-1) | 8$ or $(p-1)|12$.
Since $p > 5$, the only case left to consider then is when $p=13$ and $\Gamma/Z(\Gamma) \cong A_4$.
Another computation in \texttt{magma} shows that all such subgroups are necessarily tidy.
\end{proof}

Now continue, as in Section \ref{sect modularity results}, to fix $p \geq 3$ a prime which splits completely in a totally real field $F$, and fix a quadratic extension $K/F$ in which $p$ is unramified.
Continue to let $S_p$ denote the set of places of $F$ dividing $p$.
We recall the idea of twisting $\bar{\varrho}: G_K \to \GL_2(\mathbb{F}_p)$ by a quadratic character in order to ensure that the induction $\bar{\rho}$ has sufficiently large image (as in \cite[7.5.24]{surfaces}), which is the content of the following propositions.

\begin{proposition} \phantomsection \label{D8 extension}
Let $L/K$ be an arbitrary finite extension and $\Sigma \subset \{v \in S_p | v \text{ splits in } K\}$.
There exists a finite extension $H/K$ with the following properties.
\begin{enumerate}
    \item $H/F$ is Galois with group $\Gal(H/F) \cong D_8$, the dihedral group of order 8
    \item $\Gal(H/K) \cong (\mathbb{Z}/2\mathbb{Z})^2$
    \item $H$ is the Galois closure over $F$ of a quadratic extension $M/K$
    \item For each $v | p$, $H/F$ is unramified at $v$
    \item For each $v | p$, $\Frob_v \in \Gal(H/F)$ is central if and only if $v \in \Sigma$ \label{sigma central primes}
    \item $H/K$ is disjoint from $L$. \label{disjoint from barrho}
\end{enumerate}
\end{proposition}
\begin{proof}
We follow the proof of \cite[Proposition 7.5.25]{surfaces}, with the above statement being a modification only in properties \ref{sigma central primes} and \ref{disjoint from barrho}.
In that proof an intermediate extension $F(\sqrt{\beta})$ is introduced, where $\beta = x^2 - \alpha y^2 \not\in (F^*)^2$ for some $x,y \in F$ and $K= F(\sqrt{\alpha})$.
The field $M$ is then taken to be $K(\sqrt{x + y\sqrt{\alpha}})$.

We deviate from this proof by choosing $\beta$ so that, for $v \in S_p$, $\beta \mod v$ is a quadratic residue if and only if $v \in \Sigma$.
We then see that $v \in \Sigma$ will split in the compositum $F(\sqrt{\beta}) \cdot K$, which we saw is the fixed field of the center of $\Gal(H/F)$.
Thus $\Frob_v$ is necessarily central if $v \in \Sigma$ and the remaining parts of the proposition follow in almost the same way by appropriate choice of $x$ and $y$.
\end{proof}

\begin{proposition} \phantomsection \label{inductions adequate}
Let $\bar{\varrho}: G_K \to \GL_2(\mathbb{F}_p)$ be a continuous representation with determinant $\varepsilon^{-1}$.
Let $L/K$ be such that $\ker(\bar{\varrho}) \cap \ker(\bar{\varrho}^\sigma) = G_L$, where $ \Gal(K/F) = \{1,\sigma\}$.
Let $\Sigma \subset \{v \in S_p | v \text{ splits in } K\}$.
Form the extension $H/K$ as in Proposition \ref{D8 extension}, disjoint from $L/K$.
Let $\delta_{M/K}$ be the character of $G_K$ associated to the arising quadratic extension $M/K$ and let $\bar{\rho} = \Ind_{G_K}^{G_F} (\bar{\varrho} \otimes \delta_{M/K})$.
Suppose that $\bar{\varrho}(G_K) \cap \SL_2(G_K) = \bar{\varrho}(G_{K(\zeta_p)})$ is absolutely irreducible.
Then $\bar{\rho}$ is $\GSp_4$-reasonable.
\end{proposition}

\begin{proof}
Let $\Gamma = \bar{\varrho}(G_K)$ and let $G = \bar{\rho}(G_F)$.
We show absolute $\GL_4$-irreducibility of $G \cap \Sp_4(\mathbb{F}_p)$.
Since $\Gamma \cap \SL_2(\mathbb{F}_p)$ is absolutely irreducible, it suffices to note that $\restr{\bar{\rho}}{G_{K(\zeta_p)}}$ is not isomorphic to its conjugate under an element of $\sigma \in G_{F(\zeta_p)} \setminus G_{K(\zeta_p)}$.
This follows, as such a $\sigma$ necessarily acts non-trivially on the character $\delta_{M(\zeta_p)/K(\zeta_p)}$.

We see by Lemma \ref{Sp4(Fp)} that $G$ is automatically $\GSp_4$-adequate if $p > 5$, and therefore $\bar{\rho}$ is $\GSp_4$-reasonable by Lemma \ref{adequate implies reasonable}.
Hence suppose that $p=3$ or $p=5$ and consider those subgroups $G_{\text{bad}} \subset \GSp_4(\mathbb{F}_p)$ for which 
\begin{enumerate}
    \item $G_{\text{bad}} \cap \Sp_4(\mathbb{F}_p)$ is absolutely $\GL_4$-irreducible
    \item $G_{\text{bad}}$ surjects onto $\mathbb{F}_p^\times$ via the similitude character
    \item there exists an index 2 subgroup $G_{\text{bad}}' \leq G_{\text{bad}}$ such that $G_{\text{bad}}' \subset \GL_4(\mathbb{F}_p)$ is reducible over $\mathbb{F}_p$
    \item at least one of the following hold
    \begin{enumerate}
        \item $H^1(G_{\text{bad}} \cap \Sp_4(\mathbb{F}_p), \mathbb{F}_p) \neq 0$
        \item $H^1(G_{\text{bad}} \cap \Sp_4(\mathbb{F}_p),\splie_4) \neq 0$.
    \end{enumerate}
\end{enumerate}
Any other possible image will be $\GSp_4$-adequate and hence our $\bar{\rho}$ will still be $\GSp_4$-reasonable.

 As computed in \texttt{magma}, the only failure of vanishing of the above cohomology groups when $p=3$ is the existence of such subgroups $G_{\text{bad}}$ for which $H^1(G_{\text{bad}} \cap \Sp_4(\mathbb{F}_p), \mathbb{F}_p) \neq 0$ (as in all cases $H^1(G_{\text{bad}} \cap \Sp_4(\mathbb{F}_p),\splie_4) = 0$).
 By \cite[Lemma 7.5.5]{surfaces}, to show that $\bar{\rho}$ is $\GSp_4$-reasonable in these cases, it suffices to check that for every normal subgroup $H_{\text{bad}} \subset G_{\text{bad}} \cap \Sp_4(\mathbb{F}_p)$ of index 3 or index 9 we have $H^1(H_{\text{bad}},\splie_4) = 0$ and $H_{\text{bad}}$ is absolutely $\GL_4$-irreducible.
 A further computation shows that this is always the case.
 
When $p=5$, the failure occurs from subgroups for which $H^1(G_{\text{bad}} \cap \Sp_4(\mathbb{F}_p),\splie_4) \neq 0$, with the other cohomology group vanishing.
A further computation then shows that the image in $\PGSp_4(\mathbb{F}_p)$ of every such subgroup $G_{\text{bad}}$ does not admit a surjective quotient onto $\mathbb{Z}/4\mathbb{Z}$.
It follows, just as in \cite[Lemma 7.5.19]{surfaces}, that $\zeta_p$ cannot lie in the fixed field of the adjoint action on $\splie_4$ and that $\bar{\rho}$ is therefore $\GSp_4$-reasonable.
\end{proof}

\begin{definition} \phantomsection \label{defn semistable ordinary}
Suppose that $A \in \CNL_{\CalO}$ and $\varrho: G_K \to \GL_2(A)$ is a continuous representation.
We say that $\varrho$ is semistable ordinary of weight $0$ if for every place $w | p$ of $K$ there is an isomorphism
\begin{equation} \label{gl2 ordinary}
 \restr{\varrho}{G_{K_w}} \cong 
\begin{pmatrix}
\chi_w & * \\
 0 & \chi_w^{-1} \varepsilon^{-1}
\end{pmatrix}   
\end{equation}
for some unramified character $\chi_w: G_{K_w} \to A^\times$.
\end{definition}

\begin{proposition} \phantomsection \label{induction ordinary}
Suppose that we are in the setup of Proposition \ref{inductions adequate}.
Let $A \in \CNL_\CalO$ and let $\varrho: G_K \to \GL_2(A)$ be a lift of $\bar{\varrho}$ which is semistable ordinary of weight $0$.
For $w | p$, let $\chi_w: G_{K_W} \to A^\times$ denote the character defining the eigenvalue of the image of Frobenius, as in (\ref{gl2 ordinary}) and let
$\overline{\chi_w}: G_{K_w} \to \Fpbar^\times$ be the reduction of $\chi_w$ modulo the maximal ideal of $A$.
Set $$\Sigma = \{v \in S_p | v = w w' \text{ in } K \text{ and } \overline{\chi_w} \neq (\overline{\chi_{w'}})^\sigma \},$$ where $\sigma \in G_F \setminus G_K$.
Then the induced representation $$\rho = \Ind_{G_K}^{G_F} (\varrho \otimes \delta_{M/K})$$ is $p$-distinguished weight $2$ ordinary, in the sense of Definition \ref{defn: p-distinguished ordinary}.
\end{proposition}

\begin{proof}
Suppose $v \in S_p$ is inert in $K$ firstly and let $w$ be the unique place of $K$ above $v$.
Then $\restr{\rho}{G_{F_v}} = \Ind^{G_{F_v}}_{G_{K_w}} \restr{(\varrho \otimes \delta_{M/K})}{G_{K_w}}$.
This has the desired form, with the unramified characters defining the diagonal entries being given by the two choices of extensions of $\chi_w$ to $G_{F_v}$ (which differ by a sign on $\Frob_v \in G_{F_v}$).

Now suppose that $v \in S_p$ splits as $v = w w'$ in $K$, so that there are isomorphisms $F_v \cong K_w$ and $F_v \cong K_{w'}$ interchanged via $\sigma$.
Then $\restr{\rho}{G_{F_v}} = \restr{\varrho}{G_{K_w}} \oplus \restr{(\varrho^{\sigma})}{G_{K_w}}$.
This has the required form, provided that the mod $p$ reduction of these summands, $\restr{\bar{\varrho}}{G_{K_w}}$ and $\restr{(\bar{\varrho}^{\sigma})}{G_{K_w}}$, are not isomorphic.
Thus it suffices to show that the characters $\overline{\chi_w} \otimes \restr{ \delta_{M/K}}{G_{K_w}}$ and $(\overline{\chi_{w'}} \otimes \restr{ \delta_{M/K}}{G_{K_{w'}}})^\sigma $ are not equal.
By choice of $\Sigma$, the following are equivalent
\begin{enumerate}
    \item $v \in \Sigma$
    \item $\restr{\delta_{M/K}}{G_{K_w}} =(\restr{ \delta_{M/K}}{G_{K_{w'}}}) ^\sigma$
    \item $\overline{\chi_w} \neq (\overline{\chi_{w'}})^\sigma$
\end{enumerate}
and we are therefore done.
\end{proof}

Now let $E/K$ be an elliptic curve.
The curve $E$ is said to be modular if there exists a cuspidal automorphic representation $\boldsymbol{\pi}$ of $\GL_2(\mathbb{A}_K)$ of weight 0 (see \cite[2.6]{surfaces}) and trivial central character such that $\varrho_{E,p} \cong \varrho_{\boldsymbol{\pi},p}$.
Here $\varrho_{E,p}$ is the Galois representation given by the $p$-adic cohomology group $H^1(E_{\overline{K}},\Qpbar)$ and the Galois representation $\varrho_{\boldsymbol{\pi},p}$ attached to $\boldsymbol{\pi}$ is as in \cite[Theorem 2.7.3]{surfaces}.
Similarly, we say an irreducible representation $\bar{\varrho}: G_K \to \GL_2(\Fpbar)$ is modular if there exists such a $\boldsymbol{\pi}$ as above with $\overline{\varrho} \cong \overline{\varrho_{\boldsymbol{\pi},p}}$.
We also say $\bar{\varrho}$ is ordinarily modular if we can take such a $\boldsymbol{\pi}$ as above which is also semistable weight 0 ordinary, in the sense that there exists, for every place $w|p$ of $K$, a non-zero eigenvector in $\boldsymbol{\pi}_w^{\mathrm{Iw}_w}$ whose eigenvalues under the Hecke operators
\begin{align*}
   U_w^1 &= [\mathrm{Iw}_w \diag(\varpi_w,1) \mathrm{Iw}_w] \\
   U_w^2 &= [\mathrm{Iw}_w \diag(\varpi_w,\varpi_w) \mathrm{Iw}_w]
\end{align*}
lie in $\Zpbar^\times$ (under our fixed isomorphism $\iota: \mathbb{C} \cong \Qpbar$).
Here $\varpi_w$ is a uniformizer of $K_w$ and $\mathrm{Iw}_w$ is the standard Iwahori subgroup of $\GL_2(K_w)$ given by those matrices in $\GL_2(\CalO_{K_w})$ whose reduction modulo $\varpi_w$ is upper triangular.

We now apply Theorem \ref{modular galois reps} to the mod $p$ Galois representation attached to $\Res_{K/F} E$ together with our above results to deduce a modularity lifting theorem for $E$ under mild conditions on the image of $\overline{\varrho_{E,p}}$.
This is an attempt at a generalization of \cite[Theorem 10.1.4]{surfaces}.

\begin{theorem} \phantomsection \label{modularity lifting elliptic curves}
Let $F$ be a totally real field in which $p \geq 3$ splits completely and let $K/F$ be a quadratic extension in which $p$ is unramified with Galois group $\{1,\sigma\}$. 
Let $E/K$ be an elliptic curve and let $\bar{\varrho} = \overline{\varrho_{E,p}}$ be the attached mod $p$ Galois representation of determinant $\varepsilon^{-1}$.
Suppose that the following conditions hold:
\begin{enumerate}
    \item $E$ has ordinary good reduction or multiplicative reduction at every place $w | p$
    \item $\bar{\varrho}(G_{K(\zeta_p)})$ is absolutely irreducible
    \item $\bar{\varrho}$ is ordinarily modular
    \item if $p = 3$ or $p=5$ then order of the subgroup of $\GL_2(\mathbb{F}_p) \times \GL_2(\mathbb{F}_p)$ generated by the image of $\bar{\varrho} \oplus \bar{\varrho}^\sigma$ together with the matrices $\diag(I,-I), \diag(-I,I)$ is not equal to $64$.
\end{enumerate}
Then $E$ is modular.
\end{theorem}

\begin{proof}
Since $E$ has ordinary good reduction or multiplicative reduction at every place $w | p$, the representation $\varrho_{E,p}$ is semistable ordinary of weight $0$.
Thus we may form the induction $\bar{\rho} = \Ind_{G_K}^{G_F} (\bar{\varrho} \otimes \delta_{M/K})$ by taking $\Sigma$ as in Proposition \ref{induction ordinary} and then taking $M/K$ as in Proposition \ref{inductions adequate} for this choice of $\Sigma$ and $\bar{\varrho}$.
This is the mod $p$ reduction of the Galois representation attached to the abelian surface $A = \Res_{K/F} E'$, where $E'/K$ is the quadratic twist of $E$ corresponding to $\delta_{M/K}$.
Modularity of $\rho = \rho_{A,p}$ then implies modularity of $E'$, just as in the proof of \cite[Theorem 9.3.4]{surfaces}, and hence modularity of $E$.
So we just have to show that $\rho$ satisfies the hypotheses of Theorem \ref{modular galois reps} to deduce modularity of $E$.

We see from Proposition \ref{inductions adequate} that $\bar{\rho}$ is $\GSp_4$-reasonable.
By Lemma \ref{tidy inductions}, we see that $\bar{\rho}$ is tidy, except possibly if $p \leq 5$ and $\# \bar{\rho}(G_F) = 128$.
The image of $\restr{\bar{\rho}}{G_K}$ is exactly the subgroup of $\GL_2(\mathbb{F}_p) \times \GL_2(\mathbb{F}_p)$ generated by the image of $\bar{\varrho} \oplus \bar{\varrho}^\sigma$ and the matrices $\diag(I,-I), \diag(-I,I)$, since the extensions of $K$ given by the Galois closure of $M/F$ and the fixed field of $\ker(\bar{\varrho} \oplus \bar{\varrho}^\sigma)$ are linearly disjoint by Proposition \ref{D8 extension}.
Since $\bar{\rho}(G_K)$ is an index $2$ subgroup of $\bar{\rho}(G_F)$, it follows from our hypotheses that $\bar{\rho}(G_F)$ does not have order $128$ and $\bar{\rho}$ is therefore tidy.

As in the proof of \cite[Theorem 10.1.4]{surfaces} we can find, by automorphic induction and our assumption that $\bar{\varrho}$ is ordinarily modular, a parallel weight 2 automorphic representation $\pi$ of $\GSp_4(\mathbb{A}_F)$ of central character $|\cdot|^2$ which is ordinary at all $v | p$ such that $\overline{\rho_{\pi,p}} \cong \bar{\rho}$ and with $\rho_{\pi,p}$ pure.
By Proposition \ref{induction ordinary}, we see that $\rho$ is $p$-distinguished weight $2$ ordinary.
Purity of $\rho$ follows from \cite[Proposition 2.8.1]{surfaces}.
Thus all of the hypotheses of Theorem \ref{modular galois reps} hold for $\rho$ and we are done.
\end{proof}

\subsection{Projectively dihedral representations} \label{subsect projectively dihedral}

One case amenable to establishing modularity of $\bar{\varrho}$ directly is when the projective image of $\bar{\varrho}$ is dihedral, in which case we can find a quadratic extension $L/K$ and a character $\bar{\psi}: G_L \to \Fpbar^\times$ such that $\bar{\varrho} \cong \Ind_{G_L}^{G_K} \bar{\psi}$.
We wish to control the Hodge--Tate weights of a lift $\psi: G_L \to \Zpbar^\times$ and use automorphic induction to deduce that $\bar{\varrho}$ is ordinarily modular.
This motivates the following definition.

\begin{definition} \phantomsection \label{def CM-compositum induced}
Let $K/F$ be a quadratic extension of number fields in which $p \geq 3$ is unramified.
We say a continuous Galois representation $\bar{\varrho}: G_K \to~ \GL_2(\Fpbar)$ is CM-compositum induced if $\bar{\varrho}(G_{K(\zeta_p)})$ is absolutely irreducible with $\det \bar{\varrho} = \varepsilon^{-1}$ and there exists
\begin{enumerate} 
    \item an imaginary CM quadratic extension $F'/F$, disjoint from $K/F$
    \item a continuous character $\bar{\psi}: G_L \to \Fpbar^\times$, where $L= F' \cdot K$ is the compositum
\end{enumerate}
such that 
\begin{enumerate} [resume]
    \item each place $w|p$ of $K$ splits in $L$
    \item for every place $z | p$ of $L$ there exists an unramified character $\overline{\chi_z}: G_{L_z} \to \Fpbar^\times$ with $\{\restr{\bar{\psi}}{G_{L_z}}, \restr{\bar{\psi}^\sigma}{G_{L_{z}}}\} = \{ \overline{\chi_z}, \overline{\chi_z}^{-1} \varepsilon^{-1} \}$, where $\Gal(L/K) = \{1,\sigma\}$ \label{character restriction at p}
    \item $\bar{\varrho} \cong \Ind_{G_L}^{G_K} \bar{\psi}$.
\end{enumerate}
\end{definition}

Note that if such an extension $F'/F$ in Definition \ref{def CM-compositum induced} exists then it is unique.
There are then two choices of $\bar{\psi}$, which are conjugate under $\sigma$ and whose product is the inverse cyclotomic character.
We also observe that if $k$ is a finite field and $\bar{\varrho}$ is CM-compositum induced and conjugate to a representation valued in $\GL_2(k)$, then $\bar{\psi}$ is also valued in $k^\times$.
Indeed, letting $q = \# k$ we must have $\{\bar{\psi},\bar{\psi}^\sigma \} = \{\bar{\psi}^q,(\bar{\psi}^\sigma)^q \}$ by considering $\restr{\bar{\varrho}}{G_L}$ and the action of Frobenius.
On the other hand we must also have that $\bar{\psi}^\sigma \neq \bar{\psi}^q$ due to condition \ref{character restriction at p}.
In particular, there can never exist a CM-compositum induced representation valued in $\GL_2(\mathbb{F}_3)$, since $\bar{\psi}$ would have image of order at most 2 and thus $\bar{\varrho}(G_{K(\zeta_p)})$ would have order at most $4$ (noting that $D_8 \subset \PGL_2(\mathbb{F}_3)$ is irreducible over $\mathbb{F}_3)$.
In what follows, if $\bar{\varrho}$ is CM-compositum induced then we shall use notation as in Definition \ref{def CM-compositum induced}.

\begin{proposition} \phantomsection \label{modular residual rep}
Suppose that $F$ is a totally real field and $K/F$ is a quadratic extension in which $p \geq 3$ is unramified.
Suppose that $\bar{\varrho}: G_K \to \GL_2(\Fpbar)$ is CM-compositum induced.
Then 
\begin{enumerate}

    \item for every place $w | p$ of $K$ there is an isomorphism
$$
\restr{\bar{\varrho}}{G_{K_w}} \cong 
\begin{pmatrix}
\overline{\chi_w} & 0 \\
 0 & \overline{\chi_w}^{-1} \varepsilon^{-1}
\end{pmatrix}
$$
for some unramified character $\overline{\chi_w}: G_{K_w} \to \Fpbar^\times$
    \item $\bar{\varrho}$ is ordinarily modular.
\end{enumerate}

\end{proposition}

\begin{proof}
If $w | p$ is a place of $K$ then by Definition \ref{def CM-compositum induced} $w$ splits in $L$ into places $z,z'$ and we can find an unramified character $\overline{\chi_z}: G_{L_z} \to \Fpbar^\times$ such that $\restr{\bar{\psi}}{G_{L_z}} = \overline{\chi_z}$ and $\restr{\bar{\psi}}{G_{L_{z'}}} = (\overline{\chi_z}^{-1})^\sigma \varepsilon^{-1}$.
It follows that $\bar{\varrho}$ is of the required form for the first part of the proposition.
Next observe that $z$ and $z'$ lie above distinct places of $F'$ (as if $v$, the place of $F$ below $w$, splits completely in $L$ then the action of $\Gal(L/F)$ is transitive on the places of $L$ above $v$, but $\{z,z'\}$ would be an orbit if these places both lie over the same places of $K$ and $F'$).
We claim that we can find an algebraic $p$-adic character $\varphi: G_L \to \Zpbar^\times$ such that $\bar{\varphi}$ differs from $\bar{\psi}$ by a character which is unramified at the places dividing $p$.
Indeed, suppose that for every $\tau: F' \to \Qpbar$ we have $k_\tau \in \{0,1\}$ satisfying $k_{\tau} + k_{\tau \circ c} =1$ with $c \in \Gal(F'/F)$ complex conjugation.
Then we can find a character of $G_{F'}$ with Hodge--Tate weights $\{k_\tau\}_{\tau}$ since $F'/F$ is a quadratic extension of a totally real field by an imaginary CM field (by \cite[Corollary 2.3.16]{patrikis2019variations}, for example).
Composing such a character with the inclusion $G_L \to G_{F'}$ yields the desired $\varphi$ by our above observation about the places of $L$ dividing $p$ in relation to those of $F'$ and $K$.

Now let $\theta: \Fpbar^\times \to \Zpbar^\times$ be the Teichmuller map and let $\psi = \varphi \cdot \theta(\bar{\psi}/\bar{\varphi})$, an algebraic lift of $\bar{\psi}$.
Let $\psi^\vee : \mathbb{A}_L^\times \to \mathbb{C}^\times$ denote the algebraic Hecke character corresponding to $\psi$.
Then the automorphic induction of $\psi^\vee \otimes |\cdot |^{1/2}$ to $\GL_2(\mathbb{A}_K)$ is a weight $0$ cuspidal automorphic representation $\boldsymbol{\pi}$ with trivial central character for which the attached Galois representation $\varrho_{\boldsymbol{\pi},p}$ lifts $\bar{\varrho}$.
Moreover, for $w|p$ a place of $K$, $\boldsymbol{\pi}_w$ can be identified with $(\psi^\vee_z \otimes | \cdot ~| ^{1/2}) \boxplus (\psi^\vee_{z'} \otimes | \cdot | ^{1/2})$ and we shall distinguish $z$ from $z'$ by supposing that $\restr{\psi}{I_{L_z}}$ is trivial.
Since $\psi^\vee$ is unramified at places dividing $p$ by construction, $\boldsymbol{\pi}_w$ is Iwahori-spherical and there exists an eigenvector whose eigenvalues $u_w^i$ under the operators $U_w^i$ have $p$-adic valuation given by
\begin{align*}
    v_p(u_w^1) &= v_p(\iota (\psi^\vee_z \otimes |\cdot|^{1/2}) (\varpi_w)) - \sum_{\tau: K \hookrightarrow \mathbb{C}, \iota \tau = w} \frac{1}{2} \\
    v_p(u_w^2) &= v_p( \iota (\psi^\vee_{z'} \otimes |\cdot|^{1/2}) (\varpi_w)) - (\sum_{\tau: K \hookrightarrow \mathbb{C}, \iota \tau = w} -\frac{1}{2}) + v_p(u_w^1)
\end{align*}
by (the proof of) \cite[Lemma 2.5]{clozel2014level}.
Here $\iota: \mathbb{C} \xrightarrow{\sim} \Qpbar$ was our fixed choice of isomorphism and we have identified embeddings $K \to \Qpbar$ with places of $K$ dividing $p$.

We compute $v_p( \iota \psi^\vee_{z'}(\varpi_{z'}))$; the computation for $z$ is similar.
Choose $\alpha \in L^\times$ whose valuation is non-zero only at the place $z'$ and equal to some $m > 0$.
Raising $\alpha$ to a sufficiently large power, we can suppose that $\psi^\vee_{v}(\alpha) = 1$ for every finite place $v \neq z'$ and $\psi^\vee_{z'}(\alpha) = \psi^\vee_{z'}(\varpi_{z'}^m)$.
Then $v_p( \iota \psi^\vee_{v}(\alpha)) = - v_p( \iota \psi^\vee_{\infty}(\alpha))$ since $\psi^\vee$ is trivial on $L^\times$.
We have $\psi^\vee_{\infty}(\alpha) = \prod_{\tau: L \to \mathbb{C}} \tau(\alpha)^{n_\tau}$ where $n_{\tau} = -1$ for every $\tau: L \to \mathbb{C}$ inducing the place $z'$ via $\iota$, since we supposed that $\restr{\psi^\vee}{I_{L_{z'}}} = \varepsilon^{-1}$.
We see that $$v_p(\iota \psi^\vee_{\infty}(\alpha)) = \sum_{\tau: L \hookrightarrow \mathbb{C}} m \cdot (-1)$$ where the sum is over complex embeddings $\tau$ of $L$ whose composition with $\iota$ induces the place $z'$ (as $v_p(\iota(\tau(\alpha))) = 0$ for any $\tau$ not inducing the place $z'$ by choice of $\alpha$).
We deduce that $$v_p( \iota (\psi^\vee_{z'} \otimes |\cdot|^{1/2}) (\varpi_w)) = - f_w + f_w/2$$ and similarly that $$v_p( \iota (\psi^\vee_{z} \otimes |\cdot|^{1/2}) (\varpi_w)) = 0 + f_w/2,$$ where $f_w = |K_w:\mathbb{Q}_p|$.
Thus $v_p(u_w^1) = v_p(u_w^2) = 0$, which shows the second point and concludes the proof.
\end{proof}

We deduce a modularity theorem for some elliptic curves over $K$ with CM-compositum induced Galois representation.

\begin{corollary} \phantomsection \label{cm comp ind curves modular} 

Let $F$ be a totally real field in which $p \geq 7$ splits completely and let $K/F$ be a quadratic extension in which $p$ is unramified. 
Suppose that $E/K$ is an elliptic curve which has ordinary good reduction or multiplicative reduction at every place $w | p$ and that $\overline{\varrho_{E,p}}$ is CM-compositum induced in the sense of Definition \ref{def CM-compositum induced}.
Then $E$ is modular.
\end{corollary}

\begin{proof}
This follows immediately from Theorem \ref{modularity lifting elliptic curves} and Proposition \ref{modular residual rep}.
\end{proof}

In order to apply Theorem \ref{modular galois reps} we required the elliptic curve $E/K$ to have ordinary good reduction or multiplicative reduction at the places dividing $p$.
We show that this automatically holds in some cases when $\overline{\rho_{E,p}}$ is CM-compositum induced.

\begin{lemma} \phantomsection \label{lemma ordinary good mult red}
Let $p \geq 5$ be a prime number with $p \equiv 2 \mod 3$ and let $K/\mathbb{Q}_p$ be a finite unramified extension.
Suppose that $E/K$ is an elliptic curve such that there is an isomorphism
$$
\overline{\varrho_{E,p}} \cong 
\begin{pmatrix}
\overline{\chi} & 0 \\
 0 & \overline{\chi}^{-1} \varepsilon^{-1}
\end{pmatrix}
$$
over $\mathbb{F}_p$, where $\overline{\chi}: G_K \to \mathbb{F}_p$ is an unramified character.
Then
\begin{enumerate}
    \item if $p \geq 11$ then either $E$ has  multiplicative reduction or ordinary good reduction
    \item if $p=5$ then either $E$ has multiplicative reduction, or the degree of a minimal totally ramified extension over which $E$ acquires good reduction is either $1$ or $4$ and $E$ acquires ordinary good reduction over this extension.
\end{enumerate}
\end{lemma}
\begin{proof}
Let $v$ be the valuation on $K$ and suppose firstly that $v(j(E)) < 0$, so that $E$ has potentially multiplicative reduction.
Then $E$ is the quadratic twist of an elliptic curve $A/K$ with split multiplicative reduction by a quadratic character $\delta: G_K \to \{\pm 1\}$.
The mod $p$ Galois representation of $A$ is, on the one hand, the twist of $\overline{\varrho_{E,p}}$ by $\delta$ and thus can be written as a direct sum of the characters $\overline{\chi} \delta$ and $\overline{\chi}^{-1} \delta \varepsilon^{-1}$.
On the other hand, since $A$ has split multiplicative reduction we have an isomorphism $$\restr{\overline{\varrho_{A,p}}}{I_K} \cong \begin{pmatrix}
1 & * \\
 0 &  \varepsilon^{-1}
\end{pmatrix}.$$
Since $\varepsilon: I_K \to \mathbb{F}_p^\times$ has order $p-1 \geq 4$, we must have that $\delta$ is unramified and that $E$ has multiplicative reduction over $K$.

Now suppose that $E$ does not have (potentially) multiplicative reduction.
The elliptic curve $E$ acquires good reduction on base change to a totally ramified extension $L/K$ of degree $e \in \{1,2,3,4,6\}$, with $E$ having bad reduction on base change to any proper subfield of $L$ containing $K$.
We can take $L = K(\varpi^{1/e})$, where $\varpi$ is any choice of uniformizer for $K$.
We firstly show that $E_L$ has ordinary good reduction.
Suppose not, so that $E_L$ has supersingular reduction.
The same argument as in the proof of \cite[Theorem 3.1 2]{lozano2013field} then shows that either the image of inertia under the dual representation $\overline{\varrho_{E,p}}^{\vee}(I_L)$ contains an element of order $p$, or is the given by matrices of the form
\begin{equation} \label{non split cartan subgroup}
 \left\{ \begin{pmatrix}
a & \eta b \\
b & a
\end{pmatrix}^e 
| (a,b) \in \mathbb{F}_p^2 \setminus \{(0,0)\} \right\}   
\end{equation}
with respect to some basis for $E[p]$ and where $\eta \in \mathbb{F}_p^\times$ is a non-quadratic residue.
By hypothesis, $\overline{\varrho_{E,p}}(I_L)$ does not contain any element of order $p$ and the only scalar matrix it contains is the identity.
On the other hand, the group of matrices \ref{non split cartan subgroup} contains a subgroup of scalar matrices of order $p-1$ (since $(p^2-1)/e \geq p-1$, as $p \geq 5$ and $e \leq 6$) which gives a contradiction.

Next we show that $e \in \{1,2,4\}$.
Again we suppose not, so that $e \in \{3,6\}$.
This can only happen when $E$ has a Weierstrass equation of the form
$E: y^2 = x^3 + ax + b$
with $3v(A) > 2v(B)$.
In that case, a Weierstrass equation for the reduction of $E$ modulo $\varpi^{1/e}$ is given by $\tilde{E_L}: y^2 = x^3 + B$ for some $B \in \Fpbar^\times$.
This defines a supersingular elliptic curve by \cite[V.4.1(a)]{silverman2009arithmetic} since the coefficient of $x^{(p-1)/2}$ in $(x^3+B)^{p-1}$ is zero, as $p \equiv 2 \mod 3$.
We have therefore reached a contradiction since $\tilde{E}_L$ was already shown to be ordinary.

Now suppose that $e \in \{2,4\}$.
Let $K' = K(\varpi^{2/e})$, so that $L/K'$ is a ramified quadratic extension.
Then $E_{K'}$ is the quadratic twist of an elliptic curve $A/K'$ with good reduction.
It follows that
$$
\restr{\overline{\varrho_{A,p}}}{I_{K'}} \cong 
\begin{pmatrix}
\delta & 0 \\
 0 & \delta \varepsilon^{-1}
\end{pmatrix}
$$
where $\delta: I_{K'} \to \{\pm 1\}$ is the quadratic character arising from $L/K'$.
Firstly suppose that $A$ has supersingular reduction.
Again, the same argument as in \cite[Theorem 3.1 2]{lozano2013field} shows that $\overline{\varrho_{A,p}}^\vee(I_{K'})$ either contains an element of order $p$ or is of the form \ref{non split cartan subgroup} where the exponent is $e_{K'/\mathbb{Q}_p} \in \{1,2\}$.
In either case, this gives an immediate contradiction.

We are left with the case that $A$ has ordinary reduction.
The only possibility is that $\delta = \varepsilon^{-1}$ as characters of $I_{K'}$, so we must have $p=5$ and $e = 4$ since $\# \varepsilon(I_{K'}) = 2(p-1)/e$.
We deduce the desired statement when $p=5$.
When $p \geq 11$, we see that $e=1$ so that $E/K$ has good reduction and $\tilde{E}$ was already shown to be ordinary.
\end{proof}

\begin{lemma} \phantomsection \label{cm comp induced implies semistable}
Let $p \geq 5$ with $p \equiv 2 \mod 3$.
Let $F$ be a totally real field in which $p$ splits completely and let $K/F$ be a quadratic extension in which $p$ is unramified. 
Suppose that $E/K$ is an elliptic curve such that $\overline{\varrho_{E,p}}$ is CM-compositum induced.
If $p=5$ and $E$ has potentially good reduction suppose further that $E$ has good reduction over $K$.
Then $E$ has either ordinary good reduction or multiplicative reduction at every place $w | p$.
\end{lemma}
\begin{proof}
This is immediate from Lemma \ref{lemma ordinary good mult red} and Proposition \ref{modular residual rep}.
\end{proof}

We can also satisfy the constraints on $\bar{\varrho}$ in our modularity lifting theorem arising from tidiness considerations when $p=5$ by imposing a condition on $\bar{\psi}$.

\begin{lemma} \phantomsection \label{tidy CM comp ind}
Suppose that $p=5$, $\bar{\varrho}: G_K \to \GL_2(\mathbb{F}_5)$ is CM-compositum induced and let $\tau \in G_F \setminus G_K$.
If $\restr{(\bar{\psi}^2)^\tau}{G_{L(\zeta_5)}} \neq \restr{\bar{\psi}^2}{G_{L(\zeta_5)}}$ then $\bar{\rho} = \Ind_{G_F}^{G_K} (\bar{\varrho} \otimes \delta)$ is tidy whenever $\delta: G_K \to \{\pm 1\}$ is chosen such that $\bar{\rho}$ is absolutely irreducible.
Moreover, if $\bar{\rho}$ is absolutely irreducible then $\bar{\rho}(G_{F(\zeta_p)}) \subset \GSp_4(\mathbb{F}_5)$ is not enormous (in the sense of \cite[Definition 7.5.2]{surfaces}).
\end{lemma}
\begin{proof}
We claim that it suffices to show the order of the image of $\restr{\bar{\rho}}{G_{K(\zeta_5)}} = \restr{\bar{\varrho}}{G_{K(\zeta_5)}} \oplus \restr{\bar{\varrho}^\tau}{G_{K(\zeta_5)}}$ is at least $32$.
Indeed, then $\bar{\rho}(G_F)$ would be order at least $32 \cdot 2 \cdot 4 = 256$, since $\bar{\rho}$ is absolutely irreducible and the similitude character surjects onto $\mathbb{F}_5^\times$.
It follows from Lemma \ref{tidy inductions} that $\bar{\rho}$ is then tidy.
Since $\bar{\varrho}(G_{K(\zeta_5)})$ is necessarily of order $8$, it suffices to show that the extensions of $L(\zeta_p)$ cut out by $\restr{\bar{\varrho}}{G_{L(\zeta_5)}}$ and $\restr{\bar{\varrho}^\tau}{G_{L(\zeta_5)}}$ are linearly disjoint.
Since $\restr{\bar{\psi}^\sigma}{G_{L(\zeta_5)}} = \restr{\bar{\psi}^{-1}}{G_{L(\zeta_5)}}$, we just need to show that the extensions of $L(\zeta_5)$ cut out by $\restr{\bar{\psi}}{G_{L(\zeta_5)}}$ and $\restr{\bar{\psi}^\tau}{G_{L(\zeta_5)}}$ are linearly disjoint.
If they were not linearly disjoint, then certainly the squares of both characters would cut out the same quadratic extension of $L(\zeta_5)$, contradicting our hypothesis.
The failure of $\bar{\rho}(G_{F(\zeta_p)})$ to be enormous follows from a computation in \texttt{magma}, together with the observation that $\bar{\rho}(G_{K(\zeta_p)}) \not\cong (Q_8)^2$.
\end{proof}

\begin{example} \label{cm comp ind construction}
Continue to fix $F$ a totally real field in which $p \geq 5$ splits completely and $K/F$ a quadratic extension in which $p$ is unramified.
We give a construction of some CM-compositum induced representations $\bar{\varrho}: G_K \to \GL_2(\mathbb{F}_p)$.

Let $F'/F$ be the compositum of $F$ with an imaginary quadratic field $N$ which is disjoint from $K$ and in which $p$ splits completely.
Set $L= K \cdot F'$ and $\Gal(L/K) = \{1,\sigma\}$.
We wish to construct characters $\bar{\psi}: G_L \to \mathbb{F}_p^\times$ as in Definition \ref{def CM-compositum induced}.
We can find $\bar{\psi}$ such that for every place $z | p$ of $L$ we have $\{\restr{\bar{\psi}}{G_{L_z}}, \restr{\bar{\psi}^\sigma}{G_{L_{z}}}\} = \{1, \varepsilon^{-1} \}$ and $\bar{\psi} \bar{\psi}^\sigma = \varepsilon^{-1}$ by taking the restriction of the mod $p$ reduction of a $p$-adic algebraic character on $G_{F'}$.
We just need $\restr{\bar{\psi}}{G_{L(\zeta_p)}} \neq \restr{\bar{\psi}^\sigma}{G_{L(\zeta_p)}}$ for $\bar{\varrho} = \Ind_{G_L}^{G_K} (\bar{\psi})$ to be CM-compositum induced, so suppose this is not the case.

Suppose that $\bar{\chi}: G_L \to \mathbb{F}_p^\times$ is unramified at every place dividing $p$ and satisfies $\bar{\chi} \bar{\chi}^\sigma = 1$.
Setting $\bar{\varphi} = \bar{\psi} \bar{\chi}$, $\bar{\varrho} = \Ind_{G_L}^{G_K} (\bar{\varphi})$ will be CM-compositum induced provided that $\restr{\bar{\varphi}}{G_{L(\zeta_p)}} \neq \restr{\bar{\varphi}^\sigma}{G_{L(\zeta_p)}}$.
If $\bar{\chi}$ (and hence $\restr{\bar{\chi}}{G_{L(\zeta_p)}}$) has image of order at least $3$ then $\restr{\bar{\varphi}}{G_{L(\zeta_p)}}$ will not be equal to its conjugate, since
$$\restr{\bar{\varphi}^{-1}\bar{\varphi}^\sigma}{G_{L(\zeta_p)}} = \restr{\bar{\psi}^{-1}\bar{\psi}^\sigma}{G_{L(\zeta_p)}} \restr{\bar{\chi}^{-1}\bar{\chi}^\sigma}{G_{L(\zeta_p)}} = \restr{\bar{\chi}^{-2}}{G_{L(\zeta_p)}} \neq 1.$$

We can obtain such a $\bar{\chi}$ by considering a dihedral extension $M/K$ which contains $L$.
If $p \equiv 1 \mod 4$, then similarly to Proposition \ref{D8 extension} we can find an extension $M/L$ for which
\begin{enumerate}
    \item $M/K$ is Galois with $\Gal(M/K) \cong D_8$ and $\Gal(M/L) \cong \mathbb{Z}/4\mathbb{Z}$
    \item $p$ is unramified in $M$.
\end{enumerate}
If $p \equiv 3 \mod 4$, choose an odd prime $l | (p-1)$.
Then by \cite[Theorem 3.3]{bar09} (applied with $k_2$ greater than the number of rational primes which ramify in $F$) we can find a Galois extension $N'/\mathbb{Q}$ containing our imaginary quadratic field $N$ such that
\begin{enumerate} [resume]
    \item $\Gal(N'/\mathbb{Q}) = D_{2l}$ and $\Gal(N'/N) \cong \mathbb{Z}/l\mathbb{Z}$
    \item $p$ is unramified in $N'$
    \item there exists a rational prime $r$ which is unramified in $K$ and whose ramification index in $N'$ divisible by $l$.
\end{enumerate}
Taking $M = N' \cdot K$, we have $\Gal(M/K) \cong D_{2l}$ with $\Gal(M/L) \cong \mathbb{Z}/l\mathbb{Z}$ and $p$ is unramified in $M$.
Then in either of the above cases let $\bar{\chi}$ be the character corresponding to $\Gal(M/L)$.

Suppose now that $p=5$ and that $\bar{\psi}: G_L \to \mathbb{F}_5^\times$ gives rise to a CM-compositum induced representation.
We show that we can find a $\bar{\varphi}: G_L \to \mathbb{F}_5^\times$ for which the corresponding CM-compositum induced representation has a quadratic twist whose induction will be tidy by Lemma \ref{tidy CM comp ind}.
We require that $\restr{(\bar{\varphi}^2)^\tau}{G_{L(\zeta_p)}} \neq \restr{\bar{\varphi}^2}{G_{L(\zeta_p)}}$ where $\tau \in G_F \setminus G_K$.
We suppose that $\restr{(\bar{\psi}^2)^\tau}{G_{L(\zeta_p)}} = \restr{\bar{\psi}^2}{G_{L(\zeta_p)}}$ else we take $\bar{\varphi} = \bar{\psi}$.
We again twist by some surjective $\bar{\chi}: G_L \to \mathbb{F}_5^\times$ which cuts out a dihedral extension $M/K$ in which $p$ is unramified.
By choosing $M/L$ to be disjoint from the extension cut out by $\bar{\psi}$, it will follow that $\bar{\varphi} = \bar{\chi} \bar{\psi}$ still gives rise to a CM-compositum induced representation.
The character $\bar{\chi}^2$ will cut out a quadratic extension $L(\sqrt{\alpha})/L$ for some $\alpha \in L^\times$.
If $\alpha$ is chosen such that $\alpha^\tau \alpha^{-1} \not\in (L(\zeta_5)^\times)^2$ then $\bar{\varphi}$ will have the desired properties and finding such an extension $M/K$ is straightforward.

\end{example}

\subsection{Modularity when $p=5$} \label{p=5 modularity subsec}

For now, suppose again that $K$ is an arbitrary number field and that $\bar{\varrho}: G_K \to \GL_2(\mathbb{F}_p)$ is a continuous representation of determinant $\varepsilon^{-1}$.
Let $V$ denote the underlying $\mathbb{F}_p$-vector space of $\bar{\varrho}$ and fix an isomorphism $\iota_0: \bigwedge^2 \bar{\varrho} \cong \mu_p^\vee$.
Let $Y_{\bar{\varrho}}$ denote the uncompactified moduli space of pairs $(E,\iota)$ where $E$ is an elliptic curve and $\iota: V \xrightarrow{\sim} E[p]^\vee$ satisfies $\bigwedge^2 \iota = \iota_0$ via the isomorphism $\bigwedge^2 E[p]^\vee \cong \mu_p^\vee$ coming from the Weil pairing.

Now suppose that $p=5$.
The compactified curve $X_{\bar{\varrho}}$ is isomorphic to $\mathbb{P}^1$ over $K$ (by the same argument as in \cite[Lemma 1.1]{shepherd1997mod}).
We use this to give a notion of a positive density of a subset of $X_{\bar{\varrho}}(K)$. 
For $v$ a place of $K$ lying above a place $w$ of $\mathbb{Q}$ define the valuation $||\cdot ||_v$ to equal $| \cdot |_v^{|K_v:\mathbb{Q}_w|}$, where $| \cdot |_v$ extends the usual absolute value on $\mathbb{Q}$ defined by $w$.
Then the absolute height of $[x_0: \ldots : x_n] \in \mathbb{P}^n(K)$ is given by
$$
H([x_0: \ldots : x_n]) = \prod_{v} \max(||x_0||_v,\ldots,||x_n||_v)^{1/|K:\mathbb{Q}|},
$$
where the product is over all places of $K$. 
We have the following lemma on the asymptotic behaviour of the number of points of bounded height.

\begin{lemma} \phantomsection \label{height asymptotics}
Let $K$ be a number field of degree $d$ over $\mathbb{Q}$ and let $\Sigma$ be a finite set of finite places of $K$.
Consider the subset $S = \{(\alpha:1) \in \mathbb{P}^1(K) | v_w(\alpha) = 0 \text{ for all } w \in \Sigma\}$ of $\mathbb{P}^1(K)$, where $v_w$ denotes the valuation corresponding to the place $w$.
Then there exists a constant $C_\Sigma > 0$ such that
$$
\# \{P \in S | H(P) \leq X\} = C_\Sigma X^{2d} + 
\begin{cases}
O(X^d \log X) &\text{ if } d=1\\
O(X^{2d-1}) &\text{ if } d>1.
\end{cases}
$$
where $H$ is the absolute height on $\mathbb{P}^1(K)$.
An explicit value for $C_\Sigma$ is given by
$$C_\Sigma =  \frac{hR}{w \zeta_K(2)} \gamma_{K,2} \prod_{w \in \Sigma} 
\frac{(N (\mathfrak{p}_w) - 1)^2}{N (\mathfrak{p}_w)^2 - 1}
$$
where
\begin{itemize}
    \item $h$ is the class number of $K$
    \item $R$ is the regulator of $K$
    \item $w$ is the number of roots of unity in $K$   
    \item $\zeta_K$ is the Dedekind zeta function of $K$
    \item $\gamma_{K,2} = \frac{(2^{r_1} (2 \pi)^{r_2})^2 2^r}{|D_K|}$ where

    \begin{itemize}
    \item  $r_1$ is the number of real embeddings and $r_2$ is the number of pairs of complex embeddings of $K$
        \item $r = r_1 + r_2 - 1$
        \item $D_K$ is the discriminant of $K$
    \end{itemize}
    \item $N(\mathfrak{p}_w)$ is the norm of the prime ideal corresponding to $w$.
\end{itemize}
\end{lemma}
\begin{proof}
The case where $\Sigma = \emptyset$ is stated in \cite[Chapter 3, Theorem 5.3]{lang2013fundamentals}.
We show how the lemma follows by the same proof of this case, which was due to Schanuel.
If $P = (x_0:x_1) \in \mathbb{P}^1(K)$ with $x_0,x_1 \in K$ then we let $\mathfrak{a}_P$ denote the fractional ideal of $K$ generated by $x_0,x_1$.
In what follows we say two points in $K^2$ are equivalent if they differ by an element of $\mathcal{O}_K^\times$.
It suffices to show that the same asymptotic, on dividing through by a factor of $h$, holds when restricting only to $P \in \mathbb{P}^1(K)$ for which $\mathfrak{a}_P$ lies in a fixed class in the ideal class group of $K$.
This is the same as fixing an ideal $\mathfrak{a}$ (which we may assume is prime to the prime ideals defined by the places in $\Sigma$) and counting $(x_0,x_1) \in \mathfrak{a}^2$ up to equivalence for which $\mathfrak{a}_P = \mathfrak{a}$, $v_w(x_0) = v_w(x_1) = 0$ and $H(x_0:x_1) \leq X$.
Now suppose we have an asymptotic formula with the same order of leading term and same error term for counting pairs $(x_0,x_1) \in \mathfrak{a}^2$ up to equivalence for which $v_w(x_0) = v_w(x_1) = 0$ and $H(x_0:x_1) \leq X$.
Then by Möbius inversion (see \cite[p73]{lang2013fundamentals} for the $\Sigma = \emptyset$ case; in our case we can restrict to sums over integral ideals coprime to $\mathfrak{p}_w$ for $w \in \Sigma$) the same asymptotic will hold for the sum of interest above, but we gain a factor of $\zeta_K(2)^{-1} \prod_{w \in \Sigma} (1 - N \mathfrak{p}_w^{-2})^{-1} $.

For $i \in \{0,1\}$, $w \in \Sigma$ and $X > 0$, let $$A_{i,w,X} = \{(x_0,x_1) \in \mathfrak{a}^2 | x_i \in \mathfrak{p}_w, H(x_0:x_1) \leq X\},$$
viewed as a subset of the set of such pairs whose height is bounded by $X$ but without any conditions at places in $\Sigma$.
For $T \subset \{0,1\} \times \Sigma$, the asymptotic behaviour of $\# \bigcap_{(i,w) \in T} A_{i,w,X}$ can be computed in exactly the same way as the $T = \emptyset$ computation as in \cite[p442]{schanuel}, since the size of this intersection is also given by counting the number of lattice points in a homogeneously expanding domain.
The formula \cite[Thm 2]{schanuel} and the $T= \emptyset$ computation then gives that
$$
\# \bigcap_{(i,w) \in T} A_{i,w,X} = \big(\prod_{(i,w) \in T} N( \mathfrak{p}_w)^{-1}\big) \frac{R \gamma_{K,2}}{w} X^{2d} +
\begin{cases}
O(X^d \log X) &\text{ if } d=1\\
O(X^{2d-1}) &\text{ if } d>1.
\end{cases}$$
A straightforward application of the inclusion-exclusion principle then yields the result.
\end{proof}

We state one more fact about $H$ that will be of repeated use.
If $f : \mathbb{P}^1_K \to \mathbb{P}^1_K$ is an isomorphism over $K$ then there exists $C_1,C_2 > 0$ depending only on $f$ such that for every $P \in \mathbb{P}^1(K)$ we have the inequalities
\begin{equation} \phantomsection \label{iso gives bounded height}
C_1 H(P) \leq H(f(P)) \leq C_2 H(P).    
\end{equation}
\begin{definition} \phantomsection \label{positive density}

Let $K$ be a number field and let $\bar{\varrho}:G_K \to \GL_2(\mathbb{F}_5)$ be a continuous representation of determinant $\varepsilon^{-1}$.
We say a subset $S \subset X_{\bar{\varrho}}(K)$ has positive density if there exists a constant $C > 0$ such that for every $X$ sufficiently large
$$
\frac{\# \{ P \in S : H(P) \leq X \} }{\# \{ P \in  X_{\bar{\varrho}}(K) : H(P) \leq X \}} \geq C
$$
where $H$ is a height function on $X_{\bar{\varrho}}$ arising from some (equivalently, any by Lemma \ref{height asymptotics} and \ref{iso gives bounded height}) isomorphism $X_{\bar{\varrho}} \cong \mathbb{P}^1$ over $K$ via the absolute height on $\mathbb{P}^1$.
\end{definition}

The following lemma provides us with a subset of positive density on this modular curve, which we use to deduce a positive density of modular elliptic curves under certain hypotheses.

\begin{lemma} \phantomsection \label{positive density subset}
Let $K$ be a number field such that each place $w|5$ is unramified and every such residue field $k(w)$ is an extension of $\mathbb{F}_5$ of degree at most $2$.
Let $\bar{\varrho}: G_K \to \GL_2(\mathbb{F}_5)$ be a representation of determinant $\varepsilon^{-1}$ such that for each $w|5$,
$$\restr{\bar{\varrho}}{G_{K_w}} \cong \psi_w \oplus \psi_w^{-1} \varepsilon^{-1}$$
with $\psi_w: G_{K_w} \to \mathbb{F}_5^\times$ unramified.
\begin{enumerate}
    \item There exists non-empty $p$-adic open subsets $\Omega_w \subset X_{\bar{\varrho}}(K_w)$ for each place $w|5$ such that if $P \in \Omega_w$ then $E_P$ is an elliptic curve with ordinary good reduction or potentially multiplicative reduction.
    \item Let $\Sigma$ be a finite set of finite places of $K$ and let $\Omega_w \subset X_{\bar{\varrho}}(K_w)$ be a non-empty $p$-adic open subset for each $w \in \Sigma$. Then $$\Omega = \{P \in X_{\bar{\varrho}}(K) | P \in \Omega_w \text{ for every } w \in \Sigma\}$$ has positive density in the sense of Definition \ref{positive density}.
\end{enumerate}
\end{lemma}

\begin{proof}
Let $w|5$ and view $\psi_w$ as a character of $G_{k(w)}$.
We can find an ordinary elliptic curve $\tilde{E}$ over $k(w)$ such that $\tilde{E}[5](\overline{k(w)})$ is isomorphic, as a $G_{k(w)}$-module, to $\mathbb{F}_5[\psi_w^{-1}]$.
Indeed, $\psi_w$ has order dividing $4$ and there are only $4$ possibilities for $\psi_w$; examples are listed in Table \ref{ordinary elliptic curves f5} below.
Here when $|k(w): \mathbb{F}_5| = 2$, $\alpha \in k(w)\setminus \mathbb{F}_{5}$ satisfies $\alpha^2 + 4\alpha + 2 = 0$.
\begin{table}[h!]
\caption{Some ordinary elliptic curves over finite fields of characteristic $5$ with specified $5$-torsion}
\phantomsection
\label{ordinary elliptic curves f5}
\begin{tabular}{@{}ccl@{}}
\toprule
$|k(w): \mathbb{F}_5|$ & $\# \psi_w(G_{k(w)})$ & $\tilde{E}$              \\ \midrule
1        & 1                     & $y^2 = x^3+3x$           \\
1        & 2                     & $y^2 = x^3 + 2x$         \\
1        & 4                     & $y^2 = x^3 \pm x$        \\
2       & 1                     & $y^2 = x^3 + 3x$         \\
2       & 2                     & $y^2 = x^3 + x$          \\
2       & 4                     & $y^2 = x^3 + \alpha x$ \\
      &                      & $y^2 = x^3 + \alpha^3 x$\\ \bottomrule
\end{tabular}
\end{table}

Now consider the Serre--Tate canonical lifting of $\tilde{E}$ to an elliptic curve $\mathcal{E}$ defined over $\mathcal{O}_{K_w}$.
This has the property that the $p$-divisible group of $\mathcal{E}$ splits as a direct sum of a connected group scheme and an étale group scheme.
Letting $E = \mathcal{E}_{K_w}$, it therefore follows that the action of $G_{K_w}$ on $E[5]^\vee$ is given by a direct sum of characters which must be $\psi_w$ and $\psi_w^{-1} \varepsilon^{-1}$.
Thus $E$ is an elliptic curve over $K_w$ with ordinary good reduction that defines a point of $X_{\bar{\varrho}}(K_w)$.

Consider the modular curve $Y_{\bar{\varrho},\overline{\varrho_{E,7}}}$ parameterising triples $(A,\iota,\iota_7)$ where $A$ is an elliptic curve, $\iota$ is a symplectic isomorphism between $A[5]^\vee$ and $\bar{\varrho}$, and $\iota_7$ is a symplectic isomorphism between $A[7]^\vee$ and $E[7]^\vee$ compatible with the action of Galois.
The natural map $Y_{\bar{\varrho},\overline{\varrho_{E,7}}} \to (Y_{\bar{\varrho}})_{K_w}$ is étale (since after base change by a separable extension the map defines a $\PGL_2(\mathbb{F}_7)$-torsor).
Hence there is an open neighbourhood in the $5$-adic topology of the point of $Y_{\bar{\varrho},\overline{\varrho_{E,7}}}(K_w)$ defined by $E$ which is mapped isomorphically onto an open subset $\Omega_w \subset Y_{\bar{\varrho}}(K_w)$ under the map $Y_{\bar{\varrho},\overline{\varrho_{E,7}}}(K_w) \to Y_{\bar{\varrho}}(K_w)$.
 Let $P \in \Omega_w$ define an elliptic curve $A_P$ and suppose that $A_P$ has potentially multiplicative reduction.
Then $A_P$ must in fact have multiplicative reduction over $K$, as the quadratic twist of $A_P$ by a ramified character cannot have the mod $5$ Galois representation of an elliptic curve with split multiplicative reduction.
Now suppose instead that $A_P$ has potentially good reduction. Then since $A_P[7]$ is unramified as a $G_{K_w}$-module, it follows by a variation of the Néron--Ogg--Shafarevich criterion that $A_P$ has good reduction and the reduction $\tilde{A}_P$ is ordinary by Lemma \ref{lemma ordinary good mult red}.
This shows the first part of the lemma.

For the second part, choose an isomorphism $f: X_{\bar{\varrho}} \xrightarrow{\sim} \mathbb{P}^1$ over $K$ such that $P \in X_{\bar{\varrho}}(K)$ lies inside each $\Omega_w$ whenever $f(P) = (x:1)$ has $v_w(x) = 0$ for every $w \in \Sigma$.
It then follows from Lemma \ref{height asymptotics} that $\Omega$ has positive density.
\end{proof}

\begin{corollary} \phantomsection \label{p small inf modular curves}
Let $F$ be a totally real field in which $5$ splits completely and let $K/F$ be a quadratic extension in which $5$ is unramified.
There exists representations $\bar{\varrho}:~G_K \to \GL_2(\mathbb{F}_5)$ which are CM-compositum induced from a character $\bar{\psi}: G_L \to \mathbb{F}_5^\times$ (see Definition \ref{def CM-compositum induced}), where $L/K$ is a quadratic extension, and $\bar{\psi}$ satisfies $\restr{(\bar{\psi}^2)^\tau}{G_{L(\zeta_p)}} \neq \restr{\bar{\psi}^2}{G_{L(\zeta_p)}}$ for $\tau \in G_F \setminus G_K$.
Fix such a representation $\bar{\varrho}$ and suppose that $E/K$ is an elliptic curve for which $\overline{\varrho_{E,5}} \cong \bar{\varrho}$ and that $E$ has ordinary good reduction or multiplicative reduction at every $w|5$.
Then $E$ is modular.
Moreover, the subset of $X_{\bar{\varrho}}(K)$ given by such elliptic curves which in addition do not have complex multiplication nor have $j$-invariant contained in a proper subfield of $K$ is of positive density (in the sense of Definition \ref{positive density}).
\end{corollary}

\begin{proof}
Such representations $\bar{\varrho}$ were constructed in Example \ref{cm comp ind construction}.
A point $P \in Y_{\bar{\varrho}}(K)$ for which $E_P$ has ordinary good reduction or multiplicative reduction at every $w|5$ is modular by (the proofs of) Theorem \ref{modularity lifting elliptic curves}, Lemma \ref{tidy CM comp ind} and Proposition \ref{modular residual rep}.

We now turn our attention to the positive density statement.
Note that only finitely many points $P \in Y_{\bar{\varrho}}(K)$ correspond to a CM elliptic curve $E_P$.
Indeed, there are only finitely many possible values of $j(E_P) \in K$ and if $Q \in Y_{\bar{\varrho}}(K)$ satisfies $j(E_P) = j(E_Q)$ then $E_Q$ is a twist of $E_P$ which must be equal to $E_P$ since $\overline{\varrho_{E_P,5
}} \cong \overline{\varrho_{E_Q,5
}}$.
Thus we can ignore the existence of CM elliptic curves in showing the density statement.

Now let $l \neq 5$ be a prime number that splits completely in $K$.
For each place $v$ of $K$ above $l$ we can find non-empty $l$-adic open subsets $\Omega_v \subset X_{\bar{\varrho}}(K_v)$ whose images under the $j$-invariant, $j(\Omega_v) \subset \mathbb{P}^1(\mathbb{Q}_l)$, are pairwise disjoint.
We can therefore find a positive density subset $\Omega \subset X_{\bar{\varrho}}(K)$ such that each point $P \in \Omega$ defines an elliptic curve with ordinary good reduction or multiplicative reduction and whose image in $X_{\bar{\varrho}}(K_v)$ lies in $\Omega_v$ for each $v|l$ by Lemma \ref{positive density subset} (which we may apply by Proposition \ref{modular residual rep}).
We are done, since the elliptic curves corresponding to points of $\Omega$ cannot have $j$-invariant contained in a proper subfield of $K$ and were already shown to be modular. (One can also show further that the set of elliptic curves arising from points of $X_{\bar{\varrho}}(K)$ whose $j$-invariant is contained inside a proper subfield of $K$ is of density zero, but we have not done so since it makes no difference for the purpose of this corollary.) 
\end{proof}

\appendix
\section{} \label{sect appendix}

In this appendix we record the output of computations carried out in the computer algebra software \texttt{magma}, as explained in Section \ref{sect subgroups}.

\FloatBarrier

\begin{table}[h!]
\caption{Conjugacy classes of absolutely $\GL_4$-irreducible subgroups $\Gamma \leq \Sp_4(\mathbb{F}_p)$ which are not $\Sp_4$-adequate}
\label{subgroups of sp4(f3)}
\begin{tabular}{cclccl}
\toprule
$p$ & $\#\Gamma$ & Label                           & $h^1(\Gamma,\splie_4)$ & $h^1(\Gamma,\mathbb{F}_p)$ & ID        \\ 
\hline
3   & 96    & $D_4.A_4$                       & 0                      & 1                          & 96,202    \\ 
3   & 96    & $\SL(2,3).C_2^2$                 & 1                      & 0                          & 96,191    \\ 
3   & 240   & $C_2.S_5$                       & 1                      & 0                          & 240,89    \\ 
3   & 384   & $Q_8^2.C_6$                     & 0                      & 1                          & 384,618   \\ 
3   & 1440  & $C_2.A_6.C_2$                   & 1                      & 0                          & 1440,4591 \\ 
3   & 1152  & $C_2^2.A_4{\rm wrC}_2$          & 0                      & 1                          &           \\ 
5   & 160   & $(C_4.C_2^3):C_5$               & 0                      & 1                          & 160, 199  \\ 
5   & 480   & $D_4.A_5$                       & 1                      & 0                          & 480,957   \\ 
5   & 480   & $(C_2\times {\rm SL}(2,5)):C_2$ & 1                      & 0                          & 480,953   \\ 
5   & 720   & $S_3\times {\rm SL}(2,5)$       & 1                      & 0                          & 720,417   \\ 
5   & 28800 & $C_2^2.A_5^2.C_2$               & 1                      & 0                          &           \\ 
\bottomrule
\end{tabular}
\end{table}

\begin{table}[h] 
\footnotesize
\caption{Conjugacy classes of subgroups $\Gamma' \leq \GSp_4(\mathbb{F}_3)$ with $\nu(\Gamma') \neq 1$ for which $\Gamma = \Gamma' \cap \Sp_4(\mathbb{F}_3)$ is absolutely $\GL_4$-irreducible}
\label{subgroups of gsp4(f3)}
\begin{tabular}{lccccccccc}

\toprule
$\#\Gamma'$ & $\#\Gamma$ & (A) & (B)  & $h^1(\Gamma,\splie_4)$ & $h^1(\Gamma,\mathbb{F}_3$) & Adequate & Tidy  & Induced & Split-induced\\ \hline
64     & 32    & TRUE & FALSE & 0 & 0 & TRUE  & TRUE  & TRUE  & TRUE  \\
64     & 32    & TRUE & FALSE & 0 & 0 & TRUE  & TRUE  & TRUE  & TRUE  \\
64     & 32    & TRUE & FALSE & 0 & 0 & TRUE  & TRUE  & TRUE  & TRUE  \\
64     & 32    & TRUE & FALSE & 0 & 0 & TRUE  & TRUE  & TRUE  & TRUE  \\
80     & 40    & TRUE & TRUE  & 0 & 0 & TRUE  & TRUE  & TRUE  & FALSE \\
128    & 64    & TRUE & FALSE & 0 & 0 & TRUE  & TRUE  & TRUE  & TRUE  \\
128    & 64    & TRUE & FALSE & 0 & 0 & TRUE  & TRUE  & TRUE  & TRUE  \\
128    & 64    & TRUE & FALSE & 0 & 0 & TRUE  & FALSE & TRUE  & TRUE  \\
128    & 64    & TRUE & FALSE & 0 & 0 & TRUE  & FALSE & TRUE  & TRUE  \\
192    & 96    & TRUE & FALSE & 1 & 0 & FALSE & TRUE  & TRUE  & FALSE \\
192    & 96    & TRUE & FALSE & 0 & 1 & FALSE & TRUE  & TRUE  & TRUE  \\
192    & 96    & TRUE & FALSE & 1 & 0 & FALSE & TRUE  & TRUE  & FALSE \\
192    & 96    & TRUE & FALSE & 0 & 1 & FALSE & TRUE  & TRUE  & FALSE \\
192    & 96    & TRUE & FALSE & 0 & 0 & TRUE  & TRUE  & TRUE  & TRUE  \\
256    & 128   & TRUE & TRUE  & 0 & 0 & TRUE  & TRUE  & TRUE  & TRUE  \\
384    & 192   & TRUE & TRUE  & 0 & 0 & TRUE  & TRUE  & TRUE  & FALSE \\
480    & 240   & TRUE & TRUE  & 1 & 0 & FALSE & TRUE  & TRUE  & FALSE \\
480    & 240   & TRUE & TRUE  & 0 & 0 & TRUE  & TRUE  & TRUE  & FALSE \\
640    & 320   & TRUE & TRUE  & 0 & 0 & TRUE  & FALSE & FALSE & FALSE \\
768    & 384   & TRUE & TRUE  & 0 & 0 & TRUE  & TRUE  & TRUE  & TRUE  \\
768    & 384   & TRUE & TRUE  & 0 & 1 & FALSE & TRUE  & TRUE  & TRUE  \\
2304   & 1152  & TRUE & TRUE  & 0 & 1 & FALSE & TRUE  & TRUE  & TRUE  \\
2880   & 1440  & TRUE & TRUE  & 1 & 0 & FALSE & TRUE  & TRUE  & FALSE \\
3840   & 1920  & TRUE & TRUE  & 0 & 0 & TRUE  & TRUE  & FALSE & FALSE \\
103680 & 51840 & TRUE & TRUE  & 0 & 0 & TRUE  & TRUE  & FALSE & FALSE \\
\bottomrule
\end{tabular}
\end{table}

\FloatBarrier

\bibliographystyle{alpha}
\bibliography{references}

\end{document}